\documentclass[a4paper,12pt,reqno]{amsart}
\usepackage[a4paper,margin=1.9cm,top=2.5cm,bottom=2.5cm,centering,vcentering]{geometry}
\usepackage{graphics}
\usepackage{graphicx}
\usepackage{pstricks}
\usepackage[colorlinks=true]{hyperref}

\allowdisplaybreaks[1]

\newcounter{pr}
\newcounter{pf}
\newcounter{sc}
\newcounter{Ylist}

\newtheorem{theorem}{Theorem}[section]
\newtheorem{corollary}[theorem]{Corollary}
\newtheorem{proposition}[theorem]{Proposition}
\newtheorem{lemma}[theorem]{Lemma}
\newtheorem{conjecture}[theorem]{Conjecture}
\numberwithin{equation}{section}

\newcommand{\ASM}{\mathrm{ASM}}
\newcommand{\VSASM}{\mathrm{VSASM}}
\newcommand{\VHSASM}{\mathrm{VHSASM}}
\newcommand{\HTSASM}{\mathrm{HTSASM}}
\newcommand{\QTSASM}{\mathrm{QTSASM}}
\newcommand{\DSASM}{\mathrm{DSASM}}
\newcommand{\ASASM}{\mathrm{ASASM}}
\newcommand{\DASASM}{\mathrm{DASASM}}
\newcommand{\TSASM}{\mathrm{TSASM}}
\newcommand{\OSASM}{\mathrm{OSASM}}
\newcommand{\HSASM}{\mathrm{HSASM}}
\newcommand{\DSPM}{\mathrm{DSPM}}
\newcommand{\SVC}{\mathrm{6V}}
\newcommand{\SVCO}{\mathrm{6VO}}
\newcommand{\odd}{{\chi_{\mathrm{odd}}}}
\newcommand{\even}{{\chi_{\mathrm{even}}}}
\renewcommand{\u}{\bar{u}}
\newcommand{\q}{\bar{q}}

\newcommand{\w}{\bar{w}}
\newcommand{\z}{\bar{z}}
\newcommand{\G}{\mathcal{T}}
\newcommand{\Z}{\widetilde{Z}}
\newcommand{\X}{\,\overline{\!X}}
\newcommand{\XO}{X^\mathrm{O}}
\newcommand{\ZO}{\widetilde{Z}^{\,\mathrm{O}}}
\newcommand{\XPM}{X^\mathrm{PM}}
\newcommand{\sigmah}{\widehat{\sigma}}
\newcommand{\Pf}{\mathop{\mathrm{Pf}}}
\newcommand{\ui}{{\color{blue}u_1}}
\newcommand{\uii}{{\color{green}u_2}}
\newcommand{\uiii}{{\color{red}u_3}}
\newcommand{\uiv}{{\color{brwn}u_4}}

\newcommand{\ubii}{\bar{\color{green}u}{\color{green}\mbox{}_2}}

\renewcommand{\ss}{\scriptstyle}
\newcommand{\fracs}[2]{\frac{\rule[-0.9ex]{0ex}{0ex}#1}{\rule{0ex}{1.45ex}#2}}

\newcommand{\fracsden}[2]{\frac{#1}{\rule{0ex}{1.45ex}#2}}
\definecolor{brwn}{rgb}{0.59,0.29,0.0}
\definecolor{orng}{cmyk}{0,0.2,0.4,0}
\definecolor{grn}{cmyk}{0.4,0,0.3,0}

\newcommand{\Wi}{\raisebox{-2mm}{\psset{unit=5.8mm}\pspicture(0.1,0)(1,1)
\rput(0.5,0.5){\psline[linewidth=0.5pt](0,-0.5)(0,0.5)\psline[linewidth=0.5pt](-0.5,0)(0.5,0)}
\rput(0.5,0.5){$\ss\bullet$}
\psdots[dotstyle=triangle*,dotscale=0.9](0.5,0.8)\psdots[dotstyle=triangle*,dotscale=0.9](0.5,0.2)
\psdots[dotstyle=triangle*,dotscale=0.9,dotangle=-90](0.8,0.5)\psdots[dotstyle=triangle*,dotscale=0.9,dotangle=-90](0.2,0.5)
\endpspicture}}
\newcommand{\Wii}{\raisebox{-2mm}{\psset{unit=5.8mm}\pspicture(0.1,0)(1,1)
\rput(0.5,0.5){\psline[linewidth=0.5pt](0,-0.5)(0,0.5)\psline[linewidth=0.5pt](-0.5,0)(0.5,0)}
\rput(0.5,0.5){$\ss\bullet$}
\psdots[dotstyle=triangle*,dotscale=0.9,dotangle=180](0.5,0.8)\psdots[dotstyle=triangle*,dotscale=0.9,dotangle=180](0.5,0.2)
\psdots[dotstyle=triangle*,dotscale=0.9,dotangle=90](0.8,0.5)\psdots[dotstyle=triangle*,dotscale=0.9,dotangle=90](0.2,0.5)
\endpspicture}}
\newcommand{\Wiii}{\raisebox{-2mm}{\psset{unit=5.8mm}\pspicture(0.1,0)(1,1)
\rput(0.5,0.5){\psline[linewidth=0.5pt](0,-0.5)(0,0.5)\psline[linewidth=0.5pt](-0.5,0)(0.5,0)}
\rput(0.5,0.5){$\ss\bullet$}
\psdots[dotstyle=triangle*,dotscale=0.9](0.5,0.8)\psdots[dotstyle=triangle*,dotscale=0.9](0.5,0.2)
\psdots[dotstyle=triangle*,dotscale=0.9,dotangle=90](0.8,0.5)\psdots[dotstyle=triangle*,dotscale=0.9,dotangle=90](0.2,0.5)
\endpspicture}}
\newcommand{\Wiv}{\raisebox{-2mm}{\psset{unit=5.8mm}\pspicture(0.1,0)(1,1)
\rput(0.5,0.5){\psline[linewidth=0.5pt](0,-0.5)(0,0.5)\psline[linewidth=0.5pt](-0.5,0)(0.5,0)}
\rput(0.5,0.5){$\ss\bullet$}
\psdots[dotstyle=triangle*,dotscale=0.9,dotangle=180](0.5,0.8)\psdots[dotstyle=triangle*,dotscale=0.9,dotangle=180](0.5,0.2)
\psdots[dotstyle=triangle*,dotscale=0.9,dotangle=-90](0.8,0.5)\psdots[dotstyle=triangle*,dotscale=0.9,dotangle=-90](0.2,0.5)
\endpspicture}}
\newcommand{\Wv}{\raisebox{-2mm}{\psset{unit=5.8mm}\pspicture(0.1,0)(1,1)
\rput(0.5,0.5){\psline[linewidth=0.5pt](0,-0.5)(0,0.5)\psline[linewidth=0.5pt](-0.5,0)(0.5,0)}
\rput(0.5,0.5){$\ss\bullet$}
\psdots[dotstyle=triangle*,dotscale=0.9](0.5,0.8)\psdots[dotstyle=triangle*,dotscale=0.9,dotangle=180](0.5,0.2)
\psdots[dotstyle=triangle*,dotscale=0.9,dotangle=-90](0.2,0.5)\psdots[dotstyle=triangle*,dotscale=0.9,dotangle=90](0.8,0.5)
\endpspicture}}
\newcommand{\Wvi}{\raisebox{-2mm}{\psset{unit=5.8mm}\pspicture(0.1,0)(1,1)
\rput(0.5,0.5){\psline[linewidth=0.5pt](0,-0.5)(0,0.5)\psline[linewidth=0.5pt](-0.5,0)(0.5,0)}
\rput(0.5,0.5){$\ss\bullet$}
\psdots[dotstyle=triangle*,dotscale=0.9,dotangle=180](0.5,0.8)\psdots[dotstyle=triangle*,dotscale=0.9](0.5,0.2)
\psdots[dotstyle=triangle*,dotscale=0.9,dotangle=90](0.2,0.5)\psdots[dotstyle=triangle*,dotscale=0.9,dotangle=-90](0.8,0.5)
\endpspicture}}
\newcommand{\WLout}{\raisebox{-0.5mm}{\psset{unit=5.8mm}\pspicture(0.4,0.4)(1.02,1.1)
\rput(0.5,0.5){\psline[linewidth=0.5pt](0.5,0)(0,0)(0,0.5)}
\rput(0.5,0.5){$\ss\bullet$}\psdots[dotstyle=triangle*,dotscale=0.9](0.5,0.8)
\psdots[dotstyle=triangle*,dotscale=0.9,dotangle=-90](0.8,0.5)
\endpspicture}}
\newcommand{\WLin}{\raisebox{-0.5mm}{\psset{unit=5.8mm}\pspicture(0.4,0.4)(1.02,1.1)
\rput(0.5,0.5){\psline[linewidth=0.5pt](0.5,0)(0,0)(0,0.5)}
\rput(0.5,0.5){$\ss\bullet$}\psdots[dotstyle=triangle*,dotscale=0.9,dotangle=180](0.5,0.8)
\psdots[dotstyle=triangle*,dotscale=0.9,dotangle=90](0.8,0.5)
\endpspicture}}
\newcommand{\WLup}{\raisebox{-0.5mm}{\psset{unit=5.8mm}\pspicture(0.4,0.4)(1.02,1.1)
\rput(0.5,0.5){\psline[linewidth=0.5pt](0.5,0)(0,0)(0,0.5)}
\rput(0.5,0.5){$\ss\bullet$}\psdots[dotstyle=triangle*,dotscale=0.9](0.5,0.8)
\psdots[dotstyle=triangle*,dotscale=0.9,dotangle=90](0.8,0.5)
\endpspicture}}
\newcommand{\WLdown}{\raisebox{-0.5mm}{\psset{unit=5.8mm}\pspicture(0.4,0.4)(1.02,1.1)
\rput(0.5,0.5){\psline[linewidth=0.5pt](0.5,0)(0,0)(0,0.5)}
\rput(0.5,0.5){$\ss\bullet$}\psdots[dotstyle=triangle*,dotscale=0.9,dotangle=180](0.5,0.8)
\psdots[dotstyle=triangle*,dotscale=0.9,dotangle=-90](0.8,0.5)
\endpspicture}}

\newcommand{\Vi}{\raisebox{-3.2mm}{\pspicture(-0.1,-0.1)(1.1,1.1)
\rput(0.5,0.5){\psline[linewidth=0.5pt](0,-0.5)(0,0.5)\psline[linewidth=0.5pt](-0.5,0)(0.5,0)}
\rput(0.5,0.5){$\ss\bullet$}
\psdots[dotstyle=triangle*,dotscale=1](0.5,0.8)\psdots[dotstyle=triangle*,dotscale=1](0.5,0.2)
\psdots[dotstyle=triangle*,dotscale=1,dotangle=-90](0.8,0.5)\psdots[dotstyle=triangle*,dotscale=1,dotangle=-90](0.2,0.5)\endpspicture}}
\newcommand{\Vii}{\raisebox{-3.2mm}{\pspicture(-0.1,-0.1)(1.1,1.1)
\rput(0.5,0.5){\psline[linewidth=0.5pt](0,-0.5)(0,0.5)\psline[linewidth=0.5pt](-0.5,0)(0.5,0)}
\rput(0.5,0.5){$\ss\bullet$}
\psdots[dotstyle=triangle*,dotscale=1,dotangle=180](0.5,0.8)\psdots[dotstyle=triangle*,dotscale=1,dotangle=180](0.5,0.2)
\psdots[dotstyle=triangle*,dotscale=1,dotangle=90](0.8,0.5)\psdots[dotstyle=triangle*,dotscale=1,dotangle=90](0.2,0.5)\endpspicture}}
\newcommand{\Viii}{\raisebox{-3.2mm}{\pspicture(-0.1,-0.1)(1.1,1.1)
\rput(0.5,0.5){\psline[linewidth=0.5pt](0,-0.5)(0,0.5)\psline[linewidth=0.5pt](-0.5,0)(0.5,0)}
\rput(0.5,0.5){$\ss\bullet$}
\psdots[dotstyle=triangle*,dotscale=1](0.5,0.8)\psdots[dotstyle=triangle*,dotscale=1](0.5,0.2)
\psdots[dotstyle=triangle*,dotscale=1,dotangle=90](0.8,0.5)\psdots[dotstyle=triangle*,dotscale=1,dotangle=90](0.2,0.5)\endpspicture}}
\newcommand{\Viv}{\raisebox{-3.2mm}{\pspicture(-0.1,-0.1)(1.1,1.1)
\rput(0.5,0.5){\psline[linewidth=0.5pt](0,-0.5)(0,0.5)\psline[linewidth=0.5pt](-0.5,0)(0.5,0)}
\rput(0.5,0.5){$\ss\bullet$}
\psdots[dotstyle=triangle*,dotscale=1,dotangle=180](0.5,0.8)\psdots[dotstyle=triangle*,dotscale=1,dotangle=180](0.5,0.2)
\psdots[dotstyle=triangle*,dotscale=1,dotangle=-90](0.8,0.5)\psdots[dotstyle=triangle*,dotscale=1,dotangle=-90](0.2,0.5)\endpspicture}}
\newcommand{\Vv}{\raisebox{-3.2mm}{\pspicture(-0.1,-0.1)(1.1,1.1)
\rput(0.5,0.5){\psline[linewidth=0.5pt](0,-0.5)(0,0.5)\psline[linewidth=0.5pt](-0.5,0)(0.5,0)}
\rput(0.5,0.5){$\ss\bullet$}
\psdots[dotstyle=triangle*,dotscale=1](0.5,0.8)\psdots[dotstyle=triangle*,dotscale=1,dotangle=180](0.5,0.2)
\psdots[dotstyle=triangle*,dotscale=1,dotangle=-90](0.2,0.5)\psdots[dotstyle=triangle*,dotscale=1,dotangle=90](0.8,0.5)\endpspicture}}
\newcommand{\Vvi}{\raisebox{-3.2mm}{\pspicture(-0.1,-0.1)(1.1,1.1)
\rput(0.5,0.5){\psline[linewidth=0.5pt](0,-0.5)(0,0.5)\psline[linewidth=0.5pt](-0.5,0)(0.5,0)}
\rput(0.5,0.5){$\ss\bullet$}
\psdots[dotstyle=triangle*,dotscale=1,dotangle=180](0.5,0.8)\psdots[dotstyle=triangle*,dotscale=1](0.5,0.2)
\psdots[dotstyle=triangle*,dotscale=1,dotangle=90](0.2,0.5)\psdots[dotstyle=triangle*,dotscale=1,dotangle=-90](0.8,0.5)\endpspicture}}
\newcommand{\Lout}{\raisebox{0.3mm}{\pspicture(0.4,0.4)(1.1,1.1)
\rput(0.5,0.5){\psline[linewidth=0.5pt](0,0)(0,0.5)\psline[linewidth=0.5pt](0,0)(0.5,0)}
\rput(0.5,0.5){$\ss\bullet$}\psdots[dotstyle=triangle*,dotscale=1](0.5,0.8)
\psdots[dotstyle=triangle*,dotscale=1,dotangle=-90](0.8,0.5)\endpspicture}}
\newcommand{\Lin}{\raisebox{0.3mm}{\pspicture(0.4,0.4)(1.1,1.1)
\rput(0.5,0.5){\psline[linewidth=0.5pt](0,0)(0,0.5)\psline[linewidth=0.5pt](0,0)(0.5,0)}
\rput(0.5,0.5){$\ss\bullet$}\psdots[dotstyle=triangle*,dotscale=1,dotangle=180](0.5,0.8)
\psdots[dotstyle=triangle*,dotscale=1,dotangle=90](0.8,0.5)\endpspicture}}
\newcommand{\Lup}{\raisebox{0.3mm}{\pspicture(0.4,0.4)(1.1,1.1)
\rput(0.5,0.5){\psline[linewidth=0.5pt](0.5,0)(0,0)(0,0.5)}
\rput(0.5,0.5){$\ss\bullet$}\psdots[dotstyle=triangle*,dotscale=1](0.5,0.8)
\psdots[dotstyle=triangle*,dotscale=1,dotangle=90](0.8,0.5)\endpspicture}}
\newcommand{\Ldown}{\raisebox{0.3mm}{\pspicture(0.4,0.4)(1.1,1.1)
\rput(0.5,0.5){\psline[linewidth=0.5pt](0,0)(0,0.5)\psline[linewidth=0.5pt](0,0)(0.5,0)}
\rput(0.5,0.5){$\ss\bullet$}\psdots[dotstyle=triangle*,dotscale=1,dotangle=180](0.5,0.8)
\psdots[dotstyle=triangle*,dotscale=1,dotangle=-90](0.8,0.5)\endpspicture}}
\newcommand{\R}{\raisebox{0.3mm}{\pspicture(-0.1,0.4)(0.6,1.1)
\psline[linewidth=0.5pt](0,0.5)(0.5,0.5)\rput(0.5,0.5){$\ss\bullet$}
\psdots[dotstyle=triangle*,dotscale=1,dotangle=90](0.2,0.5)\endpspicture}}
\newcommand{\T}{\raisebox{-0.5mm}{\pspicture(0.4,-0.1)(0.6,0.6)
\psline[linewidth=0.5pt](0.5,0.5)(0.5,0)\rput(0.5,0.5){$\ss\bullet$}
\psdots[dotstyle=triangle*,dotscale=1](0.5,0.22)\endpspicture}}
\newcommand{\NEtri}{\raisebox{-0.5mm}{\psset{unit=5.8mm}\pspicture(0.45,0.45)(1.12,1.05)
\rput(0.5,0.5){\pspolygon[linewidth=0.3pt](0.5,0)(0.5,0.5)(0,0.5)}\endpspicture}}
\newcommand{\SWtri}{\raisebox{-0.5mm}{\psset{unit=5.8mm}\pspicture(0.45,0.45)(1.05,1.05)
\rput(0.5,0.5){\pspolygon[linewidth=0.3pt](0.5,0)(0,0)(0,0.5)}\endpspicture}}
\newcommand{\NWtri}{\raisebox{-0.5mm}{\psset{unit=5.8mm}\pspicture(0.45,0.45)(1.05,1.05)
\rput(0.5,0.5){\pspolygon[linewidth=0.3pt](0.5,0.5)(0,0.5)(0,0)}\endpspicture}}
\newcommand{\SEtri}{\raisebox{-0.5mm}{\psset{unit=5.8mm}\pspicture(0.45,0.45)(1.12,1.05)
\rput(0.5,0.5){\pspolygon[linewidth=0.3pt](0.5,0)(0.5,0.5)(0,0)}\endpspicture}}
\newcommand{\Sq}{\raisebox{-0.5mm}{\psset{unit=5.8mm}\pspicture(0.4,0.45)(1.12,1.05)
\rput(0.5,0.5){\pspolygon[linewidth=0.3pt](0,0)(0,0.5)(0.5,0.5)(0.5,0)}
\rput(0.5,0.5){\psline[linewidth=0.3pt](0.25,0)(0.25,0.5)}\endpspicture}}

\title[DSASMs]{Diagonally symmetric alternating sign matrices}
\author[R.~E.~Behrend]{Roger E.~Behrend}
\address{R.~E.~Behrend, School of Mathematics, Cardiff University, Cardiff,
CF24 4AG, UK}
\email{behrendr@cardiff.ac.uk}
\author[I.~Fischer]{Ilse Fischer}
\address{I.~Fischer, Faculty of Mathematics, University of Vienna,
Oskar-Morgenstern-Platz 1, 1090 Vienna, Austria}
\email{ilse.fischer@univie.ac.at}
\author[C.~Koutschan]{Christoph Koutschan}
\address{C.~Koutschan, Johann Radon Institute for Computational and Applied
Mathematics (RICAM),
Altenberger Strasse 69, 4040 Linz, Austria}
\email{christoph.koutschan@ricam.oeaw.ac.at}
\begin{document}
\begin{abstract}
The enumeration of diagonally symmetric alternating sign matrices (DSASMs) is studied,
and a Pfaffian formula is obtained for the number of DSASMs of any fixed size,
where the entries for the Pfaffian are positive integers
given by simple binomial coefficient expressions.
This result provides the first known case of an exact enumeration formula for an alternating sign matrix symmetry class in which a simple product formula does not seem to exist.
Pfaffian formulae are also obtained for DSASM generating functions associated with several natural statistics,
including the number of nonzero strictly upper triangular entries in a DSASM,
the number of nonzero diagonal entries in a DSASM, and the column number of the unique~$1$ in the first row of a DSASM.
The proofs of these results involve introducing a version of the six-vertex model whose configurations are in
bijection with DSASMs of fixed size, and obtaining a Pfaffian expression for its partition function.
Various related topics are also studied, involving diagonally symmetric permutation matrices,
off-diagonally symmetric alternating sign matrices, certain natural involutions on DSASMs,
and the asymptotic enumeration of DSASMs and other classes of alternating sign matrices.
\end{abstract}
\maketitle

\section{Introduction}\label{intro}
An alternating sign matrix (ASM) is a square matrix in which each entry is~$0$, $1$ or~$-1$,
and along each row and column the nonzero entries alternate in sign and have a sum of~$1$.
Some simple observations are that the first and last row and column of any ASM each contain a single~$1$, with all of their other entries being $0$'s,
that any permutation matrix is an ASM,
and that the symmetry group of the square has a natural action on the set of ASMs.

ASMs were first defined in the early 1980s by Mills, Robbins and Rumsey~\cite{MilRobRum82},
who also conjectured~\cite[Conj.~1]{MilRobRum82,MilRobRum83} that,
for any positive integer~$n$, the number of $n\times n$ ASMs is
\begin{equation}\label{numASM}
\prod_{i=0}^{n-1}\frac{(3i+1)!}{(n+i)!}.\end{equation}
This conjecture was first proved by Zeilberger~\cite{Zei96a}, and further proofs involving other
methods were subsequently obtained by Kuperberg~\cite{Kup96}, Fischer~\cite{Fis06,Fis07,Fis16}, and Fischer and Konvalinka~\cite{FisKon20a,FisKon20b,FisKon22}.

A diagonally symmetric alternating sign matrix (DSASM) is simply an ASM which
is invariant under reflection in its main diagonal, i.e., under matrix transposition.  An example is
\begin{equation}\label{DSASMexample}
\begin{pmatrix}0&0&0&1&0&0&0\\
0&1&0&-1&1&0&0\\
0&0&1&0&-1&0&1\\
1&-1&0&0&1&0&0\\
0&1&-1&1&-1&1&0\\
0&0&0&0&1&0&0\\
0&0&1&0&0&0&0\end{pmatrix}.\end{equation}

One of the main results of this paper (see Corollary~\ref{numDSASMcoroll})
is that, for any positive integer $n$, the number of $n\times n$ DSASMs is
\begin{equation}\label{numDSASM1}
\Pf_{\odd(n)\le i<j\le n-1}\Biggl((j-i)\,\sum_{k=0}^i\frac{3-\delta_{k,0}}{i+j-2k}\,\binom{i+j-2k}{i-k}\Biggr),\end{equation}
where $\Pf$ denotes the Pfaffian of a triangular array
(see~\eqref{Pfdef2}), $\odd(n)$ is 1 or 0 according to whether $n$ is odd or even, and $\delta$ is the Kronecker delta.
It can be checked that the entries for the Pfaffian in~\eqref{numDSASM1} are positive integers.

It seems to be difficult to convert this Pfaffian to a simpler expression.  Furthermore, by
computing values of~\eqref{numDSASM1} for specific values of $n$, and examining their prime factorizations,
it can be conjectured that the prime factors are not bounded above by a polynomial in~$n$,
and hence that a simple product formula similar to that of~\eqref{numASM} does not exist.
(By contrast, it can be seen that the prime factors of~\eqref{numASM} are bounded above by $3n-2$.)
Nevertheless,~\eqref{numDSASM1} provides the first known enumeration formula for DSASMs,
and with the aid of a computer it can be used to calculate the number of $n\times n$
DSASMs for relatively large~$n$.  This has been done for all~$n$ up to~$1000$,
and the values are available at a webpage accompanying this paper~\cite{BehFisKou23b}.  The values for~$n$ up to~$131$ are also
available at The On-Line Encyclopedia of Integer Sequences~\cite{OEIS-A005163}.  To provide an idea of the magnitude
of these values, the number of $1000\times1000$ DSASMs has 56930 decimal digits.

Other primary results in this paper include Pfaffian formulae for DSASM generating functions associated with various natural statistics.
See Theorems~\ref{Xrstheorem} and~\ref{Xrsttheorem} for formulae for a three-statistic generating function, where
for any DSASM, these statistics are the number of nonzero strictly upper triangular entries,
the number of nonzero diagonal entries, and the column number of the~$1$ in the first row.
See Theorems~\ref{Xprstheorem} and~\ref{Xprsttheorem} for generalizations of these formulae
in which a further statistic (specifically, the number of strictly upper triangular inversions in a DSASM,
as defined in Section~\ref{FurthStatSect}) is included,
and in which the number of~$1$'s and number of~$-1$'s on the main diagonal of a DSASM are taken to be two separate statistics.
Note that determinant formulae with a form similar to that of the Pfaffian formulae of
Theorems~\ref{Xrstheorem},~\ref{Xrsttheorem},~\ref{Xprstheorem} and~\ref{Xprsttheorem} are
known for certain ASM generating functions.  See~\eqref{ASMX} for an example.

The method of ASM enumeration introduced and used by Kuperberg~\cite{Kup96,Kup02} involves the statistical mechanical six-vertex model.
(For important subsequent variations of this method, see, for example, Colomo and Pronko~\cite{ColPro05a,ColPro06},
Okada~\cite{Oka06}, Razumov and Stroganov~\cite{RazStr04b,RazStr09},
and Stroganov~\cite{Str06}.)
A similar method is applied to DSASM enumeration in this paper, and an outline of the main steps which are used is as follows.
\begin{list}{$\bullet$}{\setlength{\labelwidth}{4mm}\setlength{\leftmargin}{8mm}\setlength{\labelsep}{3mm}\setlength{\topsep}{0.9mm}}
\item The three-statistic generating function for DSASMs is defined in Section~\ref{Xsect}.
\item A version of the six-vertex model is introduced in Sections~\ref{sixvertexmodelconfig}--\ref{partitionfun}.
The configurations of this version are in bijection with DSASMs of fixed size,
and consist of certain orientations of the edges of a grid graph on a triangle.
The associated partition function consists of a sum, over all such configurations, of products of
parameterized bulk and boundary weights, as given in Table~\ref{weights}, where these weights satisfy the Yang--Baxter and reflection equations.
\item An expression for the partition function as a Pfaffian multiplied by a prefactor is obtained in Theorem~\ref{ZPftheorem}, which can
be regarded as a further primary result of this paper.
This expression is an analogue, for the version of the six-vertex model on a triangle studied in this paper,
of the Izergin--Korepin determinant formula (see Izergin~\cite[Eq.~(5)]{Ize87},
Izergin, Coker and Korepin~\cite[Eqs.~(5.1) \& (5.2)]{IzeCokKor92},
Kirillov and Smirnov~\cite{KirSmi89}, and Korepin~\cite{Kor82})
for the partition function of the six-vertex model on a square with domain-wall boundary conditions.
Accordingly, Theorem~\ref{ZPftheorem} plays a role in the proof of the DSASM enumeration formula~\eqref{numDSASM1} which is
analogous to that played by the Izergin--Korepin formula in Kuperberg's proof~\cite{Kup96} of the ASM enumeration formula~\eqref{numASM},
since the configurations of the six-vertex model on a square with domain-wall boundary conditions are in bijection with ASMs of fixed size.
Theorem~\ref{ZPftheorem} was also recently obtained independently by Garbali, de Gier, Mead and Wheeler~\cite[Thm.~1.5, Thm.~4.8 \& Rem.~4.9]{GarDegMeaWhe25},~\cite[Thm.~2.28 \& Rem.~2.29]{Mea25}.
\item Theorem~\ref{ZPftheorem} is proved by identifying, in Section~\ref{GenProp}, particular properties which
uniquely characterize the partition function, and then, in Section~\ref{Pfaffianexpsect}, showing that these properties are also satisfied by the Pfaffian expression.
The properties include reduction to a lower order partition function at certain values of the main parameters,
and symmetry with respect to these parameters, where the proof of this symmetry uses the Yang--Baxter and reflection equations.
\item The exact relationship between the DSASM generating function
and the six-vertex model partition function is identified in Lemma~\ref{ZXlemma}.
\item The application of Lemma~\ref{ZXlemma} requires some care, since in the expression provided
by Theorem~\ref{ZPftheorem}, both the Pfaffian and the denominator of its prefactor become zero when
the values of the parameters required in Lemma~\ref{ZXlemma} are used.  Hence, in Theorem~\ref{Pfaffianidtheorem}, a result
is obtained for evaluating (without causing division of zero by zero) general Pfaffian expressions which have the same overall form as that of the expression given
by applying Lemma~\ref{ZXlemma} to Theorem~\ref{ZPftheorem}.
\item In Sections~\ref{proofXrs}--\ref{proofXrsttheorem}, Pfaffian expressions for the DSASM generating function are obtained
by using Theorem~\ref{ZPftheorem}, Lemma~\ref{ZXlemma} and Theorem~\ref{Pfaffianidtheorem}.  The entries for these Pfaffians
initially appear as coefficients of power series, and are then evaluated as explicit binomial coefficient expressions.\end{list}

In addition to the results which have already been described, and are obtained using mostly algebraic methods,
this paper also provides some further results which are obtained using elementary combinatorial methods.  These results include
Proposition~\ref{refprop}, in which relations are given for the numbers of $n\times n$ DSASMs with the~$1$ in the first row in column~$1$,~$2$,~$3$ or~$n$,
Proposition~\ref{diagrefprop}, in which the number of $n\times n$ DSASMs is expressed as a weighted sum over all $(n-1)\times(n-1)$ DSASMs,
and Propositions~\ref{invprop} and~\ref{invinvariantprop}, in which properties are identified for certain natural involutions on DSASMs.
Some other aspects of DSASMs have been studied using combinatorial methods by Brualdi and Kim~\cite{BruKim14}, and Rubey~\cite{Rub21}.

This paper also includes applications of some of the results to two special cases of DSASMs, for which
connections to previously-known results are identified.
In Section~\ref{DSPMsect}, DSASMs with no~$-1$'s, i.e., diagonally symmetric permutation matrices, are considered,
and in Section~\ref{OSASMsect},
off-diagonally symmetric ASMs (OSASMs) are considered.  An even-order
OSASM, as defined by Kuperberg~\cite[p.~839]{Kup02}, is an
even-order DSASM in which each diagonal entry is zero.
Note that it is impossible for each diagonal entry of an odd-order DSASM to be zero,
since if each diagonal entry of an $n\times n$ DSASM is zero, then the sum of all entries
is twice the sum of all strictly upper triangular entries, and this equals~$n$ (as the sum of entries in each row of any ASM is~1),
so that~$n$ is even.  Hence, an odd-order OSASM is defined to be an odd-order DSASM in which exactly one diagonal entry is nonzero.
It seems that odd-order OSASMs have not previously been defined or studied in the literature.
With this definition, $n\times n$ OSASMs exist for any~$n$, an example being the $n\times n$ matrix with~$1$'s on the main
antidiagonal and $0$'s elsewhere. Accordingly, OSASMs can be regarded as DSASMs with maximally-many~$0$'s on the main diagonal.
A product formula for the number of even-order OSASMs was obtained by Kuperberg~\cite[Thm.~2, second eq.~\& Thm.~5, first eq.]{Kup02},
and a product formula for the number of odd-order OSASMs is given in~\eqref{numOSASModd}, where this
was conjectured in the first arXiv version of this paper~\cite[Conj.~15]{BehFisKou23a} and subsequently
proved by Kumari~\cite[Cor.~4.4]{Kum26}.

Finally, this paper provides, in Section~\ref{AsymptSect}, a conjecture and several results for the large~$n$ asymptotics of the number of $n\times n$ DSASMs
and the number of $n\times n$ ASMs in certain other natural classes.  For example,
an asymptotic expansion of the number of DSASMs is conjectured in Conjecture~\ref{DSASMAsymptConj},
a general result for the leading term in the asymptotics of the number of ASMs invariant under the action of any subgroup of the symmetry group of the square is
obtained in Theorem~\ref{ASMLeadAsymptTh},
and the leading term in the asymptotics of the number of DSASMs whose~$1$
in the first row is in a fixed position is obtained in Proposition~\ref{DSASMRefLeadAsymptTh}.
Conjecture~\ref{DSASMAsymptConj} is based on data obtained used the Pfaffian formula~\eqref{numDSASM1},
but most of the results of Section~\ref{AsymptSect} are obtained independently of results obtained in previous sections of the paper.

An outline of the remaining sections of this paper is as follows.
\begin{list}{$\bullet$}{\setlength{\labelwidth}{3mm}\setlength{\leftmargin}{5mm}\setlength{\labelsep}{2mm}\setlength{\topsep}{0.9mm}}
\item In Section~\ref{ASMSymmClassSect}, a review is provided of the programme of
obtaining exact enumeration formulae for ASMs invariant under the action of subgroups of the symmetry group of the square.
\item In Section~\ref{PrelimSect}, the main notation and conventions of the paper are outlined,
the three-statistic DSASM generating function which will be studied in much of the paper is introduced,
and some enumerative results for DSASMs whose~$1$ in the first row is in column~$1$,~$2$,~$3$ or~$n$
are obtained using simple combinatorial methods in Proposition~\ref{refprop}.
\item In Section~\ref{mainresults}, the main results for the three-statistic
DSASM generating function and the number of DSASMs are stated, in Theorem~\ref{Xrstheorem}, Corollary~\ref{numDSASMcoroll} and Theorem~\ref{Xrsttheorem},
and alternative versions of some of these results are stated in~\eqref{XrsReform}--\eqref{numDSASMReform} and~\eqref{XrstDiv}.
\item In Section~\ref{sixvertexmodelDSASMs}, the case of the six-vertex model which is related to DSASMs is introduced.
For this version of the six-vertex model,
a simple bijection between its configurations and DSASMs is identified in~\eqref{bij},
an elementary enumerative result is obtained in Proposition~\ref{diagrefprop},
some important properties of its partition function are identified in Propositions~\ref{ZLaurent}--\ref{Zsymm},
and a Pfaffian expression for the partition function is obtained in Theorem~\ref{ZPftheorem}.
\item In Section~\ref{proofs}, the results of Section~\ref{mainresults}
are proved,
where this involves obtaining a result for the connection between the three-statistic DSASM generating function and the six-vertex model
partition function in Lemma~\ref{ZXlemma}, obtaining general identities for certain Pfaffian expressions in Theorem~\ref{Pfaffianidtheorem} and Corollary~\ref{Pfaffianidcoroll},
applying results (in particular, the bijection of~\eqref{bij} and Theorem~\ref{ZPftheorem}) from Section~\ref{sixvertexmodelDSASMs},
and using various standard algebraic methods.
\item In Section~\ref{DSPMsect}, diagonally symmetric permutation matrices are studied,
and a Pfaffian formula for an associated two-statistic generating function is obtained in~\eqref{DSPMX3}.
\item In Section~\ref{OSASMsect}, OSASMs are studied, which leads to Pfaffian formulae for a two-statistic OSASM generating function
in~\eqref{OSASMX2} and~\eqref{OSASMX2t}.  Also, a symmetry property satisfied by even-order OSASMs is identified in Proposition~\ref{evenOSASMsymmprop},
a connection between the partition function for odd-order OSASMs and symplectic characters is identified in~\eqref{symplodd},
and a product formula for the number of odd-order OSASMs is obtained in~\eqref{numOSASModd},
where these three results were conjectured in the first arXiv version of this paper~\cite[Conjs.~(15)--(17)]{BehFisKou23a} and then proved by Kumari~\cite[Thm.~4.2, Cor.~4.4 \& Thm.~5.1]{Kum26}.
\item In Section~\ref{invol}, certain natural involutions on DSASMs are studied,
and results involving some of these involutions are obtained using elementary combinatorial methods in Propositions~\ref{invprop} and~\ref{invinvariantprop}.
\item In Section~\ref{GeneralGenFuncSect}, a five-statistic generalization of the previous three-statistic generating function is introduced and studied,
and results for certain cases of this generating function are stated in Theorems~\ref{Xprstheorem} and~\ref{Xprsttheorem}.
\item In Section~\ref{AsymptSect}, the asymptotic enumeration of DSASMs and other classes of ASMs is studied, which leads to
an asymptotic expansion of the number of OSASMs in~\eqref{OSASMAsympt},
a conjectured expression for the asymptotic expansion of the number of DSASMs in Conjecture~\ref{DSASMAsymptConj},
a general result for the leading term in the asymptotics of the number of ASMs invariant under the action of any subgroup of the symmetry group of the square in Theorem~\ref{ASMLeadAsymptTh},
a result for the leading term in the asymptotics of the number of DSASMs whose~$1$ in the first row is in a fixed position in Proposition~\ref{DSASMRefLeadAsymptTh},
and the identification of certain generalizations of the results for leading asymptotic terms in~\eqref{GenLeadAsympt1}--\eqref{GenLeadAsympt3}.
\end{list}

\section{ASM symmetry class enumeration}\label{ASMSymmClassSect}
The results of this paper for the straight and refined enumeration of DSASMs contribute to the general programme of
obtaining exact enumeration formulae for ASMs invariant under the action of subgroups of the symmetry group of the square,
since DSASMs are ASMs which are invariant under the action of the subgroup consisting of the identity
and reflection in a diagonal axis. Hence, in order to provide some further context, as well as some terminology and known results which will be used in certain
parts of the paper, a review of this programme is given in this section.

The symmetry group of the square is the dihedral group $D_4\!=\!
\{\mathcal{I},\mathcal{V},\mathcal{H},\mathcal{D},\mathcal{A},\mathcal{R}_{\pi/2},\mathcal{R}_{\pi},\mathcal{R}_{-\pi/2}\}$,
where $\mathcal{I}$ denotes the identity, $\mathcal{V}$, $\mathcal{H}$, $\mathcal{D}$ and $\mathcal{A}$ denote
reflections in vertical, horizontal, diagonal and antidiagonal axes, respectively, and $\mathcal{R}_\theta$
denotes counterclockwise rotation by~$\theta$.
This group has a natural action on the set of $n\times n$ ASMs, in which, for any such ASM~$A$,
$(\mathcal{I}A)_{ij}=A_{ij}$, $(\mathcal{V}A)_{ij}=A_{i,n+1-j}$, $(\mathcal{H}A)_{ij}=A_{n+1-i,j}$, $(\mathcal{D}A)_{ij}=A_{ji}$,
$(\mathcal{A}A)_{ij}=A_{n+1-j,n+1-i}$, $(\mathcal{R}_{\pi/2}A)_{ij}=A_{j,n+1-i}$, $(\mathcal{R}_{\pi}A)_{ij}=A_{n+1-i,n+1-j}$
and $(\mathcal{R}_{-\pi/2}A)_{ij}=A_{n+1-j,i}$.
The group has ten subgroups, $\{\mathcal{I}\}$, $\{\mathcal{I},\mathcal{V}\}\approx\{\mathcal{I},\mathcal{H}\}$,
$\{\mathcal{I},\mathcal{V},\mathcal{H},\mathcal{R}_{\pi}\}$,
$\{\mathcal{I},\mathcal{R}_{\pi}\}$, $\{\mathcal{I},\mathcal{R}_{\pi/2},\mathcal{R}_{\pi},\mathcal{R}_{-\pi/2}\}$,
$\{\mathcal{I},\mathcal{D}\}\approx\{\mathcal{I},\mathcal{A}\}$,
$\{\mathcal{I},\mathcal{D},\mathcal{A},\mathcal{R}_{\pi}\}$
and~$D_4$, where $\approx$ denotes conjugacy of subgroups.

The programme of studying numbers of ASMs of fixed size which are invariant under the action of a subgroup of~$D_4$ was suggested by Richard Stanley (see Robbins~\cite[p.~18]{Rob91a},~\cite[p.~2]{Rob00}).
Specifically, for a subgroup~$G$ of~$D_4$, let $\ASM^G(n)$ denote the set of $n\times n$ ASMs which are invariant under the action of all elements of~$G$.
Then a primary objective of the programme is to obtain an expression for $|\ASM^G(n)|$, for each positive integer~$n$ and each subgroup~$G$ of~$D_4$.

The number of invariant ASMs is the same for conjugate subgroups (i.e., $|\ASM^{G}(n)|=|\ASM^{G'}(n)|$ if $G\approx G'$), which leads to
eight inequivalent ASM symmetry classes.  The names and acronyms which will be used for these are as follows.
\begin{list}{$\bullet$}{\setlength{\topsep}{0.8mm}\setlength{\labelwidth}{8mm}\setlength{\labelsep}{2.5mm}\setlength{\leftmargin}{8.2mm}}
\item Unrestricted ASMs (ASMs): $\ASM^{\{\mathcal{I}\}}(n)$.
\item Vertically symmetric ASMs (VSASMs): $\ASM^{\{\mathcal{I},\mathcal{V}\}}(n)$.\\
Or, by conjugacy, horizontally symmetric ASMs (HSASMs): $\ASM^{\{\mathcal{I},\mathcal{H}\}}(n)$.
\item Vertically and horizontally symmetric (VHSASMs): $\ASM^{\{\mathcal{I},\mathcal{V},\mathcal{H},\mathcal{R}_{\pi}\}}(n)$.
\item Half-turn symmetric ASMs (HTSASMs): $\ASM^{\{\mathcal{I},\mathcal{R}_{\pi}\}}(n)$.
\item Quarter-turn symmetric ASMs (QTSASMs): $\ASM^{\{\mathcal{I},\mathcal{R}_{\pi/2},\mathcal{R}_{\pi},\mathcal{R}_{-\pi/2}\}}(n)$.
\item Diagonally symmetric ASMs (DSASMs): $\ASM^{\{\mathcal{I},\mathcal{D}\}}(n)$.\\
Or, by conjugacy, antidiagonally symmetric ASMs (ASASMs): $\ASM^{\{\mathcal{I},\mathcal{A}\}}(n)$.
\item Diagonally and antidiagonally symmetric ASMs (DASASMs): $\ASM^{\{\mathcal{I},\mathcal{D},\mathcal{A},\mathcal{R}_{\pi}\}}(n)$.
\item Totally symmetric ASMs (TSASMs): $\ASM^{D_4}(n)$.
\end{list}
For a symmetry class XASM, the set $\ASM^{G}(n)$ will also be denoted as $\mathrm{XASM}(n)$.

The current state of knowledge of the number of $n\times n$ ASMs in each of the symmetry classes can be categorized as follows.
\begin{list}{(\roman{sc})}{\usecounter{sc}\setlength{\labelwidth}{8.2mm}\setlength{\leftmargin}{8.2mm}\setlength{\labelsep}{1.5mm}\setlength{\topsep}{0.9mm}}
\item \textit{The set is empty.}  This occurs for VSASMs with~$n$ even, VHSASMs with~$n$ even, TSASMs with~$n$ even,
and QTSASMs with $n\equiv2\pmod{4}$.  This can be verified using simple combinatorial arguments, as follows.
If $\VSASM(n)$ is nonempty, then the first row of any $A\in\VSASM(n)$ is invariant under reversal of the order of the entries,
and (as is the case for the first row of any $n\times n$ ASM) the entries consist of a single~$1$ and $n-1$ $0$'s, which implies that $n$ is odd.
Therefore, $\VSASM(n)$ and the subsets $\VHSASM(n)$ and $\TSASM(n)$ of $\VSASM(n)$ are empty for~$n$ even.
If $n$ is even and $\QTSASM(n)$ is nonempty, then the sum of all entries in any $A\in\QTSASM(n)$ is four times the sum of entries in the top-left quarter of $A$,
and this equals~$n$ (as the sum of entries in each row of any ASM is~1), so that $n\equiv0\pmod{4}$. Therefore, $\QTSASM(n)$ is empty for $n\equiv2\pmod{4}$.
\item \textit{A product formula of a type similar to~\eqref{numASM} is known.}
In addition to unrestricted ASMs with~$n$ arbitrary (for which~\eqref{numASM} applies),
this occurs for VSASMs with~$n$ odd, VHSASMs with~$n$ odd, HTSASMs with~$n$ arbitrary,
QTSASMs with $n\not\equiv2\pmod{4}$, and DASASMs with~$n$ odd.  All of these
product formulae (or recurrence relations which lead to these formulae)
appeared as conjectures in a preprint by Robbins from the mid 1980s~\cite{Rob00},
as well as in a review article by Stanley~\cite{Sta86b} from 1986, and
a review paper by Robbins~\cite{Rob91a} from 1991,
with the conjectures for ASMs and HTSASMs being obtained
 by Mills, Robbins and Rumsey~\cite[Conj.~1]{MilRobRum82,MilRobRum83},~\cite[p.~285]{MilRobRum86}.
Since then, all these formulae have been proved.  The first proofs were obtained by Zeilberger~\cite[p.~5]{Zei96a} (unrestricted ASMs),
Kuperberg~\cite[Thm.~2]{Kup02} (VSASMs with~$n$ odd, HTSASMs with~$n$ even, and QTSASMs with $n\equiv0\pmod{4}$),
Okada~\cite[Thm.~1.2 (A5) \&~(A6)]{Oka06} (VHSASMs with~$n$ odd),
Razumov and Stroganov~\cite[p.~1197]{RazStr06a},~\cite[p.~1649]{RazStr06b} (HTSASMs with~$n$ odd, and QTSASMs with~$n$ odd),
and Behrend, Fischer and Konvalinka~\cite[Cor.~5]{BehFisKon17} (DASASMs with~$n$ odd).
All these proofs, except Zeilberger's~\cite{Zei96a} for unrestricted ASMs, involve the six-vertex model,
which was first used in an alternative proof of Kuperberg~\cite{Kup96} for unrestricted ASMs.
(By contrast, the approach taken by Zeilberger~\cite{Zei96a} involved using various constant term identities to show that
a certain refined enumeration of $n\times n$ ASMs coincides with a certain refined enumeration of totally symmetric self-complementary plane partitions (TSSCPPs)
in a $2n\times2n\times2n$ box, as had been conjectured by Mills, Robbins and Rumsey~\cite[Conj.~7]{MilRobRum86}.
A special case of this result is that the total numbers of such ASMs and TSSCPPs are equal, and the fact that~\eqref{numASM} gives the number of ASMs then
follows from a result of Andrews~\cite{And94} that~\eqref{numASM} gives the number of TSSCPPs.)

The specific product formulae for these cases (including that given by~\eqref{numASM}) are
\begin{align}
\label{numASM2}|\ASM(n)|&=\prod_{i=0}^{n-1}\frac{(3i+1)!}{(n+i)!},\\
\label{numVSASM}|\VSASM(2n+1)|&=\prod_{i=1}^n\frac{(6i-2)!}{(2n+2i)!},\\
\label{numVHSASM}|\VHSASM(2n+1)|&=\prod_{i=0}^{\lfloor n/2\rfloor-1}\!\frac{(3i+2)!\,(3\lceil n/2\rceil+3i)!}{(n+i)!\,(\lceil3n/2\rceil+i)!},\\
\label{numHTSASM}|\HTSASM(n)|&=\prod_{i=0}^{\lceil n/2\rceil-1}\!\frac{(3i)!}{(\lfloor n/2\rfloor+i)!}\;\prod_{i=0}^{\lfloor n/2\rfloor-1}\!\frac{(3i+2)!}{(\lceil n/2\rceil+i)!},\\
\label{numQTSASM}|\QTSASM(4n+k)|&=|\ASM(n)|^2\,|\HTSASM(2n+k)|,\ k\in\{-1,0,1\},\\
\label{numDASASM}|\DASASM(2n+1)|&=\prod_{i=0}^n\frac{(3i)!}{(n+i)!}.\end{align}
Note that these formulae appear in the literature in various different forms, some examples being as follows.
The formulae are sometimes written such that the terms in the product are independent of $n$ (such as $|\ASM(n)|=\prod_{i=0}^{n-1}\frac{i!\,(3i+1)!}{(2i)!\,(2i+1)!}$).
The formulae given by
Kuperberg~\cite{Kup02} are expressed as double products (such as $|\ASM(n)|=(-3)^{n(n-1)/2}\prod_{i,j=1}^n\frac{3i-3j+1}{i-j+n}$).
The formulae given by Okada~\cite{Oka06} are expressed in terms of dimensions of irreducible representations of classical groups
(such as $|\ASM(n)|=3^{-n(n-1)/2}\times$(dimension of the irreducible polynomial representation of the general linear group $\mathrm{GL}(2n,\mathbb{C})$ indexed by the partition $(n-1,n-1,n-2,n-2,\ldots,1,1,0,0)$)),
where these can also be expressed as numbers of certain tableaux
(such as $|\ASM(n)|=3^{-n(n-1)/2}\times$(number of semistandard Young tableaux of shape $(n-1,n-1,n-2,n-2,\ldots,1,1,0,0)$ with entries from $\{1,\ldots,2n\}$)).
In such cases, product formulae follow using Weyl's dimension formula for classical groups or standard formulae for numbers of certain tableaux.
\item \textit{A product formula is not known.}  This occurs for DSASMs with~$n$ arbitrary,
DASASMs with~$n$ even, and TSASMs with~$n$ odd.  In each of these cases, data for
relatively small values of~$n$ is given by Robbins~\cite[Sec.~4.1]{Rob00},
and data for somewhat larger values of~$n$ is given by Bousquet-M{\'e}lou and Habsieger~\cite[Tab.~1]{BouHab93},
and (for TSASMs with $n$ odd) Li\'{e}nardy~\cite[Tab.~A.1]{Lie20}.
By examining the prime factorizations of the numbers which appear in these cases,
it seems that the existence of product formulae similar to~\eqref{numASM2}--\eqref{numDASASM} can be ruled out.
The Pfaffian expression in~\eqref{numDSASM1}, for the case of DSASMs with $n$ arbitrary,
provides the first known enumerative formula for any of these three cases.
Recently, an expression involving multiple contour integrals has also been obtained for the case of TSASMs with~$n$ odd
by Li\'{e}nardy and Walmsley Hagendorf~\cite[Cor.~4.3]{LieWal25}.
\end{list}

As well as results involving straight enumeration of the eight standard ASM symmetry classes,
results are also known for refined enumeration (associated with various statistics) of the standard symmetry classes,
and for straight and refined enumeration of certain symmetry subclasses (for example, OSASMs form a subclass of DSASMs)
and quasi-symmetry classes.
For surveys and comprehensive references regarding the enumeration of such cases, see, for example, Behrend~\cite[Sec.~3]{Beh13a}, Behrend, Fischer and Konvalinka~\cite[Sec~1.2]{BehFisKon17}, and
Fischer and Saikia~\cite[Secs.~1~\&~10]{FisSai21}.  Several ASM subclasses are also considered, in the context of asymptotics, in Section~\ref{ASMLeadAsymptFurthSect2}.

\section{Preliminaries}\label{PrelimSect}
Within this section, the notation and conventions of the paper are outlined in Section~\ref{notation},
some elementary aspects of ASMs and DSASMs of fixed size are illustrated in Section~\ref{ASMDSASMsect},
the main statistics which will be used for ASMs and DSASMs are defined in Section~\ref{DSASMstats},
the main DSASM generating function which will be studied is defined in Section~\ref{Xsect},
and some enumerative results for DSASMs with the~$1$ in the first row in column~$1$,~$2$,~$3$ or~$n$ are obtained in Section~\ref{elempropsect}.

\subsection{Notation and conventions}\label{notation}
The following general notation and conventions will be used throughout this paper.

The Kronecker delta function is denoted, as usual, by $\delta_{i,j}$.
For an integer $n$, let
\begin{equation}\label{oddeven}
\odd(n)=\begin{cases}1,&n\text{ odd,}\\0,&n\text{ even,}\end{cases}\qquad\text{and}\qquad\even(n)=\begin{cases}0,&n\text{ odd,}\\
1,&n\text{ even.}\end{cases}\end{equation}

For a positive integer $n$, the Pfaffian $\Pf A$ of a $2n\times 2n$ skew-symmetric matrix~$A$,
and the Pfaffian $\Pf_{1\le i<j\le2n}(B_{ij})$ of a triangular array $(B_{ij})_{1\le i<j\le 2n}$
are defined in~\eqref{Pfdef1} and~\eqref{Pfdef2}. The Pfaffian of an empty matrix or empty triangular array is defined to be~1.

For integers $n$ and $k$, the binomial coefficient is taken to be
\begin{equation}\label{binomial}\binom{n}{k}=\begin{cases}\displaystyle\frac{n!}{k!\,(n-k)!},&0\le k\le n,\\0,&\text{otherwise.}\end{cases}\end{equation}
This differs from some other standard conventions, in which $\binom{n}{k}$ is nonzero for $n<0$ and $k\ge0$, or for $k\le n<0$.

For a power series $f(x_1,\ldots,x_n)$, the coefficient of $x_1^{i_1}\ldots x_n^{i_n}$
in the expansion of $f(x_1,\ldots,x_n)$ is denoted $[x_1^{i_1}\ldots x_n^{i_n}]\,f(x_1,\ldots,x_n)$.
For an indeterminate~$x$, let
\begin{equation}\label{not}\bar{x}=x^{-1}\qquad\text{and}\qquad\sigma(x)=x-\bar{x}.\end{equation}
Also, for a fixed indeterminate~$q$, which will be introduced in Section~\ref{vertweights}, let
\begin{equation}\label{sigmah}\sigmah(x)=\frac{\sigma(x)}{\sigma(q^4)}.\end{equation}

\subsection{ASMs and DSASMs}\label{ASMDSASMsect}
As in Section~\ref{ASMSymmClassSect}, for each positive integer $n$, let $\ASM(n)$ and $\DSASM(n)$ denote the sets of $n\times n$ ASMs and DSASMs, respectively.
For example, the sets $\DSASM(n)$ for $n=1$, $2$ and $3$ are
\begin{gather}\notag\DSASM(1)=\{(1)\},\quad\DSASM(2)=\left\{\begin{pmatrix}1&0\\0&1\end{pmatrix},\,\begin{pmatrix}0&1\\1&0\end{pmatrix}\right\},\\
\label{DSASM123}\DSASM(3)=\left\{\begin{pmatrix}1&0&0\\0&1&0\\0&0&1\end{pmatrix},\,
\begin{pmatrix}1&0&0\\0&0&1\\0&1&0\end{pmatrix},\,\begin{pmatrix}0&1&0\\1&0&0\\0&0&1\end{pmatrix},\,
\begin{pmatrix}0&0&1\\0&1&0\\1&0&0\end{pmatrix},\,\begin{pmatrix}0&1&0\\1&-1&1\\0&1&0\end{pmatrix}\right\},\end{gather}
and a specific element of $\DSASM(7)$ is given in~\eqref{DSASMexample}.  The sets $\ASM(n)$ for $n=1$, $2$ and $3$ are
\begin{gather}\notag\ASM(1)=\DSASM(1),\qquad\ASM(2)=\DSASM(2),\\
\ASM(3)=\DSASM(3)\cup\left\{\begin{pmatrix}0&1&0\\0&0&1\\1&0&0\end{pmatrix},\,\begin{pmatrix}0&0&1\\1&0&0\\0&1&0\end{pmatrix}\right\}.\end{gather}
The cardinalities of $\ASM(n)$ and $\DSASM(n)$ for $n=1,\ldots,10$ are given in Table~\ref{card}.

\begin{table}[h]\centering
$\begin{array}{|@{\;\;}c@{\;\;}|@{\;\;\;}c@{\;\;\;}c@{\;\;\;}c@{\;\;\;}c@{\;\;\;}c@{\;\;\;}c@{\;\;\;}
c@{\;\;\;}c@{\;\;\;}c@{\;\;\;}c@{\;\;}|}\hline\rule{0ex}{4.7mm}
n&1&2&3&4&5&6&7&8&9&10\\[1mm]\hline
\rule{0ex}{5.9mm}|\ASM(n)|&1&2&7&42&429&7436&218348&10850216&911835460&129534272700\\[1.8mm]\hline
\rule{0ex}{5.9mm}|\DSASM(n)|&1&2&5&16&67&368&2630&24376&293770&4610624\\[1.8mm]\hline
\end{array}$\\[2.4mm]\caption{Numbers of ASMs and DSASMs.}\label{card}
\end{table}

\subsection{ASM and DSASM statistics}\label{DSASMstats}
The three main statistics, denoted $R$, $S$ and $T$, which will be considered for any $A\in\ASM(n)$ are
\begin{align}
\notag R(A)&=\text{number of nonzero strictly upper triangular entries in }A\\
\label{RA}&=\sum_{1\le i<j\le n}|A_{ij}|,\\
\notag S(A)&=\text{number of nonzero diagonal entries in }A\\
\label{SA}&=\sum_{i=1}^n|A_{ii}|,\\
\label{TA}T(A)&=\text{column number of the 1 in the first row of }A.\end{align}

These statistics will mostly be applied to a DSASM~$A$. In this case,
due to the diagonal symmetry of $A$, $R(A)$ is also the number of nonzero strictly lower triangular entries in~$A$,
and~$T(A)$ is also the row number of the~1 in the first column of $A$. It will be seen in~\eqref{RC}--\eqref{TC} that~$R(A)$,~$S(A)$ and~$T(A)$
have natural interpretations in terms of six-vertex model configurations.

As an example, for the DSASM $A$ in~\eqref{DSASMexample}, $R(A)=7$, $S(A)=3$ and $T(A)=4$.

Statistics analogous to~$R$ and~$S$ have recently been defined for TSASMs by Li\'{e}nardy and Walmsley Hagendorf~\cite[Eq.~(199)]{HagLie21,HagLie22},~\cite[Eq.~(9.58)]{Lie20},~\cite[Eq.~(3.7)]{LieWal25}.
The statistic~$T$, as applied to ASMs in several classes, has been widely studied in various contexts.
Early work includes a proof by Zeilberger~\cite{Zei96b} of
a conjecture of Mills, Robbins and Rumsey~\cite[Conj.~2]{MilRobRum82,MilRobRum83}
for the number of ASMs of fixed size with prescribed value of~$T$.
Recent work includes that of Fischer and Saikia~\cite{FisSai21},
involving results for the number of ASMs of fixed size in several symmetry and quasi-symmetry classes with prescribed value of~$T$.

In Sections~\ref{elempropsect} and~\ref{ASMLeadAsymptFurthSect1}, the notation $\DSASM(n,k)$ will be used for the set of $n\times n$ DSASMs whose $1$ in the first row is in column $k$, i.e.,
\begin{equation}\label{DSASMnk}\DSASM(n,k)=\{A\in\DSASM(n)\mid T(A)=k\}.\end{equation}

Another natural statistic for any ASM $A$ is
\begin{equation}\label{MA}M(A)=\text{number of $-1$'s in }A.\end{equation}
For any $A\in\ASM(n)$, the number of nonzero entries in~$A$ is $2M(A)+n$
(since in each row of~$A$, the number of~$1$'s is one more than the number of~$-1$'s),
and for $A\in\DSASM(n)$, this is also equal to $2R(A)+S(A)$, so that
\begin{equation}\label{MRS}M(A)=R(A)+(S(A)-n)/2.\end{equation}
It follows that, for any $A\in\DSASM(n)$, $S(A)\equiv n\pmod{2}$
(since $S(A)=n+2(M(A)-R(A))$), and that the number of zero diagonal entries in~$A$ is even (since $n-S(A)=2(R(A)-M(A))$).

The statistic~$M$, as applied to ASMs in several classes, has also been widely studied. Early work includes
a result of Mills, Robbins and Rumsey~\cite[Corollary]{MilRobRum83},
and results of Kuperberg~\cite[Thm.~3]{Kup96},~\cite[Thms.~3--5]{Kup02},
some of which were conjectured by Mills, Robbins and Rumsey~\cite[Conjs.~4 \&~5]{MilRobRum83}.
More recent work includes that of Hagendorf and Morin--Duchesne~\cite[Eqs.~(2.23), (2.29), (2.33) \& (2.36)]{HagMor16},
and Aigner~\cite[Thm.~1.1]{Aig21}.

Several further ASM and DSASM statistics will be introduced in Sections~\ref{invol} and~\ref{FurthStatSect}.

\subsection{The DSASM generating function}\label{Xsect}
The DSASM generating function associated with the statistics in~\eqref{RA}--\eqref{TA} is defined as
\begin{equation}\label{X}X_n(r,s,t)=\sum_{A\in\DSASM(n)}r^{R(A)}\,s^{S(A)}\,t^{T(A)},\end{equation}
for indeterminates $r$, $s$ and $t$.

For example, the $n=1$, $2$ and $3$ cases of the DSASM generating function~\eqref{X} are
\begin{equation}\label{X123}X_1(r,s,t)=st,\quad X_2(r,s,t)=s^2t+rt^2,\quad
X_3(r,s,t)=s^3t+rst+rst^2+rst^3+r^2st^2,\end{equation}
where the terms on each RHS are written in an order corresponding to that used for the DSASMs in~\eqref{DSASM123}.

\subsection{DSASMs whose $1$ in the first row is in column $1$, $2$, $3$ or $n$}\label{elempropsect}
In the following proposition, some elementary enumerative results are obtained, using simple combinatorial arguments, for
DSASMs whose~$1$ in the first row is in column~$1$,~$2$,~$3$ or~$n$.

\begin{proposition}\label{refprop}
The numbers of $n\times n$ DSASMs whose $1$ in the first row is in column~$1$,~$2$,~$3$ or~$n$ satisfy
\begin{enumerate}
\item $|\DSASM(n,1)|=|\DSASM(n-1)|$,
\item $|\DSASM(n,2)|=|\DSASM(n-1)|$,
\item $|\DSASM(n,n)|=|\DSASM(n-2)|$,
\item $|\DSASM(n,3)|=2\,|\DSASM(n-1)|-3\,|\DSASM(n-2)|$,
\end{enumerate}
for $n\ge2$ in~(i) and~(ii), and $n\ge3$ in~(iii) and~(iv).

Furthermore,~(i)--(iii) can be generalized to the DSASM generating function identities
\begin{itemize}
\item[(i$'$)] $[t^1]\,X_n(r,s,t)=s\,X_{n-1}(r,s,1)$,
\item[(ii$'$)] $[t^2]\,X_n(r,s,t)=rs\,X_{n-1}(r,s,1)+r\,(1-s^2)\,X_{n-2}(r,s,1)$,
\item[(iii$'$)] $[t^n]\,X_n(r,s,t)=r\,X_{n-2}(r,s,1)$,
\end{itemize}
for $n\ge2$ in~(i\,$'\!$), and $n\ge3$ in~(ii\,$'\!$) and~(iii\,$'\!$).
\end{proposition}
\noindent\textit{Proof}. The identities~(i)--(iv) can be obtained using the following bijections (where the fact that each of the mappings described is bijective can be checked easily).
\begin{list}{(\roman{pr})}{\usecounter{pr}\setlength{\labelwidth}{7mm}\setlength{\leftmargin}{11mm}\setlength{\labelsep}{3mm}}
\item A bijection from $\DSASM(n,1)$ to $\DSASM(n-1)$ is given by the mapping which
deletes the first row and first column of each matrix in $\DSASM(n,1)$.
\item A bijection from $\DSASM(n,2)$ to $\DSASM(n-1)$ is given by the mapping
which, for each $A\in\DSASM(n,2)$,
replaces the entry $A_{22}$ (which is $-1$ or $0$) by $A_{22}+1$, and deletes
the first row and first column of the matrix.
\item A bijection from $\DSASM(n,n)$ to $\DSASM(n-2)$ is given by the mapping which
deletes the first and last row and column of each matrix in $\DSASM(n,n)$.
\item In any $A\in\DSASM(n,3)$, $A_{23}$ is either $-1$ or $0$, so $\DSASM(n,3)$ can be partitioned into the subsets
$\DSASM(n,3)_{-1}=\{A\in\DSASM(n,3)\mid A_{23}=-1\}$ and
$\DSASM(n,3)_0=\{A\in\DSASM(n,3)\mid A_{23}=0\}$.   Bijections from
$\DSASM(n,3)_{-1}$ and $\DSASM(n,3)_0$ to certain subsets of $\DSASM(n-1)$ are as follows.
A bijection from $\DSASM(n,3)_{-1}$ to $\{A\in\DSASM(n-1)\mid A_{11}=A_{12}=0\}$
is given by the mapping which replaces the entries $A_{22}$ ($=1$), $A_{23}$
($=-1$) and $A_{32}$ ($=-1$) by~0's, and deletes the
first row and first column of the matrix.
A bijection from $\DSASM(n,3)_0$ to $\{A\in\DSASM(n\nolinebreak-\nolinebreak1)\mid A_{12}=0\}$
is given by the mapping which replaces the entry $A_{33}$ (which is~$-1$
or~$0$) by $A_{33}+1$, and deletes the first row and column of the matrix.
The first bijection, the equivalence of the condition $A_{11}=A_{12}=0$ to the condition $A_{11}\ne1$ and $A_{12}\ne1$, and the identities from~(i) and~(ii), give
$|\DSASM(n,3)_{-1}|=|\DSASM(n\nolinebreak-\nolinebreak1)|-|\DSASM(n\nolinebreak-\nolinebreak1,1)|-|\DSASM(n\nolinebreak-\nolinebreak1,2)|=
|\DSASM(n-1)|-2|\DSASM(n-2)|$.
The second bijection, the equivalence of the condition $A_{12}=0$ to the
condition $A_{12}\ne1$, and the identity from~(ii), give
$|\DSASM(n,3)_0|=|\DSASM(n\nolinebreak-\nolinebreak1)|-|\DSASM(n\nolinebreak-\nolinebreak1,2)|=
|\DSASM(n\nolinebreak-\nolinebreak1)|-|\DSASM(n\nolinebreak-\nolinebreak2)|$.
Adding the previous two equations now gives the required identity for $|\DSASM(n,3)|$.
\end{list}
The identities (i$'$)--(iii$'$) can be obtained using the bijections for the proofs of (i)--(iii), as follows.
\begin{list}{(\roman{pr}$'$)}{\usecounter{pr}\setlength{\labelwidth}{7mm}\setlength{\leftmargin}{11mm}\setlength{\labelsep}{3mm}}
\item Let $\vartheta\colon\DSASM(n,1)\to\DSASM(n-1)$ be the bijection for the proof of~(i).
Then $R\bigl(\vartheta(A)\bigr)=R(A)$ and $S\bigl(\vartheta(A)\bigr)=S(A)-1$, for any $A\in\DSASM(n,1)$, which gives
\begin{align*}[t^1]\,X_n(r,s,t)&=\sum_{A\in\DSASM(n,1)}r^{R(A)}\,s^{S(A)}=\sum_{A\in\DSASM(n-1)}r^{R(\vartheta^{-1}(A))}\,s^{S(\vartheta^{-1}(A))}\\
&=\sum_{A\in\DSASM(n-1)}r^{R(A)}\,s^{S(A)+1}=s\,X_{n-1}(r,s,1),\end{align*}
as required.
\item Let $\varphi\colon\DSASM(n,2)\to\DSASM(n-1)$ be the bijection for the proof of~(ii). Then
\begin{equation*}R\bigl(\varphi(A)\bigr)=R(A)-1\quad\text{and}\quad S\bigl(\varphi(A)\bigr)=\begin{cases}S(A)-1,&A_{22}=-1,\\S(A)+1,&A_{22}=0,\end{cases}\end{equation*}
for any $A\in\DSASM(n,2)$,
which gives
\begin{align*}[t^2]\,X_n(r,s,t)&=\sum_{A\in\DSASM(n,2)}r^{R(A)}\,s^{S(A)}=\sum_{A\in\DSASM(n-1)}r^{R(\varphi^{-1}(A))}\,s^{S(\varphi^{-1}(A))}\\
&=\sum_{\substack{A\in\DSASM(n-1)\\A_{11}=0}}r^{R(A)+1}\,s^{S(A)+1}\:+\!\sum_{\substack{A\in\DSASM(n-1)\\A_{11}=1}}r^{R(A)+1}\,s^{S(A)-1}\\
&=rs\bigl(X_{n-1}(r,s,1)-[t^1]\,X_{n-1}(r,s,t)\bigr)+rs^{-1}\,[t^1]\,X_{n-1}(r,s,t)\\
&=rs\,X_{n-1}(r,s,1)+r\,(1-s^2)\,X_{n-2}(r,s,1),\end{align*}
where~(i$'$) is used for the last equality, as required.
\item Let $\psi\colon\DSASM(n,n)\to\DSASM(n-2)$ be the bijection for the proof of~(iii).
Then $R\bigl(\psi(A)\bigr)=R(A)-1$ and $S\bigl(\psi(A)\bigr)=S(A)$, for any $A\in\DSASM(n,n)$, which gives
\begin{align*}[t^n]\,X_n(r,s,t)&=\sum_{A\in\DSASM(n,n)}r^{R(A)}\,s^{S(A)}=\sum_{A\in\DSASM(n-2)}r^{R(\psi^{-1}(A))}\,s^{S(\psi^{-1}(A))}\\
&=\sum_{A\in\DSASM(n-2)}r^{R(A)+1}\,s^{S(A)}=r\,X_{n-2}(r,s,1),\end{align*}
as required.\hfill$\Box$
\end{list}

\section{Main results for the DSASM generating function and number of DSASMs}\label{mainresults}
In this section, the main results for the DSASM generating function $X_n(r,s,t)$ and the number of $n\times n$ DSASMs are stated.
The primary statements of the results are given in Section~\ref{mainresults1}, and certain alternative statements of the results are given in Sections~\ref{mainresults2} and~\ref{mainresults3}.
The notation and conventions used in the results are defined in Section~\ref{notation}.
The proofs of the results are given in Section~\ref{proofs}, and depend on six-vertex model results
which are obtained in Section~\ref{sixvertexmodelDSASMs}.

\subsection{Primary statement of results for \texorpdfstring{$X_n(r,s,1)$}{Xn(r,s,1)}, $|\DSASM(n)|$ and \texorpdfstring{$X_n(r,s,t)$}{Xn(r,s,t)}}\label{mainresults1}
The main results for $X_n(r,s,1)$, $X_n(r,s,t)$ and $|\DSASM(n)|$ are as follows.
\begin{theorem}\label{Xrstheorem}
The DSASM generating function $X_n(r,s,t)$ at $t=1$ is given by
\begin{equation}\label{Xrs}X_n(r,s,1)=s^{\odd(n)}
\Pf_{\odd(n)\le i<j\le n-1}\biggl([u^iv^j]\,\frac{v-u}{1-uv}\biggl(s^2+\frac{r(1+u)(1+v)}{(1-ru)(1-rv)-uv}\biggr)\biggr),\end{equation}
where the coefficients which comprise the entries for the Pfaffian can be evaluated explicitly as
\begin{multline}\label{Xrsentries}
[u^iv^j]\,\frac{v-u}{1-uv}\biggl(s^2+\frac{r(1+u)(1+v)}{(1-ru)(1-rv)-uv}\biggr)\\
\shoveleft{\;=s^2\,\delta_{i+1,j}+\sum_{0\le l\le k\le i}\biggl[\biggl(\binom{k}{l}\binom{j-i+k-1}{l}-\binom{k-1}{l}\binom{j-i+k}{l}\biggr)r^{j-i+2k-2l}}\\
\:\;\;+\biggl(\binom{k}{l}\binom{j-i+k-2}{l}-\binom{k-2}{l}\binom{j-i+k}{l}\biggr)r^{j-i+2k-2l-1}\\
+\biggl(\binom{k-1}{l}\binom{j-i+k-2}{l}-\binom{k-2}{l}\binom{j-i+k-1}{l}\biggr)r^{j-i+2k-2l-2}\biggr],\end{multline}
for $0\le i<j$.
\end{theorem}
The proof of Theorem~\ref{Xrstheorem} is given in Sections~\ref{proofXrs} and~\ref{Xrsentriesproof}.  In Section~\ref{Xrsentriesproof},
recurrence relations (see~\eqref{DSASMentryrecurr1}--\eqref{DSASMentryrecurr3}) which can be used for the efficient computation of
the numbers in~\eqref{Xrsentries} are also derived.

\begin{corollary}\label{numDSASMcoroll}
The number of $n\times n$ DSASMs is given by
\begin{equation}\label{numDSASM2}
|\DSASM(n)|=\Pf_{\odd(n)\le i<j\le n-1}\biggl([u^iv^j]\,\frac{(v-u)(2+uv)}{(1-uv)(1-u-v)}\biggr),\end{equation}
where the coefficients which comprise the entries for the Pfaffian can be evaluated explicitly as
\begin{align}
\notag[u^iv^j]\,\frac{(v-u)(2+uv)}{(1-uv)(1-u-v)}&=\sum_{k=0}^i(3-\delta_{k,0})\,\biggl(\binom{i+j-2k-1}{i-k}-\binom{i+j-2k-1}{j-k}\biggr)\\
\label{numDSASMentries}&=(j-i)\,\sum_{k=0}^i\frac{3-\delta_{k,0}}{i+j-2k}\,\binom{i+j-2k}{i-k},\end{align}
for $0\le i<j$.
\end{corollary}
The proof of Corollary~\ref{numDSASMcoroll} is given in Section~\ref{numDSASMentriesproof}.

Note that it follows from the first equality in~\eqref{numDSASMentries} that all entries for the Pfaffian in~\eqref{numDSASM2}
are integers, and it follows from the second equality in~\eqref{numDSASMentries} that these integers are positive.
More generally, it can be shown straightforwardly using~\eqref{Xrsentries} that the entries for the Pfaffian in~\eqref{Xrs} are polynomials in~$r$ and~$s$
with nonnegative integer coefficients.
This suggests that it may be possible to show,
using Pfaffian analogues of the Lindstr\"{o}m--Gessel--Viennot theorem due to Okada~\cite{Oka89} and Stembridge~\cite{Ste90},
that $X_n(r,s,1)$ is also a generating function for certain families of nonintersecting lattice paths.

\begin{theorem}\label{Xrsttheorem}
For $n\ge2$, the DSASM generating function $X_n(r,s,t)$ satisfies
\begin{multline}\label{Xrst}
(t-1+rt)\,(t-1)^{n-1}\,X_n(r,s,t)=st\bigl((t-1)^n+\odd(n)\,r^n\,t^n\bigr)\,X_{n-1}(r,s,1)+r^n\,s^{\odd(n)}\,t^{n+1}\\[1mm]
\times\!\Pf_{\odd(n)\le i<j\le n-1}\left(\begin{cases}\displaystyle[u^iv^j]\,\frac{v-u}{1-uv}\biggl(s^2+\frac{r(1+u)(1+v)}{(1-ru)(1-rv)-uv}\biggr),&j\le n-2\\[4mm]
\displaystyle[u^i]\,\frac{t-1-rtu}{rt-(t-1)u}\biggl(s^2+\frac{r(t-1+rt)(1+u)}{r-(t-1+r^2)u}\biggr),&j=n-1\end{cases}\right),\end{multline}
where the coefficients which comprise the entries for the
Pfaffian can be evaluated explicitly using~\eqref{Xrsentries} for $0\le i<j\le n-2$, and as
\begin{align}\label{Xrstentries}
\notag&\!\![u^i]\,\frac{t-1-rtu}{rt-(t-1)u}\biggl(s^2+\frac{r(t-1+rt)(1+u)}{r-(t-1+r^2)u}\biggr)\\[1mm]
\notag&=\begin{cases}\displaystyle\frac{(t-1)(s^2+t-1+rt)}{rt},&i=0,\\[4mm]
\displaystyle\frac{t-1+rt}{(rt)^{i+1}}\,\Bigl(s^2\,(t-1-rt)\,(t-1)^{i-1}\\
\hspace*{11mm}\mbox{}+(t-1+rt)^2\,(t-1-rt)\sum_{k=0}^{i-2}(t-1)^k\,t^{i-k-2}\,(t-1+r^2)^{i-k-2}\\
\hspace*{34mm}\mbox{}+\bigl((t-1)^2(t+1)+rt(t-1-r)\bigr)\,t^{i-1}\,(t-1+r^2)^{i-1}\Bigr),&i\ge1\end{cases}\\[2mm]
&=\begin{cases}\displaystyle\frac{(t-1)(s^2+t-1+rt)}{rt},&i=0,\\[4mm]
\displaystyle\frac{t-1+rt}{(rt)^{i+1}}\,\Bigl(s^2\,(t-1-rt)\,(t-1)^{i-1}+\Bigl((t-1+rt)^2\,(t-1-rt)\,(t-1)^{i-1}\\
\hspace*{22mm}\mbox{}-(t-1+r+r^2)\,\bigl((t-1)^2-r^2\bigr)\,t^{i+1}\,(t-1+r^2)^{i-1}\Bigr)\Big/\\
\hspace*{90mm}\mbox{}\bigl(t-1-t(t-1+r^2)\bigr)\Bigr),&i\ge1,\end{cases}
\end{align}
for $0\le i<j=n-1$.
\end{theorem}
The proof of Theorem~\ref{Xrsttheorem} is given in Section~\ref{proofXrsttheorem}.

Note that in the Pfaffian of~\eqref{Xrst}, the function for the entries with $j=n-1$ is related simply to the function for
the entries with $j\le n-2$ by
\begin{multline}\label{Xrstentriesfun}\frac{t-1-rtu}{rt-(t-1)u}\biggl(s^2+\frac{r(t-1+rt)(1+u)}{r-(t-1+r^2)u}\biggr)\\
=\frac{v-u}{1-uv}\biggl(s^2+\frac{r(1+u)(1+v)}{(1-ru)(1-rv)-uv}\biggr)\bigg|_{v=(t-1)/(rt)}.\end{multline}

Note also that if~\eqref{Xrst} is used to compute $X_n(r,s,t)$, then $X_{n-1}(r,s,1)$, which appears on the RHS, needs to be obtained first,
which can be done directly using~\eqref{Xrs}.  Alternatively, by starting with $X_1(r,s,1)=s$ from~\eqref{X123}, $X_k(r,s,t)$ for $k=2,\ldots,n$ can be
successively computed using~\eqref{Xrst} alone, where the computation for $X_k(r,s,t)$ uses the already-computed expression for $X_{k-1}(r,s,t)$ with~$t$ set to~$1$.
In this approach, it needs to be ensured that, for $k\ge3$, the factor $(t-1)^{k-2}$ on the LHS of the $n=k-1$ case of~\eqref{Xrst} is canceled from the expression provided by the RHS of
this case of~\eqref{Xrst} before setting $t=1$,
to avoid division of zero by zero.

\subsection{Alternative statement of results for \texorpdfstring{$X_n(r,s,1)$}{Xn(r,s,1)} and $|\DSASM(n)|$}\label{mainresults2}
An alternative statement of Theorem~\ref{Xrstheorem} is that
\begin{align}\notag&\hspace*{-1.5mm}X_n(r,s,1)\\
\label{XrsReform}&\hspace*{-1.3mm}=s^{\odd(n)}\Pf_{0\le i<j\le n-\even(n)}\left(\begin{cases}\displaystyle[u^iv^j]\,\frac{v-u}{1-uv}\biggl(\frac{s^2}{(1+u)(1+v)}\\
\displaystyle\hspace{31mm}\mbox{}+\frac{r}{(1-ru)(1-rv)-uv}\biggr),&j\le n-1\\[3mm]
(-1)^i,&j=n\end{cases}\right),\end{align}
where the coefficients which comprise the entries for the Pfaffian with $j\le n-1$ are explicitly given by
\begin{multline}\label{XrsentriesReform}
[u^iv^j]\,\frac{v-u}{1-uv}\biggl(\frac{s^2}{(1+u)(1+v)}+\frac{r}{(1-ru)(1-rv)-uv}\biggr)\\
\shoveleft{\;=s^2\,(-1)^{i+j+1}+\sum_{0\le l\le k\le i}\biggl(\binom{k}{l}\binom{j-i+k-1}{l}-\binom{k-1}{l}\binom{j-i+k}{l}\biggr)r^{j-i+2k-2l}},\end{multline}
for $0\le i<j$.

Some differences between the Pfaffian expression for~$X_n(r,s,1)$ provided by~\eqref{Xrs} and~\eqref{Xrsentries},
and that provided by~\eqref{XrsReform} and~\eqref{XrsentriesReform}, are as follows. The explicit expression~\eqref{XrsentriesReform} for the entries in~\eqref{XrsReform}
is considerably shorter than the explicit expression~\eqref{Xrsentries} for the entries in~\eqref{Xrs}.
For $n$ even, the range for the entries
is $0\le i<j\le n-1$ in both~\eqref{Xrs} and~\eqref{XrsReform}.  However, for~$n$ odd, the range for the entries is
$1\le i<j\le n-1$ in~\eqref{Xrs}, but $0\le i<j\le n$ in~\eqref{XrsReform}, and in~\eqref{XrsReform} the entries have a separate form for $j=n$.

A formula for $|\DSASM(n)|$ which follows from~\eqref{XrsReform} and~\eqref{XrsentriesReform},
and which provides an alternative to the formulae of Corollary~\ref{numDSASMcoroll}, is
\begin{multline}\label{numDSASMReform}
|\DSASM(n)|\\
=\Pf_{0\le i<j\le n-\even(n)}\left(\begin{cases}
\displaystyle(-1)^{i+j+1}+\sum_{k=0}^i\frac{j-i}{i+j-2k}\,\binom{i+j-2k}{i-k},&j\le n-1\\[3mm]
(-1)^i,&j=n\end{cases}\right).\end{multline}

The proofs of~\eqref{XrsReform}--\eqref{numDSASMReform} are given in Section~\ref{proofXrsReform}.

It would also be possible to obtain an alternative statement of Theorem~\ref{Xrsttheorem} with a similar form to the alternative
statement of Theorem~\ref{Xrstheorem} provided by~\eqref{XrsReform} and~\eqref{XrsentriesReform}, but this will not be done here. Instead, an alternative statement
of~\eqref{Xrst} from Theorem~\ref{Xrsttheorem} with a somewhat different form will be given in Section~\ref{mainresults3}.

Note that~\eqref{XrsReform} leads to a compact determinant formula for $X_n(r,s,1)\,X_{n+1}(r,s,1)$
(and, similarly,~\eqref{numDSASMReform} leads to a compact determinant formula for $|\DSASM(n)|\,|\DSASM(n+1)|$).
Specifically, using~\eqref{XrsReform} together with a general determinant-Pfaffian identity obtained by Okada and Krattenthaler~\cite[Lem.~8]{OkaKra98} gives
\begin{multline}\label{Xrsdet}X_n(r,s,1)\,X_{n+1}(r,s,1)\\=
s\det_{0\le i,j\le n}\left(\begin{cases}\displaystyle[u^iv^j]\,\frac{v-u}{1-uv}\biggl(\frac{s^2}{(1+u)(1+v)}+\frac{r}{(1-ru)(1-rv)-uv}\biggr),&j\le n-1\\[3mm]
(-1)^i,&j=n\end{cases}\right).\end{multline}
The identity in~\cite[Lem.~8]{OkaKra98} can itself be obtained easily (up to overall sign factors, which can be determined by considering special cases)
by applying the Desnanot--Jacobi determinant identity to the matrix
$\left(\begin{smallmatrix}A&b&c\\-{}^tb&0&-d\\-{}^tc&d&0\end{smallmatrix}\right)$ in~\cite[Lem.~8]{OkaKra98},
and using standard determinant and Pfaffian properties.

\subsection{Alternative statement of result for \texorpdfstring{$X_n(r,s,t)$}{Xn(r,s,t)}}\label{mainresults3}
An alternative statement of the result~\eqref{Xrst} for $X_n(r,s,t)$
will be given in terms of a function $f_n(u,v)$, defined as follows.
For any nonnegative integer~$n$, let $f_n(u,v)$ denote the $n$th divided difference of $\frac{v-u}{1-uv}\bigl(s^2+\frac{r(1+u)(1+v)}{(1-ru)(1-rv)-uv}\bigr)$ with respect to $v$,
evaluated at $\underbrace{0,\ldots,0}_{n},v$.
Accordingly, $f_n(u,v)$ can be defined recursively using
\begin{equation}f_n(u,v)=\begin{cases}\displaystyle\frac{v-u}{1-uv}\biggl(s^2+\frac{r(1+u)(1+v)}{(1-ru)(1-rv)-uv}\biggr),&n=0,\\[3.8mm]
\displaystyle\frac{f_{n-1}(u,v)-f_{n-1}(u,0)}{v},&n\ge1,\end{cases}\end{equation}
or defined as a power series in~$v$ with coefficients
\begin{equation}\label{fddcoeff}[v^i]f_n(u,v)=[v^{i+n}]\,\frac{v-u}{1-uv}\biggl(s^2+\frac{r(1+u)(1+v)}{(1-ru)(1-rv)-uv}\biggr),\end{equation}
for each nonnegative integer~$i$. It follows from~\eqref{fddcoeff} that an explicit expression for $f_n(u,v)$ is
\begin{multline}f_n(u,v)=\frac{v-u}{v^n\,(1-uv)}\biggl(s^2+\frac{r(1+u)(1+v)}{(1-ru)(1-rv)-uv}\biggr)\\
-\sum_{i=0}^{n-1}\,[w^i]\,\frac{w-u}{1-uw}\biggl(s^2+\frac{r(1+u)(1+w)}{(1-ru)(1-rw)-uw}\biggr)\,v^{i-n}.\end{multline}
The alternative statement of~\eqref{Xrst} is now that, for $n\ge2$,
\begin{multline}\label{XrstDiv}
(t-1+rt)\,X_n(r,s,t)=st(t-1)\,X_{n-1}(r,s,1)+r\,s^{\odd(n)}\,t^2\\
\times\Pf_{\odd(n)\le i<j\le n-1}\left(\begin{cases}\displaystyle[u^iv^j]\,\frac{v-u}{1-uv}\biggl(s^2+\frac{r(1+u)(1+v)}{(1-ru)(1-rv)-uv}\biggr),&j\le n-2\\[4.5mm]
\displaystyle[u^i]\,f_{n-1}\biggl(\!u,\frac{t-1}{rt}\biggr),&j=n-1\end{cases}\right).\end{multline}
The proof of~\eqref{XrstDiv} is given in Section~\ref{proofXrstDiv}.

Some advantages of~\eqref{XrstDiv} relative to~\eqref{Xrst} are as follows.
The prefactors of $X_n(r,s,t)$ on the LHS and $X_{n-1}(r,s,1)$ on the RHS are simpler in~\eqref{XrstDiv} than in~\eqref{Xrst}.
Also, by setting $t=1$ in~\eqref{XrstDiv}, and noting from~\eqref{fddcoeff} that $f_{n-1}(u,0)=[v^{n-1}]\,\frac{v-u}{1-uv}\bigl(s^2+\frac{r(1+u)(1+v)}{(1-ru)(1-rv)-uv}\bigr)$,
it follows that~\eqref{XrstDiv} directly reduces to the expression~\eqref{Xrs} for $X_n(r,s,1)$,
whereas the process of obtaining~\eqref{Xrs} by taking $t\to1$ in~\eqref{Xrst} is not immediately clear.
On the other hand, an advantage of~\eqref{Xrst} is
that the expression for the entries for the Pfaffian for $j=n-1$ is simpler than in~\eqref{XrstDiv}.

\section{The six-vertex model for DSASMs}\label{sixvertexmodelDSASMs}
In this section, a version of the six-vertex model whose configurations are in bijection with DSASMs of fixed size is introduced and studied.
The configurations of the model are defined in Section~\ref{sixvertexmodelconfig},
an elementary enumerative result is obtained in Section~\ref{DSASMsumsect},
the bulk and boundary vertex weights for the model are defined in Section~\ref{vertweights},
the DSASM partition function is defined in Section~\ref{partitionfun},
and some results involving fundamental properties of this partition function are obtained in Section~\ref{GenProp}.
In Section~\ref{Pfaffians}, the definition of a Pfaffian is given, and some general properties of Pfaffians
are outlined.  In Section~\ref{Pfaffianexpsect}, a Pfaffian expression for the DSASM partition function
is obtained.

\subsection{Six-vertex model configurations}\label{sixvertexmodelconfig}
Define a grid graph on a triangle as
\psset{unit=6mm}
\begin{equation}\label{Tn}\G_n=\raisebox{-17mm}{\pspicture(0.1,0.5)(8,7)
\psline[linewidth=0.5pt](1,6)(1,5)(3,5)\psline[linewidth=0.5pt,linestyle=dashed,dash=3pt 2pt](3,5)(4,5)\psline[linewidth=0.5pt](4,5)(6,5)
\psline[linewidth=0.5pt](2,6)(2,4)(3,4)\psline[linewidth=0.5pt,linestyle=dashed,dash=3pt 2pt](3,4)(4,4)\psline[linewidth=0.5pt](4,4)(6,4)
\psline[linewidth=0.5pt](3,6)(3,3)\psline[linewidth=0.5pt,linestyle=dashed,dash=3pt 2pt](3,3)(4,3)\psline[linewidth=0.5pt](4,3)(6,3)
\psline[linewidth=0.5pt](4,6)(4,3)\psline[linewidth=0.5pt,linestyle=dashed,dash=3pt 2pt](4,3)(4,2)\psline[linewidth=0.5pt](4,2)(6,2)
\psline[linewidth=0.5pt](5,6)(5,3)\psline[linewidth=0.5pt,linestyle=dashed,dash=3pt 2pt](5,3)(5,2)\psline[linewidth=0.5pt](5,2)(5,1)(6,1)
\multirput(1,6)(1,0){5}{$\ss\bullet$}\multirput(1,5)(1,0){6}{$\ss\bullet$}\multirput(2,4)(1,0){5}{$\ss\bullet$}
\multirput(3,3)(1,0){4}{$\ss\bullet$}\multirput(4,2)(1,0){3}{$\ss\bullet$}\multirput(5,1)(1,0){2}{$\ss\bullet$}
\rput(0.9,6.4){$\scriptscriptstyle(0,1)$}\rput(2.1,6.4){$\scriptscriptstyle(0,2)$}\rput(5,6.4){$\scriptscriptstyle(0,n)$}
\rput(0.9,4.7){$\scriptscriptstyle(1,1)$}\rput(7,5){$\scriptscriptstyle(1,n+1)$}
\rput(1.9,3.7){$\scriptscriptstyle(2,2)$}\rput(7,4){$\scriptscriptstyle(2,n+1)$}
\rput(4.8,0.7){$\scriptscriptstyle(n,n)$}\rput(7.05,1){$\scriptscriptstyle(n,n+1)$}
\multirput(3.25,6.4)(0.25,0){3}{.}\multirput(6.5,2.25)(0,0.25){3}{.}\multirput(3.02,2.38)(0.18,-0.18){3}{.}
\rput(8.2,3.5){.}\endpspicture}\end{equation}

The vertices of~$\G_n$ consist of top vertices $(0,j)$, $j=1,\ldots,n$, of degree~1, left boundary vertices
$(i,i)$, $i=1,\ldots,n$, of degree~2,
bulk vertices $(i,j)$, $1\le i<j\le n$, of degree~4,
and right boundary vertices $(i,n+1)$, $i=1,\ldots,n$, of degree~1.

Now define a six-vertex model configuration on~$\G_n$ to be an orientation of
the edges of~$\G_n$, such that:
\begin{itemize}
\item There are two edges directed into and two edges directed out of each bulk vertex, i.e., the so-called six-vertex rule is
satisfied at each bulk vertex.
\item Each edge incident with a top vertex is directed upward.
\item Each edge incident with a right boundary vertex is directed leftward.
\end{itemize}
The set of six-vertex model configurations on~$\G_n$ will be denoted as $\SVC(n)$.

\psset{unit=7mm}
For $C\in\SVC(n)$ and a vertex $(i,j)$ of~$\G_n$, define the local
configuration $C_{ij}$ at $(i,j)$ to be the orientation of the edges incident
to $(i,j)$.
Hence, the possible local configurations are~\T\ at a top vertex,
\Lup, \Ldown, \Lout\ or \Lin\ at a left boundary vertex, \Vv, \Vvi, \Vi,
\Vii, \Viii\ or \Viv\ at a bulk vertex,
and \R\ at a right boundary vertex.

There is a well-known bijection from $\ASM(n)$ to the set of six-vertex model configurations
with so-called domain-wall boundary conditions on a certain grid graph $\mathcal{S}_n$ on an $n\times n$ square.
This bijection was first observed (in a different form) by Robbins and Rumsey~\cite[pp.~179--180]{RobRum86}
and then (in the now standard form) by Elkies, Kuperberg, Larsen and Propp~\cite[Sec.~7]{ElkKupLarPro92b}.
The restriction of this bijection to $\DSASM(n)$, and the restriction of each corresponding six-vertex model
configuration on $\mathcal{S}_n$ to the part which lies on or above the main diagonal of~$\mathcal{S}_n$,
provides a bijection from $\DSASM(n)$ to~$\SVC(n)$.
In this bijection, the DSASM~$A$ which corresponds to configuration $C\in\SVC(n)$ has entries given by
\begin{equation}\label{bij}A_{ij}=A_{ji}=
\left\{\rule[-10mm]{0mm}{20mm}\right.\!\!\begin{array}{@{}r@{\;\;\;\;}l@{}}1,&C_{ij}=\Vv \text{ or }\Lup,\\
-1,&C_{ij}=\Vvi\text{ or }\Ldown,\\
0,&C_{ij}=\Vi, \Lout, \Vii, \Lin, \Viii\text{ or }\Viv,\end{array}\end{equation}
for $1\le i\le j\le n$.

Note that the (fixed) local configurations at the top and right boundary vertices are not associated with entries of $A$.

Note also that, in the six-vertex model configuration on $\mathcal{S}_n$
which corresponds to an $n\times n$ DSASM, the diagonal symmetry implies that the local configurations \Viii\ and \Viv\ cannot
occur at vertices on the main diagonal, and the associated local configurations at left boundary vertices of~$\G_n$
are simply the restrictions of \Vv, \Vvi, \Vi\ or \Vii\ to the upper and right edges.

\psset{unit=5mm}
As examples, the sets $\SVC(n)$ for $n=1$, $2$ and $3$ are
\begin{gather}\notag\SVC(1)=\left\{\raisebox{-1.8mm}{
\pspicture(0.9,0.9)(2.5,2.1)\multirput(1,2)(1,-1){2}{$\ss\bullet$}\rput(1,1){$\ss\bullet$}
\psline[linewidth=0.5pt](1,2)(1,1)(2,1)
\psdots[dotstyle=triangle*,dotscale=1.1](1,1.5)\psdots[dotstyle=triangle*,dotscale=1.1,dotangle=90](1.5,1)\endpspicture}\right\},\quad
\SVC(2)=\left\{\raisebox{-4.5mm}{
\pspicture(0.9,0.9)(3.8,3.1)\multirput(1,3)(1,0){2}{$\ss\bullet$}\multirput(1,2)(1,0){3}{$\ss\bullet$}\multirput(2,1)(1,0){2}{$\ss\bullet$}
\psline[linewidth=0.5pt](1,3)(1,2)(3,2)\psline[linewidth=0.5pt](2,3)(2,1)(3,1)
\multirput(1,2)(1,0){2}{\psdots[dotstyle=triangle*,dotscale=1.1](0,0.5)}\multirput(2,1)(0,1){2}{\psdots[dotstyle=triangle*,dotscale=1.1,dotangle=90](0.5,0)}
\psdots[dotstyle=triangle*,dotscale=1.1](2,1.5)\psdots[dotstyle=triangle*,dotscale=1.1,dotangle=90](1.5,2)\rput(3.4,1.9){,}\endpspicture
\pspicture(0.9,0.9)(3.5,3.1)\multirput(1,3)(1,0){2}{$\ss\bullet$}\multirput(1,2)(1,0){3}{$\ss\bullet$}\multirput(2,1)(1,0){2}{$\ss\bullet$}
\psline[linewidth=0.5pt](1,3)(1,2)(3,2)\psline[linewidth=0.5pt](2,3)(2,1)(3,1)
\multirput(1,2)(1,0){2}{\psdots[dotstyle=triangle*,dotscale=1.1](0,0.5)}\multirput(2,1)(0,1){2}{\psdots[dotstyle=triangle*,dotscale=1.1,dotangle=90](0.5,0)}
\psdots[dotstyle=triangle*,dotscale=1.1,dotangle=180](2,1.5)\psdots[dotstyle=triangle*,dotscale=1.1,dotangle=-90](1.5,2)\endpspicture}\right\},\\[3mm]
\label{SV3}\SVC(3)=\left\{\raisebox{-6.7mm}{
\pspicture(0.9,0.9)(4.8,4.1)\multirput(1,4)(1,0){3}{$\ss\bullet$}\multirput(1,3)(1,0){4}{$\ss\bullet$}\multirput(2,2)(1,0){3}{$\ss\bullet$}\multirput(3,1)(1,0){2}{$\ss\bullet$}
\psline[linewidth=0.5pt](1,4)(1,3)(4,3)\psline[linewidth=0.5pt](2,4)(2,2)(4,2)\psline[linewidth=0.5pt](3,4)(3,1)(4,1)
\multirput(1,3)(1,0){3}{\psdots[dotstyle=triangle*,dotscale=1.1](0,0.5)}
\multirput(3,1)(0,1){3}{\psdots[dotstyle=triangle*,dotscale=1.1,dotangle=90](0.5,0)}
\psdots[dotstyle=triangle*,dotscale=1.1](2,2.5)(3,2.5)(3,1.5)
\psdots[dotstyle=triangle*,dotscale=1.1,dotangle=90](1.5,3)(2.5,3)(2.5,2)
\rput(4.4,1.9){,}\endpspicture
\pspicture(0.9,0.9)(4.8,4.1)\multirput(1,4)(1,0){3}{$\ss\bullet$}\multirput(1,3)(1,0){4}{$\ss\bullet$}\multirput(2,2)(1,0){3}{$\ss\bullet$}\multirput(3,1)(1,0){2}{$\ss\bullet$}
\psline[linewidth=0.5pt](1,4)(1,3)(4,3)\psline[linewidth=0.5pt](2,4)(2,2)(4,2)\psline[linewidth=0.5pt](3,4)(3,1)(4,1)
\multirput(1,3)(1,0){3}{\psdots[dotstyle=triangle*,dotscale=1.1](0,0.5)}
\multirput(3,1)(0,1){3}{\psdots[dotstyle=triangle*,dotscale=1.1,dotangle=90](0.5,0)}
\psdots[dotstyle=triangle*,dotscale=1.1](2,2.5)(3,2.5)
\psdots[dotstyle=triangle*,dotscale=1.1,dotangle=180](3,1.5)
\psdots[dotstyle=triangle*,dotscale=1.1,dotangle=90](1.5,3)(2.5,3)
\psdots[dotstyle=triangle*,dotscale=1.1,dotangle=-90](2.5,2)
\rput(4.4,1.9){,}\endpspicture
\pspicture(0.9,0.9)(4.8,4.1)\multirput(1,4)(1,0){3}{$\ss\bullet$}\multirput(1,3)(1,0){4}{$\ss\bullet$}\multirput(2,2)(1,0){3}{$\ss\bullet$}\multirput(3,1)(1,0){2}{$\ss\bullet$}
\psline[linewidth=0.5pt](1,4)(1,3)(4,3)\psline[linewidth=0.5pt](2,4)(2,2)(4,2)\psline[linewidth=0.5pt](3,4)(3,1)(4,1)
\multirput(1,3)(1,0){3}{\psdots[dotstyle=triangle*,dotscale=1.1](0,0.5)}
\multirput(3,1)(0,1){3}{\psdots[dotstyle=triangle*,dotscale=1.1,dotangle=90](0.5,0)}
\psdots[dotstyle=triangle*,dotscale=1.1](3,2.5)(3,1.5)
\psdots[dotstyle=triangle*,dotscale=1.1,dotangle=180](2,2.5)
\psdots[dotstyle=triangle*,dotscale=1.1,dotangle=90](2.5,3)(2.5,2)
\psdots[dotstyle=triangle*,dotscale=1.1,dotangle=-90](1.5,3)
\rput(4.4,1.9){,}\endpspicture
\pspicture(0.9,0.9)(4.8,4.1)\multirput(1,4)(1,0){3}{$\ss\bullet$}\multirput(1,3)(1,0){4}{$\ss\bullet$}\multirput(2,2)(1,0){3}{$\ss\bullet$}\multirput(3,1)(1,0){2}{$\ss\bullet$}
\psline[linewidth=0.5pt](1,4)(1,3)(4,3)\psline[linewidth=0.5pt](2,4)(2,2)(4,2)\psline[linewidth=0.5pt](3,4)(3,1)(4,1)
\multirput(1,3)(1,0){3}{\psdots[dotstyle=triangle*,dotscale=1.1](0,0.5)}
\multirput(3,1)(0,1){3}{\psdots[dotstyle=triangle*,dotscale=1.1,dotangle=90](0.5,0)}
\psdots[dotstyle=triangle*,dotscale=1.1](2,2.5)
\psdots[dotstyle=triangle*,dotscale=1.1,dotangle=180](3,2.5)(3,1.5)
\psdots[dotstyle=triangle*,dotscale=1.1,dotangle=90](2.5,2)
\psdots[dotstyle=triangle*,dotscale=1.1,dotangle=-90](1.5,3)(2.5,3)
\rput(4.4,1.9){,}\endpspicture
\pspicture(0.9,0.9)(4.5,4.1)\multirput(1,4)(1,0){3}{$\ss\bullet$}\multirput(1,3)(1,0){4}{$\ss\bullet$}\multirput(2,2)(1,0){3}{$\ss\bullet$}\multirput(3,1)(1,0){2}{$\ss\bullet$}
\psline[linewidth=0.5pt](1,4)(1,3)(4,3)\psline[linewidth=0.5pt](2,4)(2,2)(4,2)\psline[linewidth=0.5pt](3,4)(3,1)(4,1)
\multirput(1,3)(1,0){3}{\psdots[dotstyle=triangle*,dotscale=1.1](0,0.5)}
\multirput(3,1)(0,1){3}{\psdots[dotstyle=triangle*,dotscale=1.1,dotangle=90](0.5,0)}
\psdots[dotstyle=triangle*,dotscale=1.1](3,2.5)
\psdots[dotstyle=triangle*,dotscale=1.1,dotangle=180](2,2.5)(3,1.5)
\psdots[dotstyle=triangle*,dotscale=1.1,dotangle=90](2.5,3)
\psdots[dotstyle=triangle*,dotscale=1.1,dotangle=-90](1.5,3)(2.5,2)\endpspicture}\right\},\end{gather}
where the elements of each set are written in an order corresponding to that
used for the DSASMs in~\eqref{DSASM123}, and
the element of $\SVC(7)$ which corresponds to the $7\times7$ DSASM
in~\eqref{DSASMexample} is
\begin{equation}\label{configex}\raisebox{-17mm}{\pspicture(0.8,0.8)(8.2,8.2)
\multirput(1,8)(1,0){7}{$\ss\bullet$}\multirput(1,7)(1,0){8}{$\ss\bullet$}\multirput(2,6)(1,0){7}{$\ss\bullet$}\multirput(3,5)(1,0){6}{$\ss\bullet$}
\multirput(4,4)(1,0){5}{$\ss\bullet$}\multirput(5,3)(1,0){4}{$\ss\bullet$}\multirput(6,2)(1,0){3}{$\ss\bullet$}\multirput(7,1)(1,0){2}{$\ss\bullet$}
\psline[linewidth=0.5pt](1,8)(1,7)(8,7)\psline[linewidth=0.5pt](2,8)(2,6)(8,6)\psline[linewidth=0.5pt](3,8)(3,5)(8,5)\psline[linewidth=0.5pt](4,8)(4,4)(8,4)
\psline[linewidth=0.5pt](5,8)(5,3)(8,3)\psline[linewidth=0.5pt](6,8)(6,2)(8,2)\psline[linewidth=0.5pt](7,8)(7,1)(8,1)
\multirput(1,7)(1,0){7}{\psdots[dotstyle=triangle*,dotscale=1.1](0,0.5)}
\multirput(7,1)(0,1){7}{\psdots[dotstyle=triangle*,dotscale=1.1,dotangle=90](0.5,0)}
\psdots[dotstyle=triangle*,dotscale=1.1](2,6.5)(3,6.5)(5,6.5)(6,6.5)(7,6.5)(3,5.5)(4,5.5)(6,5.5)(7,5.5)(4,4.5)(5,4.5)(6,4.5)(6,3.5)
\psdots[dotstyle=triangle*,dotscale=1.1,dotangle=180](4,6.5)(5,5.5)(7,4.5)(5,3.5)(7,3.5)(6,2.5)(7,2.5)(7,1.5)
\psdots[dotstyle=triangle*,dotscale=1.1,dotangle=-90](1.5,7)(2.5,7)(3.5,7)(4.5,6)(5.5,5)(6.5,5)(4.5,4)(5.5,3)
\psdots[dotstyle=triangle*,dotscale=1.1,dotangle=90](4.5,7)(5.5,7)(6.5,7)(2.5,6)(3.5,6)(5.5,6)(6.5,6)(3.5,5)(4.5,5)(5.5,4)(6.5,4)(6.5,3)(6.5,2)
\rput(8.6,4.5){.}
\endpspicture}\end{equation}

\psset{unit=7mm}
It follows immediately from~\eqref{bij} that for $A\in\DSASM(n)$ and $C\in\SVC(n)$ which correspond under the bijection,
the statistics~\eqref{RA}--\eqref{TA} can be expressed in terms of~$C$ as
\begin{align}\label{RC}R(A)&=\text{number of local configurations \Vv\ and \Vvi\ in }C,\\
\label{SC}S(A)&=\text{number of local configurations \Lup\ and \Ldown\ in }C,\\
\label{TC}T(A)&=\text{column number of the unique local configuration \Vv\ in the first row of }C.\end{align}

A simple property of any $C\in\SVC(n)$ is that
\begin{equation}\label{inout}\text{number of local configurations \Lout\ in }C=\text{number of local configurations \Lin\ in }C.\end{equation}
This follows by applying to~$C$ the general identity that, in any directed graph, the sum of indegrees over all vertices equals the sum of outdegrees
over all vertices (since each sum equals the number of edges).
In particular, in~$C$ the bulk vertices and the left boundary vertices
with local configurations \Lup\ or \Ldown\ each have equal indegree and outdegree, and so do not contribute to the identity.
Also, the contributions of the~$n$ top vertices (which each have indegree~1 and outdegree~0)
are canceled by the contributions of the~$n$ right boundary vertices (which each have indegree~0 and outdegree~1),
leaving only the contributions of the left boundary vertices with local configurations \Lout\ and \Lin, as in~\eqref{inout}.

By considering any fixed $k$ with $1\le k\le n-1$, applying the previous argument to the
induced subgraph of~$\G_n$ whose vertices consist of those vertices $(i,j)$ of~$\G_n$ with $j\ge i+k$,
and defining $V_{n,k}=\{(1,k+1),(2,k+2),\ldots,(n-k,n)\}$, it follows that for any $C\in\SVC(n)$,
\begin{multline}\label{inoutk}\text{number of local configurations \Wi\ in $C$ on vertices in }V_{n,k}\\
=\text{number of local configurations \Wii\ in $C$ on vertices in }V_{n,k}.\end{multline}
(Note that the subgraph used to obtain~\eqref{inoutk} includes isolated vertices, $(0,k)$ and $(n-k+\nolinebreak1,n+1)$, which could be removed.)

Using the bijection of~\eqref{bij}, it follows from~\eqref{inout} that the number of zero diagonal entries in a DSASM is even,
as already observed in Section~\ref{DSASMstats}.

\subsection{A result for $|\DSASM(n)|$ as a weighted sum over \texorpdfstring{$(n-1)\times(n-1)$}{(n-1)x(n-1)} DSASMs}\label{DSASMsumsect}
In the following proposition, some elementary results for $|\DSASM(n)|$ and $X_n(r,1,t)$ are obtained using simple combinatorial arguments,
some of which involve the bijection of~\eqref{bij}.
\begin{proposition}\label{diagrefprop}
For $n\ge 2$, the number of $n\times n$ DSASMs can be expressed as a weighted sum over all $(n-\nolinebreak1)\times(n-1)$ DSASMs as
\begin{equation}\label{diagref1}
|\DSASM(n)|=\sum_{A\in\DSASM(n-1)}2^{S(A)}.\end{equation}
Furthermore, this can be generalized to the DSASM generating function identity
\begin{equation}\label{diagref2}X_n(r,1,t)=t\,X_{n-1}(r,r+1,t)+t(1-t)\,X_{n-1}(r,1,1),\end{equation}
or equivalently
\begin{equation}\label{diagref3}X_n(r,1,t)=t\,X_{n-1}(r,r+1,t)+t(1-t)\,X_{n-2}(r,r+1,1),\end{equation}
for $n\ge3$,
where the equivalence between~\eqref{diagref2} and~\eqref{diagref3} follows
by setting $t=1$ and replacing~$n$ by $n-1$ in either equation to give $X_{n-1}(r,1,1)=X_{n-2}(r,r+1,1)$.
\end{proposition}
\begin{proof}
Proofs of~\eqref{diagref1}
and~\eqref{diagref3} will be obtained by introducing a certain set $\overline{\DSASM}(n)$, whose size
is naturally given by the RHS of~\eqref{diagref1}, and considering a bijection between $\overline{\DSASM}(n)$ and $\DSASM(n)$.

For any DSASM~$A$, let $\mathcal{P}(A)$ be the set of indices $i$ for which the diagonal entry~$A_{ii}$ of~$A$ is nonzero.
It immediately follows that $|\mathcal{P}(A)|=S(A)$.
For $n\ge2$, now define $\overline{\DSASM}(n)$ as
\begin{equation}\overline{\DSASM}(n)=\bigl\{(A,\Psi)\mid A\in\DSASM(n-1),\,\Psi\subseteq\mathcal{P}(A)\bigr\},\end{equation}
and define a function $\theta\colon\overline{\DSASM}(n)\to\DSASM(n)$, as follows.
For $(A,\Psi)\in\overline{\DSASM}(n)$, first construct an $(n-1)\times(n-1)$ matrix from $A$ by replacing each diagonal entry $A_{ii}$ of $A$ for which
$i\in\mathcal{P}(A)\setminus\Psi$ by zero,
while leaving all other entries of~$A$ unchanged.  Then, let $\theta(A,\Psi)$ be the $n\times n$ matrix
whose strictly upper and lower triangular parts are given by the upper and lower triangular parts, respectively, of the previous $(n-1)\times(n-1)$ matrix,
and whose main diagonal is obtained by requiring that the sum of entries in each row (or column) is~$1$.

For example, for $n=3$,
\begin{equation*}\overline{\DSASM}(3)=\bigl\{(\bigl(\begin{smallmatrix}1&0\\0&1\end{smallmatrix}\bigr),\emptyset),\,
\bigl(\bigl(\begin{smallmatrix}1&0\\0&1\end{smallmatrix}\bigr),\{1\}\bigr),\,
\bigl(\bigl(\begin{smallmatrix}1&0\\0&1\end{smallmatrix}\bigr),\{2\}\bigr),\,
\bigl(\bigl(\begin{smallmatrix}1&0\\0&1\end{smallmatrix}\bigr),\{1,2\}\bigl),\,
\bigl(\bigl(\begin{smallmatrix}0&1\\1&0\end{smallmatrix}\bigr),\emptyset\bigl)\bigr\},\end{equation*}
and
\begin{gather*}\theta\bigl(\bigl(\begin{smallmatrix}1&0\\0&1\end{smallmatrix}\bigr),\emptyset\bigr)=\Bigl(\begin{smallmatrix}1&0&0\\0&1&0\\0&0&1\end{smallmatrix}\Bigr),\quad
\theta\bigl(\bigl(\begin{smallmatrix}1&0\\0&1\end{smallmatrix}\bigr),\{1\}\bigr)=\Bigl(\begin{smallmatrix}0&1&0\\1&0&0\\0&0&1\end{smallmatrix}\Bigr),\quad
\theta\bigl(\bigl(\begin{smallmatrix}1&0\\0&1\end{smallmatrix}\bigr),\{2\}\bigr)=\Bigl(\begin{smallmatrix}1&0&0\\0&0&1\\0&1&0\end{smallmatrix}\Bigr),\\
\theta\bigl(\bigl(\begin{smallmatrix}1&0\\0&1\end{smallmatrix}\bigr),\{1,2\}\bigr)=\Bigl(\begin{smallmatrix}0&1&0\\1&-1&1\\0&1&0\end{smallmatrix}\Bigr),\quad
\theta\bigl(\bigl(\begin{smallmatrix}0&1\\1&0\end{smallmatrix}\bigr),\emptyset\bigr)=\Bigl(\begin{smallmatrix}0&0&1\\0&1&0\\1&0&0\end{smallmatrix}\Bigr).
\end{gather*}

The fact that $\theta$ is well-defined (i.e., that $\theta(A,\Psi)\in\DSASM(n)$, for each $(A,\Psi)\in\overline{\DSASM}(n)$) can be verified using the six-vertex model
configurations introduced in Section~\ref{sixvertexmodelconfig}.  For $(A,\Psi)\in\overline{\DSASM}(n)$, let~$A$ correspond to $C\in\SVC(n-1)$ under the bijection of~\eqref{bij}.
Then, for each $i=1,\ldots,n-1$, extend the local configuration in ~$C$ at the left boundary vertex $(i,i)$ of the graph $\G_{n-1}$ (as shown in~\eqref{Tn}) according to
\psset{unit=7mm}
\begin{align*}\Lup&\mapsto\begin{cases}\raisebox{-3.2mm}{\pspicture(-0.1,-0.1)(1.1,1.1)\rput(0.5,0.5){\psline[linewidth=0.7pt,linecolor=red](-0.5,0)(0,0)(0,-0.5)\psline[linewidth=0.7pt](0,0.5)(0,0)(0.5,0)}
\rput(0.5,0.5){$\ss\bullet$}\psdots[dotstyle=triangle*,dotscale=1](0.5,0.8)\psdots[dotstyle=triangle,dotscale=1,dotangle=180,fillcolor=red](0.5,0.2)
\psdots[dotstyle=triangle,dotscale=1,dotangle=-90,fillcolor=red](0.2,0.5)\psdots[dotstyle=triangle*,dotscale=1,dotangle=90](0.8,0.5)\endpspicture},&i\in\Psi,\\
\raisebox{-3.2mm}{\pspicture(-0.1,-0.1)(1.1,1.1)\rput(0.5,0.5){\psline[linewidth=0.7pt,linecolor=red](-0.5,0)(0,0)(0,-0.5)\psline[linewidth=0.7pt](0,0.5)(0,0)(0.5,0)}\rput(0.5,0.5){$\ss\bullet$}
\psdots[dotstyle=triangle*,dotscale=1](0.5,0.8)\psdots[dotstyle=triangle,dotscale=1,fillcolor=red](0.5,0.2)
\psdots[dotstyle=triangle*,dotscale=1,dotangle=90](0.8,0.5)\psdots[dotstyle=triangle,dotscale=1,dotangle=90,fillcolor=red](0.2,0.5)\endpspicture},&i\in\mathcal{P}(A)\setminus\Psi,\end{cases}&
\Ldown&\mapsto\begin{cases}\raisebox{-3.2mm}{\pspicture(-0.1,-0.1)(1.1,1.1)\rput(0.5,0.5){\psline[linewidth=0.7pt,linecolor=red](-0.5,0)(0,0)(0,-0.5)\psline[linewidth=0.7pt](0,0.5)(0,0)(0.5,0)}
\rput(0.5,0.5){$\ss\bullet$}\psdots[dotstyle=triangle*,dotscale=1,dotangle=180](0.5,0.8)\psdots[dotstyle=triangle,dotscale=1,fillcolor=red](0.5,0.2)
\psdots[dotstyle=triangle,dotscale=1,dotangle=90,fillcolor=red](0.2,0.5)\psdots[dotstyle=triangle*,dotscale=1,dotangle=-90](0.8,0.5)\endpspicture},&i\in\Psi,\\
\raisebox{-3.2mm}{\pspicture(-0.1,-0.1)(1.1,1.1)\rput(0.5,0.5){\psline[linewidth=0.7pt,linecolor=red](-0.5,0)(0,0)(0,-0.5)\psline[linewidth=0.7pt](0,0.5)(0,0)(0.5,0)}\rput(0.5,0.5){$\ss\bullet$}
\psdots[dotstyle=triangle*,dotscale=1,dotangle=180](0.5,0.8)\psdots[dotstyle=triangle,dotscale=1,dotangle=180,fillcolor=red](0.5,0.2)
\psdots[dotstyle=triangle*,dotscale=1,dotangle=-90](0.8,0.5)\psdots[dotstyle=triangle,dotscale=1,dotangle=-90,fillcolor=red](0.2,0.5)\endpspicture},&i\in\mathcal{P}(A)\setminus\Psi,\end{cases}\\[1.4mm]
\Lout&\mapsto\raisebox{-3.2mm}{\pspicture(-0.1,-0.1)(1.1,1.1)\rput(0.5,0.5){\psline[linewidth=0.7pt,linecolor=red](-0.5,0)(0,0)(0,-0.5)\psline[linewidth=0.7pt](0,0.5)(0,0)(0.5,0)}
\rput(0.5,0.5){$\ss\bullet$}\psdots[dotstyle=triangle*,dotscale=1](0.5,0.8)\psdots[dotstyle=triangle,dotscale=1,fillcolor=red](0.5,0.2)
\psdots[dotstyle=triangle*,dotscale=1,dotangle=-90](0.8,0.5)\psdots[dotstyle=triangle,dotscale=1,dotangle=-90,fillcolor=red](0.2,0.5)\endpspicture},&
\Lin&\mapsto\raisebox{-3.2mm}{\pspicture(-0.1,-0.1)(1.1,1.1)\rput(0.5,0.5){\psline[linewidth=0.7pt,linecolor=red](-0.5,0)(0,0)(0,-0.5)\psline[linewidth=0.7pt](0,0.5)(0,0)(0.5,0)}
\rput(0.5,0.5){$\ss\bullet$}\psdots[dotstyle=triangle*,dotscale=1,dotangle=180](0.5,0.8)\psdots[dotstyle=triangle,dotscale=1,dotangle=180,fillcolor=red](0.5,0.2)
\psdots[dotstyle=triangle*,dotscale=1,dotangle=90](0.8,0.5)\psdots[dotstyle=triangle,dotscale=1,dotangle=90,fillcolor=red](0.2,0.5)\endpspicture},\end{align*}
where the orientated edges which are added by the extension are shown in red.
Also, add an up-oriented vertical edge at the top left and a left-oriented horizontal edge at the bottom right of the entire six-vertex model configuration.
It can be seen (after relabeling each vertex $(i,j)$ as~$(i,j+1)$) that this process gives a six-vertex model configuration $\widehat{C}\in\SVC(n)$,
and that $\widehat{C}$ corresponds to $\theta(A,\Psi)$ under the bijection of~\eqref{bij}, thereby confirming that $\theta(A,\Psi)\in\DSASM(n)$.

For example, let $(A,\Psi)=\bigl(\bigl(\begin{smallmatrix}1&0\\0&1\end{smallmatrix}\bigr),\{2\}\bigr)$. Then \psset{unit=4mm}
$C=\raisebox{-3.4mm}{\pspicture(0.6,0.9)(3.3,3.1)\multirput(1,3)(1,0){2}{$\ss\bullet$}\multirput(1,2)(1,0){3}{$\ss\bullet$}\multirput(2,1)(1,0){2}{$\ss\bullet$}
\psline[linewidth=0.5pt](1,3)(1,2)(3,2)\psline[linewidth=0.5pt](2,3)(2,1)(3,1)
\multirput(1,2)(1,0){2}{\psdots[dotstyle=triangle*,dotscale=1.1](0,0.5)}\multirput(2,1)(0,1){2}{\psdots[dotstyle=triangle*,dotscale=1.1,dotangle=90](0.5,0)}
\psdots[dotstyle=triangle*,dotscale=1.1](2,1.5)\psdots[dotstyle=triangle*,dotscale=1.1,dotangle=90](1.5,2)\endpspicture}$
and
$\widehat{C}=\raisebox{-5.7mm}{\pspicture(0.6,0.6)(4.3,4.3)\psline[linewidth=0.7pt](2,4)(2,3)(4,3)\psline[linewidth=0.7pt](3,4)(3,2)(4,2)\psline[linewidth=0.7pt,linecolor=red](1,3)(2,3)(2,2)(3,2)(3,1)
\psline[linewidth=0.7pt,linecolor=blue](1,4)(1,3)\psline[linewidth=0.7pt,linecolor=blue](3,1)(4,1)
\multirput(1,4)(1,0){3}{$\ss\bullet$}\multirput(1,3)(1,0){4}{$\ss\bullet$}\multirput(2,2)(1,0){3}{$\ss\bullet$}\multirput(3,1)(1,0){2}{$\ss\bullet$}
\multirput(2,3)(1,0){2}{\psdots[dotstyle=triangle*,dotscale=1.1](0,0.5)}
\multirput(3,2)(0,1){2}{\psdots[dotstyle=triangle*,dotscale=1.1,dotangle=90](0.5,0)}
\psdots[dotstyle=triangle,dotscale=1.1,fillcolor=red](2,2.5)\psdots[dotstyle=triangle*,dotscale=1.1](3,2.5)\psdots[dotstyle=triangle,dotscale=1.1,dotangle=180,fillcolor=red](3,1.5)
\psdots[dotstyle=triangle,dotscale=1.1,dotangle=90,fillcolor=red](1.5,3)\psdots[dotstyle=triangle*,dotscale=1.1,dotangle=90](2.5,3)\psdots[dotstyle=triangle,dotscale=1.1,dotangle=-90,fillcolor=red](2.5,2)
\psdots[dotstyle=triangle,dotscale=1.1,fillcolor=blue](1,3.5)\psdots[dotstyle=triangle,dotscale=1.1,dotangle=90,fillcolor=blue](3.5,1)\endpspicture}$,
where the orientated edges which extend the local configurations in~$C$ at the two left boundary vertices are shown in red, and
the up-oriented vertical edge and left-oriented horizontal edge which are added at the top left and bottom right, respectively, are shown in blue.

\psset{unit=7mm}
It can be shown that $\theta$ is bijective by observing that it has a well-defined inverse function $\theta^{-1}\colon\DSASM(n)\to\overline{\DSASM}(n)$, as follows.
If $\theta^{-1}(\widehat{A})=(A,\Psi)$, for $\widehat{A}\in\DSASM(n)$, then
$A$ is the $(n-1)\times(n-1)$ matrix whose strictly upper triangular part is the part of~$\widehat{A}$ strictly above the superdiagonal (i.e., $A_{ij}=\widehat{A}_{i,j+1}$ for $1\le i<j\le n-1$),
whose strictly lower triangular part is the part of~$\widehat{A}$ strictly below the subdiagonal (i.e., $A_{ij}=\widehat{A}_{i+1,j}$ for $1\le j<i\le n-1$),
and whose main diagonal is obtained by requiring that the sum of entries in each row (or column) is~$1$,
while~$\Psi$ is the set of indices $i$ for which the superdiagonal entry~$\widehat{A}_{i,i+1}$
(or subdiagonal entry~$\widehat{A}_{i+1,i}$) of~$\widehat{A}$ is nonzero.
Alternatively, in terms of six-vertex model configurations, if~$\widehat{A}$ and~$A$ correspond to $\widehat{C}\in\SVC(n)$ and $C\in\SVC(n-1)$, respectively, under the bijection of~\eqref{bij},
then~$C$ is obtained from~$\widehat{C}$ by simply deleting, in~$\widehat{C}$, the vertices $(0,1)$, $(1,1)$, $(2,2)$, \ldots, $(n,n)$, $(n,n+1)$ of~$\G_n$
and all their incident oriented edges, while $\Psi$ is obtained from~$\widehat{C}$ as those indices~$i$ for which $\widehat{C}_{i,i+1}$ is \Vv\ or \Vvi.

The validity of~\eqref{diagref1} (which is the $r=t=1$ case of~\eqref{diagref2} and~\eqref{diagref3}) now follows immediately from the bijectivity of~$\theta$, since
\begin{equation*}|\DSASM(n)|=|\overline{\DSASM}(n)|=\sum_{A\in\DSASM(n-1)}2^{|\mathcal{P}(A)|}=\sum_{A\in\DSASM(n-1)}2^{S(A)}.\end{equation*}

Proceeding to the proof of~\eqref{diagref3}, it follows from the definition of $\theta$ that, for any $(A,\Psi)\in\overline{\DSASM}(n)$,
\begin{equation}\label{RTprop}R\bigl(\theta(A,\Psi)\bigr)=R(A)+|\Psi|\quad\text{and}\quad T\bigl(\theta(A,\Psi)\bigr)=
\begin{cases}1,&A_{11}=1\text{ and }1\notin\Psi,\\T(A)+1,&\text{otherwise}.\end{cases}\end{equation}
Therefore, for $n\ge3$,
\begin{align*}X_n(r,1,t)&=\sum_{A\in\DSASM(n)}r^{R(A)}\,t^{T(A)}=\sum_{(A,\Psi)\in\overline{\DSASM}(n)}r^{R(\theta(A,\Psi))}\,t^{T(\theta(A,\Psi))}\\
&=\sum_{A\in\DSASM(n-1)}\sum_{\Psi\subseteq\mathcal{P}(A)}r^{R(\theta(A,\Psi))}\,t^{T(\theta(A,\Psi))}\\
&=\sum_{A\in\DSASM(n-1)}\sum_{\Psi\subseteq\mathcal{P}(A)}r^{R(A)+|\Psi|}\,t^{T(A)+1}\:+\,
\sum_{\substack{A\in\DSASM(n-1)\\A_{11}=1}}\sum_{\substack{\Psi\subseteq\mathcal{P}(A)\\1\notin\Psi}}r^{R(A)+|\Psi|}\,(t-t^2)\\
&=t\!\sum_{A\in\DSASM(n-1)}r^{R(A)}\,(r+1)^{|\mathcal{P}(A)|}\,t^{T(A)}\:+\:t(1-t)\!\sum_{A\in\DSASM(n-1,1)}r^{R(A)}\,(r+1)^{|\mathcal{P}(A)|-1}\\
&=t\!\sum_{A\in\DSASM(n-1)}r^{R(A)}\,(r+1)^{S(A)}\,t^{T(A)}\:+\:t(1-t)\!\sum_{A\in\DSASM(n-1,1)}r^{R(A)}\,(r+1)^{S(A)-1}\\
&=t\,X_{n-1}(r,r+1,t)+t(1-t)(r+1)^{-1}[t^1]\,X_{n-1}(r,r+1,t)\bigr)\\
&=t\,X_{n-1}(r,r+1,t)+t(1-t)\,X_{n-2}(r,r+1,1),\end{align*}
where the fourth equality uses~\eqref{RTprop}, the fifth equality uses the identity $\sum_{I\subseteq J}x^{|I|}=(1+x)^{|J|}$ for any finite set~$J$, and the
last equality uses~(i$'$) of Proposition~\ref{refprop}, thereby confirming~\eqref{diagref3}, as required.
\end{proof}

Note that, by restricting the domain of the function $\theta$ to $\bigl\{(A,\emptyset)\mid A\in\DSASM(n-1)\bigr\}$,
where~$\theta$ is as defined in the proof of Proposition~\ref{diagrefprop}, it follows that $\DSASM(n-1)$ is naturally
in bijection with $\{A\in\DSASM(n)\mid A_{i,i+1}=0\text{ for }i=1,\ldots,n-1\}$, i.e., the set of $n\times n$ DSASMs whose superdiagonal (or subdiagonal) entries are all zero.

Note also that the $t=1$ case of~\eqref{diagref2} or~\eqref{diagref3} for TSASMs is
obtained by Li\'{e}nardy and Walmsley Hagendorf~\cite[Eq.~(204)]{HagLie21,HagLie22},~\cite[Eq.~(9.63)]{Lie20},~\cite[Prop.~4.4]{LieWal25}.

\subsection{Vertex weights}\label{vertweights}
Returning to the general consideration of the case of the six-vertex model which is related to DSASMs,
for each possible local configuration $c$ at a bulk or left boundary vertex of the graph~$\G_n$ in~\eqref{Tn},
and for an indeterminate~$u$, assign a weight $W(c,u)$, as given in Table~\ref{weights}
(where each left boundary weight $W(c,u)$ appears in the same row as the bulk
weight whose local configuration restricts to~$c$).  In these weights, the notation $\bar{x}$ is
defined in~\eqref{not},~$\sigma(x)$ and $\sigmah(x)$ are functions defined in~\eqref{not} and~\eqref{sigmah},~$q$,~$\alpha$,~$\beta$,~$\gamma$
and~$\delta$ are arbitrary constants (i.e., independent of~$u$), and~$\phi(u)$ is an
arbitrary function of~$u$. As will be discussed in Section~\ref{partitionfun}, it will be convenient to
assume that~$\phi(u)$ is not identically zero.

\begin{table}[h]\centering
$\begin{array}{|@{\;\;}l|@{\;\;}l|}\hline\rule{0ex}{3.5ex}
\text{Bulk weights}&\text{Left boundary weights}\\[2.2mm]
\hline\rule{0ex}{3.6ex}
W(\Wv,u)=1&W(\WLup,u)=\bigl(\alpha\,q\,u+\beta\,\q\,\u\bigr)\,\phi(u)\\[3.1mm]
W(\Wvi,u)=1&W(\WLdown,u)=\bigl(\alpha\,\q\,\u+\beta\,q\,u\bigr)\,\phi(u)\\[2.8mm]
W(\Wi,u)=\sigmah(q^2u)&W(\WLout,u)=\gamma\,\sigma(q^2u^2)\,\phi(u)\\[4.6mm]
W(\Wii,u)=\sigmah(q^2u)&W(\WLin,u)=\delta\,\sigma(q^2u^2)\,\phi(u)\\[4.6mm]
W(\Wiii,u)=\sigmah(q^2\u)&\\[4.6mm]
W(\Wiv,u)=\sigmah(q^2\u)&\\[5mm]\hline
\end{array}$\\[2.4mm]\caption{Bulk and left boundary weights.}\label{weights}
\end{table}

A fundamental property of the weights in Table~\ref{weights} is that the bulk weights satisfy
the Yang--Baxter equation, and the bulk and left boundary weights together satisfy the reflection equation.
For general information regarding the Yang--Baxter equation for the six-vertex model, see, for example, Baxter~\cite[pp.~187--190]{Bax82}.
The reflection equation was first introduced by Cherednik~\cite[Eq.~(10)]{Che84}, with certain important consequences of the
equation being first identified by Sklyanin~\cite{Skl88}.
The forms of the Yang--Baxter and reflection equations which will be used in this paper are given by Behrend, Fischer and
Konvalinka~\cite[Eq.~(47) \& first eq.~of Eq.~(48)]{BehFisKon17}.
The left boundary weights of Table~\ref{weights} are the most general weights which,
together with the bulk weights of Table~\ref{weights}, satisfy the reflection
equation.  The case of these weights with $\gamma=\delta=0$ was obtained by Cherednik~\cite[Thm.~2]{Che84},
and reformulated by Sklyanin~\cite[Eq.~(29)]{Skl88}.
The general case of the left boundary weights was first obtained, in forms slightly different from the form used here,
by de~Vega and Gonz\'{a}lez-Ruiz~\cite[Eq.~(15)]{DevGon93}, and Ghoshal and
Zamolodchikov~\cite[Eq.~(5.12)]{GhoZam94}.
For a derivation of the left boundary weights in the form used here, see Ayyer, Behrend and Fischer~\cite[Thm.~3.2]{AyyBehFis20}
(in which the constants $\alpha$, $\beta$, $\gamma$, $\delta$ and function~$\phi(u)$ appear as $\beta_\mathrm{L}$,~$\gamma_\mathrm{L}$,~$\delta_\mathrm{L}$,~$\alpha_\mathrm{L}$
and $f(u)/\sigma(q^2)$, respectively).
Various special cases of these weights have previously been used in the
enumeration of certain classes of ASMs, for example by Kuperberg~\cite{Kup02}, Behrend, Fischer and Konvalinka~\cite{BehFisKon17},
and Ayyer, Behrend and Fischer~\cite{AyyBehFis20}.

\subsection{The DSASM partition function}\label{partitionfun}
For $C\in\SVC(n)$ and indeterminates $u_1,\ldots,u_n$, weights
are defined for each vertex of~$\G_n$, as follows.
The weight of each top and right boundary vertex is~$1$, the weight of bulk
vertex $(i,j)$ is $W(C_{ij},u_iu_j)$,
and the weight of left boundary vertex $(i,i)$ is $W(C_{ii},u_i)$,
where $C_{ij}$ is the local configuration of $C$ at $(i,j)$.
Note that~$q$, $\alpha$, $\beta$, $\gamma$ and $\delta$ are the same in all
of these weights.

The weight of configuration $C$ is defined to be the product of its vertex weights over all vertices of~$\G_n$,
and the DSASM partition function $Z_n(u_1,\ldots,u_n)$ is defined to be the
sum of configuration weights, over all
configurations in $\SVC(n)$. Hence,
\begin{equation}\label{Z}Z_n(u_1,\ldots,u_n)=\sum_{C\in\SVC(n)}\,\prod_{i=1}^n
W(C_{ii},u_i)\,\prod_{1\le i<j\le n}W(C_{ij},u_iu_j).\end{equation}

The assignment of parameters in the vertex weights can be illustrated, for
$n=4$, as
\psset{unit=10mm}
\begin{equation}\label{col}\raisebox{-21.8mm}{\pspicture(0.8,-0.3)(5.2,4.2)
\psline[linewidth=0.8pt,linecolor=blue](1,4)(1,3)(5,3)\psline[linewidth=0.8pt,linecolor=green](2,4)(2,2)(5,2)
\psline[linewidth=0.8pt,linecolor=red](3,4)(3,1)(5,1)\psline[linewidth=0.8pt,linecolor=brwn](4,4)(4,0)(5,0)
\multirput(1,4)(1,0){4}{$\ss\bullet$}\multirput(1,3)(1,0){5}{$\ss\bullet$}\multirput(2,2)(1,0){4}{$\ss\bullet$}
\multirput(3,1)(1,0){3}{$\ss\bullet$}\multirput(4,0)(1,0){2}{$\ss\bullet$}
\rput[tr](1.05,2.9){$\ss\ui$}\rput[tr](1.95,2.9){$\ss\ui\uii$}\rput[tr](2.95,2.9){$\ss\ui\uiii$}\rput[tr](3.95,2.9){$\ss\ui\uiv$}
\rput[tr](2.05,1.9){$\ss\uii$}\rput[tr](2.95,1.9){$\ss\uii\uiii$}\rput[tr](3.95,1.9){$\ss\uii\uiv$}
\rput[tr](3.05,0.9){$\ss\uiii$}\rput[tr](3.95,0.9){$\ss\uiii\uiv$}\rput[tr](4.05,-0.1){$\ss\uiv$}\rput(5.6,2){.}\endpspicture}\end{equation}
The colors in~\eqref{col} indicate that $u_i$
can be naturally associated with the edges in row~$i$ and column~$i$ of~$\G_n$, and that the parameter for a
left boundary weight is the single parameter associated with the incident edges,
and the parameter for a bulk weight is the product of the two parameters associated with the incident edges.

It follows from~\eqref{Z} that if the function $\phi(u)$, which appears in
each left boundary weight, is identically zero, then $Z_n(u_1,\ldots,u_n)$ is also identically zero.
In the rest of the paper, this trivial case will be excluded by assuming that $\phi(u)$ is not identically zero.

As examples, the $n=1$, $2$ and $3$ cases of the DSASM partition function are
\begin{align}\notag Z_1(u_1)&=(\alpha qu_1+\beta\q\u_1)\,\phi(u_1),\\
\notag Z_2(u_1,u_2)&=\bigl((\alpha qu_1+\beta\q\u_1)\,(\alpha qu_2+\beta\q\u_2)\,\sigmah(q^2\u_1\u_2)+\gamma\,\sigma(q^2u_1^2)\,\delta\,\sigma(q^2u_2^2)\bigr)\phi(u_1)\,\phi(u_2),\\
\notag Z_3(u_1,u_2,u_3)&=\bigl((\alpha qu_1+\beta\q\u_1)\,(\alpha qu_2+\beta\q\u_2)\,(\alpha qu_3+\beta\q\u_3)\,\sigmah(q^2\u_1\u_2)\,\sigmah(q^2\u_1\u_3)\,\sigmah(q^2\u_2\u_3)\\
\notag&\quad\qquad\mbox{}+(\alpha qu_1+\beta\q\u_1)\,\gamma\,\sigma(q^2u_2^2)\,\delta\,\sigma(q^2u_3^2)\,\sigmah(q^2\u_1\u_2)\,\sigmah(q^2\u_1\u_3)\\
\notag&\quad\qquad\mbox{}+\gamma\,\sigma(q^2u_1^2)\,\delta\,\sigma(q^2u_2^2)\,(\alpha qu_3+\beta\q\u_3)\,\sigmah(q^2\u_1\u_3)\,\sigmah(q^2\u_2\u_3)\\
\notag&\quad\qquad\mbox{}+\gamma\,\sigma(q^2u_1^2)\,(\alpha qu_2+\beta\q\u_2)\,\delta\,\sigma(q^2u_3^2)\,\sigmah(q^2u_1u_2)\,\sigmah(q^2u_2u_3)\\
\label{Z123}&\qquad\qquad\mbox{}+ \gamma\,\sigma(q^2u_1^2)\,(\alpha\q\u_2+\beta qu_2)\,\delta\,\sigma(q^2u_3^2)\,\sigmah(q^2\u_1\u_3)\bigr)\phi(u_1)\,\phi(u_2)\,\phi(u_3),\end{align}
where the terms on each RHS are written in an order corresponding to that
used for the six-vertex model configurations in~\eqref{SV3}.

Note that, due to~\eqref{inout}, $Z_n(u_1,\ldots,u_n)$ as a function of~$\gamma$ and~$\delta$, is a polynomial in~$\gamma\,\delta$.

In order to relate the DSASM partition function~\eqref{Z} to the DSASM
generating function~\eqref{X}, which will be done in Section~\ref{ZXrelation},
the constants $\alpha$, $\beta$, $\gamma$, $\delta$ and function~$\phi(u)$ in Table~\ref{weights} will be set to
\begin{equation}\label{boundassig}\alpha=\beta=s,\quad\gamma=\delta=\frac{1}{\sigma(q)},\quad\phi(u)=\frac{1}{qu+\q\u},\end{equation}
for an indeterminate~$s$.
The DSASM partition function with the assignments of~\eqref{boundassig} will be referred
to as the specialized DSASM partition function and denoted~$\Z_n(u_1,\ldots,u_n)$, i.e.,
\begin{equation}\label{ZZ}\Z_n(u_1,\ldots,u_n)=Z_n(u_1,\ldots,u_n)|_{\alpha=\beta=s,\;\gamma=\delta=1/\sigma(q),\;\phi(u)=1/(qu+\q\u)}.\end{equation}
With the assignments of~\eqref{boundassig}, the left boundary weights are
\begin{equation}\label{leftW}W(\WLup,u)=W(\WLdown,u)=s,\quad W(\WLout,u)=W(\WLin,u)=\frac{\sigma(qu)}{\sigma(q)}.\end{equation}

\subsection{Properties of the DSASM partition function}\label{GenProp}
Some important properties of the DSASM partition function $Z_n(u_1,\ldots,u_n)$ will
now be identified in the following five propositions.

\begin{proposition}\label{ZLaurent}
The function $Z_n(u_1,\ldots,u_n)/\phi(u_1)$ is a Laurent polynomial in~$u_1$,
which is even in~$u_1$ for~$n$ even and odd in~$u_1$ for~$n$ odd,
and which (unless it is identically zero) has lower degree in~$u_1$ at
least~$-n$ and upper degree in~$u_1$ at most~$n$.
\end{proposition}
Note that the definitions of degrees being used here are that
a Laurent polynomial $\sum_{i=m}^na_ix^i$, with $m\le n$, $a_m\ne0$ and
$a_n\ne0$,
has lower and upper degrees in $x$ of $m$ and $n$, respectively (and that
degrees for the zero function are not defined).
\begin{proof}
Consider $C\in\SVC(n)$, and let $j$ be the column number of the~1 in the first row
of the DSASM $A$ which corresponds to $C$ under the bijection of~\eqref{bij}, i.e., $j=T(A)$.
The weight of $C$ consists of a product of vertex weights, among which
one left boundary weight and~$n-1$ bulk weights depend on~$u_1$.
Using the six-vertex rule and the upward orientations of all the top edges of~$\G_n$,
this product of the $u_1$-dependent vertex weights is
\begin{equation*}\begin{cases}W(\WLup,u_1)\,\prod_{i=2}^nW(\Wiii,u_1u_i),&j=1,\\
W(\WLout,u_1)\,\bigl(\prod_{i=2}^{j-1}W(\Wi,u_1u_i)\bigr)\,W(\Wv,u_1u_j)\,\bigl(\prod_{i=j+1}^n
W(\Wiii,u_1u_i)\bigr),&2\le j\le n.\end{cases}\end{equation*}
Considering the explicit vertex weights, as given in Table~\ref{weights},
and assuming that $\alpha$, $\beta$ and $\gamma$ are nonzero,
it can be seen that $W(\WLup,u_1)/\phi(u_1)$, $W(\Wiii,u_1u_i)$ and $W(\Wi,u_1u_i)$
are each odd Laurent polynomials in $u_1$ of lower and upper degrees~$-1$ and~$1$, respectively,
$W(\WLout,u_1)/\phi(u_1)$ is an even Laurent polynomial in $u_1$
of lower and upper degrees~$-2$ and~$2$, respectively, and $W(\Wv,u_1u_i)=1$.
Hence, the weight of $C$ divided by $\phi(u_1)$ is a Laurent polynomial in $u_1$
which is even in $u_1$ for $n$ even and odd in $u_1$ for $n$ odd,
and which has lower and upper degrees in $u_1$ of~$-n$ and~$n$, respectively.
The cases in which~$\alpha$,~$\beta~$ or~$\gamma$ are zero can be dealt with
easily, and the required result now follows by summing over all $C\in\SVC(n)$, as
in~\eqref{Z}.
\end{proof}

\begin{proposition}\label{Z1red}
Suppose that $u_1^2=\pm\q^2$.  Then, for $n\ge2$, the DSASM partition function satisfies
\begin{equation}\label{Z1redeq}Z_n(u_1,u_2,\ldots,u_n)=Z_1(u_1)\,Z_{n-1}(u_2,\ldots,u_n)\,\prod_{i=2}^nW(\Wiii,u_1u_i).\end{equation}
\end{proposition}
\begin{proof}\psset{unit=7mm}
Let $u_1^2=\pm\q^2$, and consider $C\in\SVC(n)$.  Then, since the top left vertical edge of $\G_n$
is oriented upward, $C_{11}$ is either~\Lup\ or~\Lout, with $W(\WLout,u_1)=0$ since $u_1^2=\pm\q^2$
(assuming $\phi(\pm\q^2)$ is finite).
It now follows, using the upward orientations of all the top edges of~$\G_n$ and the six-vertex rule,
that the weight of~$C$ is zero unless $C_{11}=\Lup$ and $C_{12}=\ldots=C_{1n}=\Viii$,
for which the product of corresponding vertex weights is
$W(\WLup,u_1)\,\prod_{i=2}^nW(\Wiii,u_1u_i)$ with $W(\WLup,u_1)=Z_1(u_1)$.
Furthermore, the deletion of the appropriate vertices and edges at the top of $\G_n$
gives a six-vertex model configuration on~$\G_{n-1}$. The required result now follows by summing over all $C\in\SVC(n)$, as in~\eqref{Z}.
(Note that the process used in this proof is closely related to that used in the proof of~(i) of Proposition~\ref{refprop}.)
\end{proof}

\begin{proposition}\label{Z2red1}
Suppose that $u_1u_n=\pm q^2$.  Then, for $n\ge3$, the DSASM partition
function satisfies
\begin{equation}\label{Z2red1eq}Z_n(u_1,u_2,\ldots,u_{n-1},u_n)=Z_2(u_1,u_n)\,Z_{n-2}(u_2,\ldots,u_{n-1})\,\prod_{i=2}^{n-1}W(\Wi,u_1u_i)\,W(\Wii,u_iu_n).\end{equation}
\end{proposition}
\begin{proof}\psset{unit=7mm}
Let $u_1u_n=\pm q^2$, and consider $C\in\SVC(n)$.  Then, since the top right vertical and horizontal edges of $\G_n$
are oriented upward and leftward, respectively, $C_{1n}$ is either~\Wv\ or~\Wiii, with $W(\Wiii,u_1u_n)=0$ since $u_1u_n=\pm q^2$.
It now follows, using the upward orientations of all the top edges of $\G_n$,
the leftward orientations of all the rightmost edges of $\G_n$ and the six-vertex rule, that the weight of $C$ is zero unless $C_{1n}=\Wv$,
$C_{11}=\Lout$, $C_{12}=\ldots=C_{1,n-1}=\Wi$, $C_{2n}=\ldots=C_{n-1,n}=\Wii$ and $C_{nn}=\Lin$,
for which the product of corresponding vertex weights is $W(\WLout,u_1)\,W(\WLin,u_n)\,\prod_{i=2}^{n-1}W(\Wi,u_1u_i)\,W(\Wii,u_iu_n)$
with $W(\WLout,u_1)\,W(\WLin,u_n)=Z_2(u_1,u_n)$.
Furthermore, the deletion of the appropriate vertices and edges at the top
and right of $\G_n$ gives a six-vertex model configuration on~$\G_{n-2}$.
The required result now follows by summing over all $C\in\SVC(n)$, as in~\eqref{Z}.
(Note that the process used in this proof is closely related to that used in
the proof of~(iii) of Proposition~\ref{refprop}.)
\end{proof}

\begin{proposition}\label{Z2red2}
Suppose that $u_1u_2=\pm\q^2$. Then, for $n\ge3$, the DSASM partition
function satisfies
\begin{equation}\label{Z2red2eq}Z_n(u_1,u_2,u_3,\ldots,u_n)=Z_2(u_1,u_2)\,Z_{n-2}(u_3,\ldots,u_n)\,\prod_{i=3}^nW(\Wiii,u_1u_i)\,W(\Wiii,u_2u_i).\end{equation}
\end{proposition}
\begin{proof}
Let $u_1u_2=\pm\q^2$, and consider $C\in\SVC(n)$.  Then, since the second (from the left) top edge of $\G_n$
is oriented upward, $C_{12}$ is~\Wv,~\Wiii\ or~\Wi, with
$W(\Wi,u_1u_2)=0$ since $u_1u_2=\pm\q^2$.
The case $C_{12}=\Wi$ can thus be excluded from consideration, and it will be assumed that
$C_{12}=\Wv$ or $C_{12}=\Wiii$.  Then there exists $C'\in\SVC(n)$ with $C'_{12}=\Wiii$ or $C'_{12}=\Wv$,
$C'_{12}\ne C_{12}$ and $C'_{ij}=C_{ij}$ for all $(i,j)\notin\{(1,1),\,(1,2),\,(2,2)\}$.
Now the sum of the weights of~$C$ and~$C'$ is zero unless the edge between
$(2,2)$ and $(2,3)$ is directed leftward in~$C$ and~$C'$, since if this edge is directed rightward, then the sum of weights contains the factor
$W(\WLout,u_1)\,W(\Wv,u_1u_2)\,W(\WLdown,u_2)+W(\WLup,u_1)\,W(\Wiii,u_1u_2)\,W(\WLout,u_2)=
W(\WLout,u_1)\,W(\WLdown,u_2)\pm W(\WLup,u_1)\,W(\WLout,u_2)$, which can be
checked to be zero for $u_1u_2=\pm\q^2$.
It now follows, using the upward orientations of all the top edges of $\G_n$,
the leftward orientations of the top two rightmost edges of $\G_n$
and the six-vertex rule, that the sum of the weights of $C$ and $C'$ is zero
unless
$C_{13}=\ldots=C_{1n}=C_{23}=\ldots=C_{2n}=C'_{13}=\ldots=C'_{1n}=C'_{23}=\ldots=C'_{2n}=\Wiii$,
for which the sum of products of corresponding vertex weights, including the vertices $(1,1$), $(1,2)$ and $(2,2)$, is
$\bigl(W(\WLout,u_1)\,W(\WLin,u_2)\pm W(\WLup,u_1)\,W(\WLup,u_2)\bigr)\,\prod_{i=3}^nW(\Wiii,u_1u_i)\,W(\Wiii,u_2u_i)$
with $W(\WLout,u_1)\,W(\WLin,u_2)\pm W(\WLup,u_1)\,W(\WLup,u_2)=Z_2(u_1,u_2)$.
Furthermore, the deletion of the appropriate vertices and edges at the top of~$\G_n$
gives six-vertex model configurations on $\G_{n-2}$.
The required result now follows by summing over all suitable pairs $C,C'\in\SVC(n)$.
(Note that the boundary unitarity equation, which is satisfied by the left
boundary weights, could have been used for some of this proof. See, for example,~\cite[Eq.~(3.7)]{AyyBehFis20} and~\cite[Eq.~(49)]{BehFisKon17}
for certain cases and forms of this equation.)
\end{proof}

\begin{proposition}\label{Zsymm}
The DSASM partition function $Z_n(u_1,\ldots,u_n)$ is symmetric in $u_1,\ldots,u_n$.
\end{proposition}

\begin{proof}
This can be proved using the same process as used by Kuperberg~\cite[Lem.~11 \& Fig.~12]{Kup02} to show that the
partition function for even-order OSASMs is symmetric in its variables.
The method involves using the Yang--Baxter and reflection equations
(which, as discussed in Section~\ref{vertweights}, are satisfied by the bulk
and left boundary weights of Table~\ref{weights}) to show that $Z_n(u_1,\ldots,u_n)$ is
symmetric in $u_i$ and $u_{i+1}$, for $i=1,\ldots,n-1$.  Specifically, $W(\Wi,q^2\u_iu_{i+1})\allowbreak
Z(u_1,\ldots,u_n)$ is shown to equal $W(\Wi,q^2\u_iu_{i+1})\,Z(u_1,\ldots
u_{i-1},u_{i+1},u_i,u_{i+2},\ldots,u_n)$ by applying, in succession,
the vertical form of the Yang--Baxter equation~\cite[first eq.~of Eq.~(47)]{BehFisKon17} $i-1$ times,
the left form of the reflection equation~\cite[first eq.~of Eq.~(48)]{BehFisKon17} once,
and the horizontal form of the Yang--Baxter equation~\cite[second eq.~of Eq.~(47)]{BehFisKon17} $n-i-1$ times.
Using the same notation as in~\cite[proof of Prop.~12]{BehFisKon17} and the same color coding
as in~\eqref{col}, this process for the case $n=4$ and $i=2$ can be illustrated as
\psset{unit=0.9mm}
\begin{multline*}\pspicture(0,-1)(61,45)\rput[r](61,23){$W(\raisebox{-2mm}{\psset{unit=5.8mm}\pspicture(0.1,0)(1,1)
\rput(0.5,0.5){\psline[linewidth=0.7pt,linecolor=red](0,-0.5)(0,0.5)\psline[linewidth=0.7pt,linecolor=green](-0.5,0)(0.5,0)}
\rput(0.5,0.5){$\scriptscriptstyle\bullet$}
\psdots[dotstyle=triangle*,dotscale=0.8](0.5,0.85)\psdots[dotstyle=triangle*,dotscale=0.85](0.5,0.2)
\psdots[dotstyle=triangle*,dotscale=0.8,dotangle=-90](0.85,0.5)\psdots[dotstyle=triangle*,dotscale=0.85,dotangle=-90](0.2,0.5)
\endpspicture},q^2\ubii\uiii)\,Z(\ui,\uii,\uiii,\uiv)$}\endpspicture
\pspicture(0,-1)(12,45)\rput(6,23){$=$}\endpspicture
\pspicture(0,-1)(40,45)\psline[linewidth=0.8pt,linecolor=blue](0,40)(0,30)(40,30)
\psline[linewidth=0.8pt,linecolor=green,linearc=2](20,44)(20,38)(10,32)(10,20)
\psline[linewidth=0.8pt,linecolor=green](10,20)(40,20)
\psline[linewidth=0.8pt,linecolor=red,linearc=2](10,44)(10,38)(20,32)(20,10)
\psline[linewidth=0.8pt,linecolor=red](20,10)(40,10)
\psline[linewidth=0.8pt,linecolor=brwn](30,40)(30,0)(40,0)
\multirput(10,41)(10,0){2}{\psdots[dotstyle=triangle*,dotscale=1.2](0,0)}
\multirput(0,36)(30,0){2}{\psdots[dotstyle=triangle*,dotscale=1.2](0,0)}
\multirput(36,0)(0,10){4}{\psdots[dotstyle=triangle*,dotscale=1.2,dotangle=90](0,0)}
\multirput(10,44)(10,0){2}{$\ss\bullet$}\multirput(0,40)(15,-5){2}{$\ss\bullet$}
\multirput(0,40)(30,0){2}{$\ss\bullet$}\multirput(0,30)(10,0){5}{$\ss\bullet$}\multirput(10,20)(10,0){4}{$\ss\bullet$}
\multirput(20,10)(10,0){3}{$\ss\bullet$}\multirput(30,0)(10,0){2}{$\ss\bullet$}\rput[b](15,32){$\ss
v$}\endpspicture
\pspicture(0,-1)(18,45)\rput(9,23){$\stackrel{\mathrm{YBE}}{=}$}\endpspicture
\pspicture(0,-1)(40,41)\psline[linewidth=0.8pt,linecolor=blue](0,40)(0,30)(40,30)
\psline[linewidth=0.8pt,linecolor=green](20,40)(20,30)
\psline[linewidth=0.8pt,linecolor=green,linearc=2](20,30)(20,28)(10,22)(10,20)
\psline[linewidth=0.8pt,linecolor=green](10,20)(40,20)
\psline[linewidth=0.8pt,linecolor=red](10,40)(10,30)
\psline[linewidth=0.8pt,linecolor=red,linearc=2](10,30)(10,28)(20,22)(20,20)
\psline[linewidth=0.8pt,linecolor=red](20,20)(20,10)(40,10)
\psline[linewidth=0.8pt,linecolor=brwn](30,40)(30,0)(40,0)
\multirput(0,36)(10,0){4}{\psdots[dotstyle=triangle*,dotscale=1.2](0,0)}
\multirput(36,0)(0,10){4}{\psdots[dotstyle=triangle*,dotscale=1.2,dotangle=90](0,0)}\rput(15,25){$\ss\bullet$}
\multirput(0,40)(10,0){4}{$\ss\bullet$}\multirput(0,30)(10,0){5}{$\ss\bullet$}\multirput(10,20)(10,0){4}{$\ss\bullet$}
\multirput(20,10)(10,0){3}{$\ss\bullet$}\multirput(30,0)(10,0){2}{$\ss\bullet$}\rput[b](15,22){$\ss
v$}\endpspicture\\
\pspicture(0,-1)(10,44)\rput(3,23){$\stackrel{\mathrm{RE}}{=}$}\endpspicture
\pspicture(0,-1)(40,44)\psline[linewidth=0.8pt,linecolor=blue](0,40)(0,30)(40,30)
\psline[linewidth=0.8pt,linecolor=green](20,40)(20,10)
\psline[linewidth=0.8pt,linecolor=green,linearc=2](20,10)(22,10)(28,20)(30,20)
\psline[linewidth=0.8pt,linecolor=green](30,20)(40,20)
\psline[linewidth=0.8pt,linecolor=red](10,40)(10,20)(20,20)
\psline[linewidth=0.8pt,linecolor=red,linearc=2](20,20)(22,20)(28,10)(30,10)
\psline[linewidth=0.8pt,linecolor=red](30,10)(40,10)
\psline[linewidth=0.8pt,linecolor=brwn](30,40)(30,0)(40,0)
\multirput(0,36)(10,0){4}{\psdots[dotstyle=triangle*,dotscale=1.2](0,0)}
\multirput(36,0)(0,10){4}{\psdots[dotstyle=triangle*,dotscale=1.2,dotangle=90](0,0)}\rput(25,15){$\ss\bullet$}
\multirput(0,40)(10,0){4}{$\ss\bullet$}\multirput(0,30)(10,0){5}{$\ss\bullet$}\multirput(10,20)(10,0){4}{$\ss\bullet$}
\multirput(20,10)(10,0){3}{$\ss\bullet$}\multirput(30,0)(10,0){2}{$\ss\bullet$}\rput[l](26.5,15){$\ss
v$}\endpspicture
\pspicture(0,-1)(16,44)\rput(8,23){$\stackrel{\mathrm{YBE}}{=}$}\endpspicture
\pspicture(0,-1)(45,44)\psline[linewidth=0.8pt,linecolor=blue](0,40)(0,30)(40,30)
\psline[linewidth=0.8pt,linecolor=green](20,40)(20,10)
\psline[linewidth=0.8pt,linecolor=green,linearc=2](20,10)(32,10)(38,20)(44,20)
\psline[linewidth=0.8pt,linecolor=red](10,40)(10,20)
\psline[linewidth=0.8pt,linecolor=red,linearc=2](10,20)(32,20)(38,10)(44,10)
\psline[linewidth=0.8pt,linecolor=brwn](30,40)(30,0)(40,0)
\multirput(0,36)(10,0){4}{\psdots[dotstyle=triangle*,dotscale=1.2](0,0)}
\multirput(41,10)(0,10){2}{\psdots[dotstyle=triangle*,dotscale=1.2,dotangle=90](0,0)}
\multirput(36,0)(0,30){2}{\psdots[dotstyle=triangle*,dotscale=1.2,dotangle=90](0,0)}\rput(35,15){$\ss\bullet$}
\multirput(44,10)(0,10){2}{$\ss\bullet$}\multirput(0,40)(10,0){4}{$\ss\bullet$}\multirput(0,30)(10,0){5}{$\ss\bullet$}\multirput(10,20)(10,0){3}{$\ss\bullet$}
\multirput(20,10)(10,0){2}{$\ss\bullet$}\multirput(30,0)(10,0){2}{$\ss\bullet$}\rput[l](36.5,15){$\ss
v$}\endpspicture
\pspicture(0,-1)(12,44)\rput(6,23){$=$}\endpspicture
\pspicture(0,-1)(61,44)\rput[l](-1,23){$W(\raisebox{-2mm}{\psset{unit=5.8mm}\pspicture(0.1,0)(1,1)
\rput(0.5,0.5){\psline[linewidth=0.7pt,linecolor=red](0,-0.5)(0,0.5)\psline[linewidth=0.7pt,linecolor=green](-0.5,0)(0.5,0)}
\rput(0.5,0.5){$\scriptscriptstyle\bullet$}
\psdots[dotstyle=triangle*,dotscale=0.8](0.5,0.85)\psdots[dotstyle=triangle*,dotscale=0.85](0.5,0.2)
\psdots[dotstyle=triangle*,dotscale=0.8,dotangle=-90](0.85,0.5)\psdots[dotstyle=triangle*,dotscale=0.85,dotangle=-90](0.2,0.5)
\endpspicture},q^2\ubii\uiii)\,Z(\ui,\uiii,\uii,\uiv)$,}\endpspicture
\end{multline*}
where $v=q^2\ubii\uiii$, YBE indicates the use of the Yang--Baxter equation
and RE indicates the use of the reflection equation.
\end{proof}

Some remarks on generalizations of the results of this section are as follows.

First, it can be seen that if the DSASM partition function for $n=0$ is defined as $Z_0()=1$, then Proposition~\ref{Z1red}
is valid for $n\ge1$, and Propositions~\ref{Z2red1} and~\ref{Z2red2} are valid for $n\ge2$.

Second, it can be seen that Propositions~\ref{ZLaurent}--\ref{Z2red2} can each be generalized by applying Proposition~\ref{Zsymm}.
Specifically, the generalized form of Proposition~\ref{ZLaurent} involves
properties of $Z(u_1,\ldots,u_n)/\allowbreak\phi(u_i)$ as a function of~$u_i$, the generalized form of Proposition~\ref{Z1red} involves
a relation satisfied by $Z(u_1,\ldots,u_n)$ with $u_i^2=\pm\q^2$, and the generalized
forms of Propositions~\ref{Z2red1} and~\ref{Z2red2} involve relations satisfied by $Z(u_1,\ldots,u_n)$ with $u_iu_j=\pm q^2$ and
$u_iu_j=\pm\q^2$, respectively, for any $1\le i,j\le n$ with $i\ne j$.

\subsection{Pfaffians}\label{Pfaffians}
The definitions of Pfaffians of even-order skew symmetric matrices and of certain triangular arrays will now be given.

The Pfaffian of an empty matrix or empty triangular array is taken to be~1.
Now let $A$ be a $2n\times 2n$ skew-symmetric matrix, with $n$ positive. Then the Pfaffian of $A$ is defined as
\begin{equation}\label{Pfdef1}
\Pf A=\sum_{\{\{i_1,j_1\},\ldots,\{i_n,j_n\}\}}\text{sgn}(i_1j_1\ldots i_nj_n)\;A_{i_1,j_1}\ldots A_{i_n,j_n},\end{equation}
where the sum is over all $(2n-1)!!$ perfect matchings
$\{\{i_1,j_1\},\ldots,\{i_n,j_n\}\}$ of $\{1,\ldots,2n\}$ (i.e., all
partitions of $\{1,\ldots,2n\}$ into $n$ two-element subsets), and $(i_1j_1\ldots i_nj_n)$ is
the permutation $1\mapsto i_1$, $2\mapsto j_1$, \ldots, $2n-1\mapsto i_n$,
$2n\mapsto j_n$, in one-line notation.
Note that, for a given perfect matching of~$\{1,\ldots,2n\}$, there are
$2^n\,n!$ ways of choosing $i_1,j_1,\ldots,i_n,j_n$,
since the elements of each two-element subset can be ordered in~2 ways, and
the~$n$ two-element subsets can be ordered in $n!$ ways. However, $\Pf A$ is well-defined, since each summand on the RHS of~\eqref{Pfdef1} is independent of these choices.
In particular, if the ordering of elements within a two-element subset is changed, then the sign of the permutation and the
product of matrix entries both change sign, whereas if the ordering of the $n$ two-elements subsets is changed, then the
sign of the permutation and the product of matrix entries both remain unchanged.

Now let $(B_{ij})_{1\le i<j\le 2n}$ be a triangular array. Then the Pfaffian
$\Pf_{1\le i<j\le2n}(B_{ij})$, with $n$ positive, is defined to be
the Pfaffian of the unique skew-symmetric matrix whose strictly upper triangular part is given by~$B$.  Hence,
\begin{equation}\label{Pfdef2}\Pf_{1\le i<j\le2n}\bigl(B_{ij}\bigr)=\sum_{\substack{\{\{i_1,j_1\},\ldots,\{i_n,j_n\}\}\\i_1<j_1,\,\ldots,\,i_n<j_n}}
\text{sgn}(i_1j_1\ldots i_nj_n)\;B_{i_1,j_1}\ldots B_{i_n,j_n},\end{equation}
where the sum is over all perfect matchings
$\{\{i_1,j_1\},\ldots,\{i_n,j_n\}\}$ of $\{1,\ldots,2n\}$ with a choice in
which $i_1<j_1$, \ldots, $i_n<j_n$ are satisfied, and $(i_1j_1\ldots i_nj_n)$ has the same
meaning as in~\eqref{Pfdef1}.

Some properties of Pfaffians which follow easily from~\eqref{Pfdef1}
and~\eqref{Pfdef2}, and which will be used in the rest of this paper,
are as follows, for any $2n\times2n$ skew-symmetric matrix $A$ and triangular array $(B_{ij})_{1\le i<j\le 2n}$.
\begin{list}{(\roman{pf})}{\usecounter{pf}\setlength{\labelwidth}{7mm}\setlength{\leftmargin}{11mm}\setlength{\labelsep}{3mm}}
\item If, for some $i\ne j$, the skew-symmetric matrix $A_{i\leftrightarrow j}$ is obtained from $A$ by simultaneously interchanging rows $i$ and $j$ and
interchanging columns $i$ and $j$, then
\begin{equation}\label{Pfint}\Pf A_{i\leftrightarrow j}=-\Pf A.\end{equation}
The special case of~\eqref{Pfint} in which rows~$i$ and~$j$ of~$A$ are equal
(and hence columns~$i$ and~$j$ of~$A$ are also equal),
so that~\eqref{Pfint} gives $\Pf A=0$, will often be used.
\item For any $c_1,\ldots,c_{2n}$,
\begin{equation}\label{Pfmult}\Pf_{1\le i<j\le2n}\bigl(c_i\,c_j\,B_{ij}\bigr)=\prod_{i=1}^{2n}c_i\,\Pf_{1\le i<j\le2n}\bigl(B_{ij}\bigr).\end{equation}
The special cases of~\eqref{Pfmult} in which all of the $c_i$'s are equal, or
in which all but one of the~$c_i$'s are~1, will often be used.
\item The Pfaffian can be expanded along the first row of an array as
\begin{equation}\label{Pfredrow}
\Pf_{1\le i<j\le2n}(B_{ij})=\sum_{k=2}^{2n}(-1)^kB_{1k}\,\Pf_{\substack{2\le i<j\le2n\\i,j\ne k}}\bigl(B_{ij}\bigr)\end{equation}
or along the last column of an array as
\begin{equation}\label{Pfredcol}
\Pf_{1\le i<j\le2n}(B_{ij})=\sum_{k=1}^{2n-1}(-1)^{k+1}B_{k,2n}\,\Pf_{\substack{1\le i<j\le2n-1\\i,j\ne k}}\bigl(B_{ij}\bigr).\end{equation}
The special cases of~\eqref{Pfredrow} or~\eqref{Pfredcol} in which all but
one of the entries in the first row or last column are~0 will often be used.
\end{list}

Note that the identities~\eqref{Pfint}--\eqref{Pfredcol} are special cases of more general identities.
For example, the identity
\begin{equation}\label{PfMAMt}\Pf(Y^tAY)=\det Y\,\Pf A,\end{equation}
for any $2n\times2n$ skew-symmetric matrix~$A$ and $2n\times2n$ matrix~$Y$, gives~\eqref{Pfint}
if~$Y$ is the permutation matrix of a transposition, and gives~\eqref{Pfmult}
if $Y$ is the diagonal matrix with diagonal entries $c_1,\ldots,c_{2n}$,
while~\eqref{Pfredrow} and~\eqref{Pfredcol} can be generalized
(using~\eqref{Pfint}) to give expansions of $\Pf A$ along any row or column of a $2n\times2n$ skew-symmetric matrix~$A$.

\subsection{A Pfaffian expression for the DSASM partition function}\label{Pfaffianexpsect}
The main Pfaffian expression for the DSASM partition function will now be obtained in Theorem~\ref{ZPftheorem}. A noteworthy
feature of this result is that it applies to the most general
boundary weights which, together with the six-vertex model bulk weights, satisfy the reflection equation.
By contrast, all other currently-known determinant or Pfaffian expressions for boundary weight-dependent partition functions
associated with classes of ASMs (such as those in~\cite{AyyBehFis20},~\cite{BehFisKon17} or~\cite{Kup02})
apply only to special cases of these boundary weights, in which one or two of the parameters~$\alpha$,~$\beta$,~$\gamma$ and~$\delta$ (in Table~\ref{weights})
are set to zero.

\begin{theorem}\label{ZPftheorem}
The DSASM partition function satisfies
\begin{multline}\label{ZPf1}Z_n(u_1,\ldots,u_n)\\[-2mm]
=\prod_{1\le i<j\le n}\!\!\frac{\sigmah(q^2u_iu_j)\,\sigmah(q^2\u_i\u_j)}{\sigmah(u_i\u_j)}
\Pf_{\even(n)\le i<j\le n}\left(\begin{cases}Z_1(u_j),&i=0\\[1.5mm]
\displaystyle\frac{\sigmah(u_i\u_j)\,Z_2(u_i,u_j)}{\sigmah(q^2u_iu_j)\,\sigmah(q^2\u_i\u_j)},&i\ge1\end{cases}\,\right).\end{multline}
More explicitly, the DSASM partition function is given by
\begin{multline}\label{ZPf2}Z_n(u_1,\ldots,u_n)=\prod_{i=1}^n\phi(u_i)\prod_{1\le i<j\le n}\frac{\sigmah(q^2u_iu_j)\,\sigmah(q^2\u_i\u_j)}{\sigmah(u_i\u_j)}\\
\times\Pf_{\even(n)\le i<j\le n}\left(\begin{cases}\alpha qu_j+\beta\bar q\bar u_j,&i=0\\[1.7mm]
\sigmah(u_i\u_j)\bigl((\alpha qu_i+\beta\q\u_i)\,(\alpha qu_j+\beta\q\u_j)\,\sigmah(q^2\u_i\u_j)\\[0.9mm]
\qquad\qquad\quad+\gamma\,\delta\,\sigma(q^2u_i^2)\sigma(q^2u_j^2)\bigr)\big/\bigl(\sigmah(q^2u_iu_j)\,\sigmah(q^2\u_i\u_j)\bigr),&i\ge1\end{cases}\right).\end{multline}
\end{theorem}
Note that if $\alpha=\beta=0$, $\gamma=\bar{b}$, $\delta=b$ and
$\phi(u)=1/\sigma(q^2u^2)$, for an indeterminate~$b$, then the left boundary weights are those used by
Kuperberg~\cite[Fig.~15]{Kup02} for even-order OSASMs, i.e.,
$W(\WLup,u)=W(\WLdown,u)=0$, $W(\WLout,u)=\bar{b}$ and $W(\WLin,u)=b$,
and~\eqref{ZPf1} then gives the Pfaffian formula of~\cite[$Z_{\mathrm{O}}(n;\vec{x})$ case of Thm.~10]{Kup02} for the even-order OSASM partition function,
i.e., $Z_n(u_1,\ldots,u_n)=\prod_{1\le i<j\le n}\fracsden{\sigmah(q^2u_iu_j)\,\sigmah(q^2\u_i\u_j)}{\sigmah(u_i\u_j)}\,
\Pf_{1\le i<j\le n}\bigl(\fracsden{\sigmah(u_i\u_j)}{\sigmah(q^2u_iu_j)\,\sigmah(q^2\u_i\u_j)}\bigr)$
for $n$ even (and $Z_n(u_1,\ldots,u_n)=0$ for $n$ odd).

\psset{unit=7mm}
Note also that a result related to~\eqref{ZPf1} is obtained by Motegi~\cite[Thm.~3.3]{Mot18}.
The case considered in~\cite{Mot18} essentially involves six-vertex model configurations on a trapezoidal extension of~$\G_n$, with certain boundary conditions
and with the left boundary weights for local configurations \Lin\ being zero.

A detailed proof of Theorem~\ref{ZPftheorem} is as follows, and an alternative version of the proof is
briefly outlined at the end of this section.
\begin{proof} First observe that (assuming that $\phi(u)$ is given)
$Z_n(u_1,\ldots,u_n)$ for~$n$ even and $n\ge4$ is uniquely characterized by the
expression for $Z_2(u_1,u_2)$ in~\eqref{Z123} and the properties in
Propositions~\ref{ZLaurent},~\ref{Z2red1},~\ref{Z2red2} and~\ref{Zsymm},
where this can be shown as follows.  Consider~$n$ even and $n\ge4$, and
assume that $Z_{n-2}(u_1,\ldots,u_{n-2})$ is known.
Due to Proposition~\ref{ZLaurent}, $Z_n(u_1,\ldots,u_n)/\phi(u_1)$ is
determined by its values at $n+1$ distinct nonzero values of~$u_1^2$,
since the coefficient of $u_1^k$ in $Z_n(u_1,\ldots,u_n)/\phi(u_1)$ can be
nonzero for only the $n+1$ cases $k=-n,-n+2,\ldots,n-2,n$.
Propositions~\ref{Z2red1},~\ref{Z2red2} and~\ref{Zsymm} give
$Z_n(u_1,\ldots,u_n)$ at $2(n-1)$ values of $u_1^2$
(with $2(n-1)\ge n+1$, since $n\ge4$), specifically at
$u_1^2=q^4\u_i^2$ and $u_1^2=\q^4\u_i^2$, for $i=2,\ldots,n$.  The required conclusion now follows by induction on~$n$.

For $n$ even, define $Y_n(u_1,\ldots,u_n)$ as the RHS of~\eqref{ZPf1}.
It follows immediately that $Y_2(u_1,u_2)=Z_2(u_1,u_2)$.
Hence, the validity of~\eqref{ZPf1} for $n$ even and $n\ge4$ can
be confirmed by showing that $Y_n(u_1,\ldots,u_n)$ satisfies properties
corresponding to those
in Propositions~\ref{ZLaurent},~\ref{Z2red1},~\ref{Z2red2} and~\ref{Zsymm}.
Specifically, these properties can be taken as follows, with $n$ even in each case.
\begin{list}{(\roman{Ylist})}{\usecounter{Ylist}\setlength{\labelwidth}{7mm}\setlength{\leftmargin}{11mm}\setlength{\labelsep}{3mm}}
\item\label{YLaurent} $Y_n(u_1,\ldots,u_n)/\phi(u_1)$ is an even Laurent polynomial in~$u_1$
which (unless identically zero) has lower degree in $u_1$ at least~$-n$ and
upper degree in~$u_1$ at most~$n$.
\item\label{Y2red} For $n\ge4$ and $u_1u_2=q^{\pm2}$,
\begin{equation}\label{Y2redeq}
Y_n(u_1,u_2,u_3,\ldots,u_n)=Y_2(u_1,u_2)\,Y_{n-2}(u_3,\ldots,u_n)\,\prod_{i=3}^n\sigmah(q^{\pm2}u_1u_i)\,\sigmah(q^{\pm2}u_2u_i).\end{equation}
\item\label{Ysym} $Y_n(u_1,\ldots,u_n)$ is symmetric in $u_1,\ldots,u_n$.
\end{list}
Note that the $u_1u_2=q^2$ case of~(ii),
together with $W(\Wi,u)=W(\Wii,u)=\sigmah(q^2u)$
and an application of~(iii) to interchange~$u_2$ and~$u_n$, corresponds to the $u_1u_n=q^2$ case of Proposition~\ref{Z2red1}, and that the $u_1u_2=\q^2$ case of~(ii),
together with $W(\Wiii,u)=\sigmah(q^2\u)=-\sigmah(\q^2u)$,
corresponds to the $u_1u_2=\q^2$ case of Proposition~\ref{Z2red2}.

In order to verify~(i),
let $Y'_n(u_1,\ldots,u_n)=\fracs{\prod_{1\le i<j\le n}\sigmah(q^2u_iu_j)\,\sigmah(q^2\u_i\u_j)}
{\phi(u_1)\,\prod_{2\le i<j\le n}\sigmah(u_i\u_j)}\,\Pf_{1\le i<j\le n}\bigl(\fracsden{\sigmah(u_i\u_j)\,Z_2(u_i,u_j)}{\sigmah(q^2u_iu_j)\,\sigmah(q^2\u_i\u_j)}\bigr)$,
so that $Y_n(u_1,\ldots,u_n)=\fracsden{\phi(u_1)}{\prod_{j=2}^n\sigmah(u_1\u_j)}\,Y'_n(u_1,\ldots,u_n)$.
Then $Y'_n(u_1,\ldots,u_n)$ is a Laurent polynomial in~$u_1$ since each term
in the expansion~\eqref{Pfdef2} of
the Pfaffian contains a single $u_1$-dependent factor
$\fracsden{\sigmah(u_1\u_j)\,Z_2(u_1,u_j)}{\sigmah(q^2u_1u_j)\,\sigmah(q^2\u_1\u_j)}$, for some $2\le j\le n$, and the product of this factor
with $\frac{\sigmah(q^2u_1u_j)\,\sigmah(q^2\u_1\u_j)}{\phi(u_1)}$ from the prefactor to the Pfaffian
is a Laurent polynomial in~$u_1$.  If $u_1$ is set equal to $\pm u_j$ for any $2\le j\le n$, then in the skew-symmetric matrix
$\bigl(\fracsden{\sigmah(u_i\u_j)\,Z_2(u_i,u_j)}{\sigmah(q^2u_iu_j)\,\sigmah(q^2\u_i\u_j)}\bigr)_{1\le i,j\le n}$,
row~1 is $\pm\frac{\phi(u_1)}{\phi(\pm u_1)}$ times row~$j$ and column~1 is
$\pm\frac{\phi(u_1)}{\phi(\pm u_1)}$ times column~$j$, so that,
using~\eqref{Pfint} and~\eqref{Pfmult},
$Y'_n(u_1,\ldots,u_n)=0$.  Hence $Y'_n(u_1,\ldots,u_n)$ is divisible by
$\prod_{j=2}^n(u_1^2-u_j^2)=\prod_{j=2}^n(u_1u_j\,\sigma(u_1\u_j))$, and so $Y_n(u_1,\ldots,u_n)/\phi(u_1)$ is a Laurent polynomial in~$u_1$.
It can also easily be checked that $Y_n(u_1,\ldots,u_n)/\phi(u_1)$ is even in~$u_1$ and
has the required degree properties in $u_1$.

In order to verify~(ii),
consider $n\ge4$, write $\prod_{1\le i<j\le n}\fracsden{\sigmah(q^2u_iu_j)\,\sigmah(q^2\u_i\u_j)}{\sigmah(u_i\u_j)}=
\fracsden{\sigmah(q^2u_1u_2)\,\sigmah(q^2\u_1\u_2)}{\sigmah(u_1\bar u_2)}\allowbreak\times\allowbreak
\bigl(\prod_{i=1}^2\prod_{j=3}^n\fracsden{\sigmah(q^2u_iu_j)\,\sigmah(q^2\u_i\u_j)}{\sigmah(u_i\u_j)}\bigr)
\bigl(\prod_{3\le i<j\le n}\fracsden{\sigmah(q^2u_iu_j)\,\sigmah(q^2\u_i\u_j)}{\sigmah(u_i\u_j)}\bigr)$,
take out $\fracsden{\sigmah(q^2u_1u_2)\,\sigmah(q^2\u_1\u_2)}{\sigmah(u_1\bar u_2)}$, and use it
to multiply row~$1$ of the array in $\Pf_{1\le i<j\le n}\bigl(\fracsden{\sigmah(u_i\u_j)\,Z_2(u_i,u_j)}{\sigmah(q^2u_iu_j)\,\sigmah(q^2\u_i\u_j)}\bigr)$.
Now set $u_1u_2=q^{\pm2}$. Then row~$1$ of the array is $(Z_2(u_1,u_2),0,\ldots,0)$,
and $\prod_{i=1}^2\prod_{j=3}^n\fracsden{\sigmah(q^2u_iu_j)\,\sigmah(q^2\u_i\u_j)}{\sigmah(u_i\u_j)}=
\prod_{j=3}^n\sigmah(q^{\pm2}u_1u_j)\,\sigmah(q^{\pm2}u_2u_j)$.
Applying~\eqref{Pfredrow}
to the Pfaffian, and recalling that $Z_2(u_1,u_2)=Y_2(u_1,u_2)$, now gives~\eqref{Y2redeq} as required.

The validity of~(iii)
can be confirmed by observing that both $\prod_{1\le i<j\le n}\fracsden{\sigmah(q^2u_iu_j)\,\sigmah(q^2\u_i\u_j)}{\sigmah(u_i\u_j)}$
and $\Pf_{1\le i,j\le n}\bigl(\fracsden{\sigmah(u_i\u_j)\,Z_2(u_i,u_j)}{\sigmah(q^2u_iu_j)\,\sigmah(q^2\u_i\u_j)}\bigr)$
are antisymmetric in $u_1,\ldots,u_n$ (where this follows for the Pfaffian by using~\eqref{Pfint}).

The $n$ even case of~\eqref{ZPf1} has now been verified.
The $n$ odd case of~\eqref{ZPf1} can be obtained from the $n$ even case using Proposition~\ref{Z1red}.
In particular,~\eqref{Z1redeq} (with $n$ replaced by $n+1$, $u_1,\ldots,u_{n+1}$ replaced by $u_0,\ldots,u_n$, and~$u_0$ taken to be~$\q$) gives
$Z_n(u_1,\ldots,u_n)=\frac{Z_{n+1}(\q,u_1,\ldots,u_n)}{Z_1(\q)\,\prod_{i=1}^n\sigmah(q^3\u_i)}$.
Taking $n$ odd, using~\eqref{ZPf1}, performing simple cancellations in the prefactor
and in row~1 of the array for the Pfaffian, and multiplying row~1 of the array by~$-1$ gives
\begin{equation*}
Z_n(u_1,\ldots,u_n)=\frac{1}{Z_1(\q)}\prod_{1\le i<j\le n}\frac{\sigmah(q^2u_iu_j)\,\sigmah(q^2\u_i\u_j)}{\sigmah(u_i\u_j)}
\Pf_{0\le i<j\le n}\left(\begin{cases}\frac{Z_2(\q,u_j)}{\sigmah(q^3\u_j)},&i=0\\[1.5mm]
\frac{\sigmah(u_i\u_j)\,Z_2(u_i,u_j)}{\sigmah(q^2u_iu_j)\,\sigmah(q^2\u_i\u_j)},&i\ge1\end{cases}\,\right).\end{equation*}
Using~\eqref{Z1redeq} to write
$Z_2(\q,u_j)=Z_1(\q)\,Z_1(u_j)\,\sigmah(q^3\u_j)$ now gives the~$n$ odd case of~\eqref{ZPf1} as required.

Finally,~\eqref{ZPf2} can be easily be obtained from~\eqref{ZPf1} by using
the explicit expressions in~\eqref{Z123} for $Z_1(u_j)$ and $Z_2(u_i,u_j)$,
and using~\eqref{Pfmult} to remove $\prod_{i=1}^n\phi(u_i)$ from the Pfaffian.
\end{proof}

Note that~\eqref{ZPf1} can be proved using an alternative approach, in which it is observed that
$Z_n(u_1,\ldots,u_n)$ for all $n\ge2$ is uniquely characterized by the
expression for $Z_1(u_1)$ in~\eqref{Z123} and the properties in
Propositions~\ref{ZLaurent},~\ref{Z1red},~\ref{Z2red1} (or~\ref{Z2red2}) and~\ref{Zsymm}.
Then, defining $Y_n(u_1,\ldots,u_n)$ as the RHS of~\eqref{ZPf1} for all~$n$, it could be noted that $Y_1(u_1)=Z_1(u_1)$, and confirmed
that $Y_n(u_1,\ldots,u_n)$ satisfies properties corresponding to those in Propositions~\ref{ZLaurent},~\ref{Z1red},~\ref{Z2red1} (or~\ref{Z2red2})
and~\ref{Zsymm} for $n\ge2$.

Also note that Theorem~\ref{ZPftheorem} was recently obtained independently, in the context of the stochastic six-vertex model in half-space,
by Garbali, de Gier, Mead and Wheeler~\cite[Thm.~1.5, Thm.~4.8~\& Rem.~4.9]{GarDegMeaWhe25},~\cite[Thm.~2.28~\& Rem.~2.29]{Mea25}.  The bulk and left boundary weights used in~\cite{GarDegMeaWhe25,Mea25}
differ from those given in Table~\ref{weights} by certain
transformations, but when these transformations are taken into account, the result of~\cite{GarDegMeaWhe25,Mea25} matches that of Theorem~\ref{ZPftheorem}.  The proof in~\cite{GarDegMeaWhe25,Mea25} is
similar to that given here, in that it involves identifying certain properties (analogous to those in Propositions~\ref{ZLaurent}--\ref{Zsymm}) which uniquely characterize the DSASM partition function, and then
showing that these are also satisfied by a Pfaffian expression, although this verification is done in~\cite{GarDegMeaWhe25,Mea25} within the framework of shuffle products.

In~\cite{GarDegMeaWhe25,Mea25}, three further expressions for the order-$n$ DSASM partition function are also obtained.
The first involves a sum over all even-sized  subsets of $\{1,\ldots,n\}$~\cite[Thm.~4.5]{GarDegMeaWhe25},~\cite[Thm.~2.25]{Mea25},
the second involves shuffle products~\cite[Prop.~4.12]{GarDegMeaWhe25},~\cite[Prop.~2.32]{Mea25}, and
the third involves a sum of two terms, each consisting of a Pfaffian multiplied by a prefactor~\cite[Prop.~4.14]{GarDegMeaWhe25},~\cite[Thm.~2.37]{Mea25}.
As an example, the result of~\cite[Prop.~4.12]{GarDegMeaWhe25},~\cite[Prop.~2.32]{Mea25} gives
\begin{align}
\notag Z_{2n}(u_1,\ldots,u_n)&=\frac{1}{n!}\,\underbrace{Z_2\diamond\ldots\diamond Z_2}_n\,(u_1,\ldots,u_{2n}),\\
\label{shu}Z_{2n+1}(u_1,\ldots,u_n)&=\frac{1}{n!}\,Z_1\diamond\underbrace{Z_2\diamond\ldots\diamond Z_2}_n\,(u_1,\ldots,u_{2n+1}),
\end{align}
where the shuffle product $\diamond$ is defined as follows.
For symmetric rational functions~$f$ and~$g$ of~$m$ and~$n$ variables, respectively, $f\diamond g$
is a symmetric rational function of $m+n$ variables given (when expressed using the bulk weights of Table~\ref{weights}) by
\begin{equation}f\diamond g\,(u_1,\ldots,u_{m+n})=\sum_{I,J}f(u_I)\,g(u_J)\prod_{k\in I,\,l\in J}\frac{\sigmah(q^2u_ku_l)\,\sigmah(q^2\u_k\u_l)}{\sigmah(u_k\u_l)},\end{equation}
where the sum is over all pairs of disjoint subsets~$I=\{i_1,\ldots,i_m\}$ and~$J=\{j_1,\ldots,j_n\}$ of $\{1,\ldots,m+n\}$,
$f(u_I)$ denotes $f(u_{i_1},\ldots,u_{i_m})$ and $g(u_J)$ denotes $f(u_{j_1},\ldots,u_{j_n})$.  It can be checked that $\diamond$ is associative, so that
the repeated shuffle products in~\eqref{shu} are well-defined.

Finally, note that it has recently been shown by Fischer and H\"{o}ngesberg~\cite[Thms.~1.1~\&~1.3]{FisHon26} that the DSASM partition function, as given by Theorem~\ref{ZPftheorem},
appears within a generalized Littlewood identity, where certain transformations are again applied to the bulk and left boundary weights of Table~\ref{weights}.

\section{Proofs of the results of Section~\ref{mainresults}}\label{proofs}
In this section, the results of Section~\ref{mainresults} are proved,
where some parts of these proofs use results (in particular, the bijection of~\eqref{bij} and
Theorem~\ref{ZPftheorem}) from Section~\ref{sixvertexmodelDSASMs}.

In Section~\ref{ZXrelation}, the exact relationship between
the specialized DSASM partition function~\eqref{ZZ} and the DSASM generating function~\eqref{X} is given in Lemma~\ref{ZXlemma}.
In Section~\ref{Pfaffianid}, a general identity for evaluating certain multivariate Pfaffian expressions
when all the variables are set to zero is obtained in Theorem~\ref{Pfaffianidtheorem}.  Section~\ref{Pfaffianid}
also includes the derivation of a result, Corollary~\ref{Pfaffianidcoroll}, for the effects of certain transformations on Pfaffians whose
entries involve coefficients of power series.

In Sections~\ref{proofXrs}--\ref{proofXrsttheorem}, the proofs of
the Pfaffian formulae~\eqref{Xrs},~\eqref{numDSASM2} and~\eqref{Xrst} in Theorem~\ref{Xrstheorem}, Corollary~\ref{numDSASMcoroll} and Theorem~\ref{Xrsttheorem} are obtained
by combining Theorem~\ref{ZPftheorem}, Lemma~\ref{ZXlemma} and Theorem~\ref{Pfaffianidtheorem},
and the proofs of the explicit expressions~\eqref{Xrsentries},~\eqref{numDSASMentries} and~\eqref{Xrstentries} for the entries for these Pfaffians
are obtained using standard techniques for evaluating coefficients of power series.
Section~\ref{Xrsentriesproof} also includes the derivation of recurrence relations satisfied by the entries for the Pfaffian in Theorem~\ref{Xrstheorem}.
In Sections~\ref{proofXrsReform} and~\ref{proofXrstDiv}, the proofs of the alternative statements~\eqref{XrsReform}--\eqref{numDSASMReform}
and~\eqref{XrstDiv} of certain results are given.

\subsection{The connection between the specialized DSASM partition function~\eqref{ZZ} and DSASM generating function~\eqref{X}}\label{ZXrelation}
\begin{lemma}\label{ZXlemma}
The specialized DSASM partition function~\eqref{ZZ} and DSASM generating function~\eqref{X} are related by
\begin{multline}\label{ZX}\Z_n(z,\underbrace{1,\ldots,1}_{n-1})=\frac{\sigma(q^2)^{(n-1)(n-2)/2}\,\sigma(q^2\z)^{n-1}}{\sigma(q^4)^{n(n-1)/2}\,\sigma(q^2z)}\\
\times\biggl[\frac{(q+\q)\,\sigma(qz)\,\sigma(q^2\z)}{\sigma(q^2z)}\,X_n\biggl(\!q^2+\q^2,s,\frac{\sigma(q^2z)}{\sigma(q^2\z)}\biggr)
-s\,\sigma(z)\,X_{n-1}\bigl(q^2+\q^2,s,1\bigr)\biggr],\end{multline}
where $z$ is arbitrary.
\end{lemma}
Note that setting $z=s=q^2+\q^2=1$ (so that $q$ is a primitive 12th root of unity) in~\eqref{ZX} gives
$\Z_n(1,\ldots,1)|_{s=q^2+\q^2=1}=X_n(1,1,1)=|\DSASM(n)|$.
\begin{proof}\psset{unit=7mm}
Consider $A\in\DSASM(n)$ and $C\in\SVC(n)$ which correspond under the bijection of~\eqref{bij}.
Now $A_{11}=0$ or $A_{11}=1$.
If $A_{11}=0$, then the local configurations in row~$1$ of $\G_n$ are
$C_{11}=\Lout$, $C_{12}=\ldots=C_{1,T(A)-1}=\Vi$, $C_{1,T(A)}=\Vv$ and $C_{1,T(A)+1}=\ldots=C_{1n}=\Viii$, where $T$ is defined in~\eqref{TA}.
Furthermore, using~\eqref{SC}, the number of local configurations \Lup\ and \Ldown\ in~$C$ is~$S(A)$, and, using~\eqref{RC}, the
number of local configurations \Vi, \Vii, \Viii\ and \Viv\ in~$C$ in rows $2,\ldots,n$ of~$\G_n$ is $(n-1)(n-2)/2-R(A)+1$.
Using the bulk weights from Table~\ref{weights} and left boundary weights from~\eqref{leftW}, where the parameters in these weights
are~$z$ for vertices in row~$1$ of $\G_n$ and $1$ for vertices in rows $2,\ldots,n$ of~$\G_n$,
it follows that the weight of~$C$ for this case is
\begin{equation*}s^{S(A)}\:\frac{\sigma(qz)}{\sigma(q)}\:\sigmah(q^2z)^{T(A)-2}\:\sigmah(q^2\z)^{n-T(A)}\:\sigmah(q^2)^{(n-1)(n-2)/2-R(A)+1}.\end{equation*}
If $A_{11}=1$, then $C_{11}=\Lup$ and $C_{12}=\ldots=C_{1n}=\Viii$, and the
weight (with parameters as before) of $C$ for this case is
\begin{equation*}s^{S(A)}\:\sigmah(q^2\z)^{n-1}\:\sigmah(q^2)^{(n-1)(n-2)/2-R(A)}.\end{equation*}

Using~\eqref{Z} and~\eqref{ZZ}, regarding $\{A\in\DSASM(n)\mid A_{11}=0\}$ as
$\DSASM(n)\setminus\{A\in\DSASM(n)\mid A_{11}=1\}$ and noting that $T(A)=1$ for $A_{11}=1$, it follows that
\begin{multline*}\Z_n(z,\underbrace{1,\ldots,1}_{n-1})
=\sum_{A\in\DSASM(n)}s^{S(A)}\:\frac{\sigma(qz)}{\sigma(q)}\:\sigmah(q^2z)^{T(A)-2}\:\sigmah(q^2\z)^{n-T(A)}\:\sigmah(q^2)^{(n-1)(n-2)/2-R(A)+1}\\
+\sum_{\substack{A\in\DSASM(n)\\A_{11}=1}}\Bigl(-s^{S(A)}\:\frac{\sigma(qz)}{\sigma(q)}\:\sigmah(q^2z)^{-1}\:\sigmah(q^2\z)^{n-1}\:\sigmah(q^2)^{(n-1)(n-2)/2-R(A)+1}\\[-4mm]
+s^{S(A)}\:\sigmah(q^2\z)^{n-1}\:\sigmah(q^2)^{(n-1)(n-2)/2-R(A)}\Bigr).\end{multline*}
Using the bijection from the proof of~(i) of Proposition~\ref{refprop} to replace the sum over
$A\in\DSASM(n)$ with $A_{11}=1$ by a sum over $A'\in\DSASM(n-1)$, noting that under this bijection $R(A)=R(A')$ and $S(A)=S(A')+1$,
performing some straightforward simplifications and applying the definition~\eqref{X}, now gives~\eqref{ZX} as required.
\end{proof}

\subsection{Identities for multivariate Pfaffian expressions}\label{Pfaffianid}
A general identity for evaluating certain multivariate Pfaffian expressions with all of the variables set to zero
will now be obtained in Theorem~\ref{Pfaffianidtheorem}, and a consequence of this identity
will then be identified in Corollary~\ref{Pfaffianidcoroll}.
Some special cases of Theorem~\ref{Pfaffianidtheorem} have appeared in a similar context in, for
example, the work of Hagendorf and Morin--Duchesne~\cite[Sec.~4.2.2]{HagMor16}, and Rosengren~\cite[Lem.~3.12]{Ros16}, and
determinant versions of certain cases have appeared in a similar context in,
for example, the work of Behrend, Di Francesco and Zinn-Justin~\cite[Eqs.~(43)--(47) \& Eq.~(79)]{BehDifZin12}.
However, functions which correspond to~$h(x)$ and $k(x)$ in Theorem~\ref{Pfaffianidtheorem} are
not included in any of these cases.  The inclusion of~$h(x)$ and~$k(x)$ enables certain useful transformations to be performed
on the functions which appear in the entries for the initial multivariate Pfaffian expression. Some comments on the use of
such a transformation, for a specific~$h(x)$, will be given at the end of Section~\ref{proofXrs}.
Determinant versions of certain cases of Corollary~\ref{Pfaffianidcoroll} have appeared in,
for example, the work of Gessel and Xin~\cite[Rules, p.~14]{GesXin06}.

\begin{theorem}\label{Pfaffianidtheorem}
For integers $m$ and $n$ with $0\le m\le2n$, power series $f(x,y)$,
$g_{m+1}(x)$, \ldots, $g_{2n}(x)$, $h(x)$ and $k(x)$ with $f(x,y)$ antisymmetric in~$x$ and~$y$, and an array $(a_{ij})_{m<i<j\le2n}$,
\begin{multline}\label{HomogPf}
\left.\raisebox{-4.6ex}{$\displaystyle\frac{\rule[-5.2ex]{0ex}{0ex}\displaystyle\Pf_{1\le i<j\le2n}\left(\begin{cases}k(x_i)\,k(x_j)\,f\bigl(h(x_i)\,x_i,h(x_j)\,x_j\bigr),&j\le m\\
k(x_i)\,g_j\bigl(h(x_i)\,x_i\bigr),&i\le m<j\\
a_{ij},&i>m\end{cases}\right)}{\prod_{1\le i<j\le m}(x_j-x_i)}$}\right|_{x_1,\ldots,x_m\rightarrow0}\\
=k(0)^m\,h(0)^{m(m-1)/2}\Pf_{1\le i<j\le2n}\left(\begin{cases}[u^{i-1}v^{j-1}]\,f(u,v),&j\le m\\
[u^{i-1}]\,g_j(u),&i\le m<j\\a_{ij},&i>m\end{cases}\right).\end{multline}\end{theorem}
\begin{proof}
The $m=0$ case of~\eqref{HomogPf} is trivial, with both sides being simply
$\Pf_{1\le i<j\le2n}(a_{ij})$,
and the $m=1$ case of~\eqref{HomogPf} can also be checked easily.  Hence,
in the rest of this proof it will be assumed that $m\ge2$.

Note that the LHS of~\eqref{HomogPf} is well-defined due to the following argument.  Let $P$ be the Pfaffian
on the LHS of~\eqref{HomogPf}, and $M=(M_{ij})_{1\le i,j\le2n}$ be the skew-symmetric matrix whose strictly upper triangular part is the array for this Pfaffian.
Due to the definition~\eqref{Pfdef2} of the Pfaffian and the fact that $f(x,y)$, $g_j(x)$, $h(x)$ and $k(x)$ are power series,~$P$ is a power series in $x_1,\ldots,x_m$.
Relevant entries of~$M$ are $M_{ij}=k(x_i)\,k(x_j)\,f\bigl(h(x_i)\,x_i,h(x_j)\,x_j\bigr)$ for $i,j\le m$ (where this depends on the antisymmetry of $f(x,y)$ in~$x$ and~$y$),
$M_{ij}=k(x_i)\,g_j\bigl(h(x_i)\,x_i\bigr)$ for $i\le m<j$, and $M_{ij}=-k(x_j)\,g_i\bigl(h(x_j)\,x_j\bigr)$ for $j\le m<i$.
It follows that if $x_i$ and $x_j$ are set equal for
$i\ne j$, then rows~$i$ and~$j$ of~$M$ are equal and columns~$i$ and~$j$ of~$M$ are equal, which implies,
using~\eqref{Pfint}, that~$P$ is zero in this case.
Therefore, $P$ is divisible by $\prod_{1\le i<j\le m}(x_j-x_i)$, and so
the quotient is a power series in $x_1,\ldots,x_m$ with a well-defined value at $x_1=\ldots=x_m=0$.

The case of~\eqref{HomogPf} with $h(x)=1$ and $k(x)=1$, i.e.,
\begin{multline}\label{HomogPf1}
\left.\raisebox{-4.6ex}{$\displaystyle\frac{\rule[-5.2ex]{0ex}{0ex}\displaystyle\Pf_{1\le i<j\le2n}\left(\begin{cases}f(x_i,x_j),&j\le m\\
g_j(x_i),&i\le m<j\\a_{ij},&i>m\end{cases}\right)}{\prod_{1\le i<j\le m}(x_j-x_i)}$}\right|_{x_1,\ldots,x_m\rightarrow0}\\
=\Pf_{1\le i<j\le2n}\left(\begin{cases}[u^{i-1}v^{j-1}]\,f(u,v),&j\le m\\
[u^{i-1}]\,g_j(u),&i\le m<j\\a_{ij},&i>m\end{cases}\right),\end{multline}
will be proved first, and this case will subsequently be used in the proof of the general case.

Let
\begin{equation}\label{N}N(x_1,\ldots,x_m)=\Pf_{1\le i<j\le2n}\left(\begin{cases}f(x_i,x_j),&j\le m\\g_j(x_i),&i\le m<j\\a_{ij},&i>m\end{cases}\right)\end{equation}
and
\begin{equation}\label{Q}Q(x_1,\ldots,x_m)=\frac{N(x_1,\ldots,x_m)}{\prod_{1\le i<j\le m}(x_j-x_i)}.\end{equation}
Then (as already shown for the more general case of terms on the LHS of~\eqref{HomogPf})
$N(x_1,\ldots,x_m)$ and $Q(x_1,\ldots,x_m)$ are power series.

Now take the coefficient of $x_1^0\cdots x_m^{m-1}$ on each side of
$Q(x_1,\ldots,x_m)\prod_{1\le i<j\le m}(x_j-x_i)=N(x_1,\ldots,x_m)$, i.e.,
\begin{equation}\label{coeff}
\textstyle[x_1^0\cdots x_m^{m-1}]\bigl(Q(x_1,\ldots,x_m)\prod_{1\le i<j\le
m}(x_j-x_i)\bigr)=[x_1^0\cdots x_m^{m-1}]\,N(x_1,\ldots,x_m).
\end{equation}
Using the facts that $\prod_{1\le i<j\le m}(x_j-x_i)$ is a homogeneous polynomial of degree $m(m-1)/2$, that
$x_1^0\cdots x_m^{m-1}$ has degree $m(m-1)/2$ and that $[x_1^0\cdots x_m^{m-1}]\prod_{1\le i<j\le m}(x_j-x_i)=1$, the LHS of~\eqref{coeff} is
\begin{align}\notag&\hspace{-6mm}\textstyle[x_1^0\cdots x_m^{m-1}]\bigl(Q(x_1,\ldots,x_m)\prod_{1\le i<j\le m}(x_j-x_i)\bigr)\\
\notag&\hspace{63mm}\textstyle=Q(0,\ldots,0)\:[x_1^0\cdots x_m^{m-1}]\prod_{1\le i<j\le m}(x_j-x_i)\\
\notag&\hspace{63mm}=Q(0,\ldots,0)\\
\label{coeffLHS}&\hspace{63mm}=\left.\frac{N(x_1,\ldots,x_m)}{\prod_{1\le i<j\le m}(x_j-x_i)}\right|_{x_1,\ldots,x_m\rightarrow0}.\end{align}
The RHS of~\eqref{coeff} is, using the same notation as in~\eqref{Pfdef2},
\begin{align}
\notag&[x_1^0\cdots x_m^{m-1}]\,N(x_1,\ldots,x_m)\\
\notag&\qquad=[x_1^0\cdots x_m^{m-1}]\left(\sum_{\substack{\{\{i_1,j_1\},\ldots,\{i_n,j_n\}\}\\i_1<j_1,\,\ldots,\,i_n<j_n}}
\text{sgn}(i_1j_1\ldots i_nj_n)\prod_{p=1}^n\left(\begin{cases}f(x_{i_p},x_{j_p}),&j_p\le m\\g_{j_p}(x_{i_p}),&i_p\le m<j_p\\a_{i_p,j_p},&i_p>m\end{cases}\right)\right)\\
\notag&\qquad=\sum_{\substack{\{\{i_1,j_1\},\ldots,\{i_n,j_n\}\}\\i_1<j_1,\,\ldots,\,i_n<j_n}}
\text{sgn}(i_1j_1\ldots i_nj_n)\prod_{p=1}^n\left(\begin{cases}[x_{i_p}^{i_p-1}x_{j_p}^{j_p-1}]\,f(x_{i_p},x_{j_p}),&j_p\le m\\
[x_{i_p}^{i_p-1}]\,g_{j_p}(x_{i_p}),&i_p\le m<j_p\\a_{i_p,j_p},&i_p>m\end{cases}\right)\\
\notag&\qquad=\sum_{\substack{\{\{i_1,j_1\},\ldots,\{i_n,j_n\}\}\\i_1<j_1,\,\ldots,\,i_n<j_n}}
\text{sgn}(i_1j_1\ldots i_nj_n)\prod_{p=1}^n\left(\begin{cases}[u^{i_p-1}v^{j_p-1}]\,f(u,v),&j_p\le m\\
[u^{i_p-1}]\,g_{j_p}(u),&i_p\le m<j_p\\a_{i_p,j_p},&i_p>m\end{cases}\right)\\
\label{coeffRHS}&\qquad=\Pf_{1\le i<j\le2n}\left(\begin{cases}[u^{i-1}v^{j-1}]\,f(u,v),&j\le m\\
[u^{i-1}]\,g_j(u),&i\le m<j\\a_{ij},&i>m\end{cases}\right).\end{align}
Combining~\eqref{coeff}--\eqref{coeffRHS} now gives~\eqref{HomogPf1}.

The general case of~\eqref{HomogPf}, i.e., with $h(x)$ and $k(x)$ arbitrary, will now be proved.
Applying~\eqref{Pfmult} (with $c_i=k(x_i)$ for $i=1,\ldots,m$, and $c_i=1$ for $i=m+1,\ldots,2n$)
on the LHS of~\eqref{HomogPf}, and noting that $x_1,\ldots,x_m$ can
immediately be set to zero in the resulting factor $\prod_{i=1}^m k(x_i)$, the LHS of~\eqref{HomogPf} becomes
\begin{equation*}
\left.\raisebox{-4.6ex}{$\displaystyle k(0)^m\:\frac{\rule[-5.2ex]{0ex}{0ex}\displaystyle\Pf_{1\le i<j\le2n}\left(\begin{cases}f\bigl(h(x_i)\,x_i,h(x_j)\,x_j\bigr),&j\le m\\
g_j\bigl(h(x_i)\,x_i\bigr),&i\le m<j\\
a_{ij},&i>m\end{cases}\right)}{\prod_{1\le i<j\le m}(x_j-x_i)}$}\right|_{x_1,\ldots,x_m\rightarrow0}.\end{equation*}
This can be written, using the definitions~\eqref{N} and~\eqref{Q}, as
\begin{equation*}k(0)^m\,Q\bigl(h(x_1)\,x_1,\ldots,h(x_m)\,x_m\bigr)
\left.\prod_{1\le i<j\le m}\frac{h(x_j)\,x_j-h(x_i)\,x_i}{x_j-x_i}\right|_{x_1,\ldots,x_m\rightarrow0},\end{equation*}
which, observing that $Q\bigl(h(x_1)\,x_1,\ldots,h(x_m)\,x_m\bigr)$ is a power series, becomes
\begin{equation*}k(0)^m\,Q(0,\ldots,0)
\left.\prod_{1\le i<j\le m}\frac{h(x_j)\,x_j-h(x_i)\,x_i}{x_j-x_i}\right|_{x_1,\ldots,x_m\rightarrow0}.\end{equation*}
Applying~\eqref{HomogPf1} (whose LHS is $Q(0,\ldots,0)$), and using the easily-verified fact that
\begin{equation*}\left.\frac{h(x_j)\,x_j-h(x_i)\,x_i}{x_j-x_i}\right|_{x_i,x_j\rightarrow0}\!=h(0),\end{equation*}
now gives the RHS of~\eqref{HomogPf} as required.\end{proof}

A useful corollary of Theorem~\ref{Pfaffianidtheorem} is as follows.

\begin{corollary}\label{Pfaffianidcoroll}
For integers $m$ and $n$ with $0\le m\le2n$, power series $f(u,v)$,
$g_{m+1}(u)$, \ldots, $g_{2n}(u)$, $h(u)$ and $k(u)$ with $f(u,v)$ antisymmetric in~$u$ and~$v$, and an array $(a_{ij})_{m<i<j\le2n}$,
\begin{multline}\label{HomogPf2}
\Pf_{1\le i<j\le2n}\left(\begin{cases}[u^{i-1}v^{j-1}]\,\bigl(k(u)\,k(v)\,f(h(u)\,u,h(v)\,v)\bigr),&j\le m\\
[u^{i-1}]\,\bigl(k(u)\,g_j(h(u)\,u)\bigr),&i\le m<j\\a_{ij},&i>m\end{cases}\right)\\
=k(0)^m\,h(0)^{m(m-1)/2}\Pf_{1\le i<j\le2n}\left(\begin{cases}[u^{i-1}v^{j-1}]\,f(u,v),&j\le m\\
[u^{i-1}]\,g_j(u),&i\le m<j\\a_{ij},&i>m\end{cases}\right).\end{multline}
\end{corollary}
This result states that removing all occurrences of $h(x)$ and $k(x)$ from the Pfaffian on the LHS of~\eqref{HomogPf2} simply causes the
appearance of the prefactor $k(0)^m\,h(0)^{m(m-1)/2}$.

Corollary~\ref{Pfaffianidcoroll} follows from Theorem~\ref{Pfaffianidtheorem} by taking $f(x,y)$ to be $k(x)k(y)f(h(x)x,h(y)y)$ and
$g_j(x)$ to be $k(x)g_j(h(x)\,x)$ in the special case~\eqref{HomogPf1} of~\eqref{HomogPf}.
The LHS of the resulting equation is equal to the LHS of~\eqref{HomogPf},
and the equality between the RHSs then gives~\eqref{HomogPf2}.

Alternatively, Corollary~\ref{Pfaffianidcoroll} can be proved directly as follows.  The expansion of $f(u,v)$ gives
$f(h(u)\,u,h(v)\,v)=\sum_{l,m=1}^\infty([u^{l-1}v^{m-1}]\,f(u,v))\,h(u)^{l-1}\,u^{l-1}\,h(v)^{m-1}\,v^{m-1}$,
and the expansion of~$g_j(u)$ gives
$g_j(h(u)\,u)=\sum_{l=1}^\infty([u^{l-1}]\,g_j(u))\,h(u)^{l-1}\,u^{l-1}$. Multiplying the first of these equations by $k(u)\,k(v)$
and taking the coefficient of $u^{i-1}v^{j-1}$ on both sides gives
\begin{multline*}[u^{i-1}v^{j-1}]\,\bigl(k(u)\,k(v)\,f(h(u)u,h(v)v)\bigr)\\
=\sum_{l=1}^i\sum_{m=1}^j\bigl([u^{l-1}v^{m-1}]\,f(u,v)\bigr)\,\bigl([u^{i-l}]\,\bigl(k(u)\,h(u)^{l-1}\bigr)\bigr)\,
\bigl([v^{j-m}]\,\bigl(k(v)\,h(v)^{m-1}\bigr)\bigr),\end{multline*}
and multiplying the second equation by~$k(u)$ and taking the coefficient of~$u^{i-1}$ on both sides gives
\begin{equation*}[u^{i-1}]\bigl(k(u)\,g_j(h(u)\,u)\bigr)=\sum_{l=1}^i\bigl([u^{l-1}]\,g_j(u)\bigr)\,\bigl([u^{i-l}]\,\bigl(k(u)\,h(u)^{l-1}\bigr)\bigr).\end{equation*}
These equations can be used to check that the skew-symmetric matrices~$A$ and $B$
whose strictly upper triangular parts are the arrays for the Pfaffians on the LHS and RHS, respectively, of~\eqref{HomogPf2}
are related by $A=Y^tBY$, where the matrix $Y=(Y_{ij})_{1\le i,j\le2n}$ is given by
\begin{equation*}Y_{ij}=\begin{cases}
[u^{j-i}]\,\bigl(k(u)\,h(u)^{i-1}\bigr),&i\le j\le m,\\
\delta_{ij},&\text{otherwise}.\end{cases}\end{equation*}
Applying~\eqref{PfMAMt}, and observing that $Y$ is upper triangular with
diagonal entries $Y_{ii}=k(0)h(0)^{i-1}$ for $i=1,\ldots,m$ and $Y_{ii}=1$ for $i=m+1,\ldots,2n$ (so that
$\det Y=k(0)^m\,h(0)^{m(m-1)/2}$), then gives~\eqref{HomogPf2}.

An example of the use of~\eqref{HomogPf2} will appear in Section~\ref{proofXrsReform}, in the derivation of
a certain identity~\eqref{Pfaffianshift}.

\subsection{Proof of~\eqref{Xrs}}\label{proofXrs}
Setting $z=1$ in~\eqref{ZX} gives
\begin{equation}\label{Zunref1}X_n\bigl(q^2+\q^2,s,1\bigr)=\bigl(q^2+\q^2\bigr)^{n(n-1)/2}\,\Z_n(\underbrace{1,\ldots,1}_n).\end{equation}

Using~\eqref{ZZ} and~\eqref{ZPf1}, and applying the identity that, for any
odd $n$, $(a_{ij})_{1\le i<j\le n}$ and $(b_j)_{1\le j\le n}$,
\begin{equation}\label{PfFirstRowToLastCol}\Pf_{0\le i<j\le n}
\left(\left\{\begin{array}{@{\:}ll@{}}b_j,&i=0\\a_{ij},&i\ge1\end{array}\right.\right)=\Pf_{1\le i<j\le n+1}
\left(\left\{\begin{array}{@{\:}ll@{}}a_{ij},&j\le
n\\b_i,&j=n+1\end{array}\right.\right),\end{equation}
where this identity follows from~\eqref{Pfint} and~\eqref{Pfmult}, gives
\begin{multline*}\Z_n(u_1,\ldots,u_n)\\[-3.5mm]
=\prod_{1\le i<j\le n}\!\!\frac{\sigmah(q^2u_iu_j)\,\sigmah(q^2\u_i\u_j)}{\sigmah(u_i\u_j)}\:\Pf_{1\le i<j\le n+\odd(n)}\left(\begin{cases}
\displaystyle\frac{\sigmah(u_i\u_j)\,\Z_2(u_i,u_j)}{\sigmah(q^2u_iu_j)\,\sigmah(q^2\u_i\u_j)},&j\le n\\[4mm]
\Z_1(u_i),&j=n+1\end{cases}\,\right).\end{multline*}
Using~\eqref{Z123} and~\eqref{ZZ} to give $\Z_1(u_i)=s$ and
$\displaystyle\Z_2(u_i,u_j)=s^2\,\sigmah(q^2\u_i\u_j)+\frac{\sigma(q\,u_i)\,\sigma(q\,u_j)}{\sigma(q)^2}$,
writing certain cases of $\sigma(x)$ and $\sigmah(x)$ explicitly, and simplifying, then gives
\begin{multline}\label{ZPfK}\Z_n(u_1,\ldots,u_n)
=\frac{s^{\odd(n)}}{\sigma(q^4)^{n(n-1)/2}}\,\prod_{i=1}^nu_i^{n-1}\\
\times\prod_{1\le i<j\le n}\!\!\frac{\sigma(q^2u_iu_j)\,\sigma(q^2\u_i\u_j)}{u_j^2-u_i^2}
\:\Pf_{1\le i<j\le n+\odd(n)}\left(\begin{cases}G(u_i^2,u_j^2),&j\le n\\
1,&j=n+1\end{cases}\,\right),\end{multline}
where
\begin{equation}\label{Kdef}G(x,y)=\frac{(y-x)\,q^2\,\bigl(s^2\,\sigma(q)\,(\q^4xy-1)-(q+\q)\,(q^2+\q^2)\,(x-\q^2)\,(y-\q^2)\bigr)}
{\sigma(q)\,(xy-q^4)\,(xy-\q^4)}.\end{equation}

The expression~\eqref{ZPfK} for~$\Z_n(u_1,\ldots,u_n)$ will now be used on the RHS of~~\eqref{Zunref1}.
Immediately setting~$u_1=\ldots=u_n=1$ in the terms~$\prod_{i=1}^nu_i^{n-1}$
and $\prod_{1\le i<j\le n}\sigma(q^2u_iu_j)\,\sigma(q^2\u_i\u_j)$,
substituting $u_i^2=x_i+1$ for the remaining occurrences of $u_1,\ldots,u_n$,
and taking $x_1,\ldots,x_n\rightarrow0$, gives
\begin{multline}\label{XK}
X_n\bigl(q^2+\q^2,s,1\bigr)\\
=s^{\odd(n)}\,\sigma(q^2)^{n(n-1)/2}\:\raisebox{2.5ex}{$\left.\raisebox{-2.5ex}{$\displaystyle\frac{\rule[-3.5ex]{0ex}{0ex}\displaystyle\Pf_{1\le i<j\le n+\odd(n)}
\left(\begin{cases}G(x_i+1,x_j+1),&j\le n\\
1,&j=n+1\end{cases}\right)}{\prod_{1\le i<j\le n}(x_j-x_i)}$}\right|_{x_1,\ldots,x_n\rightarrow0}$}\,.\end{multline}

It can easily be checked that
\begin{equation}\label{KF}G(x+1,y+1)=F\biggr(\frac{x}{q^2x+\sigma(q^2)},\frac{y}{q^2y+\sigma(q^2)}\biggr),\end{equation}
where
\begin{equation}\label{Fxy}F(x,y)=\frac{y-x}{1-xy}\biggl(s^2+\frac{(q^2+\q^2)(1+x)(1+y)}{\bigl(1-(q^2+\q^2)x\bigr)\bigl(1-(q^2+\q^2)y\bigr)-xy}\biggr).\end{equation}

Hence, applying~\eqref{HomogPf} (with $2n$ replaced by $n+\odd(n)$, $m=n$,
$f(x,y)=F(x,y)$, $g_{n+1}(x)=1$, $h(x)=1/(q^2x+\sigma(q^2))$ and $k(x)=1$), noting that
$[u^{i-1}]\,g_{n+1}(u)=[u^{i-1}]\,1=\delta_{i,1}$, using~\eqref{Pfredcol} for~$n$ odd, and replacing $i$ by $i+1$ and $j$ by $j+1$, gives
\begin{equation}\label{Zunref2}
X_n(q^2+\q^2,s,1)=s^{\odd(n)}\Pf_{\odd(n)\le i<j\le n-1}\bigl([u^iv^j]\,F(u,v)\bigr).\end{equation}
Finally, setting $r=q^2+\q^2$ (which enables arbitrary~$r$ to be parameterized in terms of~$q$) in~\eqref{Zunref2} gives the required result~\eqref{Xrs}.

As a remark on this proof, note that the use of the transformation~\eqref{KF} enables~$G$ to be replaced by~$F$,
where~$F$, unlike $G$, is explicitly a function of $q^2+\q^2$. If the transformation~\eqref{KF} (which involves taking the function~$h(x)$
in Theorem~\ref{Pfaffianidtheorem} to be $h(x)=1/(q^2x+\sigma(q^2))$) is not performed, then this leads to
\begin{equation}\label{XF}X_n(q^2+\q^2,s,1)=s^{\odd(n)}\,\sigma(q^2)^{n(n-1)/2}\Pf_{\odd(n)\le i<j\le n-1}\bigl([u^iv^j]\,G(u+1,v+1)\bigr)\end{equation}
instead of~\eqref{Zunref2}.
The RHS of~\eqref{XF} can be regarded as a Pfaffian formula for $X_n(r,s,1)$ in which~$q$ should be taken as any solution to $r=q^2+\q^2$.
For example, if $r=s=1$, then $q$ is a primitive 12th root of unity, and~\eqref{XF} gives
\begin{multline}
|\DSASM(n)|=(-1)^{\lfloor(n-1)^2/4\rfloor}\,3^{n(n-1)/4}\\[-2mm]
\times\Pf_{\odd(n)\le i<j\le n-1}\biggl([u^iv^j]\,\frac{(v-u)\bigl(\sqrt{3}\,(u+v+2)\pm\mathfrak{i}\,(u+v+uv)\bigr)}
{1+(u+1)(v+1)+(u+1)^2(v+1)^2}\biggr),\end{multline}
where $\mathfrak{i}$ denotes the imaginary unit, and the two cases given by the $\pm$ are both valid
(and related to each other by complex conjugation of the entire equation).

\subsection{Proof of~\eqref{Xrsentries} and derivation of associated recurrence relations}\label{Xrsentriesproof}
It can immediately be seen that for any power series $f(u,v)$ and $g(u,v)$, and any nonnegative integers $i$ and $j$,
\begin{equation}\label{coeffprod}[u^iv^j]\,\bigl(f(u,v)\,g(u,v)\bigr)=\sum_{k=0}^i\sum_{l=0}^j\bigl([u^kv^l]\,f(u,v)\bigr)\,\bigl([u^{i-k}v^{j-l}]\,g(u,v)\bigr)\end{equation}
and
\begin{equation}\label{coeff1}[u^iv^j]\,\frac{1}{1-uv}=\delta_{i,j}.\end{equation}
It can also easily be checked that
\begin{equation}\label{coeff2}[u^iv^j]\,\frac{1}{(1-ru)(1-rv)-uv}=\sum_{k=0}^{\min(i,j)}\binom{i}{k}\binom{j}{k}\,r^{i+j-2k}.\end{equation}
Applying~\eqref{coeffprod} to the power series in \eqref{coeff1} and~\eqref{coeff2} gives
\begin{equation}\label{coeff3}[u^iv^j]\,\frac{1}{(1-uv)((1-ru)(1-rv)-uv)}
=\sum_{k=0}^{\min(i,j)}\,\sum_{l=0}^{\min(i,j)-k}\binom{i-k}{l}\binom{j-k}{l}\,r^{i+j-2k-2l}.
\end{equation}
Using~\eqref{coeff1} and~\eqref{coeff3}, and noting that $(v-u)(1+u)(1+v)=v-u+v^2-u^2+uv^2-u^2v$, now gives
\begin{align}
\notag&[u^iv^j]\,\frac{v-u}{1-uv}\biggl(s^2+\frac{r(1+u)(1+v)}{(1-ru)(1-rv)-uv}\biggr)\\[2mm]
\notag&\;\;=s^2\,\bigl(\delta_{i+1,j}-\delta_{i,j+1}\bigr)\\
\notag&\qquad+\sum_{k=0}^{\min(i,j)}\,\sum_{l=0}^{\min(i,j)-k}\biggl[\biggl(\binom{i-k}{l}\binom{j-k-1}{l}-\binom{i-k-1}{l}\binom{j-k}{l}\biggr)r^{i+j-2k-2l}\\
\notag&\quad\qquad\qquad\qquad\qquad+\biggl(\binom{i-k}{l}\binom{j-k-2}{l}-\binom{i-k-2}{l}\binom{j-k}{l}\biggr)r^{i+j-2k-2l-1}\\
\notag&\qquad\qquad\qquad+\biggl(\binom{i-k-1}{l}\binom{j-k-2}{l}-\binom{i-k-2}{l}\binom{j-k-1}{l}\biggr)r^{i+j-2k-2l-2}\biggr]\\[2mm]
\notag&\;\;=s^2\,\bigl(\delta_{i+1,j}-\delta_{i,j+1}\bigr)\\
\notag&\qquad+\text{sign}(j-i)\sum_{k=0}^{\min(i,j)}\sum_{l=0}^k\biggl[\biggl(\binom{k}{l}\binom{|j-i|+k-1}{l}-\binom{k-1}{l}\binom{|j-i|+k}{l}\biggr)r^{|j-i|+2k-2l}\\
\notag&\quad\qquad\qquad\qquad\qquad\qquad+\biggl(\binom{k}{l}\binom{|j-i|+k-2}{l}-\binom{k-2}{l}\binom{|j-i|+k}{l}\biggr)r^{|j-i|+2k-2l-1}\\
\label{Xrsentriesgen}&\qquad\qquad+\biggl(\binom{k-1}{l}\binom{|j-i|+k-2}{l}-\binom{k-2}{l}\binom{|j-i|+k-1}{l}\biggr)r^{|j-i|+2k-2l-2}\biggr],
\end{align}
where the second equality is obtained by replacing $k$ in the summand with $\min(i,j)-k$,
the convention of~\eqref{binomial} is taken for the binomial coefficients, and sign$(j-i)$ is~$1$, $-1$ or~$0$,
according to whether $j-i$ is positive, negative or zero.

The required result~\eqref{Xrsentries} follows by taking $i<j$
in~\eqref{Xrsentriesgen}.

Since the numbers in~\eqref{Xrsentriesgen}
are coefficients of $u^iv^j$ in the expansion of a rational function of~$u$ and~$v$, recurrence relations for
these numbers can be obtained straightforwardly, as will be done now. Let
\begin{equation*}f(u,v)=\frac{v-u}{1-uv}\biggl(s^2+\frac{r(1+u)(1+v)}{(1-ru)(1-rv)-uv}\biggr)\quad\text{and}\quad
M_{ij}=\begin{cases}[u^iv^j]\,f(u,v),&i,j\ge0,\\0,&\text{otherwise}.\end{cases}\end{equation*}
Also, let $N(u,v)=(v-u)(s^2((1-ru)(1-rv)-uv)+r(1+u)(1+v))$,
$D_1(u,v)=1-uv$ and $D_2(u,v)=(1-ru)(1-rv)-uv$, so that the numerator and denominator of $f(u,v)$ are $N(u,v)$
and $D_1(u,v)\,D_2(u,v)$, respectively.

Taking the coefficient of $u^iv^j$ on both sides of the three equations $f(u,v)D_1(u,v)D_2(u,v)=N(u,v)$,
$f(u,v)D_2(u,v)=N(u,v)/D_1(u,v)$ and $f(u,v)D_1(u,v)=N(u,v)/D_2(u,v)$ gives three respective recurrence relations for $M_{ij}$.
Explicitly, these are
\begin{align}\notag&\hspace{-1mm}M_{ij}-r(M_{i,j-1}+M_{i-1,j})+(r^2-2)M_{i-1,j-1}+r(M_{i-1,j-2}+M_{i-2,j-1})+(1-r^2)M_{i-2,j-2}\\
\notag&=(r+s^2)(\delta_{i,0}\delta_{j,1}-\delta_{i,1}\delta_{j,0})+r(1-s^2)(\delta_{i,0}\delta_{j,2}-\delta_{i,2}\delta_{j,0})\\
\label{DSASMentryrecurr1}&\hspace{93mm}+(r-s^2+r^2s^2)(\delta_{i,1}\delta_{j,2}-\delta_{i,2}\delta_{j,1}),\\
\notag&\hspace{-1mm}M_{ij}-r(M_{i,j-1}+M_{i-1,j})+(r^2-1)M_{i-1,j-1}=r(rs^2+2)(\delta_{i+1,j}-\delta_{i,j+1})\\
\label{DSASMentryrecurr2}&\hspace{47mm}+r(1-s^2)(\delta_{i+2,j}-\delta_{i,j+2})+(s^2-r^2s^2-r)(\delta_{i,0}\delta_{j,1}-\delta_{i,1}\delta_{j,0})\\
\intertext{and}
\notag&\hspace{-1mm}M_{ij}-M_{i-1,j-1}=s^2(\delta_{i,0}\delta_{j,1}-\delta_{i,1}\delta_{j,0})+\sum_{k=0}^{\min(i,j)}\biggl[\biggl(\binom{i}{k}\binom{j-1}{k}-\binom{i-1}{k}\binom{j}{k}\biggr)r^{i+j-2k}\\
\notag&\hspace{78mm}+\biggl(\binom{i}{k}\binom{j-2}{k}-\binom{i-2}{k}\binom{j}{k}\biggr)r^{i+j-2k-1}\\
\label{DSASMentryrecurr3}&\hspace{60mm}+\biggl(\binom{i-1}{k}\binom{j-2}{k}-\binom{i-2}{k}\binom{j-1}{k}\biggr)r^{i+j-2k-2}\biggr],\end{align}
for all $i,j\ge0$.

Any one of~\eqref{DSASMentryrecurr1},~\eqref{DSASMentryrecurr2} or~\eqref{DSASMentryrecurr3}
provides a recurrence which, together with the initial condition $M_{ij}=0$ for $i\le0$ or $j\le0$,
can be used for the efficient computation of $M_{ij}$ for $i,j\ge0$.  The property that $M_{ji}=-M_{ij}$ could also
be used in such computations.

\subsection{Proof of~\eqref{numDSASM2} and~\eqref{numDSASMentries}}\label{numDSASMentriesproof}
Setting $r=s=1$ in~\eqref{X} and~\eqref{Xrs} immediately gives~\eqref{numDSASM2}.

Setting $r=1$ in~\eqref{coeff3} and using the Chu--Vandermonde summation for binomial coefficients,
gives
\begin{equation*}[u^iv^j]\,\frac{1}{(1-uv)(1-u-v)}=\sum_{k=0}^{\min(i,j)}\binom{i+j-2k}{i-k},\end{equation*}
which then gives
\begin{multline}\label{coeff4}[u^iv^j]\,\frac{(v-u)(2+uv)}{(1-uv)(1-u-v)}=
2\sum_{k=0}^{\min(i,j)}\biggl(\binom{i+j-2k-1}{i-k}-\binom{i+j-2k-1}{j-k}\biggr)\\
+\sum_{k=1}^{\min(i,j)}\biggl(\binom{i+j-2k-1}{i-k}-\binom{i+j-2k-1}{j-k}\biggr),\end{multline}
for any nonnegative integers $i$ and $j$.
For $i<j$, the first equality in~\eqref{numDSASMentries} follows immediately from~\eqref{coeff4},
and the second equality in~\eqref{numDSASMentries} then follows by applying a simple binomial coefficient identity.

\subsection{Proof of~\eqref{Xrst} and~\eqref{Xrstentries}}\label{proofXrsttheorem}
Setting $u_n=z$ in~\eqref{ZPfK}, and using the symmetry (from Proposition~\ref{Zsymm} and~\eqref{ZZ})
of $\Z_n(u_1,\ldots,u_n)$ in all variables to
write $\Z_n(u_1,\ldots,u_{n-1},z)$ as $\Z_n(z,u_1,\ldots,u_{n-1})$, gives
\begin{multline*}\Z_n(z,u_1,\ldots,u_{n-1})=
\frac{s^{\odd(n)}}{\sigma(q^4)^{n(n-1)/2}}\,\prod_{i=1}^{n-1}u_i^{n-1}\prod_{1\le i<j\le n-1}\!\!\frac{\sigma(q^2u_iu_j)\,\sigma(q^2\u_i\u_j)}{u_j^2-u_i^2}\\
\times z^{n-1}\,\prod_{i=1}^{n-1}\frac{\sigma(q^2u_iz)\,\sigma(q^2\u_i\z)}{z^2-u_i^2}\:\Pf_{1\le i<j\le n+\odd(n)}\left(\begin{cases}G(u_i^2,u_j^2),&j\le n-1\\
G(u_i^2,z^2),&j=n\\1,&j=n+1\end{cases}\,\right),\end{multline*}
where $G(x,y)$ is again given by~\eqref{Kdef}.
Setting~$u_1=\ldots=u_{n-1}=1$ in the terms~$\prod_{i=1}^{n-1}u_i^{n-1}$,
$\prod_{1\le i<j\le n-1}\sigma(q^2u_iu_j)\,\sigma(q^2\u_i\u_j)$
and $\prod_{i=1}^{n-1}\sigma(q^2u_iz)\,\sigma(q^2\u_i\z)/(z^2-u_i^2)$,
substituting $u_i^2=x_i+1$ for the remaining occurrences of
$u_1,\ldots,u_{n-1}$, and taking $x_1,\ldots,x_{n-1}\rightarrow0$, then gives
\begin{multline*}
\Z_n(z,\underbrace{1,\ldots,1}_{n-1})=\frac{s^{\odd(n)}\,\sigma(q^2)^{(n-1)(n-2)}}{\sigma(q^4)^{n(n-1)/2}}\,\left(\frac{\sigma(q^2z)\,\sigma(q^2\z)}{\sigma(z)}\right)^{n-1}\\
\times\raisebox{4.5ex}{$\left.\raisebox{-4.5ex}{$\displaystyle\frac{\rule[-5.3ex]{0ex}{0ex}\displaystyle\Pf_{1\le i<j\le n+\odd(n)}\left(\begin{cases}G(x_i+1,x_j+1),&j\le n-1\\
G(x_i+1,z^2),&j=n\\
1,&j=n+1\end{cases}\right)}{\prod_{1\le i<j\le n-1}(x_j-x_i)}$}\right|_{x_1,\ldots,x_{n-1}\rightarrow0}$}\,.\end{multline*}
Recalling that $G(x+1,y+1)$ satisfies~\eqref{KF} with $F(x,y)$ given
by~\eqref{Fxy}, applying~\eqref{HomogPf} (with~$2n$ replaced by $n+\odd(n)$, $m=n-1$, $f(x,y)=F(x,y)$,
$g_n(x)=F\bigl(x,(z^2-1)/(q^2(z^2-1)+\sigma(q^2))\bigr)$, $g_{n+1}(x)=1$, $h(x)=1/(q^2x+\sigma(q^2))$, $k(x)=1$
and $a_{n,n+1}=1$), and noting that $[u^{i-1}]\,g_{n+1}(u)=\delta_{i,1}$, then gives
\begin{multline*}
\Z_n(z,\underbrace{1,\ldots,1}_{n-1})=\frac{s^{\odd(n)}\,\sigma(q^2)^{(n-1)(n-2)/2}}{\sigma(q^4)^{n(n-1)/2}}\,\left(\frac{\sigma(q^2z)\,\sigma(q^2\z)}{\sigma(z)}\right)^{n-1}\\
\times\Pf_{1\le i<j\le n+\odd(n)}\left(\begin{cases}[u^{i-1}v^{j-1}]\,F(u,v),&j\le n-1\\[1.9mm]
\displaystyle[u^{i-1}]\,F\biggl(\!u,\frac{z^2-1}{q^2(z^2-1)+\sigma(q^2)}\biggr),&j=n\\
\delta_{i,1}+\delta_{i,n},&j=n+1\end{cases}\right).\end{multline*}
Using~\eqref{Pfredcol} for odd~$n$, replacing $i$ by $i+1$ and $j$ by $j+1$, and applying~\eqref{Zunref2}
to the term $\Pf_{0\le i<j\le n-2}\bigl([u^iv^j]\,F(u,v)\bigr)$ which occurs for odd~$n$, then gives
\begin{multline*}\Z_n(z,\underbrace{1,\ldots,1}_{n-1})
=\frac{s^{\odd(n)}\,\sigma(q^2)^{(n-1)(n-2)/2}}{\sigma(q^4)^{n(n-1)/2}}\,\left(\frac{\sigma(q^2z)\,\sigma(q^2\z)}{\sigma(z)}\right)^{n-1}\\
\times\Biggl(\Pf_{\odd(n)\le i<j\le n-1}\left(\begin{cases}[u^iv^j]\,F(u,v),&j\le n-2\\[1.9mm]
\displaystyle[u^i]\,F\biggl(\!u,\frac{z^2-1}{q^2(z^2-1)+\sigma(q^2)}\biggr),&j=n-1\end{cases}\right)\\
+\odd(n)\,X_{n-1}\bigl(q^2+\q^2,s,1\bigr)\Biggr).\end{multline*}
Using this expression for $\Z_n(z,\underbrace{1,\ldots,1}_{n-1})$ in~\eqref{ZX}, then gives
\begin{multline*}
\biggl(\frac{\sigma(z)}{\sigma(q^2z)}\biggr)^{\!n}\,
\Biggl(\frac{(q+\q)\,\sigma(qz)\,\sigma(q^2\z)}{\sigma(z)\,\sigma(q^2z)}\,X_n\biggl(q^2+\q^2,s,\frac{\sigma(q^2z)}{\sigma(q^2\z)}\biggr)-s\,X_{n-1}\bigl(q^2+\q^2,s,1\bigr)\Biggr)\\
\shoveleft{\qquad=s^{\odd(n)}\Pf_{\odd(n)\le i<j\le n-1}\left(\begin{cases}[u^iv^j]\,F(u,v),&j\le n-2\\[1.5mm]
\displaystyle[u^i]\,F\biggl(\!u,\frac{z^2-1}{q^2(z^2-1)+\sigma(q^2)}\biggr),&j=n-1\end{cases}\right)}\\
+\odd(n)\,s\,X_{n-1}\bigl(q^2+\q^2,s,1\bigr).\end{multline*}
Setting $r=q^2+\q^2$ and $t=\sigma(q^2z)/\sigma(q^2\z)$ (which enables arbitrary~$r$ and~$t$ to be parameterized in terms of~$q$ and~$z$), and noting that
\begin{equation*}\frac{\sigma(z)}{\sigma(q^2z)}=\frac{z^2-1}{q^2(z^2-1)+\sigma(q^2)}=\frac{t-1}{rt}
\qquad\text{and}\qquad\frac{(q+\q)\,\sigma(qz)\,\sigma(q^2\z)}{\sigma(z)\,\sigma(q^2z)}=\frac{t-1+rt}{t(t-1)},\end{equation*}
gives~\eqref{Xrst}.

Finally, the first equality in~\eqref{Xrstentries} can be obtained straightforwardly, where the computations can be simplified by using~\eqref{Xrstentriesfun}.
The second equality in~\eqref{Xrstentries} then follows by simply evaluating the sum over~$k$.

\subsection{Proofs of~\eqref{XrsReform}--\eqref{numDSASMReform}}\label{proofXrsReform}
First note that for any power series $f(u,v)$ which is antisymmetric in~$u$ and~$v$,
\begin{equation}\label{Pfaffianshift}
\Pf_{\odd(n)\le i<j\le n-1}\bigl([u^iv^j]\,f(u,v)\bigr)\\
=\Pf_{0\le i<j\le n-\even(n)}\left(\begin{cases}\displaystyle[u^iv^j]\,\frac{f(u,v)}{(1+u)(1+v)},&j\le n-1\\
(-1)^i,&j=n\end{cases}\right),\end{equation}
where this can be obtained by writing $(-1)^i=[u^i]\,(1/(1+u))$, applying~\eqref{HomogPf2}
(with $2n$ replaced by $n+\odd(n)$, $m=n$, $g_{n+1}(u)=1$, $h(u)=1$, $k(u)=1/(1+u)$,~$i$ replaced by $i+1$ and~$j$ replaced by $j+1$)
to the RHS of~\eqref{Pfaffianshift}, writing $[u^i]\,g_{n+1}(u)=[u^i]\,1=\delta_{i,0}$, and using~\eqref{Pfredcol} for~$n$ odd,
which then gives the LHS of~\eqref{Pfaffianshift}.

Applying~\eqref{Pfaffianshift} to~\eqref{Xrs} immediately gives~\eqref{XrsReform}.  It can easily be checked that
\begin{equation}\label{coeff5}[u^iv^j]\,\frac{1}{(1-uv)(1+u)(1+v)}=(-1)^{i+j}\bigl(1+\min(i,j)\bigr).\end{equation}
Using~\eqref{coeff3} and~\eqref{coeff5}, together with a similar approach to that used to obtain~\eqref{Xrsentriesgen}, now gives~\eqref{XrsentriesReform}.

Finally, setting $r=s=1$ in~\eqref{XrsReform} and~\eqref{XrsentriesReform}, and using a similar approach to that used in Section~\ref{numDSASMentriesproof}, gives~\eqref{numDSASMReform}.

\subsection{Proof of~\eqref{XrstDiv}}\label{proofXrstDiv}
The proof of~\eqref{XrstDiv} will use the identity that, for any constants $a_{ij}$ which satisfy $a_{ij}=-a_{ji}$,
\begin{multline}\label{PfDivId}
\Pf_{\odd(n)\le i<j\le n-1}\left(\begin{cases}a_{ij},&j\le n-2\\
\sum_{k=0}^\infty a_{ik}\,v^k,&j=n-1\end{cases}\right)
=-\odd(n)\Pf_{0\le i<j\le n-2}(a_{ij})\\
+v^{n-1}\Pf_{\odd(n)\le i<j\le n-1}\left(\begin{cases}a_{ij},&j\le n-2\\
\sum_{k=0}^\infty a_{i,k+n-1}\,v^k,&j=n-1\end{cases}\right),\end{multline}
where this can be obtained as follows
(and where, in the first term on the RHS, the fact that $\Pf_{0\le i<j\le n-2}(a_{ij})$ is undefined for~$n$ even
can be disregarded due to the factor $\odd(n)$).
Let $A$ be the skew-symmetric matrix whose strictly upper triangular part is the array for the Pfaffian on the LHS of~\eqref{PfDivId}, and let
$\rule[-5ex]{0ex}{11ex}Y=\left(\begin{smallmatrix}1&0&0&\ldots&0&-1\\0&1&0&\ldots&0&-v\\0&0&1&\ldots&0&-v^2\\[-1.8mm]\vdots&\vdots&\vdots&\ddots&\vdots&\vdots\\0&0&0&\ldots&1&-v^{n-2}\\0&0&0&\ldots&0&1\end{smallmatrix}\!\right)$
for~$n$ even, or
$Y=\left(\begin{smallmatrix}1&0&\ldots&0&-v\\0&1&\ldots&0&-v^2\\[-1.8mm]\vdots&\vdots&\ddots&\vdots&\vdots\\0&0&\ldots&1&-v^{n-2}\\0&0&\ldots&0&1\end{smallmatrix}\!\right)$ for $n$ odd,
where the rows and columns of $A$ and $Y$ are indexed by $\odd(n),\ldots,n-1$.
By applying~\eqref{PfMAMt} (where $Y^tAY$ is obtained from $A$ by subtracting~$v^j$ times column~$j$ from column $n-1$,
and $v^i$ times row~$i$ from row $n-1$, for each $i,j\in\{\odd(n),\ldots,n-2\}$), it follows that the LHS of~\eqref{PfDivId} is
$\Pf_{\odd(n)\le i<j\le n-1}\left(\begin{cases}a_{ij},&j\le n-2\\
\odd(n)a_{i0}+v^{n-1}\sum_{k=0}^\infty a_{i,k+n-1}\,v^k,&j=n-1\end{cases}\right)$.  The RHS of~\eqref{PfDivId} can be obtained from this by using~\eqref{Pfredcol} (to expand
along the last column of the array to give a sum of two Pfaffians, one with an overall factor of $\odd(n)$, and the other with an overall factor of~$v^{n-1}$),
and then applying~\eqref{PfFirstRowToLastCol} for $n$ odd (to the first Pfaffian in the sum).

For the main part of the proof of~\eqref{XrstDiv}, substituting $t=1/(1-rv)$ into~\eqref{Xrst}, and using~\eqref{Xrstentriesfun}, gives
\begin{multline*}
(1+v)(1-rv)\,X_n\Bigl(r,s,\frac{1}{1-rv}\Bigr)=s\biggl(\!v+\frac{\odd(n)}{v^{n-1}}\biggr)X_{n-1}(r,s,1)\\[1mm]
+\frac{s^{\odd(n)}}{v^{n-1}}\,\Pf_{\odd(n)\le i<j\le n-1}\left(\begin{cases}\displaystyle[u^iv^j]\,\frac{v-u}{1-uv}\biggl(s^2+\frac{r(1+u)(1+v)}{(1-ru)(1-rv)-uv}\biggr),&j\le n-2\\[4mm]
\displaystyle[u^i]\,\frac{v-u}{1-uv}\biggl(s^2+\frac{r(1+u)(1+v)}{(1-ru)(1-rv)-uv}\biggr),&j=n-1\end{cases}\right).\end{multline*}
Applying~\eqref{PfDivId} with $a_{ij}=[u^iv^j]\,\frac{v-u}{1-uv}\bigl(s^2+\frac{r(1+u)(1+v)}{(1-ru)(1-rv)-uv}\bigr)$, using~\eqref{Xrs} to write $\Pf_{0\le i<j\le n-2}(a_{ij})=X_n(r,s,1)$ for~$n$ odd,
and using~\eqref{fddcoeff} to write $\sum_{k=0}^\infty a_{i,k+n-1}\,v^k=f_{n-1}(u,v)$, then gives
\begin{multline*}
(1+v)(1-rv)\,X_n\Bigl(r,s,\frac{1}{1-rv}\Bigr)=sv\,X_{n-1}(r,s,1)\\[1mm]
+s^{\odd(n)}\,\Pf_{\odd(n)\le i<j\le n-1}\left(\begin{cases}\displaystyle[u^iv^j]\,\frac{v-u}{1-uv}\biggl(s^2+\frac{r(1+u)(1+v)}{(1-ru)(1-rv)-uv}\biggr),&j\le n-2\\[2.5mm]
\displaystyle[u^i]\,f_{n-1}(u,v),&j=n-1\end{cases}\right).\end{multline*}
Finally, substituting $v=(t-1)/(rt)$ gives~\eqref{XrstDiv}.

\section{Diagonally symmetric permutation matrices}\label{DSPMsect}
In this section, the special case of DSASMs with no~$-1$'s, i.e., diagonally symmetric permutation matrices (DSPMs), is considered.
A Pfaffian formula for a DSPM generating function is given in~\eqref{DSPMX3}, where this can be obtained either using Theorems~\ref{Xrstheorem} and~\ref{Xrsttheorem},
or directly.

The set of $n\times n$ DSPMs will be denoted as~$\DSPM(n)$. Note that $\DSPM(n)$ consists of the permutation matrices of all involutions on $\{1,\ldots,n\}$.

For any $A\in\DSPM(n)$, the statistics $R$ and $S$ are related by
\begin{equation}\label{DSPMRS}2R(A)+S(A)=n,\end{equation}
which can be obtained by setting $M(A)=0$ in~\eqref{MRS}.
Hence, it is sufficient to consider only one of these statistics in a generating function for DSPMs.
Taking this to be~$S$, the DSPM generating function associated with~$S$ and~$T$ is defined as
\begin{equation}\label{DSPMXdef}\XPM_n(s,t)=\sum_{A\in\DSPM(n)}s^{S(A)}\,t^{T(A)}.\end{equation}
Note that, for any DSPM $A$, $S(A)=\mathop{\mathrm{tr}}A$ is the number of fixed points in the involution associated with~$A$.

It follows from simple and well-known combinatorial arguments that
\begin{align}\label{XDSPMs}\XPM_n(s,1)&=\sum_{i=0}^{\lfloor n/2\rfloor}\frac{n!}{2^i\,i!\,(n-2i)!}\,s^{n-2i},\\
\label{XDSPMst}\XPM_n(s,t)&=s\,t\,\XPM_{n-1}(s,1)+\sum_{i=2}^{n}t^i\,\XPM_{n-2}(s,1),\end{align}
and it follows from~\eqref{DSPMRS} that $\XPM_n(s,t)$ is related to the DSASM generating function~\eqref{X} by
\begin{equation}\label{DSPMX1}\XPM_n(s,t)=\bigl(x^{-n}\,X_n(x^2,xs,t)\bigr)|_{x\rightarrow0}.\end{equation}

Using~\eqref{Xrs} and~\eqref{Xrst} in~\eqref{DSPMX1}, and performing some simplifications, gives the Pfaffian formula
\begin{multline}\label{DSPMX2}
\XPM_n(s,t)=s^{\odd(n)}\,t\Pf_{\odd(n)\le i<j\le n-1}
\biggl([u^iv^j]\,\frac{v-u}{1-uv}\biggl(s^2+\frac{(1+u)(1+v)}{1-uv}\\[-1.5mm]
\textstyle+u^{n-2}v^{n-2}(t-1)\sum_{k=0}^{n-2}(n-k-1)\,t^k\biggr)\biggr).\end{multline}
Evaluating the coefficients which comprise the entries for the Pfaffian in~\eqref{DSPMX2} then gives
\begin{multline}\label{DSPMX3}
\XPM_n(s,t)=s^{\odd(n)}\,t\\
\times\Pf_{\odd(n)\le i<j\le n-1}\textstyle\Bigl(\bigl(2i+1+s^2+(t-1)\sum_{k=0}^{n-2}(n-k-1)\,t^k\,\delta_{i,n-2}\bigr)\delta_{i+1,j}+(i+1)\,\delta_{i+2,j}\Bigr).\end{multline}

Alternatively,~\eqref{DSPMX3} can be proved directly by showing that it holds for $n=1,\ldots,4$, and that the recurrence
\begin{equation}\label{DSPMX4}
F_n(s,t)=\textstyle\bigl(n-2+s^2+\sum_{i=1}^{n-1}t^i\bigr)\,t\,F_{n-2}(s,1)-(n-2)(n-3)\,t\,F_{n-4}(s,1),\end{equation}
for $n\ge5$, is satisfied with $F_n(s,t)$ taken to be either side of~\eqref{DSPMX3}.
The case of~\eqref{DSPMX4} with $F_n(s,t)$ taken as the LHS of~\eqref{DSPMX3}
follows by replacing~$n$ in the $t=1$ case of~\eqref{XDSPMst} by $n-1$ and $n-2$,
and combining these two equations with~\eqref{XDSPMst} itself.
The case of~\eqref{DSPMX4} with $F_n(s,t)$ taken as the RHS of~\eqref{DSPMX3} follows
by first applying~\eqref{Pfredcol} to the Pfaffian
(whose underlying array has only two nonzero entries in the last column)
to give a sum of multiples of two further Pfaffians. One of these Pfaffians is $F_{n-2}(s,1)$,
and applying~\eqref{Pfredcol} to the other Pfaffian (whose underlying
array has only one nonzero entry in the last column) then gives the required case of~\eqref{DSPMX4}.

Note that the Pfaffian in~\eqref{DSPMX3} can also be written as the determinant of a $\lfloor n/2\rfloor\times\lfloor n/2\rfloor$ tridiagonal matrix using the identity
\begin{multline}\label{tridiag}
\Pf_{\odd(n)\le i<j\le n-1}\bigl(a_i\,\delta_{i+1,j}+b_i\,\delta_{i+2,j}\bigr)\\[-2mm]
=\det_{1\le i,j\le\lfloor n/2\rfloor}\Biggl(\begin{cases}
a_{2i-2}\,\delta_{i,j}+b_{2i-2}\,\delta_{i+1,j}+b_{2i-3}\,\delta_{i,j+1},&n\text{ even}\\
a_{2i-1}\,\delta_{i,j}+b_{2i-1}\,\delta_{i+1,j}+b_{2i-2}\,\delta_{i,j+1},&n\text{ odd}\end{cases}\Biggr),
\end{multline}
for any $a_{\odd(n)},\ldots,a_{n-2}$ and $b_{\odd(n)},\ldots,b_{n-3}$.  (A way of proving this identity is to show that
it holds for $n=1,\ldots,4$, and that the recurrence $F_n=a_{n-2}\,F_{n-2}-b_{n-3}\,b_{n-4}\,F_{n-4}$, for $n\ge5$,
is satisfied with~$F_n$ taken to be either side of~\eqref{tridiag}.)
Applying~\eqref{tridiag} to~\eqref{DSPMX3} gives
\begin{multline}
\XPM_n(s,t)=s^{\odd(n)}\,t\\
\times\det_{1\le i,j\le\lfloor n/2\rfloor}\left(\begin{cases}
\bigl(4i-3+s^2+(t-1)\sum_{k=0}^{n-2}(n-k-1)\,t^k\,\delta_{i,n/2}\bigr)\,\delta_{i,j}\\
\hspace*{44mm}\mbox{}+(2i-1)\,\delta_{i+1,j}+(2i-2)\,\delta_{i,j+1},&n\text{ even}\\[1.8mm]
\bigl(4i-1+s^2+(t-1)\sum_{k=0}^{n-2}(n-k-1)\,t^k\,\delta_{i,(n-1)/2}\bigr)\,\delta_{i,j}\\
\hspace*{54.4mm}\mbox{}+2i\,\delta_{i+1,j}+(2i-1)\,\delta_{i,j+1},&n\text{ odd}\end{cases}\right).\end{multline}

\section{Off-diagonally symmetric ASMs}\label{OSASMsect}
In this section, the special case of off-diagonally symmetric ASMs (OSASMs) is considered.
As indicated in Section~\ref{intro}, an even-order OSASM is defined to be an even-order DSASM in which each diagonal entry is~$0$,
and an odd-order OSASM is defined to be an odd-order DSASM in which exactly one diagonal entry is nonzero.
Some of the results for OSASMs in this section, specifically Proposition~\ref{evenOSASMsymmprop},~\eqref{symplodd} and~\eqref{numOSASModd}, were conjectured in the first arXiv version of this paper~\cite{BehFisKou23a}
and then proved by Kumari~\cite{Kum26}.

In Section~\ref{OSASMgenconsid}, some general considerations for OSASMs are outlined,
and an OSASM generating function is defined.  In Section~\ref{OSASMmainresults}, Pfaffian formulae for the OSASM generating function
and number of OSASMs are obtained using~\eqref{Xrs},~\eqref{Xrst} and~\eqref{XrsReform}.
In Section~\ref{OSASMZ}, an OSASM partition function is defined, and some results are obtained for this function, including
a Pfaffian expression in~\eqref{ZOPf}.
In Section~\ref{OSASMHSASM}, previously-known enumerative connections between even-order OSASMs,
odd-order horizontally symmetric ASMs and even-order ASMs are outlined, and explicit formulae are
given for the number of even-order OSASMs and a special case of the even-order OSASM generating function.
In Section~\ref{EvenOSASMSymmConjSect}, it is shown, in Proposition~\ref{evenOSASMsymmprop}, that the even-order OSASM generating function satisfies a certain symmetry property.
In Section~\ref{OSASMSymplSect}, a previously-known connection between the even-order OSASM partition function and certain symplectic characters is outlined,
and it is shown in~\eqref{symplodd} that an analogous connection holds for odd-order OSASMs.
In Section~\ref{OddOSASMConjSect}, it is shown in~\eqref{numOSASModd} that the number of odd-order OSASMs is given by an explicit product formula.

\subsection{General considerations for OSASMs}\label{OSASMgenconsid}
The set of $n\times n$ OSASMs will be denoted as $\OSASM(n)$.
For example, the sets $\OSASM(n)$ for $n=1,\ldots,4$ are
\begin{gather}\notag\OSASM(1)=\{(1)\},\quad\OSASM(2)=\left\{\begin{pmatrix}0&1\\1&0\end{pmatrix}\right\},\\
\notag\OSASM(3)=\left\{\begin{pmatrix}1&0&0\\0&0&1\\0&1&0\end{pmatrix},\,\begin{pmatrix}0&1&0\\1&0&0\\0&0&1\end{pmatrix},\,
\begin{pmatrix}0&0&1\\0&1&0\\1&0&0\end{pmatrix},\,\begin{pmatrix}0&1&0\\1&-1&1\\0&1&0\end{pmatrix}\right\},\\
\label{OSASM1234}\OSASM(4)=\left\{\begin{pmatrix}0&1&0&0\\1&0&0&0\\0&0&0&1\\0&0&1&0\end{pmatrix},\,
\begin{pmatrix}0&0&1&0\\0&0&0&1\\1&0&0&0\\0&1&0&0\end{pmatrix},\,
\begin{pmatrix}0&0&0&1\\0&0&1&0\\0&1&0&0\\1&0&0&0\end{pmatrix}\right\}.\end{gather}
The cardinalities of $\OSASM(n)$ for $n=1,\ldots,12$ are given in Table~\ref{cardO}.

\begin{table}[h]\centering
$\begin{array}{|@{\;\;}c@{\;\;}|@{\;\;\;}c@{\;\;\;\;}c@{\;\;\;\;}c@{\;\;\;\;}c@{\;\;\;\;}c@{\;\;\;}c@{\;\;\;}c@{\;\;\;}c@{\;\;\;}
c@{\;\;\;}c@{\;\;\;}c@{\;\;\;}c@{\;\;}|}\hline\rule{0ex}{4.7mm}
n&1&2&3&4&5&6&7&8&9&10&11&12\\[1mm]\hline
\rule{0ex}{5.9mm}|\OSASM(n)|&1&1&4&3&32&26&640&646&34048&45885&4945920&9304650\\[1.8mm]\hline\end{array}$\\[2.4mm]\caption{Numbers of OSASMs.}\label{cardO}
\end{table}

The OSASM generating function associated with the statistics~$R$ and~$T$ is defined as
\begin{equation}\label{OSASMX}\XO_n(r,t)=\sum_{A\in\OSASM(n)}r^{R(A)}\,t^{T(A)}.\end{equation}
For example, the $n=1,\ldots,4$ cases of the OSASM generating function~\eqref{OSASMX} are
\begin{equation}\XO_1(r,t)=t,\ \XO_2(r,t)=rt^2,\ \XO_3(r,t)=rt+rt^2+rt^3+r^2t^2,\ \XO_4(r,t)=r^2t^2+r^2t^3+r^2t^4,\end{equation}
where the terms on each RHS are written in an order corresponding to that used for the OSASMs in~\eqref{OSASM1234}.

It follows from the definitions of OSASMs and the statistic~$S$ that the OSASM generating function~\eqref{OSASMX} is related to the DSASM generating function~\eqref{X} by
\begin{equation}\label{OSASMX1}\XO_n(r,t)=\begin{cases}X_n(r,0,t),&n\text{ even},\\
\bigl(X_n(r,s,t)/s\bigr)|_{s\rightarrow0},&n\text{ odd}.\end{cases}\end{equation}

An elementary property of the even-order OSASM generating function is
\begin{equation}\label{OSASMXmin}X_{2n}^\mathrm{O}(-r,t)=(-1)^n\,\XO_{2n}(r,t),\end{equation}
which can be obtained by noting that, for any $A\in\OSASM(2n)$, $R(A)=M(A)+n$
(as given by~\eqref{MRS}), where~$M(A)$ is even (since $M(A)$ is twice the number of strictly upper triangular~$-1$'s in~$A$).

It follows from~\eqref{OSASMX1} and (i$'$)--(iii$'$) of Proposition~\ref{refprop} that
\begin{align}
\label{XOcoeff1}[t^1]\,\XO_n(r,t)&=\odd(n)\,\XO_{n-1}(r,1),\\
\label{XOcoeff2}[t^2]\,\XO_n(r,t)&=r\bigl(\odd(n)\,\XO_{n-1}(r,1)+\XO_{n-2}(r,1)\bigr),\\
\label{XOcoeffn}[t^n]\,\XO_n(r,t)&=r\,\XO_{n-2}(r,1),
\end{align}
for $n\ge2$ in~\eqref{XOcoeff1}, and $n\ge3$ in~\eqref{XOcoeff2} and~\eqref{XOcoeffn}.  Further relations of this type for $[t^3]\,\XO_{2n}(r,t)$
and $[t^{2n-1}]\,\XO_{2n}(r,t)$ will be conjectured in~\eqref{XO3conj}.

A relation between OSASMs and the DSASM generating function~\eqref{X} at $r=-1$ and $s=1$ is given by
\begin{equation}\label{Xmin2}
(-1)^{n(n-1)/2}\,X_n(-1,1,t)=\begin{cases}|\OSASM(n-2)|\,t(t-1),&n\text{ even},\\t\,\XO_{n-1}(1,t),&n\text{ odd},\end{cases}\end{equation}
for $n\ge3$.
This can be obtained by setting $r=-1$ in~\eqref{diagref3} to give $X_n(-1,1,t)=t\,X_{n-1}(-1,0,t)\allowbreak+t(1-t)\,X_{n-2}(-1,0,1)$,
and then using $X_n(r,0,t)=0$ for~$n$ odd, the~$n$ even case of~\eqref{OSASMX1}, and~\eqref{OSASMXmin}.

\subsection{Results for the OSASM generating function and number of OSASMs}\label{OSASMmainresults}
Pfaffian formulae for the OSASM generating function $\XO_n(r,t)$ at $t=1$ can be obtained by using~\eqref{Xrs}
and~\eqref{XrsReform} in~\eqref{OSASMX1}, which gives
\begin{align}\notag&\XO_n(r,1)\\
\notag&\qquad=r^{\lfloor n/2\rfloor}\Pf_{\odd(n)\le i<j\le n-1}\biggl([u^iv^j]\,\frac{(v-u)(1+u)(1+v)}{(1-uv)\,\bigl((1-ru)(1-rv)-uv\bigr)}\biggr)\\
\label{OSASMX2}&\qquad=r^{\lfloor n/2\rfloor}\Pf_{0\le i<j\le n-\even(n)}\left(\begin{cases}\displaystyle[u^iv^j]\,\frac{v-u}{(1-uv)\,\bigl((1-ru)(1-rv)-uv\bigr)},&j\le n-1\\[3mm]
(-1)^i,&j=n\end{cases}\right).\end{align}

A Pfaffian identity for the OSASM generating function $\XO_n(r,t)$ can be obtained by using~\eqref{Xrst} in~\eqref{OSASMX1}, which gives, after some simplification,
\begin{multline}\label{OSASMX2t}
(t-1)^{n-1}\,\XO_n(r,t)=\textstyle\odd(n)\,t\,\sum_{i=0}^{n-1}(rt)^i\,(1-t)^{n-i-1}\,\XO_{n-1}(r,1)\\
+r^{\lfloor3n/2\rfloor}\,t^{n+1}\Pf_{\odd(n)\le i<j\le n-1}\left(\begin{cases}\displaystyle[u^iv^j]\,\frac{(v-u)(1+u)(1+v)}{(1-uv)\,\bigl((1-ru)(1-rv)-uv\bigr)},&j\le n-2\\[4mm]
\displaystyle[u^i]\,\frac{(t-1-rtu)(1+u)}{\bigl(rt-(t-1)u\bigr)\bigl(r-(t-1+r^2)u\bigr)},&j=n-1\end{cases}\right),\end{multline}
for $n\ge2$.

Pfaffian formulae for the number of $n\times n$ OSASMs can be obtained by setting $r=1$ in~\eqref{OSASMX2}, which gives
\begin{align}\notag|\OSASM(n)|&=\Pf_{\odd(n)\le i<j\le n-1}\biggl([u^iv^j]\,\frac{(v-u)(1+u)(1+v)}{(1-uv)(1-u-v)}\biggr),\\
\label{numOSASM}&=\Pf_{0\le i<j\le n-\even(n)}\left(\begin{cases}\displaystyle[u^iv^j]\,\frac{v-u}{(1-uv)(1-u-v)},&j\le n-1\\
(-1)^i,&j=n\end{cases}\right).\end{align}

The coefficients which comprise the entries for the Pfaffians in~\eqref{OSASMX2}--\eqref{numOSASM} can be obtained easily
using~\eqref{Xrsentries},~\eqref{Xrstentries} and~\eqref{XrsentriesReform}, and some simplification.
For example, the coefficients in the second Pfaffian formula in~\eqref{numOSASM} are
\begin{align}\notag[u^iv^j]\,\frac{v-u}{(1-uv)(1-u-v)}&=\sum_{k=0}^i\biggl(\binom{i+j-2k-1}{i-k}-\binom{i+j-2k-1}{j-k}\biggr)\\
\label{numOSASMentries}&=(j-i)\,\sum_{k=0}^i\frac{1}{i+j-2k}\,\binom{i+j-2k}{i-k},\end{align}
for $0\le i<j$.

\subsection{The OSASM partition function}\label{OSASMZ}
For the constants $\alpha$, $\beta$, $\gamma$, $\delta$ and function~$\phi(u)$ in Table~\ref{weights}, consider the assignments
\begin{equation}\label{Oassig}\alpha=\beta=s\,\sigma(q),\quad\gamma=\delta=1,\quad\phi(u)=1/\sigma(q^2u^2),\end{equation}
for which the left boundary weights are
\begin{equation}\label{leftWO}W(\WLup,u)=W(\WLdown,u)=\frac{s\,\sigma(q)}{\sigma(qu)},\quad W(\WLout,u)=W(\WLin,u)=1.\end{equation}
Now use the assignments of~\eqref{Oassig} to define the OSASM partition function in terms of the DSASM partition function~\eqref{Z} as
\begin{equation}\label{ZO}\ZO_n(u_1,\ldots,u_n)=\Bigl(\bigl(Z_n(u_1,\ldots,u_n)|_{\alpha=\beta=s\sigma(q),\;\gamma=\delta=1,\;\phi(u)=1/\sigma(q^2u^2)}\bigr)/s^{\odd(n)}\Bigr)\Big|_{s\rightarrow0}.\end{equation}

It follows that the OSASM partition function can be written directly in terms of weights as
\begin{equation}\label{ZOdir}\ZO_n(u_1,\ldots,u_n)=\sum_{C\in\SVCO(n)}\,\prod_{i=1}^n W(C_{ii},u_i)\,\prod_{1\le i<j\le n}W(C_{ij},u_iu_j),\end{equation}
where the left boundary weights are given by
$W(\WLup,u)=W(\WLdown,u)=\sigma(q)/\sigma(qu)$ and $W(\WLout,u)=W(\WLin,u)=1$, the bulk weights are given in Table~\ref{weights},
and $\SVCO(n)$ denotes the subset of $\SVC(n)$ which corresponds to $\OSASM(n)$ under the bijection of~\eqref{bij}.
Hence, $\SVCO(n)$ is the set of six-vertex model configurations on $\G_n$ in which there are no local configurations of type~\Lup\ or~\Ldown\ at the left boundary vertices for~$n$ even,
or exactly one local configuration of type~\Lup\ or~\Ldown\ at the left boundary vertices for~$n$ odd.

It follows from~\eqref{Z},~\eqref{leftW} and~\eqref{leftWO} that the OSASM partition
function~\eqref{ZO} is related to the specialized DSASM partition function~\eqref{ZZ}, by
\begin{equation}\label{ZOZ}\ZO_n(u_1,\ldots,u_n)=\frac{\sigma(q)^n}{\prod_{i=1}^n\sigma(qu_i)}\:\bigl(\Z_n(u_1,\ldots,u_n)/s^{\odd(n)}\bigr)|_{s\rightarrow0},\end{equation}
and it follows from~\eqref{ZPf1} (and~\eqref{PfFirstRowToLastCol}) that a Pfaffian expression for the OSASM partition
function is
\begin{multline}\label{ZOPf}\ZO_n(u_1,\ldots,u_n)\\[-2mm]
=\prod_{1\le i<j\le n}\!\!\frac{\sigmah(q^2u_iu_j)\,\sigmah(q^2\u_i\u_j)}{\sigmah(u_i\u_j)}\;
\Pf_{1\le i<j\le n+\odd(n)}\left(\begin{cases}\displaystyle\frac{\sigmah(u_i\u_j)}{\sigmah(q^2u_iu_j)\,\sigmah(q^2\u_i\u_j)},&j\le n\\[4.5mm]
\displaystyle\frac{\sigma(q)}{\sigma(qu_i)},&j=n+1\end{cases}\,\right).\end{multline}
Note that for $n$ even,~\eqref{ZOPf} is
the Pfaffian expression for the even-order OSASM partition function obtained by Kuperberg~\cite[$Z_{\mathrm{O}}(n;\vec{x})$ case of Theorem~10]{Kup02},
as discussed after Theorem~\ref{ZPftheorem}.

It follows from a result of Mead~\cite[Prop.~2.30]{Mea25} that the Izergin--Korepin determinant formula
for the partition function of the six-vertex model on a square with domain-wall boundary conditions (as discussed in Section~\ref{intro})
can be regarded as a special case of the Pfaffian formula~\eqref{ZOPf} for the even-order OSASM partition function.

It was observed by Kuperberg~\cite[Remark after Thm.~10]{Kup02} that the even-order OSASM partition function satisfies
\begin{equation}\label{ZOeveninv}\ZO_{2n}(\u_1,\dots,\u_{2n})=\ZO_{2n}(u_1,\dots,u_{2n}),\end{equation}
where this follows from~\eqref{ZOPf}, with~$n$ replaced by $2n$, since by considering $\u_1,\dots,\u_{2n}$ instead of $u_1,\dots,u_{2n}$,
each term $\sigmah(u_i\u_j)$ on the RHS of~\eqref{ZOPf} becomes $\sigmah(\u_iu_j)=-\sigmah(u_i\u_j)$, so that a factor of $(-1)^{2n(2n-1)/2}=(-1)^n$ arises from the product over $1\le i<j\le2n$,
which cancels with a factor of $(-1)^n$ that arises (after applying~\eqref{Pfmult}) from the Pfaffian.

It follows from~\eqref{ZX},~\eqref{OSASMX1} and~\eqref{ZOZ} that the OSASM partition function~\eqref{ZO} and OSASM generating function~\eqref{OSASMX} are related by
\begin{multline}\label{ZXO}\ZO_n(z,\underbrace{1,\ldots,1}_{n-1})=\frac{\sigma(q)\,\sigma(q^2)^{(n-1)(n-2)/2}\,\sigma(q^2\z)^{n-1}}{\sigma(q^4)^{n(n-1)/2}\,\sigma(q^2z)}\\
\times\biggl[\frac{(q+\q)\,\sigma(q^2\z)}{\sigma(q^2z)}\,\XO_n\biggl(\!q^2+\q^2,\frac{\sigma(q^2z)}{\sigma(q^2\z)}\biggr)
-\odd(n)\,\frac{\sigma(z)}{\sigma(qz)}\,\XO_{n-1}\bigl(q^2+\q^2,1\bigr)\biggr],\end{multline}
where $z$ is arbitrary.  Setting $z=1$ in~\eqref{ZXO} gives (analogously to~\eqref{Zunref1})
$\ZO_n(\underbrace{1,\ldots,1}_n)=\bigl(q^2+\q^2\bigr)^{-n(n-1)/2}\,\XO_n\bigl(q^2+\q^2,1\bigr)$, and then setting $q^2+\q^2=1$ (and using $\XO_n(1,1)=|\OSASM(n)|$) gives
\begin{equation}\label{ZOnumOSASM}\ZO_n(\underbrace{1,\ldots,1}_n)|_{q^2+\q^2=1}=\big|\OSASM(n)|,\end{equation}
where this can also be seen directly from~\eqref{ZOdir} using $|\SVCO(n)|=|\OSASM(n)|$ and the fact that each weight on the RHS of~\eqref{ZOdir} is~1 for $u_1=\ldots=u_n=1$ and $q^2+\q^2=1$.

It was shown by Kumari~\cite[Lem.~4.1]{Kum26} that, at $q^2+\q^2=1$, the even-order OSASM partition function with parameters $u_1,\ldots,u_n,\u_1,\ldots,\u_n,1,\mathfrak{i}$
and the odd-order OSASM partition function with parameters $u_1,\ldots,u_n,\u_1,\ldots,\u_n,1$ are related by
\begin{multline}\label{ZOthm}\ZO_{2n+2}(u_1,\ldots,u_n,\u_1,\ldots,\u_n,1,\mathfrak{i})\big|_{q^2+\q^2=1}\\
=3^{-n}\,\Biggl(\prod_{i=1}^n\frac{u_i^2-1+\u_i^2}{u_i+\u_i}\Biggr)^{\!\!2}\,\ZO_{2n+1}(u_1,\ldots,u_n,\u_1,\ldots,\u_n,1)\big|_{q^2+\q^2=1},\end{multline}
where $\mathfrak{i}$ denotes the imaginary unit.

\subsection{\,\,\,\,Relations between even-order OSASMs, odd-order horizontally symmetric ASMs and even-order ASMs}\label{OSASMHSASM}
It is known that there are certain close enumerative connections between $2n\times2n$ OSASMs,
$(2n+1)\times(2n+1)$ horizontally symmetric ASMs (HSASMs) and $2n\times2n$ ASMs, and that these provide explicit formulae
for $|\OSASM(2n)|$ and $X_{2n}^\mathrm{O}(1,t)$. These connections and formulae are sometimes
stated in terms of vertically symmetric ASMs, but are stated here in terms of HSASMs, since
the connections involving the statistic~$T$ then seem more natural.

Specifically, it was shown by Kuperberg~\cite[Thm.~2, second eq.~\& Thm.~5, first eq.]{Kup02} that
\begin{align}\notag|\OSASM(2n)|&=|\HSASM(2n+1)|\\
\label{numOSASMeven2}&=\prod_{i=1}^n\frac{(6i-2)!}{(2n+2i)!},\end{align}
where the second equality is also given in~\eqref{numVSASM} (since $|\HSASM(2n+1)|=|\VSASM(2n+1)|$),
and it was shown by Razumov and Stroganov~\cite[Secs.~3~\&~4]{RazStr04b} that
\begin{align}\notag\XO_{2n}(1,t)&=\sum_{A\in\HSASM(2n+1)}t^{T(A)}\\
\notag&=\frac{2\,|\OSASM(2n)|}{(t+1)\,|\ASM(2n)|}\sum_{A\in\ASM(2n)}t^{T(A)+1}\\
\notag&=\frac{1}{(t+1)\,(2n)!}\,\prod_{i=1}^{n-1}\frac{(6i-2)!}{(2n+2i)!}\,\sum_{i=0}^{2n-1}\frac{(2n+i-1)!\,(4n-i-2)!}{i!\,(2n-i-1)!}\,t^{i+2}\\
\label{OSASMX5}&=\frac{1}{(2n)!}\,\prod_{i=1}^{n-1}\frac{(6i-2)!}{(2n+2i)!}\,\sum_{0\le i\le j\le2n-2}\frac{(2n+i-1)!\,(4n-i-2)!}{i!\,(2n-i-1)!}\:(-1)^{i+j}\,t^{j+2},\end{align}
where, using~\eqref{OSASMX}, $\XO_{2n}(1,t)=\sum_{A\in\OSASM(2n)}t^{T(A)}$.

No bijective proofs are currently known for the enumerative relations between $\OSASM(2n)$, $\HSASM(2n+1)$ and $\ASM(2n)$
given by the first equality of~\eqref{numOSASMeven2} and first two equalities of~\eqref{OSASMX5}.

It can be shown using~\eqref{OSASMX5} that $\XO_{2n}(1,-1)$ is given by the product formula
\begin{equation}\label{OSASMX6}\XO_{2n}(1,-1)=\frac{(3n-1)!}{2^n\,(2n-1)!}\,\prod_{i=1}^{n-1}\frac{(6i-2)!}{(2n+2i-1)!},\end{equation}
where the details of the derivation are as follows.  The third expression for $\XO_{2n}(1,t)$ in~\eqref{OSASMX5} gives
\begin{align*}\XO_{2n}(1,t)&=\frac{(2n-1)!}{2n\,(t+1)}\,\prod_{i=1}^{n-1}\frac{(6i-2)!}{(2n+2i)!}\,\sum_{i=0}^{2n-1}\binom{2n+i-1}{i}\,\binom{4n-i-2}{2n-i-1}\,t^{i+2}\\
&=-\frac{(2n-1)!}{2n\,(t+1)}\,\prod_{i=1}^{n-1}\frac{(6i-2)!}{(2n+2i)!}\:[x^{2n-1}]\,\frac{t^2}{\bigl((1+x)(1+xt)\bigr)^{2n}},\end{align*}
which then gives
\begin{align*}\XO_{2n}(1,-1)&=-\frac{(2n-1)!}{2n}\,\prod_{i=1}^{n-1}\frac{(6i-2)!}{(2n+2i)!}\:[x^{2n-1}]\,\biggl(\frac{1}{(1+x)^{2n}}\,\frac{d}{dt}\,\frac{t^2}{(1+xt)^{2n}}\bigg|_{t=-1}\,\biggr)\\
&=\frac{(2n-1)!}{n}\,\prod_{i=1}^{n-1}\frac{(6i-2)!}{(2n+2i)!}\:[x^{2n-1}]\,\biggl(\frac{1}{(1-x^2)^{2n}}+\frac{nx}{(1-x)(1-x^2)^{2n}}\biggr)\\
&=\frac{(2n-1)!}{n}\,\prod_{i=1}^{n-1}\frac{(6i-2)!}{(2n+2i)!}\:\biggl(0+n\,[x^{2n-2}]\,\frac{1}{(1-x)(1-x^2)^{2n}}\biggr)\\
&=(2n-1)!\,\prod_{i=1}^{n-1}\frac{(6i-2)!}{(2n+2i)!}\:[x^{n-1}]\,\frac{1}{(1-x)^{2n+1}}\\
&=(2n-1)!\,\prod_{i=1}^{n-1}\frac{(6i-2)!}{(2n+2i)!}\,\binom{3n-1}{n-1}\\
&=\frac{(3n-1)!}{2^n\,(2n-1)!}\,\prod_{i=1}^{n-1}\frac{(6i-2)!}{(2n+2i-1)!},\end{align*}
where the third-last equality follows from the identity $[x^{2i}]\,f(x^2)/(1-x)=[x^i]\,f(x)/(1-x)$, for any power series $f(x)$ and nonnegative integer~$i$.

\subsection{A symmetry property for even-order OSASMs}\label{EvenOSASMSymmConjSect}
For even-order OSASMs, a certain symmetry property holds, as follows, where this was conjectured in the first arXiv version of this
paper~\cite[Conj.~16]{BehFisKou23a} and subsequently proved by Kumari~\cite[Thm.~5.1]{Kum26}.
\begin{proposition}[Kumari~\cite{Kum26}]\label{evenOSASMsymmprop}
The even-order OSASM generating function $\XO_{2n}(r,t)$ satisfies
\begin{equation}\label{OSASMXt1}\XO_{2n}(r,t)=t^{2n+2}\,\XO_{2n}(r,\bar{t}).\end{equation}
Equivalently, by equating coefficients of $r$ and $t$ on both sides of~\eqref{OSASMXt1},
\begin{multline}\label{OSASMXt2}|\{A\in\OSASM(2n)\,|\,R(A)=\rho,\,T(A)=\tau\}|\\
=|\{A\in\OSASM(2n)\,|\,R(A)=\rho,\,T(A)=2n+2-\tau\}|,\end{multline}
for any $\rho$ and $\tau$.
\end{proposition}
\begin{proof}
It follows from~\eqref{ZOeveninv} that
\begin{equation}\label{Zid}\ZO_{2n}(z,\underbrace{1,\dots,1}_{2n-1})=\ZO_{2n}(\z,\underbrace{1,\dots,1}_{2n-1}).\end{equation}
Applying~\eqref{ZXO} (with $n$ replaced by $2n$) to each side of~\eqref{Zid} gives
\begin{equation}\label{OSASMXt3}\XO_{2n}\biggl(\!q^2+\q^2,\frac{\sigma(q^2z)}{\sigma(q^2\z)}\biggr)=
\biggl(\frac{\sigma(q^2z)}{\sigma(q^2\z)}\biggr)^{\!2n-2}\,\XO_{2n}\biggl(\!q^2+\q^2,\frac{\sigma(q^2\z)}{\sigma(q^2z)}\biggr),\end{equation}
and setting $r=q^2+\q^2$ and $t=\sigma(q^2z)/\sigma(q^2\z)$ (which enables arbitrary~$r$ and~$t$ to be parameterized in terms of~$q$ and~$z$) in~\eqref{OSASMXt3}
then gives~\eqref{OSASMXt1}.
\end{proof}
Some comments on Proposition~\ref{evenOSASMsymmprop} are as follows.
\begin{list}{$\bullet$}{\setlength{\labelwidth}{3mm}\setlength{\leftmargin}{5mm}\setlength{\labelsep}{2mm}\setlength{\topsep}{0.9mm}}
\item It follows using~\eqref{MRS} that for each OSASM in either of the sets in~\eqref{OSASMXt2}, the numbers of~$-1$'s and~$1$'s
are $\rho-n$ and $\rho+n$, respectively.
\item It can be shown straightforwardly that the set of pairs $(\rho,\tau)$ for which each side of~\eqref{OSASMXt2} is nonzero is
\begin{multline*}\textstyle\quad\Bigl\{(2i+n,j)\,\Big|\,i=0,\ldots,2\bigl\lfloor\frac{(n-1)^2}{4}\bigr\rfloor,\\[-1.5mm]
\textstyle j=\max\bigl(2,i-2\bigl\lfloor\frac{n(n-4)}{4}\bigr\rfloor\bigr),\ldots,2n+2-\max\bigl(2,i-2\bigl\lfloor\frac{n(n-4)}{4}\bigr\rfloor\bigr)\Bigr\}\end{multline*}
\item As an example of~\eqref{OSASMXt2}, consider $n=3$, $\rho=5$ and $\tau=3$.  Then the sets in~\eqref{OSASMXt2} both have size~3, and are
\begin{equation*}\left\{\left(\begin{smallmatrix}0&0&1&0&0&0\\0&0&0&1&0&0\\1&0&0&-1&1&0\\0&1&-1&0&0&1\\0&0&1&0&0&0\\0&0&0&1&0&0\end{smallmatrix}\right),\,
\left(\begin{smallmatrix}0&0&1&0&0&0\\0&0&0&1&0&0\\1&0&0&-1&0&1\\0&1&-1&0&1&0\\0&0&0&1&0&0\\0&0&1&0&0&0\end{smallmatrix}\right),\,
\left(\begin{smallmatrix}0&0&1&0&0&0\\0&0&0&0&1&0\\1&0&0&0&-1&1\\0&0&0&0&1&0\\0&1&-1&1&0&0\\0&0&1&0&0&0\end{smallmatrix}\right)\right\}\end{equation*}
on the LHS, and
\begin{equation*}\left\{\left(\begin{smallmatrix}0&0&0&0&1&0\\0&0&1&0&0&0\\0&1&0&0&-1&1\\0&0&0&0&1&0\\1&0&-1&1&0&0\\0&0&1&0&0&0\end{smallmatrix}\right),\,
\left(\begin{smallmatrix}0&0&0&0&1&0\\0&0&1&0&-1&1\\0&1&0&0&0&0\\0&0&0&0&1&0\\1&-1&0&1&0&0\\0&1&0&0&0&0\end{smallmatrix}\right),\,
\left(\begin{smallmatrix}0&0&0&0&1&0\\0&0&0&1&-1&1\\0&0&0&0&1&0\\0&1&0&0&0&0\\1&-1&1&0&0&0\\0&1&0&0&0&0\end{smallmatrix}\right)\right\}\end{equation*}
on the RHS.
\item An alternative proof of the $r=1$ case of~\eqref{OSASMXt1} (and hence also the $r=-1$ case of~\eqref{OSASMXt1}, due to~\eqref{OSASMXmin})
can be obtained from the first equality of~\eqref{OSASMX5} using vertical reflection of HSASMs,
which satisfies $T(A)+T(A^v)=2n+2$ for all $A\in\HSASM(2n+1)$, where $A^v$ denotes the reflection of $A$ in a central vertical line.
\item An alternative proof of the $r=\sqrt{2}$ case of~\eqref{OSASMXt1} has been obtained by Lee~\cite[Eq.~(1.8)]{Lee24},~\cite[Eq.~(4.4)]{Lee25}, using
a connection between this case and off-diagonally symmetric domino tilings of odd-order Aztec diamonds. Specifically, the number $|O(2n-1,k)|$ of
such domino tilings of an Aztec diamond of order $2n-1$, with a certain boundary defect indexed by $k\in\{1,\ldots,2n-1\}$, is given by~\cite[Eq.~(1.6)]{Lee24},~\cite[Eq.~(4.2)]{Lee25} as
\begin{equation}\label{Lee}|O(2n-1,k)|=2^{n/2-1}\,[t^{k+1}]\,\XO_{2n}(\sqrt{2},t),\end{equation}
where the terms $2^{\rho+n}$ in~\cite[Eq.~(1.6)]{Lee24},~\cite[Eq.~(4.2)]{Lee25} and
$2^\rho$ in~\cite[Eq.~(1.8)]{Lee24},~\cite[Eq.~(4.4)]{Lee25} should be replaced by $2^{(\rho+n)/2}$ and $2^{\rho/2}$, respectively.
For further details, see also Lee~\cite[Sec.~1.2]{Lee23},~\cite[Secs.~1.3~\&~3.1]{Lee25}.
\item For $\rho=n$ and $\tau=2,\ldots,2n$, each side of~\eqref{OSASMXt2} is $(2n-3)!!$, since the associated sets consist of all $2n\times2n$ diagonally symmetric
permutation matrices with zero trace, and the unique~$1$ in the first row~$1$ in a fixed position.
\item The $\tau=2$ and $\tau=2n$ cases of~\eqref{OSASMXt2} can alternatively be obtained from~\eqref{XOcoeff2} and~\eqref{XOcoeffn},
which give
\begin{equation}[t^2]\,\XO_{2n}(r,t)=[t^{2n}]\,\XO_{2n}(r,t)=r\,\XO_{2n-2}(r,1).\end{equation}
This can also be proved directly by observing that the sets in~\eqref{OSASMXt2} in these cases are
$\Bigl\{\biggl({\scriptsize\begin{array}{@{\,}c@{\:}|@{\:}c@{\:}|@{\:}c@{\,}}0&1&0\\[-0.4mm]\hline1&0&0\\[-0.4mm]\hline0&0&A\end{array}}\biggr)\,\Big|\,A\in\OSASM(2n-2),\,
R(A)=\rho-1\Bigr\}$ and
$\Bigl\{\biggl({\scriptsize\begin{array}{@{\,}c@{\:}|@{\:}c@{\:}|@{\:}c@{\,}}0&0&1\\[-0.4mm]\hline0&A&0\\[-0.4mm]\hline1&0&0\end{array}}\biggr)\,\Big|\,A\in\OSASM(2n-2),\,
R(A)=\rho-1\Bigr\}$.
\item For the coefficients of $t^3$ on both sides of~\eqref{OSASMXt1}, corresponding to the case $\tau=3$ in~\eqref{OSASMXt2}, it is conjectured that the second equality holds in
\begin{equation}\label{XO3conj}[t^3]\,\XO_{2n}(r,t)=[t^{2n-1}]\,\XO_{2n}(r,t)=r\bigl((n-2)r^2+1\bigr)\,\XO_{2n-2}(r,1),\end{equation}
where the first equality follows from~\eqref{OSASMXt1}.
The validity of the $r=\sqrt{2}$ case of this conjecture follows from a result of Lee~\cite[first eq.\ of~(1.13)]{Lee24},~\cite[first eq.\ of~(4.58)]{Lee25},
together with~\eqref{XOcoeff2} and~\eqref{Lee}.
\item An interesting, but seemingly difficult, open problem is to obtain a combinatorial proof of~\eqref{OSASMXt2} by finding an involution on $\OSASM(2n)$ with
the properties $R(A')=R(A)$ and $T(A)+T(A')=2n+2$, for all $A\in\OSASM(2n)$,
where~$A'$ denotes the image of~$A$ under the involution.
\item It may also be interesting to obtain an algebraic proof of~\eqref{OSASMXt1} using~\eqref{OSASMX2t}.
\end{list}

\subsection{OSASM partition functions and symplectic characters}\label{OSASMSymplSect}
It is known that for $q^2+\q^2=1$,
the partition function for even-order OSASMs is, up to a simple prefactor, a certain symplectic character.

Specifically, it follows from results of Okada~\cite[Thm.~2.5(2), second eq.]{Oka06}, and Razumov and Stroganov~\cite[Thm.~5]{RazStr04b}, that
\begin{equation}\label{sympleven}
\ZO_{2n}(u_1,\ldots,u_{2n})|_{q^2+\q^2=1}=3^{-n(n-1)}\,\mathit{sp}_{(n-1,n-1,n-2,n-2,\ldots,1,1,0,0)}(u_1^2,\ldots,u_{2n}^2),\end{equation}
where $\mathit{sp}_{(n-1,n-1,n-2,n-2,\ldots,1,1,0,0)}(u_1^2,\ldots,u_{2n}^2)$ denotes
the symplectic character (as given, for example, in Fulton and Harris~\cite[Eq.~(24.18)]{FulHar91}), indexed by the partition $(n-1,n-1,n-2,n-2,\ldots,1,1,0,0)$
and with variables $u_1^2,\ldots,u_{2n}^2$.

Setting $u_1=\ldots=u_{2n}=1$ in~\eqref{sympleven} provides a proof that $|\OSASM(2n)|$
is given by the product formula in~\eqref{numOSASMeven2}, since the LHS of~\eqref{sympleven} is then (using~\eqref{ZOnumOSASM}) $|\OSASM(2n)|$,
and the RHS of~\eqref{sympleven} is, using the dimension formula for the symplectic group
(as given, for example, in Fulton and Harris~\cite[Exercise~24.20]{FulHar91}), the product formula in~\eqref{numOSASMeven2}.

An analogue of~\eqref{sympleven} for the partition function for odd-order OSASMs is
\begin{multline}\label{symplodd}
\ZO_{2n+1}(u_1,\ldots,u_n,\u_1,\ldots,\u_n,1)|_{q^2+\q^2=1}\\
=3^{-n^2}\,\Biggl(\prod_{i=1}^n\frac{u_i+\u_i}{u_i^2-1+\u_i^2}\Biggr)^{\!\!2}\;
\mathit{sp}_{(n,n,n-1,n-1,\ldots,1,1,0,0)}(u_1^2,\ldots,u_n^2,\u_1^2,\ldots,\u_n^2,1,-1),\end{multline}
where this was conjectured in the first arXiv version of this
paper~\cite[Conj.~17]{BehFisKou23a} and subsequently proved by Kumari~\cite[Thm.~4.2]{Kum26}.  The proof consists of applying~\eqref{sympleven} on
the LHS of~\eqref{ZOthm}, which then gives~\eqref{symplodd}.
Note that the symplectic character on the RHS of~\eqref{symplodd} can also be written as $\mathit{sp}_{(n,n,n-1,n-1,\ldots,1,1,0,0)}(u_1^2,\ldots,u_n^2,u_1^2,\ldots,u_n^2,1,-1)$,
using the property that a symplectic character is invariant under the replacement of any variable by its reciprocal.

\subsection{Product formula for the number of odd-order OSASMs}\label{OddOSASMConjSect}
The number of $(2n+1)\times(2n+1)$ OSASMs is given by the product formula
\begin{equation}\label{numOSASModd}|\OSASM(2n+1)|=\frac{2^{n-1}\,(3n+2)!}{(2n+1)!}\,\prod_{i=1}^n\frac{(6i-2)!}{(2n+2i+1)!},
\end{equation}
where this was conjectured in the first arXiv version of this
paper~\cite[Conj.~15]{BehFisKou23a} and subsequently proved by Kumari~\cite[Cor.~4.4]{Kum26}.
The proof consists of setting $u_1=\ldots=u_n=1$ in~\eqref{ZOthm}, and applying~\eqref{ZXO} (and $\ZO_{2n+2}(1,\ldots,1,\mathfrak{i})=\ZO_{2n+2}(\mathfrak{i},1,\ldots,1)$,
which follows from the property that the OSASM partition function is symmetric in all its variables) on the LHS and~\eqref{ZOnumOSASM} on the RHS, to give
\begin{equation}\label{OSASMX7}\XO_{2n+2}(1,-1)=2^{-2n}\,|\OSASM(2n+1)|.\end{equation}
Applying~\eqref{OSASMX6} on the LHS of~\eqref{OSASMX7} then gives~\eqref{numOSASModd}.

Note that an analogue of~\eqref{OSASMX7} for $\XO_{2n+2}(\sqrt{2},-1)$ is
\begin{equation}\XO_{2n+2}(\sqrt{2},-1)=\sqrt{2}\,\XO_{2n}(\sqrt{2},1),\end{equation}
where this follows from an identity of Lee~\cite[Eq.~(1.14)]{Lee24},~\cite[Eq.~(4.59)]{Lee25}, together with~\eqref{XOcoeffn} and~\eqref{Lee}.

An alternative proof of~\eqref{numOSASModd}, which avoids the use of~\eqref{OSASMX6}, consists of setting $u_1=\ldots=u_n=1$ in~\eqref{symplodd}, and applying~\eqref{ZOnumOSASM} on the LHS, to give
\begin{equation}\label{numoddOSASMsympl}|\OSASM(2n+1)|=3^{-n^2}\,2^{2n}\,\mathit{sp}_{(n,n,n-1,n-1,\ldots,1,1,0,0)}(\underbrace{1,\ldots,1}_{2n+1},-1).\end{equation}
Applying an identity of Rosengren~\cite[p.~33, last eq'n with $l=1$]{Ros13} to the symplectic character on the RHS of~\eqref{numoddOSASMsympl} then gives a product formula for $|\OSASM(2n+1)|$,
which can be simplified straightforwardly to the product formula in~\eqref{numOSASModd}.  The reference to the identity in~\cite{Ros13} was provided
by an anonymous reviewer of this paper, following a comment in the first arXiv version of this paper~\cite[end of Sec.~8]{BehFisKou23a}.

Finally, note that by considering~\eqref{OSASMX7}, and using~\eqref{OSASMX2t} for $\XO_{2n+2}(1,-1)$ on the LHS
and the second Pfaffian in~\eqref{numOSASM} for $|\OSASM(2n+1)|$ on the RHS, it follows
(after applying~\eqref{HomogPf2} on the RHS with $m=2n-1$, $f(u,v)=(v-u)/((1-uv)((1-ru)(1-rv)-uv))$, $g_{2n}(u)=(t-1-rtu)/((rt-(t-1)u)(r-(t-1+r^2)u))$, $h(u)=1$ and $k(u)=1+u$)
that
\begin{multline}\label{OSASMX8}\Pf_{0\le i<j\le2n+1}\left(\begin{cases}\displaystyle[u^iv^j]\,\frac{v-u}{(1-uv)(1-u-v)},&j\le2n\\\bigl((-1)^i+2^i\bigr)/2,&j=2n+1\end{cases}\right)\\
=\Pf_{0\le i<j\le2n+1}\left(\begin{cases}\displaystyle[u^iv^j]\,\frac{v-u}{(1-uv)(1-u-v)},&j\le2n\\(-1)^i,&j=2n+1\end{cases}\right),\end{multline}
where the coefficients which comprise the entries with $j\le2n$ in these Pfaffians are given in~\eqref{numOSASMentries}.
Although the two sides of~\eqref{OSASMX8} are similar, it still seems to be difficult to prove directly that they are equal.

\section{Involutions on DSASMs}\label{invol}
In this section, certain involutions on $\DSASM(n)$ are studied.
The specific involutions under consideration are antidiagonal reflection and two
involutions, denoted~$\ast$ and~$\dagger$, each of which acts only on diagonal, superdiagonal and subdiagonal DSASM entries.

Antidiagonal reflection is considered in Section~\ref{AntidiagreflSect}, and the involutions~$\ast$ and~$\dagger$ are defined in Sections~\ref{astSect} and~\ref{daggerSect}.
Properties of~$\ast$ and~$\dagger$ are studied in Section~\ref{astdaggerPropSect}, and the numbers of $\ast$-invariant, $\dagger$-invariant and $\ast\dagger$-invariant DSASMs and OSASMs
are determined in Section~\ref{astdaggerNumInvSect}.

For an involution~$\nu$ on $\DSASM(n)$, the image of a DSASM~$A$ under~$\nu$ will be denoted~$A^\nu$,
and the sets of $\nu$-invariant $n\times n$ DSASMs and OSASMs will be denoted $\DSASM^\nu(n)$ and $\OSASM^\nu(n)$, respectively, i.e.,
\begin{equation*}
\DSASM^\nu(n)=\{A\in\DSASM(n)\mid A^\nu=A\},\quad\OSASM^\nu(n)=\{A\in\OSASM(n)\mid A^\nu=A\}.
\end{equation*}

\subsection{Antidiagonal reflection on $\DSASM(n)$}\label{AntidiagreflSect}
A simple and immediately-apparent involution on $\DSASM(n)$ is antidiagonal reflection~$a$,
i.e., $(A^a)_{ij}=(\mathcal{A}A)_{ij}=A_{n+1-j,n+1-i}$, for any $A\in\DSASM(n)$, where the notation~$\mathcal{A}$ is used in Section~\ref{ASMSymmClassSect}.
It can be seen that
the statistics~$R$ and~$S$ satisfy $R(A^a)=R(A)$ and $S(A^a)=S(A)$,
and also that the further statistics which will be defined in Section~\ref{FurthStatSect} satisfy
$P(A^a)=P(A)$, $S_+(A^a)=S_+(A)$ and $S_-(A^a)=S_-(A)$.
It follows that $X_n(r,s,t)$ and $\X_n(p,r,s_+,s_-,t)$ remain unchanged if $T(A)$ on the RHS of~\eqref{Xrst} and~\eqref{Xss}
is replaced by~$T(A^a)$, i.e., by $n+1$ minus the row number of the~$1$ in the last column of~$A$.
The set of $a$-invariant $n\times n$ DSASMs is $\DSASM^a(n)=\DASASM(n)$,
and information regarding $|\DASASM(n)|$ is given in Section~\ref{ASMSymmClassSect}.

\subsection{The involution \texorpdfstring{$\ast$}{*} on $\DSASM(n)$}\label{astSect}
Another natural involution~$\ast$ on $\DSASM(n)$
will now be introduced. For $n=1$, the single element of $\DSASM(1)$ is necessarily invariant under~$\ast$.
For $n\ge2$, the mapping can be defined most easily in terms of six-vertex model configurations, as follows.
Consider any $A\in\DSASM(n)$, and let~$C$ and~$C^\ast$ be the six-vertex model configurations which, under the bijection of~\eqref{bij},
correspond to $A$ and $A^\ast$, respectively. In the graph~$\G_n$ (as shown in~\eqref{Tn}),
the vertices $(k-1,k)$, for $k=2,\ldots,n$, will be referred to as superdiagonal vertices, i.e., these are the
$n-1$ bulk vertices which are adjacent to left boundary vertices.
The definition of~$\ast$ is now that~$C^\ast$ is obtained from~$C$ by, for all $k=2,\ldots,n$, interchanging the orientations of the two edges
which connect the superdiagonal vertex $(k-1,k)$ to the left boundary vertices $(k-1,k-1)$ and $(k,k)$
(where an orientation is regarded as being into or out of $(k-1,k)$), and leaving the orientations
of all other edges unchanged. \psset{unit=7mm}
Hence, the action of $\ast$ at each superdiagonal vertex is that local configurations~\Viii\ and~\Vv\ are interchanged,
local configurations~\Vvi\ and~\Viv\ are interchanged, and local configurations~\Vi\ and~\Vii\ each remain unchanged.
It follows that~$C^\ast$ is an element of $\SVC(n)$, since
there are still two edges directed into and two edges directed out of each bulk vertex,
and the orientations of the top and right boundary edges are unchanged (and hence still directed upward and leftward, respectively).
It also follows immediately that~$\ast$ is an involution.

Note that if $A_{ij}\ne A^\ast_{ij}$,
then $i-j$ is $0$, $-1$ or $1$, i.e., entries which differ between~$A$ and~$A^\ast$ are on the diagonal, superdiagonal or subdiagonal.

Note also that~$\ast$ is the composition of $n-1$ mutually commuting local involutions $\ast_2$, \ldots, $\ast_n$, where~$\ast_k$
interchanges the orientations of the two edges which connect $(k-1,k)$ to $(k-1,k-1)$ and $(k,k)$,
while leaving the orientations of all other edges unchanged.

As an example of the action of~$\ast$, for the DSASM $A$ in~\eqref{DSASMexample}, the six-vertex model
configurations~$C$ (as given in~\eqref{configex}) and $C^\ast$ are
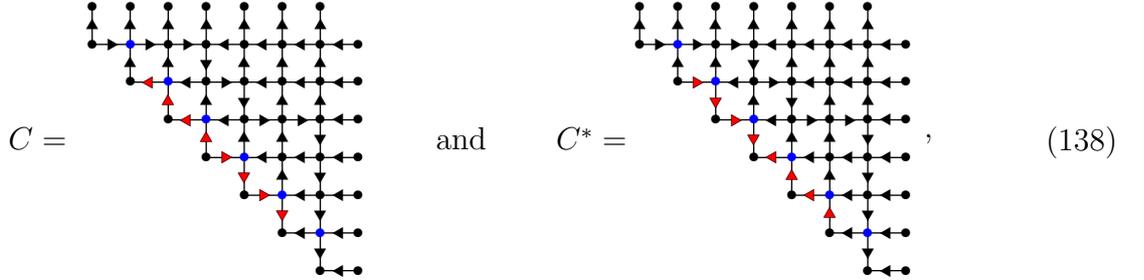
\psset{unit=5mm}
\begin{equation*}\raisebox{-17mm}{
\pspicture(-1.2,0.8)(11.9,8.2)\rput[l](-1.2,4.5){$C=$}
\multirput(1,8)(1,0){7}{$\ss\bullet$}\multirput(3,7)(1,0){6}{$\ss\bullet$}\multirput(4,6)(1,0){5}{$\ss\bullet$}\multirput(5,5)(1,0){4}{$\ss\bullet$}
\multirput(6,4)(1,0){3}{$\ss\bullet$}\multirput(7,3)(1,0){2}{$\ss\bullet$}\multirput(8,1)(0,1){2}{$\ss\bullet$}
\psline[linewidth=0.5pt](1,8)(1,7)(8,7)\psline[linewidth=0.5pt](2,8)(2,6)(8,6)\psline[linewidth=0.5pt](3,8)(3,5)(8,5)\psline[linewidth=0.5pt](4,8)(4,4)(8,4)
\psline[linewidth=0.5pt](5,8)(5,3)(8,3)\psline[linewidth=0.5pt](6,8)(6,2)(8,2)\psline[linewidth=0.5pt](7,8)(7,1)(8,1)
\multirput(1,7)(1,-1){7}{$\ss\bullet$}\multirput(2,7)(1,-1){6}{$\color{blue}\ss\bullet$}
\multirput(1,7)(1,0){7}{\psdots[dotstyle=triangle*,dotscale=1.1](0,0.5)}
\multirput(7,1)(0,1){7}{\psdots[dotstyle=triangle*,dotscale=1.1,dotangle=90](0.5,0)}
\psdots[dotstyle=triangle*,dotscale=1.1](2,6.5)(3,6.5)(5,6.5)(6,6.5)(7,6.5)(4,5.5)(6,5.5)(7,5.5)(5,4.5)(6,4.5)(6,3.5)
\psdots[dotstyle=triangle*,dotscale=1.1,dotangle=180](4,6.5)(5,5.5)(7,4.5)(7,3.5)(7,2.5)(7,1.5)
\psdots[dotstyle=triangle*,dotscale=1.1,dotangle=-90](1.5,7)(2.5,7)(3.5,7)(4.5,6)(5.5,5)(6.5,5)
\psdots[dotstyle=triangle*,dotscale=1.1,dotangle=90](4.5,7)(5.5,7)(6.5,7)(3.5,6)(5.5,6)(6.5,6)(4.5,5)(5.5,4)(6.5,4)(6.5,3)(6.5,2)
\psdots[dotstyle=triangle,dotscale=1.1,fillcolor=red](3,5.5)(4,4.5)
\psdots[dotstyle=triangle,dotscale=1.1,dotangle=180,fillcolor=red](5,3.5)(6,2.5)
\psdots[dotstyle=triangle,dotscale=1.1,dotangle=90,fillcolor=red](2.5,6)(3.5,5)
\psdots[dotstyle=triangle,dotscale=1.1,dotangle=-90,fillcolor=red](4.5,4)(5.5,3)
\rput(9.83,4.5){and}\endpspicture
\pspicture(-1.2,0.8)(8.6,8.2)\rput[l](-1.2,4.5){$C^\ast=$}
\multirput(1,8)(1,0){7}{$\ss\bullet$}\multirput(3,7)(1,0){6}{$\ss\bullet$}\multirput(4,6)(1,0){5}{$\ss\bullet$}\multirput(5,5)(1,0){4}{$\ss\bullet$}
\multirput(6,4)(1,0){3}{$\ss\bullet$}\multirput(7,3)(1,0){2}{$\ss\bullet$}\multirput(8,1)(0,1){2}{$\ss\bullet$}
\psline[linewidth=0.5pt](1,8)(1,7)(8,7)\psline[linewidth=0.5pt](2,8)(2,6)(8,6)\psline[linewidth=0.5pt](3,8)(3,5)(8,5)\psline[linewidth=0.5pt](4,8)(4,4)(8,4)
\psline[linewidth=0.5pt](5,8)(5,3)(8,3)\psline[linewidth=0.5pt](6,8)(6,2)(8,2)\psline[linewidth=0.5pt](7,8)(7,1)(8,1)
\multirput(1,7)(1,-1){7}{$\ss\bullet$}\multirput(2,7)(1,-1){6}{$\color{blue}\ss\bullet$}
\multirput(1,7)(1,0){7}{\psdots[dotstyle=triangle*,dotscale=1.1](0,0.5)}
\multirput(7,1)(0,1){7}{\psdots[dotstyle=triangle*,dotscale=1.1,dotangle=90](0.5,0)}
\psdots[dotstyle=triangle*,dotscale=1.1](2,6.5)(3,6.5)(5,6.5)(6,6.5)(7,6.5)(4,5.5)(6,5.5)(7,5.5)(5,4.5)(6,4.5)(6,3.5)
\psdots[dotstyle=triangle*,dotscale=1.1,dotangle=180](4,6.5)(5,5.5)(7,4.5)(7,3.5)(7,2.5)(7,1.5)
\psdots[dotstyle=triangle*,dotscale=1.1,dotangle=-90](1.5,7)(2.5,7)(3.5,7)(4.5,6)(5.5,5)(6.5,5)
\psdots[dotstyle=triangle*,dotscale=1.1,dotangle=90](4.5,7)(5.5,7)(6.5,7)(3.5,6)(5.5,6)(6.5,6)(4.5,5)(5.5,4)(6.5,4)(6.5,3)(6.5,2)
\psdots[dotstyle=triangle,dotscale=1.1,dotangle=180,fillcolor=red](3,5.5)(4,4.5)
\psdots[dotstyle=triangle,dotscale=1.1,fillcolor=red](5,3.5)(6,2.5)
\psdots[dotstyle=triangle,dotscale=1.1,dotangle=-90,fillcolor=red](2.5,6)(3.5,5)
\psdots[dotstyle=triangle,dotscale=1.1,dotangle=90,fillcolor=red](4.5,4)(5.5,3)
\rput(8.6,4.5){,}\endpspicture}\end{equation*}
for which
\begin{equation*}A=\left(\begin{smallmatrix}0&0&0&1&0&0&0\\0&\color{red}1&\color{red}0&-1&1&0&0\\0&\color{red}0&\color{red}1&\color{red}0&-1&0&1\\
1&-1&\color{red}0&0&\color{red}1&0&0\\0&1&-1&\color{red}1&\color{red}-1&\color{red}1&0\\0&0&0&0&\color{red}1&\color{red}0&0\\0&0&1&0&0&0&0\end{smallmatrix}\right)\quad\text{and}\quad
\rule[-5.1ex]{0ex}{0ex}A^\ast=\left(\begin{smallmatrix}0&0&0&1&0&0&0\\0&\color{red}0&\color{red}1&-1&1&0&0\\0&\color{red}1&\color{red}-1&\color{red}1&-1&0&1\\1&-1&\color{red}1&0&\color{red}0&0&0\\
0&1&-1&\color{red}0&\color{red}1&\color{red}0&0\\0&0&0&0&\color{red}0&\color{red}1&0\\0&0&1&0&0&0&0\end{smallmatrix}\right),\end{equation*}
where the superdiagonal vertices are shown in blue, the edge orientations which differ between~$C$ and~$C^\ast$ are shown in red,
and the matrix entries which differ between~$A$ and~$A^\ast$ are also shown in red.

\subsection{The involution \texorpdfstring{$\dagger$}{†} on $\DSASM(n)$}\label{daggerSect}
\psset{unit=7mm}
A related involution $\dagger$ on $\DSASM(n)$ is defined as follows.
For any DSASM~$A$, again let~$C$ and~$C^\dagger$ be the six-vertex model configurations which, under the bijection of~\eqref{bij},
correspond to~$A$ and~$A^\dagger$, respectively.
If~$A$ is $\ast$-invariant, or equivalently if the local configuration at each superdiagonal vertex in~$C$ is~$\Vi$ or~$\Vii$, then~$A^\dagger$ is taken to be~$A$.
If~$A$ is not \mbox{$\ast$-invariant}, or equivalently if there is a superdiagonal vertex at which the local configuration
in~$C$ is~$\Viii$,~$\Vv$,~$\Vvi$\ or~$\Viv$, then let $(k-1,k)$ be the first such superdiagonal vertex from the top-left of~$\G_n$,
and obtain~$C^\dagger$ from~$C$ by interchanging the orientations of the two edges which connect
$(k-1,k)$ to $(k-1,k-1)$ and $(k,k)$, while leaving the orientations
of all other edges unchanged, i.e., in this case,~$\dagger$ has the same action as the local involution~$\ast_k$.

As an example, taking
$\rule[-5.1ex]{0ex}{11.3ex}A=\left(\begin{smallmatrix}0&0&0&1&0&0&0\\0&\color{red}1&\color{red}0&-1&1&0&0\\0&\color{red}0&\color{red}1&0&-1&0&1\\1&-1&0&0&1&0&0\\
0&1&-1&1&-1&1&0\\0&0&0&0&1&0&0\\0&0&1&0&0&0&0\end{smallmatrix}\right)$ as in~\eqref{DSASMexample} gives
$\rule[-5.1ex]{0ex}{11.3ex}A^\dagger=\left(\begin{smallmatrix}0&0&0&1&0&0&0\\0&\color{red}0&\color{red}1&-1&1&0&0\\0&\color{red}1&\color{red}0&0&-1&0&1\\
1&-1&0&0&1&0&0\\0&1&-1&1&-1&1&0\\0&0&0&0&1&0&0\\0&0&1&0&0&0&0\end{smallmatrix}\right)$ (with $k=3$),
where the matrix entries which differ between~$A$ and~$A^\dagger$ are shown in red.

\subsection{Properties of \texorpdfstring{$\ast$}{*} and \texorpdfstring{$\dagger$}{†}}\label{astdaggerPropSect}
It follows from the definitions of $\ast$ and $\dagger$ that these involutions commute. Hence,
$\{1,\ast,\dagger,\ast\dagger\}$ (where $1$ denotes the identity and $\ast\dagger$ denotes the composition
of $\ast$ and $\dagger$) can be regarded as a group which is isomorphic to the Klein four-group,
and which has an action on $\DSASM(n)$.

It can also be seen that
\begin{equation}\label{astdag}\DSASM^\ast(n)=\DSASM^\dagger(n)\subseteq\DSASM^{\ast\dagger}(n),\end{equation}
where the equality between $\DSASM^\ast(n)$ and $\DSASM^\dagger(n)$ follows immediately from the definition of $\dagger$,
and their containment in $\DSASM^{\ast\dagger}(n)$ follows by observing that if~$A$ is in $\DSASM^\ast(n)=\DSASM^\dagger(n)$
then $A=A^\ast=A^\dagger$, which (applying $\dagger$ to both sides of $A^\ast=A^\dagger$) gives $A^{\ast\dagger}=A$.  It will be seen
in~\eqref{invinvariant1} and~\eqref{invinvariant3} that for $n$ even,
$\DSASM^\ast(n)=\DSASM^\dagger(n)$ is empty and not equal to $\DSASM^{\ast\dagger}(n)$,
and it will be seen in~\eqref{invinvariant2} that for $n$ odd, $\DSASM^\ast(n)=\DSASM^\dagger(n)$ is nonempty and equal to $\DSASM^{\ast\dagger}(n)$.

The behavior of certain statistics under the action of $\ast$ and $\dagger$ is given in the following result.
\begin{proposition}\label{invprop}
Consider any $A\in\DSASM(n)$, and let $U(A)$ be the sum of all strictly upper triangular entries in~$A$
(or the sum of all strictly lower triangular entries in~$A$), i.e., $U(A)=\sum_{1\le i<j\le n}A_{ij}$.
Then
\begin{align}\label{Uast}U(A)+U(A^\ast)&=n-1,\\
\label{Udag}|U(A)-U(A^\dagger)|&=1-\delta_{A,A^\dagger},\\
\label{Rast}R(A)+R(A^\ast)&\equiv n-1\pmod{2}\\
\intertext{and}
\label{Rdag}R(A)+R(A^\dagger)&\equiv1-\delta_{A,A^\dagger}\pmod{2}.
\end{align}
\end{proposition}
Note that, due to the equality in~\eqref{astdag}, the term $\delta_{A,A^\dagger}$ in~\eqref{Udag} and~\eqref{Rdag} can be replaced by~$\delta_{A,A^\ast}$.
\begin{proof}
The proof of~\eqref{Uast} will be considered first.  The result is trivial for $n=1$.
For $n\ge2$, let $C\in\SVC(n)$ correspond to~$A$ under the bijection of~\eqref{bij}.
It follows from the basic properties of this bijection that, for $2\le j\le n$,
the sum of strictly upper triangular entries in column~$j$ of~$A$ is~0 or~1
according to whether the edge between vertices $(j-1,j)$ and $(j,j)$ is oriented up or down, respectively, in~$C$.
Hence, \psset{unit=7mm}
\begin{equation*}\sum_{i=1}^{j-1}A_{ij}=\begin{cases}0,&C_{j-1,j}=\text{\Vi, \Viii\ or \Vvi},\\
1,&C_{j-1,j}=\text{\Vii, \Viv\ or \Vv}.\end{cases}\end{equation*}
By considering the action of $\ast$ at superdiagonal vertices, it now follows that
\begin{equation}\label{inv1}
\sum_{i=1}^{j-1}A_{ij}+\sum_{i=1}^{j-1}A^\ast_{ij}=\begin{cases}0,&C_{j-1,j}=\Vi,\\
1,&C_{j-1,j}=\text{\Viii, \Vvi, \Viv\ or \Vv,}\\
2,&C_{j-1,j}=\Vii.\end{cases}\end{equation}
Subtracting~1 from each side of~\eqref{inv1}, and summing over $j=2,\ldots,n$, gives
\begin{multline}\label{inv2}
U(A)+U(A^\ast)-n+1\\
=\text{(number of local configurations \Vii\ in $C$ on superdiagonal vertices)}\\
-\text{(number of local configurations \Vi\ in $C$ on superdiagonal vertices)}.\end{multline}
Using~\eqref{inoutk} with $k=1$, the RHS of~\eqref{inv2} is~0, which then gives~\eqref{Uast}.

Proceeding to the proof of~\eqref{Udag}, if $A=A^\dagger$, then the result is trivial.  If $A\ne A^\dagger$,
then it follows from the definition of~$\dagger$ that $|A_{k-1,k}-A^\dagger_{k-1,k}|=1$ for a unique $2\le k\le n$,
and $A_{ij}=A^\dagger_{ij}$ for all other $1\le i<j\le n$, so that $|U(A)-U(A^\dagger)|=|\sum_{1\le i<j\le n}(A_{ij}-A^\dagger_{ij})|=1$, as required.

Finally, for the proofs of~\eqref{Rast} and~\eqref{Rdag}, note that~$R(A)$ equals $U(A)$ plus twice the
number of strictly upper triangular $-1$'s in $A$, so that $U(A)\equiv R(A)\pmod{2}$.
Applying this to~\eqref{Uast} and~\eqref{Udag} gives~\eqref{Rast} and~\eqref{Rdag}, respectively.
\end{proof}

\subsection{Numbers of \texorpdfstring{$\ast$}{*}-invariant, \texorpdfstring{$\dagger$}{†}-invariant and \texorpdfstring{$\ast\dagger$}{*†}-invariant DSASMs and OSASMs}\label{astdaggerNumInvSect}
The numbers of $\ast$-invariant, $\dagger$-invariant and $\ast\dagger$-invariant DSASMs and OSASMs are given in the following result.
\begin{proposition}\label{invinvariantprop}
For $n\ge1$,
\begin{align}
\label{invinvariant1}|\DSASM^\ast(2n)|&=|\DSASM^\dagger(2n)|=|\OSASM^\ast(2n)|=|\OSASM^\dagger(2n)|=0,\\
\label{invinvariant2}|\DSASM^\ast(2n+1)|&=|\DSASM^\dagger(2n+1)|=|\DSASM^{\ast\dagger}(2n+1)|=|\OSASM(2n)|,\\
\label{invinvariant3}|\DSASM^{\ast\dagger}(2n)|&=2\,|\OSASM(2n-1)|\\
\intertext{and}
\label{invinvariant4}\notag|\OSASM^{\ast\dagger}(2n)|&=|\OSASM^\ast(2n+1)|=|\OSASM^\dagger(2n+1)|=|\OSASM^{\ast\dagger}(2n+1)|\\
&=|\ASM(n)|.\end{align}
\end{proposition}
\begin{proof}
For $A\in\DSASM(2n)$, it follows from~\eqref{Uast} that $U(A)+U(A^\ast)$ is odd, which implies that equality of~$A$ and~$A^\ast$ is impossible,
and hence that $\DSASM^\ast(2n)$ is empty.
It now follows, using~\eqref{astdag} and the fact that $\OSASM^\ast(n)$ and $\OSASM^\dagger(n)$ are subsets of $\DSASM^\ast(n)$ and $\DSASM^\dagger(n)$, respectively,
that $\DSASM^\dagger(2n)$, $\OSASM^\ast(2n)$ and $\OSASM^\dagger(2n)$ are also empty, thereby confirming~\eqref{invinvariant1}.

\psset{unit=7mm}
Proceeding to the proof of~\eqref{invinvariant2}, using~\eqref{astdag} gives $\DSASM^\ast(2n+1)=\DSASM^\dagger(2n+1)\subseteq\DSASM^{\ast\dagger}(2n+1)$.
It will now be shown that the containment here is an equality. Suppose there exists $A\in\DSASM^{\ast\dagger}(2n+1)$ with $A\notin
\DSASM^\ast(2n+1)=\DSASM^\dagger(2n+1)$.  Then $A^\ast=A^\dagger\ne A$. But this is impossible, since~\eqref{Uast} would give
$U(A)+U(A^\ast)=2n$, so that $U(A)+U(A^\ast)$ is even, whereas~\eqref{Udag} would give $|U(A)-U(A^\ast)|=1$, so that $U(A)+U(A^\ast)$ is odd.

Having confirmed the validity of the first two equalities in~\eqref{invinvariant2}, the validity of the third equality will now be confirmed using
a bijection from $\OSASM(2n)$ to $\DSASM^\ast(2n+1)$.
Let $A\in\OSASM(2n)$ and $C\in\SVC(2n)$ correspond under the bijection of~\eqref{bij}.  Then
each local configuration on the left boundary in~$C$ is either~\Lout\ or~\Lin\ (since
each diagonal entry in $A$ is zero), and there exists $\widehat{C}\in\SVC(2n+1)$ which is obtained from~$C$ by extending each~\Lout\ to~\Vi
and each \Lin\ to~\Vii, and adding a further up-oriented vertical edge at the top left
and left-oriented horizontal edge at the bottom right.
It follows that the $\widehat{A}\in\DSASM(2n+1)$ which corresponds to~$\widehat{C}$ under the bijection of~\eqref{bij} is
$\ast$-invariant (since the local configuration at each superdiagonal vertex in~$\widehat{C}$ is~\Vi\ or~\Vii),
and it can easily be seen that the mapping in which~$A$ is transformed to~$\widehat{A}$ is bijective from $\OSASM(2n)$ to $\DSASM^\ast(2n+1)$.
Note that $\widehat{A}$ can be constructed directly from $A$ by taking
the strictly upper and lower triangular parts of~$\widehat{A}$ to be the upper and lower triangular parts, respectively, of~$A$,
and then obtaining the main diagonal of~$\widehat{A}$ by requiring that the sum of entries in each row (or column) is~$1$.
Note also that~$\widehat{A}=\theta(A,\emptyset)$,
where~$\theta$ is the bijection which is defined in the proof of Proposition~\ref{diagrefprop}. \psset{unit=4mm}
As an example, for $n=1$,
$C=\raisebox{-3.4mm}{\pspicture(0.6,0.9)(3.3,3.3)\multirput(1,3)(1,0){2}{$\ss\bullet$}\multirput(1,2)(1,0){3}{$\ss\bullet$}\multirput(2,1)(1,0){2}{$\ss\bullet$}
\psline[linewidth=0.5pt](1,3)(1,2)(3,2)\psline[linewidth=0.5pt](2,3)(2,1)(3,1)
\multirput(1,2)(1,0){2}{\psdots[dotstyle=triangle*,dotscale=1.1](0,0.5)}\multirput(2,1)(0,1){2}{\psdots[dotstyle=triangle*,dotscale=1.1,dotangle=90](0.5,0)}
\psdots[dotstyle=triangle*,dotscale=1.1,dotangle=180](2,1.5)\psdots[dotstyle=triangle*,dotscale=1.1,dotangle=-90](1.5,2)\endpspicture}$
and $A=\bigl(\begin{smallmatrix}0&1\\1&0\end{smallmatrix}\bigr)$
give
$\widehat{C}=\raisebox{-5.7mm}{\pspicture(0.6,0.6)(4.3,4.3)\multirput(1,4)(1,0){3}{$\ss\bullet$}\multirput(1,3)(1,0){4}{$\ss\bullet$}
\multirput(2,2)(1,0){3}{$\ss\bullet$}\multirput(3,1)(1,0){2}{$\ss\bullet$}
\psline[linewidth=0.5pt](1,4)(1,3)(4,3)\psline[linewidth=0.5pt](2,4)(2,2)(4,2)\psline[linewidth=0.5pt](3,4)(3,1)(4,1)
\multirput(1,3)(1,0){3}{\psdots[dotstyle=triangle*,dotscale=1.1](0,0.5)}
\multirput(3,1)(0,1){3}{\psdots[dotstyle=triangle*,dotscale=1.1,dotangle=90](0.5,0)}
\psdots[dotstyle=triangle*,dotscale=1.1](2,2.5)
\psdots[dotstyle=triangle*,dotscale=1.1,dotangle=180](3,2.5)(3,1.5)
\psdots[dotstyle=triangle*,dotscale=1.1,dotangle=90](2.5,2)
\psdots[dotstyle=triangle*,dotscale=1.1,dotangle=-90](1.5,3)(2.5,3)\endpspicture}$
and $\widehat{A}=\Bigl(\begin{smallmatrix}0&0&1\\0&1&0\\1&0&0\end{smallmatrix}\Bigr)$.

\psset{unit=7mm}
Proceeding to~\eqref{invinvariant3}, this will be proved using a process in which
two elements of $\DSASM^{\ast\dagger}(2n)$ are constructed from each element of $\OSASM(2n-1)$.
Let $A\in\OSASM(2n-1)$ and $C\in\SVC(2n-1)$ correspond under the bijection of~\eqref{bij}.  Then there exists a unique
left boundary vertex at which the local configuration in~$C$ is~\Lup\ or~\Ldown\ (since exactly one diagonal entry in~$A$ is nonzero),
while the local configurations at all other left boundary vertices are~\Lout\ or~\Lin\ (since all other diagonal entries in~$A$ are zero).
Now construct two distinct elements $\widehat{C}_1,\widehat{C}_2\in\SVC(2n)$ from~$C$ by extending~\Lup\ (which occurs if $A$ has a~$1$ on the diagonal)
to either~\Viii\ or~\Vv, extending~\Ldown\ (which occurs if, alternatively, $A$ has a~$-1$ on the diagonal) to either~\Vvi\ or~\Viv,
extending each~\Lout\ to~\Vi\ and each \Lin\ to~\Vii,
and adding a further up-oriented vertical edge at the top left and left-oriented horizontal edge at the bottom right.
It follows that $\widehat{C}_1$ and $\widehat{C}_2$ each contain a unique superdiagonal vertex at which the local configuration
is~\Viii,~\Vv,~\Vvi\ or~\Viv.  It can then be seen, using the definitions of~$\ast$ and~$\dagger$, that
the $\widehat{A}_1,\widehat{A}_2\in\DSASM(2n)$ which correspond to~$\widehat{C}_1,\widehat{C}_2$ under the bijection of~\eqref{bij} satisfy
$\widehat{A}_1^{\ast\dagger}=\widehat{A}_1\ne\widehat{A}_1^\ast=\widehat{A}_1^\dagger$ and $\widehat{A}_2^{\ast\dagger}=\widehat{A}_2\ne\widehat{A}_2^\ast=\widehat{A}_2^\dagger$.
Also observe that any $\widehat{A}\in\DSASM^{\ast\dagger}(2n)$ satisfies $\widehat{A}^{\ast\dagger}=\widehat{A}\ne\widehat{A}^\ast=\widehat{A}^\dagger$
since, due to~\eqref{invinvariant1}, $\DSASM^\ast(2n)=\DSASM^\dagger(2n)=\emptyset$.
It can now be checked that the previous process, when applied to all elements of $\OSASM(2n-1)$,
generates each element of $\DSASM^{\ast\dagger}(2n)$ exactly once, thereby confirming~\eqref{invinvariant3}.
Note that for the $\widehat{A}_1$ and~$\widehat{A}_2$ which are constructed from $A\in\OSASM(2n-1)$, one
of these two $\ast\dagger$-invariant $2n\times2n$ DSASMs (specifically, the one associated with an extension of~\Lup\ to~\Viii\ or~\Ldown\ to~\Viv)
has all zeros on the superdiagonal, while the other (specifically, the one associated with an extension of~\Lup\ to~\Vv\ or~\Ldown\ to~\Vvi)
has a single nonzero entry on the superdiagonal. Hence, this process partitions $\DSASM^{\ast\dagger}(2n)$ into two natural
subsets, each of size $|\OSASM(2n-1)|$.

Finally, for the proof of~\eqref{invinvariant4}, it can be checked straightforwardly that
\begin{equation*}\OSASM^{\ast\dagger}(2n)=\Bigl\{\Bigl({\scriptsize\begin{array}{@{\,}c@{\:}|@{\:}c@{\,}}0&A\\[-0.4mm]\hline\rule{0ex}{2.3ex}A^t&0\end{array}}\Bigr)\,
\Big|\:A\in\ASM(n)\Bigr\}\end{equation*}
and
\begin{equation*}\OSASM^\ast(2n+1)=\OSASM^\dagger(2n+1)=\OSASM^{\ast\dagger}(2n+1)=\left\{
\biggl({\scriptsize\begin{array}{@{\,}c@{\:}|@{\:}c@{\:}|@{\:}c@{\,}}0&0&A\\[-0.4mm]\hline0&1&0\\[-0.4mm]\hline\rule{0ex}{2.3ex}A^t&0&0\end{array}}\biggr)\,\bigg|
\:A\in\ASM(n)\right\},\end{equation*}
which immediately give~\eqref{invinvariant4}.
\end{proof}

Some consequences of Proposition~\ref{invinvariantprop} are as follows.

First, it follows from~\eqref{invinvariant1} and~\eqref{invinvariant3} that the orbits of
the group $\{1,\ast,\dagger,\ast\dagger\}$ on $\DSASM(2n)$ consist of $|\OSASM(2n-1)|$ size-2 orbits of the form
$\{A=A^{\ast\dagger}\ne A^\ast=A^\dagger\}$ and $(|\DSASM(2n)|-2|\OSASM(2n-1)|)/4$ size-4 orbits,
while it follows from~\eqref{invinvariant2} that the orbits of $\{1,\ast,\dagger,\ast\dagger\}$ on $\DSASM(2n+1)$ consist of
$|\OSASM(2n)|$ size-1 orbits
and $(|\DSASM(2n+1)|-|\OSASM(2n)|)/4$ size-4 orbits.

Second, combining the $t=1$ case of~\eqref{Xmin2} with $|\DSASM^\ast(2n)|=0$ and
$|\DSASM^\ast(2n+1)|=|\OSASM(2n)|$ from~\eqref{invinvariant1} and~\eqref{invinvariant2} gives
\begin{equation}\label{Stem}|\DSASM^\ast(n)|=(-1)^{n(n-1)/2}\,X_n(-1,1,1),\end{equation}
which provides an example of the $-1$ phenomenon,
as first described for plane partitions by Stembridge~\cite{Ste94}.  This phenomenon is said to occur if
a finite set~$\mathcal{S}$, an involution~$\nu$ on $\mathcal{S}$ and an integer-valued statistic~$K$ on~$\mathcal{S}$
satisfy $|\{A\in\mathcal{S}\mid A^\nu=A\}|=|\{A\in\mathcal{S}\mid K(A)\text{ even}\}|-|\{A\in\nolinebreak\mathcal{S}\mid K(A)\text{ odd}\}|$,
i.e., the number of $\nu$-invariant elements of~$\mathcal{S}$ is $\sum_{A\in \mathcal{S}}(-1)^{K(A)}$.
For the case here, the set, involution and statistic are $\mathcal{S}=\DSASM(n)$, $\nu=\ast$ and $K(A)=n(n-1)/2+R(A)$, for each $A\in\DSASM(n)$.

Another, more direct, proof of~\eqref{Stem} is as follows.
For~$n$ even, both sides of~\eqref{Stem} are zero since, due to~\eqref{Rast},~$\ast$ provides a bijection between $\{A\in\DSASM(n)\mid R(A)\text{ even}\}$
and $\{A\in\DSASM(n)\mid R(A)\text{ odd}\}$.  Alternatively, due to~\eqref{Rdag},~$\dagger$ provides a further bijection between these sets.

For~$n$ odd,~\eqref{Stem} follows by observing that, due to~\eqref{Rast},~$R(A)\equiv(n-1)/2\pmod{2}$ for $A\in\DSASM^\ast(n)$,
and that, due to~\eqref{Rdag} and the equality in~\eqref{astdag},~$\dagger$ provides a bijection between $\{A\in\DSASM(n)\setminus\DSASM^\ast(n)\mid R(A)\text{ even}\}$ and
$\{A\in\DSASM(n)\setminus\DSASM^\ast(n)\mid R(A)\text{ odd}\}$.
Note that~$\ast$ does not provide another bijection between these sets since,
due to~\eqref{Rast}, $R(A)+R(A^\ast)$ is even for all $A\in\DSASM(n)$.

\section{Results for a generalized DSASM generating function}\label{GeneralGenFuncSect}
In this section, a generalization of the results of Section~\ref{mainresults} to include further natural statistics is outlined.
Specifically, the number of strictly upper triangular inversions in a DSASM is included, and the number of~$1$'s and number of~$-1$'s on the main diagonal of a DSASM are taken to be two separate statistics.
These statistics are introduced and studied in Section~\ref{FurthStatSect}, a generalized DSASM generating function involving these statistics is defined in Section~\ref{GeneralGenFuncSubSect},
and certain results for this generating function are stated (with the details of the proofs omitted) as Theorems~\ref{Xprstheorem} and~\ref{Xprsttheorem} in Section~\ref{GeneralGenFuncResSect}.

\subsection{Further DSASM statistics}\label{FurthStatSect}
For any $A\in\ASM(n)$, define the statistics
\begin{align}
\label{PA}P(A)&=\sum_{\substack{1\le i<i'<j\le n\\1\le j'\le j}}A_{ij}\,A_{i'j'},\\
\label{SApl}S_+(A)&=\text{number of $1$'s on the main diagonal of }A,\\
\label{SAmi}S_-(A)&=\text{number of $-1$'s on the main diagonal of }A.\end{align}
It follows trivially that the statistic $S$, as defined in~\eqref{SA}, is related to $S_+$ and $S_-$ by
\begin{equation}\label{SSS}S(A)=S_+(A)+S_-(A),\end{equation}
for any ASM $A$.

\psset{unit=7mm}
For $A\in\DSASM(n)$ and $C\in\SVC(n)$ which correspond under the bijection of~\eqref{bij},
the statistics~\eqref{PA}--\eqref{SAmi} can be expressed in terms of $C$ as
\begin{align}\notag P(A)&=\text{number of local configurations \Vi\ in }C\\
\label{PC}&=\text{number of local configurations \Vii\ in }C,\\
\label{SCpl}S_+(A)&=\text{number of local configurations \Lup\ in }C,\\
\label{SCmi}S_-(A)&=\text{number of local configurations \Ldown\ in }C.\end{align}
The bijection of~\eqref{bij} immediately gives~\eqref{SCpl} and~\eqref{SCmi},
while~\eqref{PC} can be obtained straightforwardly using basic properties of DSASMs
and their corresponding six-vertex model configurations.  Note that the first equality of~\eqref{PC}
implies that $P(A)$ is a nonnegative integer (with the nonnegativity not being immediately obvious from~\eqref{PA}),
and that the second equality of~\eqref{PC} follows by summing over $k=1,\ldots,n-1$ in~\eqref{inoutk}.

The statistics~\eqref{PA}--\eqref{SAmi} can be used to access several other natural statistics.  For example, for any ASM~$A$, the trace of~$A$ is
\begin{equation}\label{trA}\mathop{\mathrm{tr}}A=S_+(A)-S_-(A),\end{equation}
and, for any $A\in\DSASM(n)$, the sum of all strictly upper triangular entries in~$A$, as used in Proposition~\ref{invprop}, is
\begin{equation}\label{sumuppA}\sum_{1\le i<j\le n}A_{ij}=\bigl(n-S_+(A)+S_-(A)\bigr)/2,\end{equation}
since the sum of all entries in~$A$ is $n=\mathop{\mathrm{tr}}A+2\sum_{1\le i<j\le n}A_{ij}$.

For any $A\in\ASM(n)$, the number of inversions in $A$ is defined as
\begin{equation}\label{IA}I(A)=\sum_{\substack{1\le i<i'\le n\\1\le j'\le j\le n}}\!A_{ij}\,A_{i'j'},\end{equation}
and it can be shown straightforwardly that for any $A\in\DSASM(n)$,
\begin{align}\notag I(A)&=2P(A)+\bigl(n-S_+(A)-S_-(A)\bigr)/2\\
&=2P(A)+\bigl(n-S(A)\bigr)/2.\end{align}
The statistic~$I$ was first defined by Robbins and Rumsey~\cite[Eq.~(18)]{RobRum86}.
A closely-related statistic, $I'(A)=\rule[-1.5ex]{0ex}{0ex}\sum_{1\le i<i'\le n;\;1\le j'<j\le n}A_{ij}\,A_{i'j'}=I(A)+M(A)$
for any $A\in\ASM(n)$, where $M(A)$ is given in~\eqref{MA}, was previously defined by Mills, Robbins and Rumsey~\cite[p.~344]{MilRobRum83},
and is sometimes instead referred to as the number of inversions.
The use of the name inversions is motivated by the fact that if~$A$ is a permutation matrix,
then $I(A)$ (which in this case equals~$I'(A)$, since $M(A)=0$) is the number of inversions in the permutation associated with~$A$.

It can be shown that if $C$ is the six-vertex model configuration with domain-wall boundary conditions on~$\mathcal{S}_n$
(as described in Section~\ref{sixvertexmodelconfig}) which corresponds to $A\in\ASM(n)$, then
$I(A)$ is the number of local configurations \Vi\ in $C$ (which is also equal to the number of local configurations \Vii\ in~$C$).
Accordingly, each occurrence of \Vi\ can be regarded as a single inversion,
and due to~\eqref{PC}, it is natural to refer to the statistic $P(A)$, for a DSASM~$A$, as the number
of strictly upper triangular inversions in~$A$.

\subsection{The generalized DSASM generating function}\label{GeneralGenFuncSubSect}
The generalized DSASM generating function associated with the statistics $P$,~$R$,~$S_+$,~$S_-$ and~$T$ is defined as
\begin{equation}\label{Xss}\X_n(p,r,s_+,s_-,t)=\sum_{A\in\DSASM(n)}p^{P(A)}\,r^{R(A)}\,s_+^{S_+(A)}\,s_-^{S_-(A)}\,t^{T(A)},\end{equation}
for indeterminates~$p$, $r$,~$s_+$,~$s_-$ and $t$.
This generalizes the DSASM generating function $X_n(r,s,t)$, as defined in~\eqref{X}, with~\eqref{SSS} implying that
\begin{equation}X_n(r,s,t)=\X_n(1,r,s,s,t).\end{equation}

As examples, the $n=1$, $2$ and $3$ cases of the generalized DSASM generating function~\eqref{Xss} are
\begin{align}\notag\X_1(p,r,s_+,s_-,t)&=s_+t,\quad\X_2(p,r,s_+,s_-,t)=s_+^2t+rt^2,\\
\X_3(p,r,s_+,s_-,t)&=s_+^3t+rs_+t+rs_+t^2+prs_+t^3+r^2s_-t^2,\end{align}
where the terms on each RHS are written in an order corresponding to that used for the DSASMs in~\eqref{DSASM123}.

In order to relate the generalized DSASM generating function~\eqref{Xss} to the DSASM partition function~\eqref{Z},
the constants $\alpha$, $\beta$, $\gamma$, $\delta$ and function~$\phi(u)$ in the left boundary weights in Table~\ref{weights} are set to
\begin{equation}\label{boundassigref}\alpha=\frac{s_+\,q\,w-s_-\,\q\,\w}{\sigma(q^2w^2)},\quad\beta=\frac{s_-\,q\,w-s_+\,\q\,\w}{\sigma(q^2w^2)},\quad
\gamma=\delta=\frac{1}{\sigma(q^2w^2)},\quad\phi(u)=1,\end{equation}
for an arbitrary constant~$w$, which gives
\begin{gather}\notag W(\WLup,u)=\frac{s_+\,\sigma(q^2wu)+s_-\,\sigma(w\u)}{\sigma(q^2w^2)},\qquad
W(\WLdown,u)=\frac{s_+\,\sigma(w\u)+s_-\,\sigma(q^2wu)}{\sigma(q^2w^2)},\\
\label{leftWref}W(\WLout,u)=W(\WLin,u)=\frac{\sigma(q^2u^2)}{\sigma(q^2w^2)}.\end{gather}

Note that while the left boundary weights~\eqref{leftW} at $u=1$ are
$W(\WLup,1)=W(\WLdown,1)=s$ and $W(\WLout,1)=W(\WLin,1)=1$,
the left boundary weights~\eqref{leftWref} at $u=w$ are $W(\WLup,w)=s_+$, $W(\WLdown,w)=\nolinebreak s_-$ and $W(\WLout,w)=W(\WLin,w)=1$.
Using this property, and a process analogous to that used in the proof of Lemma~\ref{ZXlemma}, it can be shown that
the DSASM partition function and generalized DSASM generating function~\eqref{Xss} are related by
\begin{multline}\label{ZXgen}\overline{\!Z}_n(z,\underbrace{w,\ldots,w}_{n-1})=\frac{\sigma(q^2\w^2)^{(n-1)(n-2)/2}\,\sigma(q^2\w\z)^{n-1}}{\sigma(q^4)^{n(n-1)/2}\,\sigma(q^2wz)}\\
\times\Biggl[\frac{\sigma(q^2w^2)\,\sigma(q^2z^2)\,\sigma(q^2\w\z)}{\sigma(q^2\w^2)\,\sigma(q^2wz)}\:
\X_n\biggl(\biggl(\frac{\sigma(q^2w^2)}{\sigma(q^2\w^2)}\biggr)^{\!2},\frac{\sigma(q^4)}{\sigma(q^2\w^2)},s_+,s_-,
\frac{\sigma(q^2wz)\,\sigma(q^2\w^2)}{\sigma(q^2\w\z)\,\sigma(q^2w^2)}\biggr)\qquad\quad\\
+\,\frac{\bigl(s_+\,\sigma(w\z)+s_-\,\sigma(q^2wz)\bigr)\,\sigma(w\z)}{\sigma(q^2w^2)}\:
\X_{n-1}\biggl(\biggl(\frac{\sigma(q^2w^2)}{\sigma(q^2\w^2)}\biggr)^{\!2},\frac{\sigma(q^4)}{\sigma(q^2\w^2)},s_+,s_-,1\biggr)\Biggr],\end{multline}
where $\,\overline{\!Z}_n(u_1,\ldots,u_n)$ denotes the DSASM partition function~\eqref{Z} in which the assignments of~\eqref{boundassigref} are used.

\subsection{Pfaffian equations for the generalized DSASM generating function}\label{GeneralGenFuncResSect}
A process analogous to that used in Sections~\ref{proofXrs} and~\ref{proofXrsttheorem}
to obtain the Pfaffian equations~\eqref{Xrs} and~\eqref{Xrst} for~$X_n(r,s,1)$ and~$X_n(r,s,t)$
can be used, together with~\eqref{ZXgen}, to derive Pfaffian equations for the generalized DSASM generating function $\X_n(p,r,s_+,s_-,t)$.
In principle, it should be possible to obtain an equation which enables $\X_n(p,r,s_+,s_-,t)$ to be determined for arbitrary~$p$,~$r$,~$s_+$,~$s_-$ and~$t$.  However, this seems to be
reasonably complicated, and the results presented here enable
$\X_n(p,r,s_+,s_-,t)$ to be determined only for the special cases of $t=1$ (in Theorem~\ref{Xprstheorem}) or $s_+=s_-$ (in Theorem~\ref{Xprsttheorem}).
The details of the proofs will be omitted, but the eventual results are as follows.
\begin{theorem}\label{Xprstheorem}
For $n\ge2$, the generalized DSASM generating function $\X_n(p,r,s_+,s_-,t)$ at $t=1$ is given by
\begin{multline}\label{Xprs}
\X_n(p,r,s_+,s_-,1)=\Pf_{\odd(n)\le i<j\le n-1}\biggl([u^iv^j]\,(s_++s_-\,u)^{\odd(n)\,\delta_{i,1}}\,\frac{v-u}{1-uv}\\
\times\biggl((s_+-s_-\,u)(s_+-s_-\,v)+\frac{r(1-u^2)(1-v^2)}{(1-ru)(1-rv)-puv}\biggr)\biggr).\end{multline}
\end{theorem}

Note that Theorem~\ref{Xprstheorem} has recently been applied, in the context of a generalized Littlewood identity, by Fischer and H\"{o}ngesberg~\cite[Sec.~6]{FisHon26},
and that this leads to certain results for DSASMs (see~\cite[Thm.~1.4~\& Prob.~6.3]{FisHon26}) involving statistics which can all be expressed in terms of~$P$,~$R$,~$S_+$ and~$S_-$.

\begin{theorem}\label{Xprsttheorem}
For $n\ge2$, the generalized DSASM generating function $\X_n(p,r,s_+,s_-,t)$ at $s_+=s_-$
satisfies
\begin{multline}\label{Xprst}
\!\!\!(t-1+rt)\,(t-1)^{n-1}\,\X_n(p,r,s,s,t)=st\bigl((t-1)^n+\odd(n)\,r^n\,t^n\bigr)\,\X_{n-1}(p,r,s,s,1)\,+\,r^n\,s^{\odd(n)}\,t^{n+1}\\[1mm]
\times\!\Pf_{\odd(n)\le i<j\le n-1}\left(\begin{cases}\displaystyle[u^iv^j]\,\frac{v-u}{1-uv}\biggl(s^2+\frac{r(1+u)(1+v)}{(1-ru)(1-rv)-puv}\biggr),&j\le n-2\\[4mm]
\displaystyle[u^i]\,\frac{t-1-rtu}{rt-(t-1)u}\biggl(s^2+\frac{r(t-1+rt)(1+u)}{r-(pt-p+r^2)u}\biggr),&j=n-1\end{cases}\right).\end{multline}
\end{theorem}

Note that determinant expressions with a form similar to that of the Pfaffian expressions~\eqref{Xrs}, \eqref{Xrst},~\eqref{Xprs} or~\eqref{Xprst} for DSASM generating functions
are also known for ASM generating functions.
For example, a determinant expression for the ASM generating function associated with the statistics~$M$ and~$I$,
as defined in~\eqref{MA} and~\eqref{IA}, is
\begin{equation}\label{ASMX}\sum_{A\in\ASM(n)}r^{2M(A)}p^{I(A)}=\det_{0\le i<j\le n-1}\biggl([u^iv^j]\,\frac{1}{(1-uv)\,\bigl((1-ru)(1-rv)-puv\bigr)}\biggr).\end{equation}
This follows from a result of
Behrend, Di Francesco and Zinn-Justin~\cite[Prop.~1, Prop.~3, Eq.~(68)~\& Eq.~(70)]{BehDifZin12} that
$\sum_{A\in\ASM(n)}r^{2M(A)}p^{I(A)}=\det_{0\le i<j\le n-1}([u^iv^j]\,g(u,v))$,
where $g(u,v)=(1-u)/((1-v)(1-uv))+pu/((1-v)(1-r^2u-v-(p-r^2)uv))$. In this identity, $g(u,v)$
can be replaced by $f(u,v)=1/((1-uv)((1-ru)(1-rv)-puv))$ to give~\eqref{ASMX},
since these functions are related by $g(u,v)=(1+(p-r^2-1)u+r^2u^2)f(ru,\bar{r}v)$,
which implies that the matrices $A=([u^iv^j]g(u,v))_{0\le i,j\le n-1}$
and $B=([u^iv^j]f(u,v))_{0\le i,j\le n-1}$ are related by $A=Y_1BY_2$,
where $Y_1=(r^j(\delta_{i,j}+(p-r^2-1)\delta_{i,j+1}+r^2\delta_{i,j+2}))_{0\le i,j\le n-1}$
and $Y_2=(r^{-i}\delta_{i,j})_{0\le i,j\le n-1}$, so that
$\det Y_1=r^{n(n-1)/2}$, $\det Y_2=r^{-n(n-1)/2}$ and $\det A=\det B$.
(This replacement of $g(u,v)$ by $f(u,v)$ corresponds to an application of the general
identity $\det_{0\le i,j\le n-1}([u^iv^j]\,(k_1(u)k_2(v)f(h_1(u)u,h_2(v)v)))
=(k_1(0)k_2(0))^n(h_1(0)h_2(0))^{n(n-1)/2}\det_{0\le i,j\le n-1}([u^iv^j]\,f(u,v))$,
for any power series $f(u,v)$, $h_1(u)$, $h_2(u)$, $k_1(u)$ and $k_2(u)$,
which can be regarded as a determinant version of the $m=2n$ case of~\eqref{HomogPf2}.)
The case of~\eqref{ASMX} with $r=p=1$ is also given by Barry~\cite[Prop.~5]{Bar21b}.

\section{Asymptotics}\label{AsymptSect}
In this section, the large $n$ asymptotics of the number of $n\times n$ DSASMs is considered, and related information on the large~$n$ asymptotics of the number of $n\times n$ ASMs
in other classes is also presented.

In Section~\ref{ASMAsymptSect}, the asymptotics of the number of unrestricted ASMs is reviewed, and an asymptotic expansion for this case is obtained
using an approach involving the Barnes $G$-function.
In Section~\ref{OSASMAsymptSect}, the same approach is used to obtain an asymptotic expansion of the number of OSASMs.
In Section~\ref{DSASMAsymptSect}, a conjecture for the asymptotic expansion of the number of DSASMs is given, as Conjecture~\ref{DSASMAsymptConj}.
In Section~\ref{ASMLeadAsymptSect}, a general result, Theorem~\ref{ASMLeadAsymptTh}, is obtained for the leading term in the asymptotics of the number of ASMs in any symmetry class.
In Section~\ref{ASMLeadAsymptFurthSect1}, the leading term in the asymptotics of the number of DSASMs whose~$1$ in the first row is in a fixed position is obtained in Proposition~\ref{DSASMRefLeadAsymptTh}.
In Section~\ref{ASMLeadAsymptFurthSect2}, further aspects of the results of Sections~\ref{ASMLeadAsymptSect} and~\ref{ASMLeadAsymptFurthSect1} are
discussed, and generalizations of these results are given in~\eqref{GenLeadAsympt1}--\eqref{GenLeadAsympt3}.

The only parts of Section~\ref{AsymptSect} that depend crucially on results for DSASMs obtained elsewhere in this paper are~\eqref{OSASMAsympt} for the asymptotics
of $|\OSASM(n)|$ (which depends on the product formula~\eqref{numOSASModd}),
and Conjecture~\ref{DSASMAsymptConj} for the conjectured asymptotics of $|\DSASM(n)|$ (which depends on data obtained using the Pfaffian formula~\eqref{numDSASM2}).

\subsection{Asymptotics of the number of ASMs}\label{ASMAsymptSect}
The large $n$ asymptotics of the number of $n\times n$ ASMs, and the number of $n\times n$ ASMs in
symmetry classes for which a product formula exists,
has been studied by Ayyer, Cori and Gouyou-Beauchamps~\cite[App.~B]{AyyCorGou11},
Bleher and Fokin~\cite[App.~A]{BleFok06}, Bleher and Liechty~\cite{BleLie18}, Bogoliubov, Kitaev and Zvonarev~\cite[Sec.~V]{BogKitZvo02},
de Gier~\cite[Sec.~5.2]{Deg09}, Hone~\cite[Eq.~(2.6)]{Hon06}, Korepin and Zinn-Justin~\cite[Sec.~5.1]{KorZin00}, Mitra~\cite{Mit09},
Mitra and Nienhuis~\cite[App.]{MitNie04}, and Ribeiro and Korepin~\cite[Sec.~5]{RibKor15}.  Some of this work will now be reviewed.

The asymptotics of numbers given by a product formula of the type in~\eqref{numASM2}--\eqref{numDASASM} can be studied effectively using an approach which
involves the Barnes $G$-function $G(z)$.  For information regarding this function, see, for example, Adamchik~\cite{Ada01a,Ada01b,Ada04,Ada14}.

This approach will be illustrated for $|\ASM(n)|$.  In this case,
\begin{align}\label{ASMasymp1}
\notag|\ASM(n)|&=\prod_{i=0}^{n-1}\frac{(3i+1)!}{(n+i)!}\\
\notag&=\prod_{i=0}^{n-1}\frac{i!\,(3i+1)!}{(2i)!\,(2i+1)!}\\
\notag&=\prod_{i=0}^{n-1}\frac{\bigl(\prod_{j=1}^ij\bigr)\,\bigl(\prod_{j=1}^i(3j)\bigr)\,\bigl(\prod_{j=0}^i(3j+1)\bigr)\,\bigl(\prod_{j=0}^{i-1}(3j+2)\bigr)}
{\bigl(\prod_{j=1}^i(2j)\bigr)\,\bigl(\prod_{j=0}^{i-1}(2j+1)\bigr)\,\bigl(\prod_{j=1}^i(2j)\bigr)\,\bigl(\prod_{j=0}^i(2j+1)\bigr)}\\
\notag&=\frac{3^{n(3n-1)/2}}{2^{n(2n-1)}}\:\prod_{i=0}^{n-1}\frac{\bigl(\prod_{j=0}^i(j+\frac{1}{3})\bigr)\,\bigl(\prod_{j=0}^{i-1}(j+\frac{2}{3})\bigr)}
{\bigl(\prod_{j=0}^{i-1}(j+\frac{1}{2})\bigr)\,\bigl(\prod_{j=0}^i(j+\frac{1}{2})\bigr)}\\
\notag&=\frac{3^{n(3n-1)/2}}{2^{n(2n-1)}}\:\prod_{i=0}^{n-1}\frac{\Gamma\bigl(\frac{1}{2}\bigr)^2\,\Gamma\bigl(i+\frac{4}{3}\bigr)\,\Gamma\bigl(i+\frac{2}{3}\bigr)}
{\Gamma\bigl(\frac{1}{3}\bigr)\,\Gamma\bigl(\frac{2}{3}\bigr)\,\Gamma\bigl(i+\frac{1}{2}\bigr)\,\Gamma\bigl(i+\frac{3}{2}\bigr)}\\
\notag&=\frac{3^{n(3n-1)/2}\,\Gamma\bigl(\frac{1}{2}\bigr)^{2n}\,G\bigl(\frac{1}{2}\bigr)\,G\bigl(\frac{3}{2}\bigr)\,G\bigl(n+\frac{4}{3}\bigr)\,G\bigl(n+\frac{2}{3}\bigr)}
{2^{n(2n-1)}\,\Gamma\bigl(\frac{1}{3}\bigr)^n\,\Gamma\bigl(\frac{2}{3}\bigr)^n\,
G\bigl(\frac{4}{3}\bigr)\,G\bigl(\frac{2}{3}\bigr)\,G\bigl(n+\frac{1}{2}\bigr)\,G\bigl(n+\frac{3}{2}\bigr)}\\
&=\frac{G\bigl(\frac{1}{2}\bigr)\,G\bigl(\frac{3}{2}\bigr)}{G\bigl(\frac{2}{3}\bigr)\,G\bigl(\frac{4}{3}\bigr)}\:
\biggl(\frac{3\sqrt{3}}{4}\biggr)^{\!n^2}\,\frac{G\bigl(n+\frac{2}{3}\bigr)\,G\bigl(n+\frac{4}{3}\bigr)}{G\bigl(n+\frac{1}{2}\bigr)\,G\bigl(n+\frac{3}{2}\bigr)},
\end{align}
where the first equality uses~\eqref{numASM2},
the second, third and fourth equalities involve simple rearrangements, the
fifth equality uses the identity $\prod_{i=0}^{n-1}(i+z)=\Gamma(n+z)/\Gamma(z)$ (where $\Gamma(z)$ is the gamma function), the sixth equality uses the identity
$\prod_{i=0}^{n-1}\Gamma(i+z)=G(n+z)/G(z)$ (where~$G(z)$ is the Barnes $G$-function), and the last equality uses the properties $\Gamma(1/2)=\sqrt{\pi}$
and $\Gamma(1/3)\,\Gamma(2/3)=2\pi/\sqrt{3}$.

For Re$(z)>0$ and $|z|\to\infty$, the Barnes $G$-function has the asymptotic expansion
\begin{multline}\label{Gasmp}
\log\bigl(G(1+z)\bigr)\\
=\frac{z^2\log z}{2}-\frac{3z^2}{4}+\frac{\log(2\pi)\,z}{2}-\frac{\log z}{12}+\zeta'(-1)+\sum_{i=1}^{N}\frac{B_{2i+2}}{4i(i+1)\,z^{2i}}
+O\Bigl(\frac{1}{z^{2N+2}}\Bigr),\end{multline}
where $\zeta'(-1)$ is the derivative of the Riemann zeta function at $-1$ (which can be written as $\zeta'(-1)=1/12-\log A$,
where $A$ is the Glaisher--Kinkelin constant), and $B_i$ are the Bernoulli numbers. For information on this expansion, see, for example,
Adamchik~\cite[p.~11]{Ada01a},~\cite[p.~6]{Ada01b},~\cite[Eq.~(5.19)]{Ada14} (noting that a sign correction is needed for the sum over~$i$ in these references).

Applying~\eqref{Gasmp}, together with expansions of standard functions, to~\eqref{ASMasymp1} gives
the asymptotic expansion of $|\ASM(n)|$ as
\begin{multline}\label{ASMasymp2}|\ASM(n)|=\frac{G\bigl(\frac{1}{2}\bigr)\,G\bigl(\frac{3}{2}\bigr)}{G\bigl(\frac{2}{3}\bigr)\,G\bigl(\frac{4}{3}\bigr)}\:
\biggl(\frac{3\sqrt{3}}{4}\biggr)^{\!n^2}\,n^{-5/36}\\
\times\biggl(1-\frac{115}{15552n^2}+\frac{796873}{483729408n^4}-\frac{23733315595}{22568879259648n^6}+O\bigl(n^{-8}\bigr)\biggr),\end{multline}
where the algebraic manipulations and simplifications which lead to this expression can be done efficiently by a computer.
Using standard properties of $G(z)$ and $\Gamma(z)$, the overall constant in~\eqref{ASMasymp2} can be rewritten as
\begin{equation}\label{ASMconst}\frac{G\bigl(\frac{1}{2}\bigr)\,G\bigl(\frac{3}{2}\bigr)}{G\bigl(\frac{2}{3}\bigr)\,G\bigl(\frac{4}{3}\bigr)}
=\frac{2^{5/12}\,\pi^{1/3}\,e^{\zeta'(-1)/3}}{3^{7/36}\,\Gamma\bigl(\frac{1}{3}\bigr)^{2/3}},\end{equation}
whose value, to~$10$ decimal places, is $0.7746696632$.

It follows from~\eqref{ASMasymp2} that
\begin{equation}\label{ASMasymp3}\lim_{n\to\infty}|\ASM(n)|^{1/n^2}=\frac{3\sqrt{3}}{4}.\end{equation}

\subsection{Asymptotics of the number of OSASMs}\label{OSASMAsymptSect}
The same approach as that used to obtain the asymptotic expansion~\eqref{ASMasymp2} of $|\ASM(n)|$ from the product formula~\eqref{numASM2}
can also be used to obtain an expansion of $|\OSASM(n)|$ from the product formulae~\eqref{numOSASMeven2} and~\eqref{numOSASModd}, for~$n$ even and odd, respectively.
Specifically, this gives
\begin{align}
\notag&\hspace{-3mm}|\OSASM(n)|\\
\notag&\hspace{-1mm}=\begin{cases}\displaystyle
\frac{G\bigl(\frac{3}{4}\bigr)\,G\bigl(\frac{5}{4}\bigr)^2\,G\bigl(\frac{7}{4}\bigr)}{G\bigl(\frac{4}{3}\bigr)\,G\bigl(\frac{5}{3}\bigr)\,G\bigl(\frac{5}{6}\bigr)\,
G\bigl(\frac{7}{6}\bigr)}\,\biggl(\frac{3\sqrt{3}}{4}\biggr)^{\!n^2/2}\biggl(\frac{3}{4}\biggr)^{\!3n/4}\\[4mm]
\displaystyle\qquad\qquad\qquad\qquad\qquad\times\,
\frac{G\bigl(\frac{n}{2}+\frac{4}{3}\bigr)\,G\bigl(\frac{n}{2}+\frac{5}{3}\bigr)\,G\bigl(\frac{n}{2}+\frac{5}{6}\bigr)\,G\bigl(\frac{n}{2}+\frac{7}{6}\bigr)}
{G\bigl(\frac{n}{2}+\frac{3}{4}\bigr)\,G\bigl(\frac{n}{2}+\frac{5}{4}\bigr)^2\,G\bigl(\frac{n}{2}+\frac{7}{4}\bigr)},&n\text{ even}\\[7mm]
\displaystyle\frac{2^{5/2}\,G\bigl(\frac{5}{4}\bigr)\,G\bigl(\frac{7}{4}\bigr)^2\,G\bigl(\frac{9}{4}\bigr)}
{3^{3/2}\,G\bigl(\frac{7}{3}\bigr)\,G\bigl(\frac{8}{3}\bigr)\,G\bigl(\frac{5}{6}\bigr)\,
G\bigl(\frac{7}{6}\bigr)}\,\biggl(\frac{3\sqrt{3}}{4}\biggr)^{\!n^2/2}\biggl(\frac{3}{4}\biggr)^{\!3n/4}\\[4mm]
\displaystyle\qquad\qquad\qquad\qquad\qquad\times\,
\frac{G\bigl(\frac{n}{2}+\frac{1}{3}\bigr)\,G\bigl(\frac{n}{2}+\frac{2}{3}\bigr)\,G\bigl(\frac{n}{2}+\frac{11}{6}\bigr)\,G\bigl(\frac{n}{2}+\frac{13}{6}\bigr)}
{G\bigl(\frac{n}{2}+\frac{3}{4}\bigr)\,G\bigl(\frac{n}{2}+\frac{5}{4}\bigr)^2\,G\bigl(\frac{n}{2}+\frac{7}{4}\bigr)},&n\text{ odd}\end{cases}\\[4mm]
\label{OSASMAsympt}&\hspace{-1mm}=\begin{cases}\displaystyle\frac{3^{11/72}\,\pi^{1/6}\,e^{\zeta'(-1)/6}}{2^{1/24}\,\Gamma\bigl(\frac{1}{3}\bigr)^{1/3}}\,
\biggl(\frac{3\sqrt{3}}{4}\biggr)^{\!n^2/2}\biggl(\frac{3}{4}\biggr)^{\!3n/4}\,n^{-5/72}\\[5mm]
\displaystyle\quad\times\biggl(1-\frac{5}{144n}-\frac{385}{124416n^2}+\frac{57365}{5971968n^3}+\frac{15026011}{30958682112n^4}+O\bigl(n^{-5}\bigr)\biggr),&n\text{ even}\\[6mm]
\displaystyle\frac{3^{83/72}\,\pi^{1/6}\,e^{\zeta'(-1)/6}}{2^{49/24}\,\Gamma\bigl(\frac{1}{3}\bigr)^{1/3}}\,
\biggl(\frac{3\sqrt{3}}{4}\biggr)^{\!n^2/2}\biggl(\frac{3}{4}\biggr)^{\!3n/4}\,n^{67/72}\\[5mm]
\displaystyle\quad\times\biggl(1+\frac{67}{144n}+\frac{6095}{124416n^2}-\frac{173635}{5971968n^3}-\frac{417738629}{30958682112n^4}+O\bigl(n^{-5}\bigr)\biggr),&n\text{ odd}.\end{cases}
\end{align}

It follows from~\eqref{OSASMAsympt} that
\begin{equation}\label{OSASMLeadAsympt}\lim_{n\to\infty}|\OSASM(n)|^{2/n^2}=\frac{3\sqrt{3}}{4}.\end{equation}

\subsection{Asymptotics of the number of DSASMs}\label{DSASMAsymptSect}
As indicated in Section~\ref{intro}, the formula for $|\DSASM(n)|$ given by Corollary~\ref{numDSASMcoroll} was used, together with a computer, to calculate
$|\DSASM(n)|$ for all~$n$ up to~$1000$, and the values are provided at a webpage accompanying this paper~\cite{BehFisKou23b}.
Using this data, and motivated by the forms of asymptotic expansions such as~\eqref{ASMasymp2} and~\eqref{OSASMAsympt} of numbers of ASMs in other classes,
the following conjecture regarding the asymptotic expansion of~$|\DSASM(n)|$ was then obtained.
\begin{conjecture}\label{DSASMAsymptConj}
It is conjectured that
\begin{multline}\label{DSASMAsympt}|\DSASM(n)|\\
=C\,\biggl(\frac{3\sqrt{3}}{4}\biggr)^{\!n^2/2}\,3^{n/4}\,n^{-5/72}
\begin{cases}\displaystyle1-\frac{385}{31104n^2}+\frac{(-1)^{n/2}\,a}{n^{5/2}}+O\bigl(n^{-3}\bigr),&n\text{ even},\\[4.5mm]
\displaystyle1-\frac{385}{31104n^2}+\frac{(-1)^{(n+1)/2}\,b}{n^{5/2}}+O\bigl(n^{-3}\bigr),&n\text{ odd},\end{cases}\end{multline}
where $C$, $a$ and $b$ are constants, estimated as $C\approx0.72352852136732419487494728616080$,
$a\approx0.07834956265$ and $b\approx0.18915257676$.  It is also conjectured that $b=(1+\sqrt{2})a$.
\end{conjecture}
The validity of the leading term $\bigl(3\sqrt{3}/4\bigr)^{n^2/2}$ in~\eqref{DSASMAsympt} can be proved, as will be seen in Section~\ref{ASMLeadAsymptSect}. Specifically, this follows
from the $G=\{\mathcal{I},\mathcal{D}\}$ case of~\eqref{ASMLeadAsympt}.

Note that, in contrast to the asymptotic expansions~\eqref{ASMasymp2} and~\eqref{OSASMAsympt}, which arise from product formulae, it seems that in the expansion of $|\DSASM(n)|$,
the final sum contains fractional exponents and irrational coefficients.

Note also that it has not so far been possible to conjecture a closed expression for the overall constant~$C$ in~\eqref{DSASMAsympt}, whereas such expressions exist for the overall constants in expansions
such as~\eqref{ASMasymp2} and~\eqref{OSASMAsympt}. For example, the overall constant in~\eqref{ASMasymp2} is given by~\eqref{ASMconst}.

\subsection{Leading asymptotics of numbers of ASMs in any symmetry class}\label{ASMLeadAsymptSect}
A result which gives the leading term in the large $n$ asymptotics of the number of $n\times n$ ASMs in any symmetry class will now be obtained.  The necessary notation and background for ASM symmetry classes
are given in Section~\ref{ASMSymmClassSect}. Also, for each subgroup~$G$ of~$D_4$, let $\mathcal{N}(G)$ denote the set of positive integers~$n$ for which $\ASM^G(n)$ is nonempty. The sets $\mathcal{N}(G)$
and other relevant information are summarized in Table~\ref{symmclass}, in which~$\mathbb{N}$ denotes the set of positive integers and $\mathbb{N}_{\mathrm{odd}}$ denotes the set of odd positive integers.

\begin{table}[h]\centering
$\begin{array}{|c|c|c|c|}\hline\rule{0ex}{3.5ex}
\text{Subgroup }G\text{ of }D_4&|G|&\ASM^G(n)&\mathcal{N}(G)=\{n\in\mathbb{N}\mid\ASM^G(n)\ne\emptyset\}\\[2.2mm]
\hline\rule{0ex}{3ex}
\{\mathcal{I}\}&1&\ASM(n)&\mathbb{N}\\[1.9mm]
\{\mathcal{I},\mathcal{V}\}\ (\text{or }\{\mathcal{I},\mathcal{H}\})&2&\VSASM(n)\ (\text{or }\HSASM(n))&\mathbb{N}_{\mathrm{odd}}\\[1.9mm]
\{\mathcal{I},\mathcal{V},\mathcal{H},\mathcal{R}_{\pi}\}&4&\VHSASM(n)&\mathbb{N}_{\mathrm{odd}}\\[1.9mm]
\{\mathcal{I},\mathcal{R}_{\pi}\}&2&\HTSASM(n)&\mathbb{N}\\[1.9mm]
\{\mathcal{I},\mathcal{R}_{\pi/2},\mathcal{R}_{\pi},\mathcal{R}_{-\pi/2}\}&4&\QTSASM(n)&\{n\in\mathbb{N}\mid n\not\equiv2\pmod{4}\}\\[1.9mm]
\{\mathcal{I},\mathcal{D}\}\ (\text{or }\{\mathcal{I},\mathcal{A}\})&2&\DSASM(n)\ (\text{or }\ASASM(n))&\mathbb{N}\\[1.9mm]
\{\mathcal{I},\mathcal{D},\mathcal{A},\mathcal{R}_{\pi}\}&4&\DASASM(n)&\mathbb{N}\\[1.9mm]
D_4&8&\TSASM(n)&\mathbb{N}_{\mathrm{odd}}\\[2.1mm]\hline
\end{array}$\\[2.4mm]\caption{ASM symmetry classes}\label{symmclass}\end{table}

\begin{theorem}\label{ASMLeadAsymptTh}
For any subgroup $G$ of~$D_4$,
\begin{equation}\label{ASMLeadAsympt}\lim_{\substack{n\to\infty\\[0.7mm]n\in\mathcal{N}(G)}}|\ASM^G(n)|^{|G|/n^2}=\frac{3\sqrt{3}}{4}.\end{equation}
\end{theorem}

\noindent\emph{Proof}. The two main results for sequences which will be used are as follows.
\begin{list}{$\bullet$}{\setlength{\topsep}{0.8mm}\setlength{\labelwidth}{2mm}\setlength{\leftmargin}{4mm}}
\item For a real positive sequence $(x_n)_{n=1}^\infty$,
\begin{equation}\label{seq}\text{if }\lim_{n\to\infty}\frac{x_{n+2}\,x_n}{x_{n+1}^2}=L,\text{ then }\lim_{n\to\infty}x_n^{1/n^2}=L^{1/2},\end{equation}
where this can be verified straightforwardly, for example by using the Stolz--Ces\`{a}ro theorem for sequences.
\item For real sequences $(x_n)_{n=1}^\infty$, $(y_n)_{n=1}^\infty$ and $(z_n)_{n=1}^\infty$,
\begin{equation}\label{squeeze}\text{if }x_n\le y_n\le z_n\text{ for all $n$ and }\lim_{n\to\infty}x_n=\lim_{n\to\infty}z_n=L,\text{ then }\lim_{n\to\infty}y_n=L,\end{equation}
which is the standard squeeze theorem for sequences.
\end{list}
The result~\eqref{seq} will be used for cases in category~(ii) in Section~\ref{ASMSymmClassSect}, i.e., those in which $|\ASM^G(n)|$ is given by a product formula.
The typical approach in these cases is that $x_n$ is taken to be a certain subsequence of $|\ASM^G(n)|$, the associated product formula from~\eqref{numASM2}--\eqref{numDASASM}
is used to obtain a rational expression for $x_{n+2}\,x_n/x_{n+1}^2$ from which a value for $\lim_{n\to\infty}x_{n+2}\,x_n/x_{n+1}^2$ can immediately be extracted,
and~\eqref{seq} is then applied.  An alternative, but less efficient, approach for these cases would be to use the product formula to obtain a full asymptotic expansion of $|\ASM^G(n)|$,
and then to identify its leading term.
This expansion is given for $G=\{\mathcal{I}\}$ in~\eqref{ASMasymp2}, from which~\eqref{ASMLeadAsympt} for this case has already been obtained in~\eqref{ASMasymp3}.
However, for completeness, the approach involving~\eqref{seq} is still provided below for $G=\{\mathcal{I}\}$.

The squeeze theorem~\eqref{seq} will be used for cases in category~(iii) in Section~\ref{ASMSymmClassSect}, i.e., those in which a product formula for $|\ASM^G(n)|$ is not known.
The typical approach in these cases is that, for particular subgroups $H$ and $K$ of $D_4$, $x_n$ and $z_n$ are taken to be certain subsequences of $|\ASM^H(n)|^{|H|/n^2}$ and $|\ASM^K(n)|^{|K|/n^2}$, respectively,
such that $\lim_{n\to\infty}x_n=\lim_{n\to\infty}z_n=3\sqrt{3}/4$ has already been established.  Then $y_n$ is taken to be a certain subsequence of $|\ASM^G(n)|^{|G|/n^2}$,
and it is shown using simple combinatorial arguments that $x_n\le y_n\le z_n$ for all $n$, so that~\eqref{squeeze} can be applied.

A further simple result which will be used several times, often without being identified explicitly, is that, for a real positive sequence $(x_n)_{n=1}^\infty$ and real constants~$a$,~$b$ and~$c$ with $a$ nonzero,
\begin{equation}\label{lim}\lim_{n\to\infty}x_n^{1/n^2}=L\text{ \  if and only if \ }\lim_{n\to\infty}x_n^{a/(an^2+bn+c)}=L,\end{equation}
which follows from $a/(an^2+bn+c)=(1/n^2)\,(an^2/(an^2+bn+c))$ and $\lim_{n\to\infty}an^2/(an^2+bn+c)=1$.

It is only necessary to consider the eight inequivalent (up to conjugacy) cases of $G$ in~\eqref{ASMLeadAsympt}, and this will now be done for each case separately, as follows.
\begin{list}{$\bullet$}{\setlength{\topsep}{0.8mm}\setlength{\labelwidth}{2mm}\setlength{\leftmargin}{4mm}}
\item For $G=\{\mathcal{I}\}$,~\eqref{numASM2} gives
\begin{equation*}\frac{|\ASM(n+2)|\,|\ASM(n)|}{|\ASM(n+1)|^2}=\frac{3(3n+2)(3n+4)}{4(2n+1)\,(2n+3)},\end{equation*}
so that
\begin{equation*}\lim_{n\to\infty}\frac{|\ASM(n+2)|\,|\ASM(n)|}{|\ASM(n+1)|^2}=\frac{3^3}{2^4},\end{equation*}
Applying~\eqref{seq} with $x_n=|\ASM(n)|$ then gives $\lim_{n\to\infty}|\ASM(n)|^{1/n^2}=3\sqrt{3}/4$, as required.
\item For $G=\{\mathcal{I},\mathcal{V}\}$,~\eqref{numVSASM} gives
\begin{equation*}\frac{|\VSASM(2n+3)|\,|\VSASM(2n-1)|}{|\VSASM(2n+1)|^2}=\frac{9(3n+1)(3n+2)(6n-1)(6n+1)}{4(4n-1)(4n+1)^2(4n+3)},\end{equation*}
so that
\begin{equation*}\lim_{n\to\infty}\frac{|\VSASM(2n+3)|\,|\VSASM(2n-1)|}{|\VSASM(2n+1)|^2}=\frac{3^6}{2^8}.\end{equation*}
Applying~\eqref{seq} with $x_n=|\VSASM(2n-1)|$ gives $\lim_{n\to\infty}|\VSASM(2n-1)|^{1/n^2}=3^3/2^4$.
Applying~\eqref{lim} with $x_n=|\VSASM(2n-1)|$, $a=4$, $b=-4$ and $c=1$ then gives
$\lim_{n\to\infty}\allowbreak|\VSASM(2n-1)|^{4/(2n-1)^2}=3^3/2^4$, so that
$\lim_{n\to\infty}|\VSASM(2n-1)|^{2/(2n-1)^2}=(3^3/2^4)^{1/2}=3\sqrt{3}/4$, as required.
\item For $G=\{\mathcal{I},\mathcal{V},\mathcal{H},\mathcal{R}_{\pi}\}$,~\eqref{numVHSASM} gives
\begin{equation*}\frac{|\VHSASM(2n+3)|\,|\VHSASM(2n-1)|}{|\VHSASM(2n+1)|^2}
=\frac{3(3n+2(-1)^n)(3n-(-1)^n)}{4(2n-1)(2n+1)},\end{equation*}
so that
\begin{equation*}\lim_{n\to\infty}\frac{|\VHSASM(2n+3)|\,|\VHSASM(2n-1)|}{|\VHSASM(2n+1)|^2}=\frac{3^3}{2^4}.\end{equation*}
Applying~\eqref{seq} with $x_n=|\VHSASM(2n-1)|$ gives $\lim_{n\to\infty}|\VHSASM(2n-1)|^{1/n^2}=3\sqrt{3}/4$, from which
it follows (using~\eqref{lim}) that $\lim_{n\to\infty}|\VHSASM(2n-1)|^{4/(2n-1)^2}=3\sqrt{3}/4$, as required.
\item For $G=\{\mathcal{I},\mathcal{R}_{\pi}\}$,~\eqref{numHTSASM} gives
\begin{align*}\frac{|\HTSASM(2n+4)|\,|\HTSASM(2n)|}{|\HTSASM(2n+2)|^2}&=\frac{9(3n+1)(3n+2)(3n+4)(3n+5)}{16(2n+1)^2(2n+3)^2},\\
\intertext{and}
\frac{|\HTSASM(2n+3)|\,|\HTSASM(2n-1)|}{|\HTSASM(2n+1)|^2}&=\frac{9(3n+1)^2(3n+2)^2}{16(2n+1)^4},\end{align*}
so that
\begin{equation*}\lim_{n\to\infty}\frac{|\HTSASM(n+4)|\,|\HTSASM(n)|}{|\HTSASM(n+2)|^2}=\frac{3^6}{2^8}.\end{equation*}
Applying~\eqref{seq} separately with $x_n=|\HTSASM(2n)|$ and with $x_n=|\HTSASM(2n-1)|$ gives $\lim_{n\to\infty}|\HTSASM(2n)|^{1/n^2}=3^3/2^4$ and
$\lim_{n\to\infty}|\HTSASM(2n-1)|^{1/n^2}=3^3/2^4$, from which it follows that $\lim_{n\to\infty}|\HTSASM(n)|^{2/n^2}=3\sqrt{3}/4$,
as required.
\item For $G=\{\mathcal{I},\mathcal{R}_{\pi/2},\mathcal{R}_{\pi},\mathcal{R}_{-\pi/2}\}$, it follows straightforwardly from~\eqref{numQTSASM}, using the
$G=\{\mathcal{I}\}$ and $G=\{\mathcal{I},\mathcal{R}_{\pi}\}$ cases of~\eqref{ASMLeadAsympt}, that
$\lim_{n\to\infty}|\QTSASM(4n+k)|^{4/(4n+k)^2}=3\sqrt{3}/4$, for $k\in\{-1,0,1\}$, as required.
\item For $G=\{\mathcal{I},\mathcal{D},\mathcal{A},\mathcal{R}_{\pi}\}$,~\eqref{numDASASM} gives
$|\DASASM(2n-1)|^2=3^{n-1}\,|\HTSASM(2n-1)|$, so that $\lim_{n\to\infty}|\DASASM(2n-1)|^{4/(2n-1)^2}=\lim_{n\to\infty}3^{2(n-1)/(2n-1)^2}\,|\HTSASM(2n-1)|^{2/(2n-1)^2}
=\lim_{n\to\infty}|\HTSASM(2n-1)|^{2/(2n-1)^2}$.  The $G=\{\mathcal{I},\mathcal{R}_{\pi}\}$ case of~\eqref{ASMLeadAsympt} then gives
\begin{equation}\label{DASASModdlim}\lim_{n\to\infty}|\DASASM(2n-1)|^{4/(2n-1)^2}=\frac{3\sqrt{3}}{4}.\end{equation}
Now consider the inequality
\begin{equation}\label{DASASMInequ}|\DASASM(n)|\le|\DASASM(n+1)|,\end{equation}
which can be verified by considering a simple bijection from the subset $\DASASM(n+1)'=\{A\in\DASASM(n+1)\mid A_{i,\lfloor n/2\rfloor+1}=0,\;\text{for}\;i=1,\ldots,\lceil n/2\rceil\}$
of $\DASASM(n+1)$ to $\DASASM(n)$, as follows.
For $n$ even, the bijection is the mapping which deletes row and column $n/2+1$ (in which each
entry is~0, except for a $1$ in the center) from each matrix in $\DASASM(n+1)'$.
For $n$ odd and~$A\in\DASASM(n+1)'$, the image $\tilde{A}$ of $A$ under the bijection has the following rows.
For $1\le i\le(n-1)/2$, row~$i$ of~$\tilde{A}$ is row~$i$ of~$A$ with its $(n+1)/2$-th entry (which is~$0$) deleted.
Row $(n+1)/2$ of~$\tilde{A}$ is $(A_{(n+3)/2,1},\ldots,A_{(n+3)/2,(n-1)/2},2A_{(n+1)/2,(n+3)/2}-1,A_{(n+1)/2,(n+5)/2},\ldots,A_{(n+1)/2,n+1})$.
For $(n+3)/2\le i\le n$, row $i$ of $A'$ is row $i+1$ of $A$ with its $(n+3)/2$-th entry (which is~$0$) deleted.

It follows from~\eqref{DASASModdlim} that $\lim_{n\to\infty}|\DASASM(2n-1)|^{1/n^2}=\lim_{n\to\infty}|\DASASM(2n+1)|^{1/n^2}=3\sqrt{3}/4$, and
from~\eqref{DASASMInequ} that $|\DASASM(2n-1)|^{1/n^2}\le|\DASASM(2n)|^{1/n^2}\le|\DASASM(2n+1)|^{1/n^2}$.
Hence, by applying the squeeze theorem~\eqref{squeeze} with $x_n=|\DASASM\allowbreak(2n-1)|^{1/n^2}$, $y_n=|\DASASM(2n)|^{1/n^2}$ and $z_n=|\DASASM(2n+1)|^{1/n^2}$, it follows that
$\lim_{n\to\infty}\allowbreak|\DASASM(2n)|^{1/n^2}=\lim_{n\to\infty}|\DASASM(2n)|^{4/(2n)^2}=3\sqrt{3}/4$.

Finally, combining this and~\eqref{DASASModdlim} gives $\lim_{n\to\infty}|\DASASM(n)|^{4/n^2}=3\sqrt{3}/4$, as required.
\item For $G=\{\mathcal{I},\mathcal{D}\}$, consider functions
$\NEtri\colon\ASM(n)\to\DSASM(n)$, $\SWtri\colon\ASM(n)\to\DSASM(n)$,
$\NWtri\colon\ASM(n)\to\ASASM(n)$ and $\SEtri\colon\ASM(n)\to\ASASM(n)$,
where, for each $A\in\ASM(n)$, $\NEtri(A)$ (respectively, $\SWtri(A)$) is the unique $n\times n$ DSASM
whose strictly upper (respectively, lower) triangular part is the same as that part of~$A$,
and $\NWtri(A)$ (respectively, $\SEtri(A)$) is the unique $n\times n$ ASASM whose part strictly above (respectively, below) the main antidiagonal
is the same as that part of~$A$. Explicitly, $\NEtri(A)$ is given by
\begin{equation*}
\NEtri(A)_{ij}=\begin{cases}A_{ij},&i<j,\\
1-\sum_{k=1}^{i-1}A_{ki}-\sum_{k=i+1}^nA_{ik},&i=j,\\
A_{ji},&i>j,\end{cases}\end{equation*}
for $1\le i,j\le n$, and the other three functions can be expressed in terms of $\NEtri$
as $\SWtri=\NEtri\mathcal{D}=\mathcal{R}_{\pi}\NEtri\mathcal{R}_{\pi}$,
$\NWtri=\mathcal{R}_{\pi/2}\NEtri\mathcal{R}_{-\pi/2}$ and $\SEtri=\mathcal{R}_{-\pi/2}\NEtri\mathcal{R}_{\pi/2}$.
It can easily be checked that these functions are well-defined.

Now consider functions
$E\colon\ASM(n)\to\DSASM(n)^2$ and $F\colon\DSASM(n)\to\DASASM(n)^2$ given by
\begin{align}\notag E(A)&=(\NEtri(A),\SWtri(A)),\text{ for each }A\in\ASM(n),\\
\label{FGdef}F(A)&=(\NWtri(A),\SEtri(A)),\text{ for each }A\in\DSASM(n).\end{align}
These are illustrated schematically in Figure~\ref{FGJKFig}.
It can be verified straightforwardly that~$E$ and~$F$ are each injective, which implies that
\begin{equation}\label{DSASMineq}|\ASM(n)|\le|\DSASM(n)|^2\le|\DASASM(n)|^4.\end{equation}
It now follows using the $G=\{\mathcal{I}\}$ and $G=\{\mathcal{I},\mathcal{D},\mathcal{A},\mathcal{R}_{\pi}\}$ cases of~\eqref{ASMLeadAsympt},
and by applying the squeeze theorem~\eqref{squeeze} with $x_n=|\ASM(n)|^{1/n^2}$, $y_n=|\DSASM(n)|^{2/n^2}$ and $z_n=|\DASASM(n)|^{4/n^2}$, that
$\lim_{n\to\infty}|\DSASM(n)|^{2/n^2}=3\sqrt{3}/4$, as required.

Note that an alternative approach in this case involves observing that $|\OSASM(n)|\le|\DSASM(n)|$, since $\OSASM(n)$ is a subset of $\DSASM(n)$.
Hence, the inequalities
\begin{equation*}|\OSASM(n)|^2\le|\DSASM(n)|^2\le|\DASASM(n)|^4\end{equation*}
can be used instead of~\eqref{DSASMineq}.  Then, using~\eqref{OSASMLeadAsympt} and the $G=\{\mathcal{I},\mathcal{D},\mathcal{A},\mathcal{R}_{\pi}\}$ case of~\eqref{ASMLeadAsympt},
and applying the squeeze theorem~\eqref{squeeze} with $x_n=|\OSASM(n)|^{2/n^2}$, $y_n=|\DSASM(n)|^{2/n^2}$ and $z_n=|\DASASM(n)|^{4/n^2}$, again gives
 $\lim_{n\to\infty}|\DSASM(n)|^{2/n^2}=3\sqrt{3}/4$.

\psset{unit=2.8mm}
\begin{figure}[h]
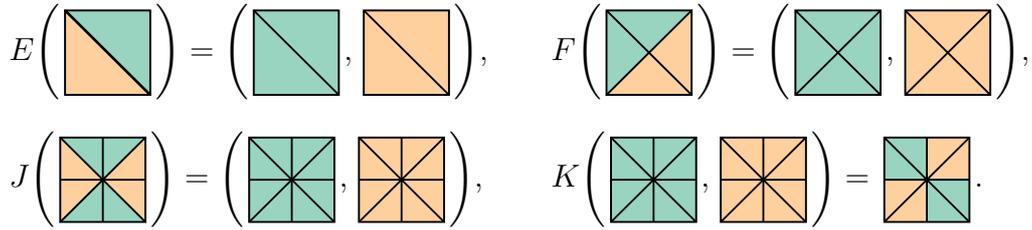
\centering
$\begin{array}{l@{\qquad}l}E\Biggl(\raisebox{-4.65mm}{\pspicture(-0.25,0)(4.25,4)
\pspolygon[linewidth=0,fillstyle=solid,fillcolor=orng](0,0)(4,0)(0,4)\pspolygon[linewidth=0,fillstyle=solid,fillcolor=grn](4,0)(4,4)(0,4)
\pspolygon[linewidth=0.7pt](0,0)(4,0)(4,4)(0,4)\psline[linewidth=1pt](0,4)(4,0)\endpspicture}\Biggr)
=\Biggl(\raisebox{-4.65mm}{\pspicture(-0.25,0)(4.25,4)\pspolygon[linewidth=0,fillstyle=solid,fillcolor=grn](0,0)(0,4)(4,4)(4,0)\pspolygon[linewidth=0.7pt](0,0)(4,0)(4,4)(0,4)\psline[linewidth=0.7pt](0,4)(4,0)
\endpspicture},
\raisebox{-4.65mm}{\pspicture(-0.25,0)(4.25,4)\pspolygon[linewidth=0,fillstyle=solid,fillcolor=orng](0,0)(0,4)(4,4)(4,0)\pspolygon[linewidth=0.7pt](0,0)(4,0)(4,4)(0,4)\psline[linewidth=0.7pt](0,4)(4,0)
\endpspicture}\Biggr),&
F\Biggl(\raisebox{-4.65mm}{\pspicture(-0.25,0)(4.25,4)\pspolygon[linewidth=0,fillstyle=solid,fillcolor=grn](0,0)(4,4)(0,4)\pspolygon[linewidth=0,fillstyle=solid,fillcolor=orng](0,0)(4,4)(4,0)
\pspolygon[linewidth=0.7pt](0,0)(4,0)(4,4)(0,4)\psline[linewidth=0.7pt](0,4)(4,0)\psline[linewidth=0.7pt](0,0)(4,4)\endpspicture}\Biggr)
=\Biggl(\raisebox{-4.65mm}{\pspicture(-0.25,0)(4.25,4)\pspolygon[linewidth=0,fillstyle=solid,fillcolor=grn](0,0)(0,4)(4,4)(4,0)
\pspolygon[linewidth=0.7pt](0,0)(4,0)(4,4)(0,4)\psline[linewidth=0.7pt](0,4)(4,0)\psline[linewidth=0.7pt](0,0)(4,4)\endpspicture},
\raisebox{-4.65mm}{\pspicture(-0.25,0)(4.25,4)\pspolygon[linewidth=0,fillstyle=solid,fillcolor=orng](0,0)(0,4)(4,4)(4,0)
\pspolygon[linewidth=0.7pt](0,0)(4,0)(4,4)(0,4)\psline[linewidth=0.7pt](0,4)(4,0)\psline[linewidth=0.7pt](0,0)(4,4)\endpspicture}
\Biggr),\\[8mm]
J\Biggl(\raisebox{-4.65mm}{\pspicture(-0.25,0)(4.25,4)
\pspolygon[linewidth=0,fillstyle=solid,fillcolor=grn](2,2)(0,4)(4,4)\pspolygon[linewidth=0,fillstyle=solid,fillcolor=orng](2,2)(4,4)(4,0)
\pspolygon[linewidth=0,fillstyle=solid,fillcolor=grn](2,2)(4,0)(0,0)\pspolygon[linewidth=0,fillstyle=solid,fillcolor=orng](2,2)(0,4)(0,0)
\pspolygon[linewidth=0.7pt](0,0)(4,0)(4,4)(0,4)\psline[linewidth=0.7pt](0,4)(4,0)\psline[linewidth=0.7pt](0,0)(4,4)\psline[linewidth=0.7pt](2,0)(2,4)\psline[linewidth=0.7pt](0,2)(4,2)\endpspicture}\Biggr)
=\Biggl(
\raisebox{-4.65mm}{\pspicture(-0.25,0)(4.25,4)\pspolygon[linewidth=0,fillstyle=solid,fillcolor=grn](0,0)(0,4)(4,4)(4,0)
\pspolygon[linewidth=0.7pt](0,0)(4,0)(4,4)(0,4)\psline[linewidth=0.7pt](0,4)(4,0)\psline[linewidth=0.7pt](0,0)(4,4)\psline[linewidth=0.7pt](2,0)(2,4)\psline[linewidth=0.7pt](0,2)(4,2)\endpspicture},
\raisebox{-4.65mm}{\pspicture(-0.25,0)(4.25,4)\pspolygon[linewidth=0,fillstyle=solid,fillcolor=orng](0,0)(0,4)(4,4)(4,0)
\pspolygon[linewidth=0.7pt](0,0)(4,0)(4,4)(0,4)\psline[linewidth=0.7pt](0,4)(4,0)\psline[linewidth=0.7pt](0,0)(4,4)\psline[linewidth=0.7pt](2,0)(2,4)\psline[linewidth=0.7pt](0,2)(4,2)\endpspicture}\Biggr),&
K\Biggl(\raisebox{-4.65mm}{\pspicture(-0.25,0)(4.25,4)\pspolygon[linewidth=0,fillstyle=solid,fillcolor=grn](0,0)(0,4)(4,4)(4,0)
\pspolygon[linewidth=0.7pt](0,0)(4,0)(4,4)(0,4)\psline[linewidth=0.7pt](0,4)(4,0)\psline[linewidth=0.7pt](0,0)(4,4)\psline[linewidth=0.7pt](2,0)(2,4)
\psline[linewidth=0.7pt](0,2)(4,2)\endpspicture},
\raisebox{-4.65mm}{\pspicture(-0.25,0)(4.25,4)\pspolygon[linewidth=0,fillstyle=solid,fillcolor=orng](0,0)(0,4)(4,4)(4,0)
\pspolygon[linewidth=0.7pt](0,0)(4,0)(4,4)(0,4)\psline[linewidth=0.7pt](0,4)(4,0)\psline[linewidth=0.7pt](0,0)(4,4)\psline[linewidth=0.7pt](2,0)(2,4)\psline[linewidth=0.7pt](0,2)(4,2)\endpspicture}\Biggr)=
\raisebox{-4.65mm}{\pspicture(-0.25,0)(4.25,4)\pspolygon[linewidth=0,fillstyle=solid,fillcolor=orng](0,0)(0,2)(2,2)(2,0)\pspolygon[linewidth=0,fillstyle=solid,fillcolor=grn](0,2)(0,4)(2,4)(2,2)
\pspolygon[linewidth=0,fillstyle=solid,fillcolor=orng](2,2)(2,4)(4,4)(4,2)\pspolygon[linewidth=0,fillstyle=solid,fillcolor=grn](2,0)(2,2)(4,2)(4,0)
\pspolygon[linewidth=0.7pt](0,0)(4,0)(4,4)(0,4)\psline[linewidth=0.7pt](0,4)(4,0)\psline[linewidth=0.7pt](0,0)(4,4)\psline[linewidth=0.7pt](2,0)(2,4)\psline[linewidth=0.7pt](0,2)(4,2)\endpspicture}.
\end{array}$
\caption{Schematic illustrations of the functions $E\colon\ASM(n)\to\DSASM(n)^2$,
$F\colon\DSASM(n)\to\DASASM(n)^2$, $J\colon\VHSASM(2n-1)\to\TSASM(2n-1)^2$ and
$K\colon\TSASM(2n-1)^2\to\DASASM(2n-1)$ in~\eqref{FGdef} and~\eqref{JKdef}. For each function,
all of the smallest triangular regions of the same color contain the same
ASM entries, up to reflection in symmetry axes.}\label{FGJKFig}
\end{figure}
\item For $G=D_4$, consider a function $\Sq\colon\VSASM(2n-1)^2\to\ASM(2n-1)$, where for each $(A_1,A_2)\in\VSASM(2n-1)^2$, $\Sq(A_1,A_2)$ is the unique $(2n-1)\times(2n-1)$ ASM
whose part on or left (respectively, right) of the central column is the same as that part of $A_1$ (respectively,~$A_2$), i.e.,
\begin{equation*}\bigl(\Sq(A_1,A_2)\bigr)_{ij}=\begin{cases}(A_1)_{ij},&j\le n,\\
(A_2)_{ij},&j\ge n,\end{cases}\end{equation*}
for $1\le i,j\le2n-1$. It can easily be checked, by observing that the central column of any VSASM is $(1,-1,1,-1,\ldots,-1,1)$,
that $\Sq$ is well defined.

Now consider functions
$J\colon\VHSASM(2n-1)\to\TSASM(2n-1)^2$ and $K\colon\TSASM(2n-1)^2\allowbreak\to\DASASM(2n-1)$ given by
\begin{align}\notag J(A)&=\bigl((\NWtri\NEtri)(A),(\SEtri\NEtri)(A)\bigr),\text{ for each }A\in\VHSASM(2n-1),\\
\label{JKdef}K(A_1,A_2)&=(\NWtri\NEtri\Sq)(A_1,A_2),\text{ for each }(A_1,A_2)\in\TSASM(2n-1)^2.\end{align}
(Note that for $J$,
$(\NWtri\NEtri)(A)=(\NEtri\NWtri)(A)=(\SWtri\SEtri)(A)=(\SEtri\SWtri)(A)$ and
$(\SEtri\NEtri)(A)=(\NEtri\SEtri)(A)=(\SWtri\NWtri)(A)=(\NWtri\SWtri)(A)$,
and for $K$, $(\NWtri\NEtri\Sq)(A_1,A_2)=(\NEtri\NWtri\Sq)(A_1,A_2)$.)
These functions are illustrated schematically in Figure~\ref{FGJKFig}.
It can be verified straightforwardly that~$J$ and~$K$ are each injective, which implies that
\begin{equation*}|\VHSASM(2n-1)|\le|\TSASM(2n-1)|^2\le|\DASASM(2n-1)|.\end{equation*}
It now follows using the $G=\{\mathcal{I},\mathcal{V},\mathcal{H},\mathcal{R}_{\pi}\}$ and $G=\{\mathcal{I},\mathcal{D},\mathcal{A},\mathcal{R}_{\pi}\}$ cases of~\eqref{ASMLeadAsympt},
and the squeeze theorem~\eqref{squeeze}, that
$\lim_{n\to\infty}|\TSASM(2n-1)|^{8/(2n-1)^2}=3\sqrt{3}/4$, as required.\qed
\end{list}

\subsection{Leading asymptotics of $|\DSASM(n,k)|$}\label{ASMLeadAsymptFurthSect1}
Theorem~\ref{ASMLeadAsymptTh} for the leading asymptotics of the numbers of ASMs in the eight standard symmetry classes was proved in some cases by using known product formulae,
and in other cases by using combinatorial arguments to obtain certain inequalities
and then applying the squeeze theorem.  The same methods can be used to obtain the leading asymptotics for the numbers of ASMs
in many other well-studied ASM subclasses.  This will be illustrated in the following proposition for the DSASM subclasses $\DSASM(n,k)$ (as defined in~\eqref{DSASMnk}), and
some further ASM subclasses will then be considered in Section~\ref{ASMLeadAsymptFurthSect2}.

\begin{proposition}\label{DSASMRefLeadAsymptTh}For each $k\in\{1\ldots n\}$, the number of $n\times n$ DSASMs whose~$1$ in the first row is in column~$k$ satisfies
\begin{equation}\label{DSASMRefLeadAsympt}\lim_{n\to\infty}|\DSASM(n,k)|^{2/n^2}=\frac{3\sqrt{3}}{4}.\end{equation}
\end{proposition}
\begin{proof}The $k=1$, $k=2$ and $k=n$ cases of~\eqref{DSASMRefLeadAsympt} follow immediately from~(i),~(ii) and~(iii), respectively, of Proposition~\ref{refprop},
and the $G=\{\mathcal{I},\mathcal{D}\}$ case of~\eqref{ASMLeadAsympt}.
For $3\le k\le n-1$, note that
\begin{equation}\label{DSASMRefInequ}|\DSASM(n-2)|\le|\DSASM(n,k)|\le|\DSASM(n)|,\end{equation}
where the second inequality holds since $\DSASM(n,k)$ is a subset of $\DSASM(n)$, and the first inequality
can be verified by considering a simple bijection from $\{A\in\DSASM(n,k)\mid A_{i,k}=0,\;\text{for}\;i=2,\ldots,n\}$ to
$\DSASM(n-2)$, in which rows and columns~$1$ and~$k$ are deleted from each matrix in the first set.
(In fact, this is also a well-defined bijection for $k=2$ and $k=n$.) It follows from the $G=\{\mathcal{I},\mathcal{D}\}$ case of~\eqref{ASMLeadAsympt} that
$\lim_{n\to\infty}|\DSASM(n-2)|^{2/n^2}=|\DSASM(n)|^{2/n^2}=3\sqrt{3}/4$.  Therefore,~\eqref{DSASMRefInequ} and the squeeze theorem~\eqref{squeeze} give~\eqref{DSASMRefLeadAsympt} for
$3\le k\le n-1$, as required.\end{proof}

\subsection{Discussion of the leading asymptotics of numbers of ASMs and other objects}\label{ASMLeadAsymptFurthSect2}
Some further aspects of the leading asymptotics of numbers of ASMs and related objects
(including six-vertex model configurations and rhombus tilings)
will now be discussed, and more general frameworks for the results of Sections~\ref{ASMLeadAsymptSect} and~\ref{ASMLeadAsymptFurthSect1}
will be outlined. The sets of positive integers and odd positive integers will again be denoted as~$\mathbb{N}$ and~$\mathbb{N}_{\mathrm{odd}}$, respectively.

Let $C$ denote an ASM subclass considered in~\eqref{OSASMLeadAsympt}, Theorem~\ref{ASMLeadAsymptTh} or Proposition~\ref{DSASMRefLeadAsymptTh},
and~$C(n)$ denote the associated subset of the set of $n\times n$ ASMs, i.e., $C(n)$ is $\OSASM(n)$, $\ASM^G(n)$ or $\DSASM(n,k)$, for a subgroup~$G$ of~$D_4$ or $k\in\{1\ldots n\}$.
Also, let $N_C$ be $\mathbb{N}$ if $C(n)$ is $\OSASM(n)$ or $\DSASM(n,k)$, or be $\mathcal{N}(G)$ if $C(n)$ is $\ASM^G(n)$ .

It can easily be checked, using the defining properties of each class~$C$, and~\eqref{OSASMLeadAsympt},~\eqref{ASMLeadAsympt} and~\eqref{DSASMRefLeadAsympt},
that there exist a (not necessarily unique) quadratic polynomial~$P_C(n)$ in~$n$, and a subset $S_C(n)$ of $\{1,\ldots,n\}^2$ for $n\in N_C$, with the following properties.
\begin{list}{$\bullet$}{\setlength{\labelwidth}{4mm}\setlength{\leftmargin}{8mm}\setlength{\labelsep}{3mm}\setlength{\topsep}{0.9mm}}
\item For $n\in N_C$, any $A\in C(n)$ is uniquely determined by its entries $A_{ij}$ with $(i,j)\in S_C(n)$.
\item For $n\in N_C$, $|S_C(n)|=P_C(n)$.
\item The leading asymptotics of $|C(n)|$ are given by
\begin{equation}\label{GenLeadAsympt1}
\lim_{\substack{n\to\infty\\[0.5mm]n\in N_C}}|C(n)|^{1/P_C(n)}=\frac{3\sqrt{3}}{4}.\end{equation}
\end{list}
The set $S_C(n)$ can be regarded as a fundamental domain for the ASMs in $C(n)$, and a natural interpretation of~\eqref{GenLeadAsympt1} is that,
for large $n\in N_C$ and any $(i,j)\in S_C(n)$, there are on average $3\sqrt{3}/4$ choices for the entry $A_{ij}$ of an ASM~$A$ in~$C(n)$.

Some examples are as follows. If $C(n)=\ASM(n)$, then natural choices for $S_C(n)$ are $\{1,\ldots,n\}^2$ or (using the property that the sum of entries in each row and column of any ASM is~$1$)
$\{1,\ldots,n-1\}^2$, which give $P_C(n)=n^2$ or $P_C(n)=(n-1)^2$, respectively.
If $C(n)=\DSASM(n)$, then natural choices for $S_C(n)$ are $\{(i,j)\mid 1\le i\le j\le n\}^2$ (since any DSASM is determined by its upper triangular part)
or $\{(i,j)\mid 1\le i<j\le n\}^2$ (using also the property that the sum of entries in each row and column of any DSASM is~$1$), which give $P_C(n)=n(n+1)/2$ or
$P_C(n)=n(n-1)/2$, respectively.

Note that for fixed $C$, $P_C(n)$ has a unique leading coefficient $\ell_C=[n^2]P_C(n)$,
and~\eqref{OSASMLeadAsympt},~\eqref{ASMLeadAsympt} and~\eqref{DSASMRefLeadAsympt} state that
\begin{equation}\label{GenLeadAsympt2}
\lim_{\substack{n\to\infty\\[0.5mm]n\in N_C}}|C(n)|^{1/(\ell_C\,n^2)}=\frac{3\sqrt{3}}{4},
\end{equation}
due to~\eqref{lim}.

There are also several further, well-studied ASM subclasses $C$ for which~\eqref{GenLeadAsympt1} holds, together with an associated set $N_C$ and quadratic polynomial $P_C(n)$. Some examples are as follows.
\begin{list}{(\roman{sc})}{\usecounter{sc}
\setlength{\labelwidth}{8.2mm}\setlength{\leftmargin}{8.2mm}\setlength{\labelsep}{1.5mm}\setlength{\topsep}{0.9mm}}
\item The set of $n\times n$ ASMs whose $1$ in the first row is in column $k$, for fixed $k\in\{1,\ldots,n\}$.  A product formula for the numbers of these ASMs
was conjectured by Mills, Robbins and Rumsey~\cite[Conj.~2]{MilRobRum82,MilRobRum83}, and first proved by Zeilberger~\cite[Main Thm.]{Zei96b}.
In this case, $N_C=\mathbb{N}$, and some natural choices for $S_C(n)$ are $\{2,\ldots,n\}\times\{1,\ldots,n\}$ or $\{2,\ldots,n-1\}\times\{1,\ldots,n-1\}$,
which give $P_C(n)=n(n-1)$ or $P_C(n)=(n-1)(n-2)$, respectively.
\item The set of $n\times n$ quasi-QTSASMs for $n\equiv2\pmod{4}$ (so that $N_C=\{n\in\mathbb{N}\mid n\equiv2\pmod{4}\}$).
Such ASMs are defined by Aval and Duchon~\cite{AvaDuc10}, who also obtained a product formula for their number~\cite[Thm.~3]{AvaDuc10}.
\item The set of $n\times n$ DASASMs with $n$ odd (so that $N_C=\mathbb{N}_{\mathrm{odd}}$) and maximally-many $k$'s on the main diagonal and main antidiagonal for
fixed $k\in\{-1,0,1\}$, or minimally-many $0$'s on the main diagonal and main antidiagonal.  Product formulae for the numbers of ASMs in these four cases were obtained
by Ayyer, Behrend and Fischer~\cite[Thms.~1.2,~1.3~\&~1.5, Cor.~1.4]{AyyBehFis20}.
\end{list}
The validity of~\eqref{GenLeadAsympt1} for these cases can be verified by using the associated product formulae, and applying the same methods as those used
in Sections~\ref{ASMLeadAsymptSect} and~\ref{ASMLeadAsymptFurthSect1} for cases with known product formulae. Alternatively, some of the verifications can be done
by using combinatorial arguments to obtain certain inequalities, and applying the same methods as those used
in Sections~\ref{ASMLeadAsymptSect} and~\ref{ASMLeadAsymptFurthSect1} for cases without known product formulae.

The result~\eqref{GenLeadAsympt1} (or~\eqref{GenLeadAsympt2}), which generalizes the results of Sections~\ref{ASMLeadAsymptSect} and~\ref{ASMLeadAsymptFurthSect1},
can also be interpreted in terms of the six-vertex model. In particular, each set~$C(n)$ is in simple bijection with the set of
six-vertex model configurations on a certain graph, with certain boundary conditions, for $n\in N_C$.  For example, if $C(n)=\DSASM(n)$, then the graph is~$\G_n$, as given in~\eqref{Tn}, and the
six-vertex model configurations on $\G_n$
are as defined in Section~\ref{sixvertexmodelconfig}.  Another example is that if $C(n)=\ASM(n)$, then the six-vertex model configurations are those with domain-wall boundary conditions on the graph~$\mathcal{S}_n$,
as also discussed in Section~\ref{sixvertexmodelconfig}.  When~\eqref{GenLeadAsympt1} is interpreted in terms of the six-vertex model, $C(n)$ is taken to be the relevant set of six-vertex model configurations,
and~$P_C(n)$ is the number of vertices in the associated graph.

Although~\eqref{GenLeadAsympt1} holds for the six-vertex model configurations associated with all cases already considered in this section, there are further well-studied cases for which it does not hold.  For example,
let~$C(n)$ be the set of six-vertex model configurations on an $n\times n$ grid graph with toroidal boundary conditions (which are simply Eulerian orientations of an $n\times n$ grid graph on a torus).
Then $N_C=\mathbb{N}$, and $p_C(n)=n^2$ (since the associated graph has~$n^2$ vertices), and it was shown by Lieb~\cite{Lie67a,Lie67b} that $\lim_{n\to\infty}|C(n)|^{1/n^2}=8\sqrt{3}/9$.  Accordingly,
this case is associated with a class of boundary conditions which differs from the single class associated with all cases previously considered in this section.  For further information on such classes of boundary
conditions for the six-vertex model (in the more general context of partition functions, i.e., certain weighted enumerations of numbers of configurations), see, for example,
Korepin and Zinn-Justin~\cite[Sec.~6]{KorZin00}, Ribeiro and Korepin~\cite[Secs.~5~\&~6]{RibKor15}, Tavares, Ribeiro and Korepin~\cite[Sec.~7]{TavRibKor15a},
Tavares, Ribeiro and Korepin~\cite{TavRibKor15b}, and Zinn-Justin~\cite{Zin02}.

Also observe that~\eqref{ASMLeadAsympt} and~\eqref{GenLeadAsympt1} can be regarded as ASM-related cases of more general properties which may be satisfied by further types of combinatorial sets.
More specifically, let $(Q(n))_{n=1}^\infty$ be a sequence of finite sets,
and suppose that, for a positive integer~$d$, $|Q(n)|^{1/n^d}$ converges to a finite positive limit as $n\to\infty$.
Furthermore, suppose that there exists a finite group $\mathcal{G}$ which has an action on each $Q(n)$.  Then, for any subgroup~$G$ of~$\mathcal{G}$, let $Q^G(n)$ denote the set of elements of~$Q(n)$
which are invariant under the action of all elements of~$G$, and let~$\mathcal{N}(G)$ denote the set of positive integers~$n$ for which~$Q^G(n)$ is nonempty (where this is assumed to be an infinite set).

A possible property of the asymptotics of $|Q^G(n)|$ is that, for any subgroup $G$ of $\mathcal{G}$,
\begin{equation}\label{GenLeadAsymptGroup}\lim_{\substack{n\to\infty\\[0.7mm]n\in\mathcal{N}(G)}}|Q^G(n)|^{|G|/n^d}=\lim_{n\to\infty}|Q(n)|^{1/n^d}.\end{equation}
This provides a generalized context for~\eqref{ASMLeadAsympt}, since it follows from~\eqref{ASMLeadAsympt} that~\eqref{GenLeadAsymptGroup} holds
for the case $Q(n)=\ASM(n)$, $\mathcal{G}=D_4$ and $d=2$.

Similarly, a generalized context for~\eqref{GenLeadAsympt1} is to consider a combinatorial class~$Q$, which is associated with an infinite subset $N_Q$ of $\mathbb{N}$,
a sequence $(Q(n))_{n\in N_Q}$ of finite sets, a polynomial~$P_Q(n)$ in~$n$, and a positive constant $L_Q$, such that
\begin{equation}\label{GenLeadAsympt3}
\lim_{\substack{n\to\infty\\[0.5mm]n\in N_Q}}|Q(n)|^{1/P_Q(n)}=L_Q.\end{equation}
Typically, $P_Q(n)$ gives the size of a fundamental domain on which each element of $Q(n)$ can be defined.
An important aspect of~\eqref{GenLeadAsympt3} is that there may be several related classes~$Q$ for which $L_Q$ is independent of~$Q$.
For example, if $Q(n)$ is taken to be $\OSASM(n)$, $\ASM^G(n)$, $\DSASM(n,k)$ or any of the cases in~(i)--(iii) of the list earlier in this section, then $L_Q=3\sqrt{3}/4$, as in~\eqref{GenLeadAsympt1}.

Note that each~$Q(n)$ in~\eqref{GenLeadAsympt3} may arise as the set of states of a discrete statistical mechanical model, defined
on $P_Q(n)$ sites. In this context, if each state has zero energy, then the partition function is the number $|Q(n)|$ of states,
and $k_{\mathrm{B}}\log L_Q$ is the entropy per site in the thermodynamic (i.e., large~$n$) limit,
where~$k_{\mathrm{B}}$ is the Boltzmann constant. See, for example, Baxter~\cite[Eqs.~(8.1.1)--(8.1.2)]{Bax82} for this perspective
in relation to the six-vertex model.

Finally, a brief review will be given of rhombus tilings, since these provide a further example of combinatorial objects for which~\eqref{GenLeadAsymptGroup} or~\eqref{GenLeadAsympt3} hold in many well-studied cases.
Specifically, the rhombus tilings under consideration are tilings of finite connected regions of the regular triangular lattice,
in which each rhombic tile is the union of two equilateral unit triangles which share an edge.
Such tilings are in simple bijection with perfect matchings of an associated region of the dual hexagonal lattice,
and certain cases are also in bijection with plane partitions.  It is known that rhombus tilings share many enumerative properties with alternating sign matrices,
with some natural finite sets of one object being equinumerous with those of the other, but the underlying combinatorial reasons for this remain poorly understood.
For further details regarding such tilings, perfect matchings and plane partitions, see, for example Krattenthaler~\cite{Kra16}, and references therein.

Let $\mathrm{RT}(n)$ denote the set of rhombus tilings of a regular hexagon of side length~$n$
on the regular triangular lattice.  Then the symmetry group~$D_6$ of the hexagon has a natural action on $\mathrm{RT}(n)$.
There are ten conjugacy classes of subgroups of~$D_6$,
which correspond to ten inequivalent symmetry classes~$\mathrm{RT}^G(n)$, for subgroups~$G$ of~$D_6$
(where $\mathrm{RT}^G(n)$ denotes the set of rhombus tilings in~$\mathrm{RT}(n)$ which are invariant under the action of all elements of~$G$).
For subgroups~$G$ which contain rotation by~$\pi$ or reflection through an axis which bisects two
opposite edges of the hexagon, $\mathrm{RT}^G(n)$ is nonempty if and only if $n$ is even, and so $\mathcal{N}(G)$ is taken to be the set of positive even integers.  For the remaining subgroups~$G$ (which can
contain only rotation by $\pm2\pi/3$ or reflection in an axis which passes through two opposite vertices of the hexagon), $\mathrm{RT}^G(n)$ is nonempty for all~$n$, and so $\mathcal{N}(G)$ is taken to be $\mathbb{N}$.
Product formulae for $|\mathrm{RT}^G(n)|$ are known for each of the ten symmetry classes, as discussed, for example, by Krattenthaler~\cite[Sec.~6]{Kra16}.  Using these formulae, and applying the same methods
as those used in Sections~\ref{ASMLeadAsymptSect} and~\ref{ASMLeadAsymptFurthSect1} for cases with known product formulae, it can be shown that, for any subgroup of $D_6$,
\begin{equation}\label{RTLeadAsmpt}\lim_{\substack{n\to\infty\\[0.7mm]n\in\mathcal{N}(G)}}|\mathrm{RT}^G(n)|^{|G|/n^2}=\biggl(\frac{3\sqrt{3}}{4}\biggr)^{\!3}.\end{equation}
(Alternatively, some cases of~\eqref{RTLeadAsmpt} can be obtained from others by applying the same methods as those used in Sections~\ref{ASMLeadAsymptSect} and~\ref{ASMLeadAsymptFurthSect1}
for cases without known product formulae.)  It follows from~\eqref{RTLeadAsmpt} that~\eqref{GenLeadAsymptGroup} holds for the case $Q(n)=\mathrm{RT}(n)$, $\mathcal{G}=D_6$ and $d=2$.

It can be seen that a tiling of a regular hexagon of side length~$n$ contains $3n^2$ unit rhombi,
and that, for a subgroup $G$ of $D_6$, any tiling in $\mathrm{RT}^G(n)$ is uniquely determined by that part of the tiling which lies on a
fundamental region containing $p_G(n)$ rhombi, where $p_G(n)$ is a quadratic polynomial in~$n$ with leading coefficient $3/|G|$.
Hence,~\eqref{RTLeadAsmpt} provides an example of~\eqref{GenLeadAsympt3}, with $Q(n)=\mathrm{RT}^G(n)$, $P_Q(n)=p_G(n)$ and $N_Q=\mathcal{N}(G)$,
and where $L_Q=3\sqrt{3}/4$ is independent of~$G$.  It can also be shown that there are many other well-studied classes~$Q$ of rhombus tilings
for which~\eqref{GenLeadAsympt3} holds with $L_Q=3\sqrt{3}/4$, and with suitable $N_Q(n)$ and $P_Q(n)$, where $P_Q(n)$ is a quadratic polynomial
in~$n$ which gives the number of rhombi in a fundamental domain.

\section*{Acknowledgements}
REB was supported by Austrian Science Fund (FWF) SFB grant F50, and Leverhulme Trust grant RPG-2019-083.
IF was supported by Austrian Science Fund (FWF) START grant Y463, SFB grant F50, grant P34931 and SFB grant 10.55776/F1002.
CK was supported by Austrian Science Fund (FWF) SFB grant F50 and grant I6130-N.
We thank the anonymous reviewers for helpful comments on the paper,
Mihai Ciucu for useful discussions regarding some of the content of Sections~\ref{ASMLeadAsymptSect}--\ref{ASMLeadAsymptFurthSect2},
and the authors of~\cite{GarDegMeaWhe25,Mea25} and~\cite{Kum26} for helpful discussions regarding the work in those papers.

\section*{Note on References}
In the reference list below, for papers which have both published and arXiv versions, bibliographic data
is provided for each version.
However, for citations of such papers within this paper, details such as theorem, equation or page numbers refer to those in the published version,
which may differ from those in the arXiv version.

\let\oldurl\url
\makeatletter
\renewcommand*\url{%
        \begingroup
        \let\do\@makeother
        \dospecials
        \catcode`{1
        \catcode`}2
        \catcode`\ 10
        \url@aux
}
\newcommand*\url@aux[1]{%
        \setbox0\hbox{\oldurl{#1}}%
        \ifdim\wd0>\linewidth
                \strut
                \\
                \vbox{%
                        \hsize=\linewidth
                        \kern-\lineskip
                        \raggedright
                        \strut\oldurl{#1}%
                }%
        \else
                \hskip0pt plus\linewidth
                \penalty0
                \box0
        \fi
        \endgroup
}
\makeatother
\gdef\MRshorten#1 #2MRend{#1}
\gdef\MRfirsttwo#1#2{\if#1M
MR\else MR#1#2\fi}
\def\MRfix#1{\MRshorten\MRfirsttwo#1 MRend}
\renewcommand\MR[1]{\relax\ifhmode\unskip\spacefactor3000 \space\fi
  \MRhref{\MRfix{#1}}{{\tiny \MRfix{#1}}}}
\renewcommand{\MRhref}[2]{
 \href{https://mathscinet.ams.org/mathscinet-getitem?mr=#1}{#2}}
\bibliography{Bibliography}
\bibliographystyle{amsplainhyper}
\end{document}